\documentclass{amsart}
\usepackage{paper}
\usepackage{mathdots}

\definecolor{refnotecolor}{rgb}{0.80,0.33,0.0}

\title[Satellites and Telescopes]{Satellites and Telescopes: A Concordance formula for Bordered Floer Homology}

\author[Gary Guth]{Gary Guth}
\address{Department of Mathematics\\Stanford University\\
		Building 380\\
		Stanford, California 94305}
\email{gmguth@stanford.edu}

\author{Sungkyung Kang}
\address{Department of Pure Mathematics and Mathematical Statistics \\ University of Cambridge \\ Cambridge CB3 0WB}
\email{sungkyung.kang@dpmms.cam.ac.uk}

\begin{document}

\begin{abstract}
Given a knot $K$ in $S^3$, its knot Floer completely determines the bordered Floer homology of its complement by a work of Lipshitz--Ozsv\'{a}th--Thurston. Furthermore, the determination is combinatorial: given a model for $\CFK(S^3, K)$ there is a method for producing an explicit model for $\CFDh(\KC)$. In this paper, we show that a similar formula holds between certain classes of chain endomorphisms of knot Floer chain complex and type D endomorphisms of bordered Floer homology, up to a 1-dimensional ambiguity in the type D side; both the formula and the ambiguity can be computed combinatorially. It follows that, for any concordance from a knot to itself and any satellite pattern, we can combinatorially compute the knot Floer cobordism map of the satellite concordance (up to conjugation) from the knot Floer cobordism map of the given concordance and the bordered Floer homology of the pattern complement. 
\end{abstract}

\maketitle
\setcounter{tocdepth}{1}
\tableofcontents

\section{Introduction}\label{sec:intro}
Given a knot $K\subset S^3$, its exterior $S^3 \smallsetminus \nu(K)$ is a central object of study in knot theory. Indeed, according to \cite{mca1989knots}, the oriented homeomorphism type of $S^3 \smallsetminus \nu(K)$ determines the knot type of $K$. In terms of Heegaard Floer homology and its variants, the knot $K$ gives rise to its \emph{knot Floer chain complex} $\CFK(S^3,K)$, which is an object in the derived category $\mathcal{D}^{perf}(\mathbb{F}_2[U,V])$ of bigraded perfect dg modules over the bigraded ring $\mathbb{F}_2[U,V]$, where
\[
\deg(U) = (-2,0),\quad \deg(V) = (0,-2),\quad \deg(\partial)=(-1,-1)
\]
in terms of its symmetric bigrading, or equivalently $\deg(U)=(-2,1)$, $\deg(V)=(0,-1)$, $\deg(\partial)=(-1,0)$ in terms of its $(\text{Maslov},\text{Alexander})$ bigrading. On the other hand, for any framing $n$ on $K$ (and thus also on $\nu(K)$), we have the \emph{bordered Floer homology} $\CFDh(S^3 \smallsetminus K,n)$ (we will usually suppress the framing unless necessary) of $S^3 \smallsetminus \nu(K)$, which is a type D module over the torus algebra
\[
\cA:=\mathcal{A}(T^2)=\mathrm{Span}_{\mathbb{F}_2}\{\iota_0,\iota_1,\rho_1,\rho_2,\rho_3,\rho_{12},\rho_{23},\rho_{123}\}
\]
which is graded by a nonabelian group (which depends on $n$) whose center is $\mathbb{Z}$, generated by the degree of the differential. According to the work of Lipshitz, Ozsv\'{a}th, and Thurston \cite[Theorem 11.36]{LOT_bordered_HF}, for any sufficiently large negative integer $n$, the type D module $\CFDh(S^3\smallsetminus K,n)$ can be computed combinatorially from the $\mathcal{R}$-coefficient knot Floer complex 
\[
\CFK_\mathcal{R}(S^3,K) = \CFK(S^3,K)\otimes^L_{\mathbb{F}_2[U,V]} \mathcal{R},\quad \mathcal{R}:=\mathbb{F}_2[U,V]/(UV).
\]
By tensoring with Dehn twist bimodules, it follows that $\CFDh(S^3\smallsetminus K,n)$ can be combinatorially computed from $\CFK_\mathcal{R}(S^3,K)$ for any framing $n$. Conversely, given $\CFDh(S^3\smallsetminus K,n)$, one can recover $\CFK_\cR(S^3, K)$ by way of the pairing theorem. See also \cite{mnemonic_LOT,hanselman2023bordered}. This dictionary has been used to solve many interesting problems in low dimensional topology, including but not limited to \cite{levine_doubling_operators,OzStSz_concordanceHoms,guth2023doubled}.

There is an analogous story a dimension higher. Suppose we are given a self-concordance $C \sub S^3 \times [0,1]$ from $K$ to itself (we restrict to self-concordances here purely for the sake of exposition). By the functoriality of knot Floer homology, after choosing a basepoint $p\in K$ and a simple arc $a\subset C$ whose boundary is $\{p\}\times \partial I$, the concordance $C$ induces a homotopy class of chain endomorphisms (i.e. its $\mathcal{R}$-coefficient knot Floer cobordism map)
\[
F_{C,a}^\mathcal{R}:\CFK_\mathcal{R}(S^3,K)\rightarrow \CFK_\mathcal{R}(S^3,K),
\]
which depends only on the smooth isotopy class of $(C,a)$. On the other hand, it follows from \cite[Theorem 2]{guth_one_not_enough_exotic_surfaces} and the naturality of bordered Floer homology (shown in \cite{guthkang2024invariantsplittingprinciples}) that the bordered cobordism $(S^3 \times I)\smallsetminus \nu(C)$ induces a homotopy class of a type D endomorphism
\[
F_{S^3 \times I \smallsetminus \nu(C),a}^D:\CFDh(S^3 \smallsetminus K)\rightarrow \CFDh(S^3 \smallsetminus K),
\]
which depends only on the oriented diffeomorphism type of $(S^3 \times I\smallsetminus \nu(C),a)$\footnote{In \cite{guth_one_not_enough_exotic_surfaces}, the author did not prove the invariance of this map. This is established in \Cref{sec:Doubling}.}. It is thus a natural question to ask whether one can combinatorially compute $F_{S^3 \times I\smallsetminus \nu(C),a}^D$ from $F_{C,a}^\mathcal{R}$.

As a first step towards answering this question, we show in \Cref{subsec: LOT for morphisms} that the formula of Lipshitz--Ozsv\'{a}th--Thurston for computing $\CFDh(S^3\smallsetminus K,n)$ from $\CFK_\cR(S^3, K)$ (in its basis-free form) is functorial. In \cite{LOT_bordered_HF}, Lipshitz-Ozsv\'ath-Thurston show that a homotopy equivalence of $\CFK_\cR(S^3,K)$ induces an equivalence of the basis free model for $\CFDh(\KC)$. We show that this constrution can be used to define a (homotopy) ring map
\begin{align*}
    \Omega:\mathrm{End}^{ls}_\mathcal{R}(\CFK_\cR(S^3,K)) \xrightarrow{\simeq} \mathrm{End}^\cA(\CFDh(S^3 \smallsetminus K))
\end{align*}
from space of \emph{locally symmetric} chain endomorphisms of $\CFK_\mathcal{R}(S^3,K)$ (whose definition is given in \Cref{def:locally symmetric}). Of course, since there are infinitely many homotopy classes of chain endomorphisms of $\CFK_\mathcal{R}(S^3,K)$ and only finitely many homotopy classes of type D endomorphisms of $\CFDh(S^3 \smallsetminus K)$, it is necessary to restrict to a special class of endomorphisms if we are to formulate any reasonable correspondence between the two. Indeed, in \Cref{thm:main thm}, we prove that $\Omega$ gives a one-to-one correspondence between homotopy classes of locally symmetric chain endomorphisms of $\CFK_\mathcal{R}(S^3,K)$ and homotopy classes of type D endomorphisms of $\CFDh(S^3\smallsetminus K)$ modulo the ``theta-minus class'' (defined in \Cref{sec: the map Lambda} and discussed in \Cref{sec: theta class}). In previous work \cite{guthkang2024invariantsplittingprinciples}, the authors defined a map $\Lambda$ in the opposite direction, associating $\cR$-complex endomorphisms to type D endomorphisms. \Cref{thm:main thm} is established by showing that $\Omega$ is a right inverse for $\Lambda$. Furthermore, we prove in \Cref{prop: omega and concordance maps} that $\Omega(F^\mathcal{R}_{C,a}) = F^D_{S^3 \times I\smallsetminus \nu(C),a}$\footnote{This is an equality up to a reparametrization of $\CFDh(S^3 \smallsetminus K)$ that does not depend on the choice of $(C,a)$.} and thus $F^D_{S^3 \times I\smallsetminus \nu(C),a}$ can be computed combinatorially from $F^\mathcal{R}_{C,a}$. Our primary work is summed up in the following theorem.

\begin{introthm} \label{thm: main1}
    Consider the $\mathbb{F}_2$-algebras $\mathrm{End}^{ls}_\mathcal{R}(\CFK_\cR(S^3,K))$ of homotopy classes of locally symmetric chain endomorphisms of $\CFK_\mathcal{R}(S^3,K)$ and $\mathrm{End}^\cA(\CFDh(S^3 \smallsetminus K))$ of homotopy classes of type D endomorphisms of $\CFDh(S^3\smallsetminus K)$. There exists a combinatorially computable, central element $[\theta^-_K]\in \mathrm{End}^\cA(\CFDh(S^3 \smallsetminus K))$ satisfying $[\theta^-_K]^2=0$ so that the combinatorially defined ring map
    \[
    \Omega:\mathrm{End}^{ls}_\mathcal{R}(\CFK_\cR(S^3,K)) \xrightarrow{\simeq} \mathrm{End}^\cA(\CFDh(S^3 \smallsetminus K))/\langle [\theta^-_K]\rangle
    \]
    is a homotopy equivalence.

    Furthermore, for any concordance $C$ of a pointed knot $(K,p)$ to itself and a simple arc $a\subset C$ satisfying $\partial a =\{p\}\times \partial I$, we have $\Omega([F^\mathcal{R}_{C,a}]) = [F^D_{S^3 \times I\smallsetminus \nu(C),a}]$. In particular, the homotopy class of $F^D_{S^3 \times I\smallsetminus \nu(C),a}$ can be combinatorially computed from the homotopy class of $F^\mathcal{R}_{C,a}$.
\end{introthm}

As a topological application, we show in \Cref{thm:satellite-concordances} that for any satellite pattern $P$, the $\mathcal{R}$-coefficient knot Floer cobordism map $F^\mathcal{R}_{P(C),a}$ for the satellite concordance $P(C)$ (of $P(K)$ to itself) can be combinatorially computed from the homotopy class of $F^\mathcal{R}_{C,a}$ and the homotopy equivalence class of the type DA bimodule ${\CFDAh}((S^1 \times D^2) \smallsetminus P)$. We describe this in terms of the following theorem, follows from \Cref{thm: main1}, together with \Cref{prop: Lambda and concordance maps} and \Cref{prop: omega and concordance maps}.

\begin{introthm}\label{thm: main2}
    Let $K$ be a knot, $(C, a)$ a decorated concordance from $K$ to itself, and $P$ a pattern. Suppose that we are given the following algebraic data:
    \begin{itemize}
        \item A finite-dimensional operationally bounded\footnote{See \cite[Lemma 2.3.11]{LOT_bimodules} for a definition.} model of the type DA bimodule ${\CFDAh}((S^1 \times D^2)\smallsetminus P)$;
        \item Finitely generated free models of the chain complex $\CFK_\mathcal{R}(S^3,K)$;
        \item A representative of the homotopy class of the cobordism map $F_{C,a}:\CFK_\mathcal{R}(S^3,K)\rightarrow \CFK_\mathcal{R}(S^3,K)$.
    \end{itemize}
    Then one can combinatorially compute the homotopy class of the satellite cobordism map
    \[
    F_{P(C,a)}:\CFK_\mathcal{R}(S^3,P(K))\rightarrow \CFK_\mathcal{R}(S^3,P(K)) 
    \]
    up to conjugation.
\end{introthm}

It is straightforward, using results from \cite{guthkang2024invariantsplittingprinciples}, to generalize \Cref{thm: main2} to concordance maps between any pairs of knots. However, there is a caveat: in that setting, there are no canonical identifications between knot Floer chain complexes of incoming and outgoing boundaries of such concordances, and so we of course cannot say that $F_{P(C)}$ is determined up to conjugation (the domain and codomain are now different). Instead, it is determined only up to precompositions and postcompositions of homotopy autoequivalences of its domain and codomain.

One other possible straightforward generalization is about concordances $\Sigma$ from knots to themselves that are not necessarily in $S^3 \times I$, but rather in any twice-punctured closed smooth $\mathrm{Spin}^c$ 4-manifold $(W,\mathfrak{s})$ satisfying $c_1(\mathfrak{s}) = PD[\Sigma]$. We expect the proof in this generalized setting is essentially the same, though we do not pursue this.

\begin{rem}
    The connection between knot Floer homology and bordered Floer homology has been explored in many contexts, and it is natural to seek to understand the relation of the present work to the existing literature \cite{LOT_bordered_HF,zemke2023bordered,hanselman2023bordered,mnemonic_LOT}. In particular, in concluding this project, we were referred to the work of Hockenhull \cite{hockenhull}. He constructs an $\cA_\infty$-module $\mathrm{Poly}(L, \Lambda)$ associated to a framed link $(L, \Lambda)$ in $S^3$. In a future work, we plan to explore the connection between his work and our own.
\end{rem}

\subsection*{Note on AI use} During the revision process, we used Claude Fable 5 to look for typographical and mathematical errors, which we then fixed.

\subsection*{Organization} In Sections \ref{sec:tel_hyp} and \ref{sec: attaching_curves} we set up the algebraic machinery needed for the proof of \Cref{thm:main thm}: hypercubes and telescopes. Very roughly, we consider infinite families of attaching curves, representing all large surgeries on $K$. This data can be organized to give a model for $\CFK_\cR(S^3, K)$ in terms of $\CFh(S^3_n(K))$ for infinitely many $n$. In \Cref{sec:main_hypercube}, we study a particular hypercube of attaching curves which gives rise to the $\F[U,V]$-action on our model, and in \Cref{sec: Telescopes and large surgeries} we prove that this model actually computes $\CFK_\cR(S^3, K)$. In Sections \ref{sec: the LOT Correspondence} and \ref{sec: the LOT hypercube}, we recall the (basis free) $\CFK_\cR$-to-$\CFDh$ formula for objects and morphisms, and show that their construction can be generalized to locally symmetric, homogeneous endomorphisms. We then use the bordered framework to compute the composition $\Lambda \Omega$ in Sections \ref{sec: End actions} and \ref{sec: the map Lambda}. In \Cref{sec: theta class}, we study the class $\theta_K^-$ which exactly captures the discrepancy between $\mathrm{End}^{ls}_\mathcal{R}(\CFK_\cR(S^3,K))$ and  $\mathrm{End}_h(\CFDh(S^3 \smallsetminus K))$. Finally, we put the pieces together in \Cref{sec:Doubling} and prove that the $\CFK_\cR$-to-$\CFDh$ formula for morphisms can be used to combinatorially compute the maps induced by concordance complements from the maps induced by concordances.

\subsection*{Acknowledgments} The authors would like to thank Jesse Cohen, Robert Lipshitz, and Ian Zemke for helpful discussions. The authors are also grateful to the Max Planck Institute for Mathematics in Bonn for their hospitality while this work was in progress. GG is partially supported by the Simons Collaboration Grant on New Structures in Low-Dimensional Topology; SK is partially supported by the Royal Society University Research Fellowship URF\textbackslash R1\textbackslash 251501.

\section{Hypercubes and telescopes}\label{sec:tel_hyp}
In this section, we provide a brief review of the algebraic objects which are most relevant to the work of this paper: hypercubes and mapping telescopes. For more details and different presentations, see \cite{HHSZ_surgery_exact_invol, manolescu2024heegaardfloerhomologyinteger,Varolgunes_MV_rel_symp_cohomology}.

\subsection{Hyperboxes and Hypercubes}

In this paper, we will follow the conventions and terminology of \cite[Section 5]{HHSZ_surgery_exact_invol}, though the notions there are adapted from those in \cite{manolescu2024heegaardfloerhomologyinteger}. Let $\E^n = \{0,1\}^n$ be the unit cube in $\R^n$. If $\bm d = (d_1, \hdots, d_n) \in \Z^n_{\ge 0}$, define $\E(\bm d)$ to be 
\begin{align*}
    {\E}(\bm d) = \prod_{i=1}^n\{0, \hdots, d_i\}
\end{align*}
{the} cube of size $\bm d$. We write $||\bm d|| = \sum_i d_i$.

\begin{defn}
    An $n$-dimensional \emph{pre-hyperbox}, $H$, of chain complexes of size $\bm d$ consists of a collection of $\F$-vector spaces $(C^\epsilon)_{\epsilon \in \E(\bm d)}$, together with maps 
    \begin{align*}
        D^{\epsilon}_{\epsilon_0}: C^{\epsilon_0} \ra C^{\epsilon_0 + \epsilon},
    \end{align*}
    for $(\epsilon_0, \epsilon) \in \E(\bm d)\times \E^n$ so that $\epsilon_0 + \epsilon \in \E(\bm d)$. We say that the pre-hyperbox $H$ is a \emph{hyperbox} if these maps satisfy the structure relation
    \begin{align}\label{eqn:hyper_cube_structure_equation}
        \sum_{\substack{\epsilon' \in \E^n,\\ \epsilon' \le \epsilon}} D^{\epsilon - \epsilon'}_{\epsilon_0 + \epsilon'}\circ D^{\epsilon'}_{\epsilon_0} = 0,
    \end{align}
    for all $\epsilon_0$ and $\epsilon + \epsilon_0$ in $\E(\bm d)$.
\end{defn}

The total space of a hyperbox $H = (C^\epsilon, D^\epsilon_{\epsilon_0})$ is the complex 
\begin{align*}
    \left( \bigoplus_\epsilon C^\epsilon, \sum_{\epsilon, \epsilon_0} D^\epsilon_{\epsilon_0} \right).
\end{align*}

Given hyperboxes $H_1$ of size $(d_1, \hdots, d_{n-1}, d_n)$ and $H_2$ of size $(d_1, \hdots, d_{n-1}, d_n')$ with the property that the restrictions of $H_1$ and $H_2$ to the hyperfaces $\E(\bm d)\times \E^1(d_n)$ and $\E(\bm d)\times \E^1(d_n')$ coincide, we can form a new hyperbox, $\St(H_1, H_2)$, a hyperbox of size $(d_1, \hdots, d_{n-1}, d_n+d_n')$ by gluing $H_1$ and $H_2$ along their common face. We refer to this operation as \emph{stacking}.

Another natural operation on hyperboxes is \emph{compression}. Given a pre-hyperbox $H = (C_\epsilon, D_{\epsilon_0}^\epsilon)$ of size $\ell = (d_1,\cdots,d_n)$, its compression is a pre-hypercube $(\widetilde{C}_\epsilon,\widetilde{D}^{\epsilon}_{\epsilon_0})$ defined as follows:

\[
\begin{split}
\widetilde{C}_{\epsilon} &= C_{\ell\times\epsilon}, \\
\widetilde{D}^0_{\epsilon} &= D^0_{\ell\times\epsilon}, \\ 
\widetilde{D}^\epsilon_{\epsilon_0} &= \sum_{\substack{ \epsilon_1+\cdots+\epsilon_k = \ell\times \epsilon \\ k = ||\ell\times\epsilon||-||\epsilon||+1 \\ \epsilon_1,\cdots,\epsilon_k\in \mathbb{E}^n \text{ and all nonzero} }} D^{\epsilon_k}_{\ell\times\epsilon_0 + \epsilon_1 + \cdots + \epsilon_{k-1}}\cdots D^{\epsilon_2}_{\ell\times\epsilon_0 + \epsilon_1} D^{\epsilon_1}_{\ell\times \epsilon_0},
\end{split}
\]
where, for any pair of integral vectors $v,w\in \mathbb{Z}^n$, we define $v\times w = (v_1 w_1,\cdots,v_n w_n)$. It is straightforward to see that compressing a hyperbox gives a hypercube. See \cite[Section 5.2]{HHSZ_surgery_exact_invol} for examples of compressing hyperboxes of size $(1, 1, d)$.

\begin{defn}
    An \emph{$n$-dimensional (pre-)hypercube} of chain complexes is a (pre-)hyperbox of size $\bm d = (1, 1, \hdots, 1).$
\end{defn}

A \emph{partially defined} $n$-dimensional (pre-)hypercube is a collection of $\F$-vector spaces $C_\epsilon$ for $\epsilon \in \E^n$ together with maps $D^\epsilon_{\epsilon_0}$ for some subset $S \sub \{(\epsilon, \epsilon_0) \subset \E^n \times \E^n: \epsilon + \epsilon_0 \in \E^n\}$ with the property that \Cref{eqn:hyper_cube_structure_equation} is satisfied whenever the maps involved are defined. A \emph{pre-filling} of a partially defined $n$-dimensional hypercube is an extension of $S$ which defines a pre-hypercube; we say that a pre-filling is a \emph{filling} if the extension is a hypercube.

\begin{defn}
    A \emph{morphism} (or \emph{map}) of $n$-dimensional hypercubes $(C_\epsilon,D_{\epsilon_0}^\epsilon)$ and $(C'_{\eta},D_{\eta_0}^\eta)$ is a pre-filling of the partially defined $(n+1)$-dimensional hypercube 
    \begin{align*}
        (C_{\epsilon\times \{0\}} \cup C'_{\eta \times \{1\}}, D_{\epsilon_0\times\{0\}}^{\epsilon\times\{0\}}\cup D_{\eta_0\times\{1\}}^{\eta\times\{1\}}).
    \end{align*}
    We say that a map is \emph{regular} if it is a filling.
\end{defn}

Given maps of $n$-dimensional hypercubes $H_1\xra{f} H_2$ and $H_2 \xra{g} H_3$, we define the composition $H_1 \xra{g\circ f} H_3$ to be the compression of the hyperbox $$\St(H_1\xra{f} H_2, H_2 \xra{g} H_3).$$ It is easy to see that composition is associative. 

For a fixed $n$, $n$-dimensional hypercubes form an additive category, $\frak C_n$, with biproduct direct sum. In particular, any hypercube, $\cC$, has an identity map.
\begin{defn}[Identity maps of hypercubes]
    Let $\cC = (C_\epsilon,D^\epsilon_{\epsilon_0})$ be an $n$-dimensional hypercube. The \emph{identity map} $\mathrm{id}_\cC:\cC\rightarrow \cC$ is defined to be the $(n+1)$-dimensional hypercube $(\widetilde{C}_\epsilon,\widetilde{D}^\epsilon_{\epsilon_0})$ where we define
    \[
    \begin{split}
        \widetilde{C}_{(\epsilon,0)} = \widetilde{C}_{(\epsilon,1)} &= C_\epsilon, \\
        \widetilde{D}^{(\epsilon,0)}_{(\epsilon_0,0)} = {\widetilde{D}^{(\epsilon,0)}_{(\epsilon_0,1)}} &= D^\epsilon_{\epsilon_0}, \\
        {\widetilde{D}^{(0,1)}_{(\epsilon,0)}} &= \mathrm{id}_{C_\epsilon}, \\
        {\widetilde{D}^{(\epsilon^\prime,1)}_{(\epsilon,0)}} &= 0 \text{ whenever } \epsilon^\prime \neq 0 \text{ and }\epsilon^\prime+\epsilon \in \mathbb{E}^n,
    \end{split}
    \]
    for any $\epsilon,\epsilon_0\in\mathbb{E}^n$ satisfying $\epsilon+\epsilon_0\in\mathbb{E}^n$. It is straightforward to check that all hypercube structure equations are satisfied. Furthermore, for any maps $f:\cC\rightarrow \cD$ and $g:\cD\rightarrow \cC$, we have $f\circ \mathrm{id}_\cC = f$ and $\mathrm{id}_\cC \circ g = g$.
\end{defn}
The cube $\cC$ comes with an additional canonical endomorphism, namely, a differential.
\begin{defn}[Hypercube differentials]
    Let $\cC = (C_\epsilon,D^\epsilon_{\epsilon_0})$ be an $n$-dimensional hypercube. The \emph{differential map} $\partial_\cC:\cC\rightarrow \cC$ is defined to be the $(n+1)$-dimensional hypercube $\widetilde{\cC}=(\widetilde{C}_\epsilon,\widetilde{D}^\epsilon_{\epsilon_0})$, given by
    \[
    \begin{split}
        \widetilde{C}_{(\epsilon,0)} = \widetilde{C}_{(\epsilon,1)} &= C_\epsilon, \\
        \widetilde{D}^{(\epsilon,0)}_{(\epsilon_0,0)} = \widetilde{D}^{(\epsilon,0)}_{(\epsilon_0,1)} &= D^\epsilon_{\epsilon_0}, \\
        \widetilde{D}^{(\epsilon,1)}_{(\epsilon_0,0)} &= D^\epsilon_{\epsilon_0},
    \end{split}
    \]
    for any $\epsilon,\epsilon_0\in\mathbb{E}^n$ satisfying $\epsilon+\epsilon_0\in\mathbb{E}^n$. It is straightforward to check that the hypercube structure equations for $\cC$ imply the structure equations for $\widetilde{\cC}$, and therefore $\partial_\cC$ is indeed an endomorphism of $\cC$. It is straightforward to check that $\partial_\cC^2=0$; this follows directly from the hypercube structure equations for $\cC$.
\end{defn}

It follows that hypercubes can essentially be treated as chain complexes. In particular, just as morphism spaces of chain complexes inherit the structure of a complex, with differential $$\partial_{Mor} (f:C \ra D) = \partial_D f + f \partial_C,$$ morphism spaces of hypercubes also inherit a differential. Similarly, the cycles are precisely the regular maps. 

\begin{lem}\label{lem:regular_maps_are_chain_maps}
    A map $f:\cC\rightarrow \cD$ between hypercubes $\cC,\cD$ is regular if and only if $f \circ \partial_\cC = \partial_\cD \circ f$.
\end{lem}
\begin{proof}
    This follows directly from the hypercube structure equations for $f$.
\end{proof}

We interpret \Cref{lem:regular_maps_are_chain_maps} to mean that a hypercube is nothing but a zero-dimensional hypercube (i.e.\ a chain complex) in the category of hypercubes. This is an instance of the following general phenomenon.
\begin{lem}
    There exists a canonically defined fully faithful additive functor $F_{m,n}$ from the category of $(n+m)$-dimensional hypercubes (in any additive category) to the category of $n$-dimensional hypercubes of $m$-dimensional hypercubes (in the same additive category).
\end{lem}
\begin{proof}
    The case $m=0$ is just the definition of hypercubes, and we have just proved the case $n=0$. Hence we may assume that $n,m>0$; we will first prove it when $n=1$.

    To that end, we construct a functor $F_{n,1}$ from $(n+1)$-dimensional hypercubes to 1-dimensional hypercubes of $n$-dimensional hypercubes. Given an $(n+1)$-dimensional hypercube $\cC=(C_\epsilon,D^\epsilon_{\epsilon_0})$, restricting to its faces corresponding to the subsets
    \[
    \mathbb{E}^n\times \{0\},\mathbb{E}^n\times \{1\}\subset \mathbb{E}^{n+1}
    \]
    gives two $n$-dimensional hypercubes, which we denote by $\cC_0$ and $\cC_1$. The hypercube $\cC$ itself then defines a regular map
    \[
    f_\cC:\cC_0\rightarrow \cC_1.
    \]
    By \Cref{lem:regular_maps_are_chain_maps}, we know that $\partial_{\cC_1}f = f\partial_{\cC_0}$. Therefore, the 1-dimensional pre-hypercube $F(\cC)$ of $n$-dimensional hypercubes, which we define as the diagram below, satisfies the structure equations and thus is a 1-dimensional hypercube of $n$-dimensional hypercubes.
    \[
    \xymatrix{
    \cC_0 \ar@(ul,ur){}^{\partial_{\cC_0}}\ar[r]^{f_\cC} & \cC_1\ar@(ul,ur){}^{\partial_{\cC_1}}
    }
    \]
    For morphisms, recall that a map $f$ between $(n+1)$-dimensional hypercubes is by definition a 1-dimensional pre-hypercube of $(n+1)$-dimensional hypercubes, which is then a 1-dimensional pre-hypercube of 1-dimensional hypercube of $n$-dimensional hypercubes by the above argument. Therefore, we can regard it as a map $F_{n,1}(f)$ between 1-dimensional hypercubes of $n$-dimensional hypercubes. This completes the construction of the functor $F_{n,1}$; it is clear from definitions that it induces bijections between morphism sets and thus is fully faithful.

    In the general case, the fully faithful functor $F_{m,n}$ can be inductively constructed by considering $m$-dimensional hypercubes as 1-dimensional hypercubes of $(m-1)$-dimensional hypercubes.
\end{proof}

\begin{rem}
    In fact, the image of $F_{m,n}$ consists precisely of $n$-dimensional hypercubes $\cC = (C_\epsilon,D^\epsilon_{\epsilon_0})$ of $m$-dimensional hypercubes $C_\epsilon$ such that $D^0_{\epsilon} = \partial_{C_\epsilon}$ for every $\epsilon\in\mathbb{E}^n$. Therefore, we will routinely draw $(n+m)$-dimensional hypercubes as $n$-dimensional hypercubes whose vertices are $m$-dimensional hypercubes. The $D^0$ maps will usually not be specified, though they should be understood as automatically filled by differential endomorphisms of $m$-dimensional hypercubes.
\end{rem}

Finally, we note that there is a notion of homotopy in this category. Given maps $H_1 \xra{f} H_2$ and $H_1 \xra{g} H_2$ of $n$-dimensional hypercubes, a \emph{homotopy} from $f$ to $g$ is a filling:
\begin{align*}
    \begin{tikzcd}[ampersand replacement = \&]
        H_1 \ar[r, "f+g"]\ar[d,"\id"]\ar[dr,"h"] \& H_2 \ar[d,"\id"]\\
        H_1 \ar[r, "0"] \& H_2.
    \end{tikzcd}
\end{align*}

\begin{lem}\label{lem: homotopies are chain homotopies}
    Let $\cC,\cD$ be hypercubes. Then two maps $f,g:\cC\rightarrow \cD$ are homotopic if and only if there exists a map $H:\cC\rightarrow \cD$ satisfying $f+g = \partial H + H \partial$.
\end{lem}
\begin{proof}
    We assume without loss of generality that $g=0$. If $f$ is homotopic to 0, then by definition, we have an $(n+2)$-dimensional hypercube, drawn as a 2-dimensional hypercube of $n$-dimensional hypercubes as follows:
    \[
    \xymatrix{
    \cC \ar[r]^f\ar[d]_{\mathrm{id}_\cC} \ar@2[rd]^{H} & \cD\ar[d]^{\mathrm{id}_\cD} \\
    \cC \ar[r]_0 & \cD.
    }
    \]
    The structure equation for this 2-dimensional hypercube then reads $f = \partial H + H \partial$. The argument also works backwards, so the lemma holds in both directions.
\end{proof}

This allows us to define the notion of homotopy equivalences of hypercubes. 
\begin{defn}
A map $f:\cC\rightarrow \cD$ between hypercubes $\cC$ and $\cD$ is a \emph{homotopy equivalence} if there exists another map $g:\cD\rightarrow \cC$ such that $f\circ g$ is homotopic to $\mathrm{id}_{\cD}$ and $g\circ f$ is homotopic to $\mathrm{id}_{\cC}$. We also say that $\cC$ and $\cD$ are \emph{homotopy equivalent}, and also that $g$ is the \emph{homotopy inverse} of $f$.
\end{defn}

We also consider hyperboxes of graded chain complexes. Let $f: C \ra C'$ be a map between $\Z$-graded complexes. Write $f = \sum_{s \in \Z} f[s]$, where $f[s]$ has degree $s$. Explicitly, $f[s]$ is the sum over all compositions 
\begin{align*}
    C_i \xra{f|_{C_i}} C' \ra C'_{i+s}.
\end{align*}

Later, we will typically work with $n$-dimensional hypercubes $\mathcal{C}=(C_\epsilon,D^\epsilon_{\epsilon_0})$, where each $C_\epsilon$ is graded by some abelian group $G$, and for each $\epsilon$ and $\epsilon_0$ with $\epsilon>0$, the maps $D^\epsilon_{\epsilon_0}$ are maps of mixed degrees, while the length zero maps $D^0_\epsilon$ are maps of constant degree, say $\lambda\in G$. Given a tuple
\[
(g_1,\cdots,g_n)\in G^{n},
\]
we can replace the structure maps $D^\epsilon_{\epsilon_0}$ as follows. We first expand them as sums of their constant-degree-shift pieces:
\[
D^\epsilon_{\epsilon_0} = \sum_{g\in G} (D^\epsilon_{\epsilon_0})_g.
\]
Then we replace $D^\epsilon_{\epsilon_0}$ by $(D^\epsilon_{\epsilon_0})_{(1-||\epsilon||)\lambda + \epsilon \cdot (g_1,\cdots,g_n)} $. In general, these maps need not satisfy the structure equations for hypercubes. However, if $G$ is totally ordered and $(1-||\epsilon||)\lambda + \epsilon \cdot (g_1,\cdots,g_n)$ is the minimal (or maximal) grading in which the associated graded component of $D^\epsilon_{\epsilon_0}$ is nonzero for every $\epsilon$ and $\epsilon_0$, the structure equations are satisfied, and thus the associated graded piece of the given hypercube $\mathcal{C}$ is well-defined.

\subsection{Rays and Mapping Telescopes}

We will often consider infinite sequences of hypercubes, and adopt some of the tools and terminology of \cite{Varolgunes_MV_rel_symp_cohomology}.

\begin{defn}
    Let $\cD_0, \cD_1, \hdots$ be a sequence of $n$-dimensional hypercubes with the property that each $\cD_i$ can be glued to $\cD_{i+1}$, i.e. $\cD_i$ and $\cD_{i+1}$ share an identical $(n-1)$-dimensional face. This data constitutes a \emph{positive $n$-ray}. Equivalently, we can think of a positive $n$-ray as a sequence of maps between $(n-1)$-cubes, $\cD_i = (\cC_i \xra{f_i} \cC_{i+1})$, which we will often express as
    \begin{align*}
        \cC_0 \ra \cC_1 \ra \cC_2 \ra \hdots.
    \end{align*} 
    We call $\cC_i$ the \emph{i}th slice of the ray. 

    Similarly, we call a sequence $\hdots, \cD_{-1},\cD_0$, with the property that each $\cD_{-i}$ can be glued to $\cD_{-(i+1)}$ a \emph{negative $n$-ray}. We will present such a ray as 
    \begin{align*}
        \hdots \ra \cC_{-2} \ra \cC_{-1} \ra \cC_0,
    \end{align*}
    in terms of maps between slices.\\
\end{defn}
\begin{defn}
    A \emph{morphism} (or a \emph{map}) of negative $n$-rays $\cC = \{\cC_i, f_i\}$ and $\cC' = \{\cC_i', f_i'\}$ is a collection of $n$-dimensional pre-hypercubes, $\cE = \{\cE_i\}$, from $(\cC_i \ra \cC_{i+1})$ to $(\cC_{i+k}' \ra \cC_{i+k+1}')$, with the property that $\cE_i$ can be glued to $\cE_{i+1}$.
     \[
    \xymatrix{
    \cdots \ar[r] & \cC_i \ar[r]\ar[drrr]\ar@{}[drrrr]^(.45){}="a"^(.55){}="b" \ar@2 "a";"b" & \cC_{i+1} \ar[r]\ar[drrr] & \cdots \ar[r] & \cC_{i+k}\ar[r] & \cC_{i+k+1}\ar[r] & \cdots \\
    \cdots \ar[r]& \cC_i' \ar[r] & \cC_{i+1}' \ar[r] & \cdots \ar[r] & \cC_{i+k}'\ar[r] & \cC_{i+k+1}'\ar[r] & \cdots
    }
    \]
    Such a map is \emph{regular} if it consists of hypercubes. A \emph{homotopy} between (positive or negative) $n$-ray morphisms $\cE$ and $\cE'$ is a collection of homotopies between $\cE_i+\cE_i'$ and 0.
\end{defn}

\begin{rem}
    As in the case of hypercubes, $(n+m)$-rays can be regarded as $n$-rays of $m$-rays. Similarly, it makes sense to talk about rays of hypercubes, and thus maps of negative $n$-rays are the same as negative $n$-rays of 1-dimensional hypercubes, and homotopies between them are negative $n$-rays of certain 2-dimensional hypercubes.
\end{rem}

\begin{defn}
    Let $\cC$ be a positive $1$-ray, 
    \begin{align*}
        \cC_0 \xra{f_0} \cC_1 \xra{f_1} \cC_2 \xra{f_2} \hdots.
    \end{align*}
    Define the \emph{telescope} $\Tel(\cC)$ to be the 0-ray (i.e. chain complex) 
    \begin{align}
        \Tel(\cC) = \Cone\left(\bigoplus_{i\ge 0} \cC_i \xra{\id + \oplus_k f_k} \bigoplus_{i\ge 0} \cC_i \right). 
    \end{align}
    Concretely, the differential on this complex is given by 
    \begin{align*}
        \begin{tikzcd}[ampersand replacement = \&]
            C_0 
            \ar[loop above,"\partial"]
            \&
            C_1
            \ar[loop above,"\partial"]
            \&
            C_2 
            \ar[loop above,"\partial"]
            \&
            { }\hdots
            \\
            C_0 
            \ar[u,"\id"]
            \ar[ur,"f_0"]
            \ar[loop below,"\partial"]
            \&
            C_1
            \ar[u,"\id"]
            \ar[ur,"f_1"]
            \ar[loop below,"\partial"]
            \&
            C_2 
            \ar[u,"\id"]
            \ar[ur,"f_2"]
            \ar[loop below,"\partial"]
            \&
            { }\hdots
        \end{tikzcd}
    \end{align*}
    Telescopes of negative $1$-rays are defined similarly, and their differential is shown graphically below.
    \begin{align}\label{eqn:negative telescope diff}
        \begin{tikzcd}[ampersand replacement = \&]
            \hdots \& 
            C_{-3} 
            \ar[loop above,"\partial"]
            \&
            C_{-2}
            \ar[loop above,"\partial"]
            \&
            C_{-1} 
            \ar[loop above,"\partial"]
            \&
            {C_0}
            \\
            \hdots 
            \ar[ur,"f_{3}"]
            \& 
            C_{-3} 
            \ar[u,"\id"]
            \ar[ur,"f_2"]
            \ar[loop below,"\partial"]
            \&
            C_{-2}
            \ar[u,"\id"]
            \ar[ur,"f_1"]
            \ar[loop below,"\partial"]
            \&
            C_{-1} 
            \ar[u,"\id"]
            \ar[ur,"f_0"]
            \ar[loop below,"\partial"]
            \&
            { }
        \end{tikzcd}
    \end{align}
\end{defn}

As the name suggests, telescoping is a collapsing operation, taking a 1-ray to a zero ray. In higher dimensions, we can do this iteratively. For instance, consider a 2-ray, which we think of as a 1-ray of 1-rays:
\begin{align*}
    (\cC_{00} \ra \cC_{01} \ra \hdots) \ra (\cC_{10} \ra \cC_{11} \ra \hdots) \ra (\cC_{20} \ra \cC_{21} \ra \hdots) \ra \hdots.
\end{align*}
We can either collapse the inner 1-rays, to obtain the 1-ray of 0-rays:
\begin{align*}
    \Tel(\cC_{0*}) \ra \Tel(\cC_{1*}) \ra \Tel(\cC_{2*}) \ra \hdots
\end{align*}
or we can collapse the outer 1-ray, to obtain a 0-ray of 1-rays
\begin{align*}
    \Tel(\cC_{*0} \ra \cC_{*1} \ra \hdots).
\end{align*}
This leads to the notion of a partial telescope.

\begin{defn}
    Given an $(n+1)$-ray $\cC$ and choice of coordinate $i \in \{1, \hdots, n+1\}$ we define the \emph{partial telescope of $\cC$ in the direction of $i$} to be the $n$-ray 
    \begin{align*}
        \Tel(C^i_0 \ra C^i_1 \ra C^i_2 \ra \hdots),
    \end{align*}
    where $(C^i_0 \ra C^i_1 \ra C^i_2 \ra \hdots)$ is the decomposition of $\cC$ into a 1-ray of $n$-rays determined by $i$. 

    More generally, given an $(n+m)$-ray $\mathcal{C}$ and a length $m$ subset $S$ of $\{1,\cdots,n+m\}$, we define the \emph{partial telescope of $\cC$ with respect to $S$}, which is denoted $\Tel_S(\cC)$, to be the $n$-ray obtained by taking the partial telescope of $\cC$ in the direction of $i$ for all $i$ in $S$.
\end{defn}
\begin{rem}
It is straightforward to see that the definition of partial telescopes is independent of the order of choices of coordinates in $S$ that we used in the process of taking iterated telescopes. 
\end{rem}

Given a map $f:\cC\rightarrow \cD$ between $(n+m)$-rays (positive or negative) $\cC$ and $\cD$ and any choice of coordinates $S\subset \{1,\cdots,n+m\}$, it follows from an iterated application of \cite[Lemma 2.2.2]{Varolgunes_MV_rel_symp_cohomology} that there is an induced map
\[
\mathrm{Tel}_S(f):\mathrm{Tel}_S(\cC)\rightarrow \mathrm{Tel}_S(\cD)
\]
between the partial telescopes with respect to coordinates in $S$. In particular, when $m=0$, we get a chain map
\[
\mathrm{Tel}(f):\mathrm{Tel}(\cC)\rightarrow \mathrm{Tel}(\cD).
\]
A similar argument can also be used to show that, if $f,g$ are homotopic, then the induced maps $\mathrm{Tel}_S(f)$ and $\mathrm{Tel}_S(g)$ are also homotopic. 

Note that even though we allowed maps between positive rays to have nonzero shifts, taking its telescope still induces a well-defined map. In particular, if we have a map $f:\cC\rightarrow \cD$ between positive 1-rays $\cC$ and $\cD$ that involves chain maps from $C_n$ to $D_{n+k}$ for each $n$, we still get a well-defined chain map (or maps of 0-rays)
\[
\mathrm{Tel}(f):\mathrm{Tel}(\cC)\rightarrow \mathrm{Tel}(\cD).
\]

\begin{rem}
    Let $\cC$ be an $n$-ray of $m$-dimensional hypercubes. Then $\mathrm{Tel}(\cC)$ is a 0-ray of $m$-dimensional hypercubes. In other words, it is an $m$-dimensional hypercube $\mathrm{Tel}(\cC)$ together with a regular map
    \[
    d:\mathrm{Tel}(\cC)\rightarrow \mathrm{Tel}(\cC)
    \]
    satisfying $d^2=0$. Write $\mathrm{Tel}(\cC)=(C_\epsilon,D^\epsilon_{\epsilon_0})$. Then the map $d$ consists of maps
    \[
    d^\epsilon_{\epsilon_0}:C_{\epsilon_0} \rightarrow C_{\epsilon_0+\epsilon};
    \]
    the condition $d^2=0$ then ensures that the new $m$-dimensional pre-hypercube $(C_\epsilon,d^\epsilon_{\epsilon_0})$ satisfies the hypercube structure equations and thus is a hypercube. In other words, we get a new $m$-dimensional hypercube structure on $\mathrm{Tel}(\cC)$; \emph{we define this new structure as the telescope $\mathrm{Tel}(\cC)$ of $\cC$}. Note that this new structure is the unique one which makes the following equation hold:
    \[
    \partial_{\mathrm{Tel}(\cC)} = d.
    \]
    In general, this is how we will regard 0-rays (or 0-dimensional hypercubes) of hypercubes as hypercubes. Similarly, we will routinely regard 0-rays of rays as rays.
\end{rem}

We now state a version of the homological perturbation lemma in the context of hypercubes. Given an $n$-dimensional hypercube $\cC=(C_\epsilon,D^\epsilon_{\epsilon_0})$, suppose that for each $\epsilon\in\mathbb{E}^n$, we are given a chain complex $(\widetilde{C}_\epsilon,\partial_\epsilon)$ and chain maps
\[
F_\epsilon:\widetilde{C}_\epsilon\rightarrow (C_\epsilon,D^0_\epsilon),\quad G_\epsilon:(C_\epsilon,D^0_\epsilon) \rightarrow\widetilde{C}_\epsilon,
\]
such that $G_\epsilon F_\epsilon=\mathrm{id}$ and $F_\epsilon G_\epsilon=\mathrm{id}+\partial_\epsilon H_\epsilon+H_\epsilon \partial_\epsilon$ for some maps $H_\epsilon$.  Given these, we can consider two partially defined $(n+1)$-dimensional hypercubes, $\cC^F=(C^F_\epsilon,(D^F)^\epsilon_{\epsilon_0})$ and $\cC^G=(C^G_\epsilon,(D^G)^\epsilon_{\epsilon_0})$, defined as follows.
\begin{itemize}
    \item For any $\epsilon,\epsilon_0\in\mathbb{E}^n$ with $\epsilon+\epsilon_0\in\mathbb{E}^n$, we define $C^G_{(\epsilon,0)}=C^F_{(\epsilon,1)}=C_\epsilon$ and $(D^G)^{(\epsilon,0)}_{(\epsilon_0,0)}=(D^F)^{(\epsilon,0)}_{(\epsilon_0,1)}=D^\epsilon_{\epsilon_0}$.
    \item For each $\epsilon\in\mathbb{E}^n$, we define $C^F_{(\epsilon,0)} = C^G_{(\epsilon,1)} = \widetilde{C}_\epsilon$, $(D^F)^0_{(\epsilon,0)} = (D^G)^0_{(\epsilon,1)} = \partial_\epsilon$.
    \item For each $\epsilon\in{\mathbb{E}^n}$, we define $(D^F)^{(0,\cdots,0,1)}_{(\epsilon,0)} = F_\epsilon$ and $(D^G)^{(0,\cdots,0,1)}_{(\epsilon,0)} = G_\epsilon$.
    \item Everything else is undefined.
\end{itemize}
Then the homological perturbation lemma for hypercubes is given as follows.
\begin{lem}[{\cite[Lemma 2.10]{hendricks2022involutive}}]
    The partially defined $(n+1)$-dimensional hypercubes $\cC^F$ and $\cC^G$ admits a filling, where the map $(D^F)^{({\epsilon},0)}_{(\epsilon_0,0)}$ coincides with the map $(D^G)^{(\epsilon,0)}_{(\epsilon_0,1)}$ for every $\epsilon$ and $\epsilon_0$.
\end{lem}
\begin{proof}
    While its proof was already given in \cite{hendricks2022involutive}, our notation is slightly different, so we present a proof for the sake of clarity and self-containedness. 
    
    We start with explicit formulae for fillings of $\cC^F$ and $\cC^G$. Then we will prove that they satisfy the structure equations. We will only discuss the filling of $\cC^F$ in detail, as the argument for $\cC^G$ is nearly identical. In order to fill $\cC^F$ and $\cC^G$, we define the following.
    \begin{itemize}
        \item For each $\epsilon,\epsilon_0\in\mathbb{E}^n$ with $\epsilon+\epsilon_0\in\mathbb{E}^n$, we define
        \[
        (D^F)^{(\epsilon,0)}_{(\epsilon_0,0)} = (D^G)^{(\epsilon,0)}_{(\epsilon_0,1)} = \sum_{0<\epsilon_1<\cdots<\epsilon_k<\epsilon} G_{\epsilon_0+\epsilon}D^{\epsilon-\epsilon_k}_{\epsilon_0+\epsilon_k}H_{\epsilon_0+\epsilon_k}\cdots H_{\epsilon_0+\epsilon_2}D^{\epsilon_2-\epsilon_1}_{\epsilon_0+\epsilon_1}H_{\epsilon_0+\epsilon_1}D^{\epsilon_1}_{\epsilon_0}F_{\epsilon_0}.
        \]
        \item For the same parameters $\epsilon,\epsilon_0$, we define
        \[
        \begin{split}
        (D^F)^{(\epsilon,1)}_{(\epsilon_0,0)} &= \sum_{0<\epsilon_1<\cdots<\epsilon_k<\epsilon} H_{\epsilon_0+\epsilon}D^{\epsilon-\epsilon_k}_{\epsilon_0+\epsilon_k}H_{\epsilon_0+\epsilon_k}\cdots H_{\epsilon_0+\epsilon_2}D^{\epsilon_2-\epsilon_1}_{\epsilon_0+\epsilon_1}H_{\epsilon_0+\epsilon_1}D^{\epsilon_1}_{\epsilon_0}F_{\epsilon_0}, \\
        (D^G)^{(\epsilon,1)}_{(\epsilon_0,0)} &= \sum_{0<\epsilon_1<\cdots<\epsilon_k<\epsilon} G_{\epsilon_0+\epsilon}D^{\epsilon-\epsilon_k}_{\epsilon_0+\epsilon_k}H_{\epsilon_0+\epsilon_k}\cdots H_{\epsilon_0+\epsilon_2}D^{\epsilon_2-\epsilon_1}_{\epsilon_0+\epsilon_1}H_{\epsilon_0+\epsilon_1}D^{\epsilon_1}_{\epsilon_0}H_{\epsilon_0}.
        \end{split}
        \]
    \end{itemize}
    In order to show that these indeed define a hypercube, we first check that the maps of the form $(D^F)^{(\epsilon,0)}_{(\epsilon_0,0)}$ satisfy the structure equations. We start by expanding the equation except for its two terms involving length 0 maps:
    \[
    \begin{split}
        \sum_{\substack{\epsilon+\epsilon_0+\epsilon^\prime_0 = \epsilon^\prime \\ \epsilon_0,\epsilon^\prime_0\ne 0}} D^{(\epsilon^\prime_0,0)}_{(\epsilon+\epsilon_0,0)} D^{(\epsilon_0,0)}_{(\epsilon,0)}  &= \sum_{\substack{\epsilon+\epsilon_0+\epsilon^\prime_0 = \epsilon^\prime \\ \epsilon_0,\epsilon^\prime_0\ne 0}} \sum_{\substack{ \epsilon_1+\cdots+\epsilon_k = \epsilon_0 \\ \epsilon^\prime_1+\cdots+\epsilon^\prime_k = \epsilon^\prime_0 \\ \epsilon_i,\epsilon^\prime_i\ne 0\text{ for all } i}} GD^{\epsilon^\prime_1}H\cdots HD^{\epsilon^\prime_k} FGD^{\epsilon_1}H\cdots HD^{\epsilon_k} F \\
        &= \sum_{\substack{\epsilon_1 + \cdots + \epsilon_k = \epsilon^\prime - \epsilon \\ \epsilon_i \ne 0 \text{ for all } i}} \sum_{i=0}^{k} GD^{\epsilon_1}H\cdots HD^{\epsilon_i} FG D^{\epsilon_{i+1}} H\cdots HD^{\epsilon_k} F \\
        &= \sum_{\substack{\epsilon_1 + \cdots + \epsilon_k = \epsilon^\prime - \epsilon \\ \epsilon_i \ne 0 \text{ for all } i}} \sum_{i=1}^{k} GD^{\epsilon_1}H\cdots HD^{\epsilon_i} (\mathrm{id}+\partial H + H \partial) D^{\epsilon_{i+1}} H\cdots HD^{\epsilon_k} F \\
        &= \sum_{\substack{\epsilon_1 + \cdots + \epsilon_k = \epsilon^\prime - \epsilon \\ \epsilon_i \ne 0 \text{ for all } i}} \sum_{i=0}^{k} GD^{\epsilon_1}H\cdots H \left( \partial D^{\epsilon_i} + D^{\epsilon_i} \partial \right) H D^{\epsilon_{i+1}} H\cdots HD^{\epsilon_k} F \\
        &\quad \quad + \sum_{\substack{\epsilon_1 + \cdots + \epsilon_k = \epsilon^\prime - \epsilon \\ \epsilon_i \ne 0 \text{ for all } i}} \sum_{i=0}^{k} GD^{\epsilon_1}H\cdots H \left(\sum_{\substack{\epsilon^\prime_1+\epsilon^\prime_2 = \epsilon_i \\ \epsilon^\prime_1,\epsilon^\prime_2\ne 0}} D^{\epsilon^\prime_1}D^{\epsilon^\prime_2} \right) H D^{\epsilon_{i+1}} H\cdots HD^{\epsilon_k} F \\
        &\quad\quad + \sum_{\substack{\epsilon_1 + \cdots + \epsilon_k = \epsilon^\prime - \epsilon \\ \epsilon_i \ne 0 \text{ for all } i}} (G\partial D^{\epsilon_1}H\cdots HD^{\epsilon_k} F  +  G D^{\epsilon_1}H\cdots HD^{\epsilon_k} \partial F) \\
        &= \sum_{\substack{\epsilon_1 + \cdots + \epsilon_k = \epsilon^\prime - \epsilon \\ \epsilon_i \ne 0 \text{ for all } i}} (\partial G D^{\epsilon_1}H\cdots HD^{\epsilon_k} F  +  G D^{\epsilon_1}H\cdots HD^{\epsilon_k} F\partial) \\
        &= D^{0}_{(\epsilon^\prime,0)}D^{(\epsilon^\prime-\epsilon,0)}_{(\epsilon,0)} + D^{(\epsilon^\prime-\epsilon,0)}_{(\epsilon,0)}D^{0}_{(\epsilon,0)}.
    \end{split}
    \]
    Here, we have rearranged the terms and also added twice (recall that we are working with mod 2 coefficients) the terms having $G\partial$ and $\partial F$ for the fourth equality and used the structure equations for $\cC$ for the fifth equality. This proves the structure equations for the maps of the form $(D^F)^{(\epsilon,0)}_{(\epsilon_0,0)}$. The structure equations for the remaining maps can be shown via very similar computations, so they are left to the reader.
\end{proof}

\subsection{Quasimodules}

We now specialize to a particularly useful class of hypercubes.

\begin{defn}
For any integer $n\ge 0$, an \emph{$\F[U_1, \hdots, U_n]$-quasimodule} is an $n$-dimensional hypercube $\cC = (C_\epsilon,D^\epsilon_{\epsilon_0})$ whose parallel facets are identical, i.e. it is a hypercube which satisfies:
\begin{align*}
    C_\epsilon = C_{\epsilon'} &\quad   \forall \epsilon,\epsilon'\in \E^n \\
    D^\epsilon_{\epsilon_0} = D^\epsilon_{\epsilon_0+\epsilon_1} &\quad  \forall \epsilon,\epsilon_0,\epsilon_1\in\mathbb{E}^n \text{ satisfying } \epsilon+\epsilon_0+\epsilon_1\in\mathbb{E}^n,
\end{align*}
\end{defn}
Due to this translation-invariance, we will drop all subscripts from our notation, and simply write $(C,D^\epsilon)$. The reader is encouraged to verify that the compression of a hyperbox of $\F[U_1, \hdots, U_n]$-quasimodules is itself an $\F[U_1, \hdots, U_n]$-quasimodule.

\begin{defn}
    A map of $\F[U_1, \hdots, U_n]$-quasimodules consists of a map between the underlying hypercubes which is translation-invariant. As before, we say a map of quasimodules is \emph{regular} if the filling defining this map is an $(n+1)$-dimensional hypercube. The composition of two quasimodule maps is defined to be their compression. Homotopies of quasi-module maps are translation invariant homotopies of the underlying hypercube maps.
\end{defn}

\begin{defn}\label{def:qmod category}
    For any $n\ge 0$, it follows by the same argumentation as above, that $\F[U_1, \hdots, U_n]$-quasimodules form an additive category which we denote $\QMod_n$. In particular, when $n = 0$, this is nothing but the category of $\F$-chain complexes. We will write $\QMod_n^h$ for the associated homotopy category. 
\end{defn}

The translation-invariance condition of $\F[U_1, \hdots, U_n]$-quasimodules allows us to define their homology.
\begin{defn}
    Let $\cC = (C,D^\epsilon)$ be a quasimodule. Define its \emph{homology} $H_\ast(\cC)$, to be the homology of the chain complex $(C,D^0)$.
\end{defn}
\begin{defn}
    Given a map $f:\cC\rightarrow \cD$ between $\F[U_1, \hdots, U_n]$-quasimodules, $\cC_1=(C_1,{D_1^\epsilon})$ and $\cC_2=(C_2,D_2^\epsilon)$, there is an induced chain map
    \[
    f:(C_1,D^0_1)\rightarrow (C_2,D_2^0),
    \]
    and thus a map $H_\ast(f):H_\ast(\cC_1)\rightarrow H_\ast(\cC_2)$ on homology. We say that $f$ is a \emph{quasi-isomorphism} if $H_\ast(f)$ is an isomorphism.
\end{defn}

As in the case of chain complexes, homotopy equivalence implies quasi-isomorphism.
\begin{lem}
    Let $f:\cC\rightarrow \cD$ be a homotopy equivalence of $\F[U_1, \hdots, U_n]$-quasimodules (of any additive category). Then $f$ is a quasi-isomorphism.
\end{lem}
\begin{proof}
    Denote $\cC_1 = (C_1,D_1^\epsilon)$, $\cC_2 = (C_2,D_2^\epsilon)$, and let $g:\cC_2\rightarrow \cC_1$ be a homotopy inverse of $f$. The hypercube structure equations imply that the induced chain maps
    \[
    f:(C_1,D_1^0)\rightarrow (C_2,D_2^0),\quad g:(C_2,D_2^0)\rightarrow (C_1,D_1^0)
    \]
    are homotopy inverses to each other, which implies that $H_\ast(f)$ and $H_\ast(g)$ are inverses to each other. In particular, $H_\ast(f)$ is an isomorphism. Therefore $f$ is a quasi-isomorphism.
\end{proof}

\subsection{Telescoping}

Quasimodules can be glued together just like hypercubes, and there is also an analogue of rays in this context.

\begin{defn} \label{defn:partial_rays}
    Given a $\F[U_1, \hdots, U_n]$-quasimodule $\cC$ and any set of coordinates $S\subset \{1,\cdots,n\}$ with $|S| = k$, one associates to it the canonical partial negative $k$-ray $\mathrm{Ray}^-_S(\cC)$ of $(n-k)$-dimensional quasimodules by concatenating infinitely many copies of $\cC$, indexed by functions $S\rightarrow \mathbb{Z}_{\le 0}$.
\end{defn}
\begin{defn}
    In the same setting as in \Cref{defn:partial_rays}, the total telescope of $\mathrm{Ray}^-_S(\cC)$ is an $(n-k)$-dimensional quasimodule, which we define as the \emph{partial $S$-telescope} of $\cC$ and denote by $\mathrm{Tel}_S(\cC)$. When $S=\{1,\cdots,n\}$, we denote $\mathrm{Tel}_S(\cC)$ by $\mathrm{Tel}(\cC)$ and call it the \emph{total telescope of $\cC$}.
\end{defn}
The following lemmas indicate that taking partial (or total) telescopes of quasimodules does not introduce new structures.

\begin{lem} \label{lem:ray_contraction_1-dim}
    Let $\cC=(C, f)$ be a 1-dimensional quasimodule (of any additive category with countable coproducts). Then the obvious inclusion map $i:C\rightarrow \mathrm{Tel}(\cC)$ is a homotopy equivalence of $\F$-complexes.
\end{lem}
\begin{proof}
    Recall that $\mathrm{Tel}(\cC)=\bigoplus_{i\le 0} \cC^1_i \oplus \bigoplus_{i\le -1} \cC^2_i$ (here, $\cC^1_i$ and $\cC^2_i$ are both just $\cC_i$; the superscripts will be useful in distinguishing the two summands). Recall that the differential on the telescope is given by 
    \[
    \begin{split}
        \partial &= \partial_1 + \partial_2; \\
        \partial_1 &= \bigoplus_{i\le 0} \partial^1_i \oplus \bigoplus_{i\le -1} \partial^2_i; \\
        \partial_2 &= \bigoplus_{i\le 0}(\mathrm{id}^{2\rightarrow 1}_{i-1} + f^{2\rightarrow 1}_{i}).
    \end{split}
    \]
    Compare to \Cref{eqn:negative telescope diff}. Here, the superscript $2\rightarrow 1$ is to indicate that we consider the maps $\mathrm{id}$ and $f$ as
    \[
    \mathrm{id}^{2\rightarrow 1}_{i-1}:\cC^2_{i-1}\rightarrow \cC^1_{i-1},\quad f^{2\rightarrow 1}_i: \cC^2_{i-1}\rightarrow \cC^1_i.
    \]
    Consider the endomorphisms of $\Tel(\cC)$:
    \begin{enumerate}
        \item Define $\delta$ to be $\id^{1 \ra 2}_{i}$ on $C_{i}^1$ for $i < 0$ and zero otherwise;
        \item Let $\rho^i_j: C^i_{j-1} \ra C^i_{j}$ be the map $f^{1 \ra 1}_{j}$ for $i \in \{1,2\}$ and $j < 0$; additionally, let $\rho^1_1 = \id_{0}^{1\ra 1}$. Define $\rho = \bigoplus_{j,k} \rho^1_j\oplus \rho^2_k$.
    \end{enumerate}
    Note that, following the notation in \Cref{eqn:negative telescope diff}, $\delta$ shifts elements down and $\rho$ shifts elements right and applies the structure map. Let $p:\mathrm{Tel}(\cC)\rightarrow \cC$ be the map whose restriction to $\cC^1_{j-1}$ is given by $(\rho^1_j \circ \hdots \circ\rho^1_1)$ and is trivial on $\cC^2_j$.

    It is clear that $p$ and $i$ are both regular, and $p\circ i=\mathrm{id}_{\cC_0}$. We construct a nullhomotopy of $\id_{\mathrm{Tel}(\cC)}+ i\circ p$ by hand.  We define a homotopy
    \begin{align*}
        H:= \delta(\id_{\Tel(\cC)} + \rho + \rho^2 +\hdots).
    \end{align*}
    As $\delta$ is nilpotent, this is well-defined. We claim that 
    \begin{align*}
        \id_{\Tel(\cC)} + i \circ p = \partial H + H \partial.
    \end{align*}
    Let $x_m \in \cC^1_m$ and $y_n \in \cC^2_n$. First, observe that, when restricted to $C^1_m$, we have that 
    \begin{align*}
        (\partial \delta \rho^k + \delta \rho^k \partial)  = \begin{cases}
            (\rho^k + \rho^{k+1}) & k < |m|\\
            0 & \text{else}
        \end{cases}
    \end{align*}
    and when restricted to $C^2_n$
    \begin{align*}
        (\partial \delta \rho^k + \delta \rho^k \partial)  = (\rho^k + \rho^{k+1}) \quad \forall k.
    \end{align*}
    It follows that on $C^1_m$
    \begin{align*}
        (\partial H + H \partial)|_{\cC^1_m} &= ((\id_{\Tel(\cC)} + \rho) + (\rho + \rho^2 )+ \hdots + (\rho^{{|m|}-1} +\rho^{{|m|}}))\\
        &=\id_{C^1_m} + (f^{1\ra 1})^{{|m|}}\\
        &= (\id_{\Tel(\cC)} + i \circ p)|_{\cC^1_m}
    \end{align*}
    and 
    \begin{align*}
        (\partial H + H \partial)|_{\cC^2_n} &= ((\id_{\Tel(\cC)} + \rho) + (\rho + \rho^2 )+ \hdots + (\rho^{{|n|}-1} +\rho^{{|n|}}) + (\rho^{{|n|}} +0))\\
        &=\id_{C^2_n}  + 0\\
        &= (\id_{\Tel(\cC)} + i \circ p)|_{\cC^2_n}.
    \end{align*}
    This proves the claim. Therefore, by \Cref{lem: homotopies are chain homotopies}, we conclude that $i\circ p$ is homotopic to $\mathrm{id}_{\Tel(\cC)}$. Therefore, $i$ is a homotopy equivalence.
\end{proof}

\begin{cor} \label{cor: inclusion_qis}
    Let $\cC=(C,D^\epsilon)$ be an $n$-dimensional quasimodule. Then the canonical inclusion of $C$ into $\mathrm{Tel}(\cC)$ is a homotopy equivalence.
\end{cor}
\begin{proof}
    We consider $\cC$ as {a} 1-ray of 1-{rays} of $\cdots$ of 1-{rays} and apply \Cref{lem:ray_contraction_1-dim} $n$ times.
\end{proof}

However, mapping telescopes of $n$-dimensional quasimodules have more structure than chain complexes of $\mathbb{F}_2$-vector spaces; they come with a natural $\F[U_1, \hdots, U_n]$-action.

\begin{lem}
    Given an $n$-{dimensional} quasimodule $\cC$, its total telescope $\mathrm{Tel}(\cC)$ admits a canonical structure of a chain complex of free $\mathbb{F}_2[U_1,U_2,\cdots,U_n]$-modules. 
\end{lem}
\begin{proof}
    We only prove the lemma for the case $n=1$; the general case can be proved in an almost identical way. So we assume that $n=1$, so that $\cC$ is just a 1-dimensional quasimodule; we draw its associated negative ray as follows.
    \[
    \cC = ( \cdots \rightarrow C \rightarrow C ).
    \]
    Recall that its mapping telescope is modelled on the following space:
    \[
    \bigoplus_{{j}\le 0} C \oplus \bigoplus_{{j}\le -1} C;
    \]
    Now we consider the action of $U$ on this space, defined as the ``negative shift'', i.e. moving $n$th $C$ to $(n-1)$th $C$. The differential of the telescope clearly commutes with this action. Therefore $\mathrm{Tel}(\cC)$ becomes a chain complex of free $\mathbb{F}_2[U]$-modules. 
\end{proof}
Therefore, from now on, for $n$-dimensional quasimodules $\cC$, we will consider $\mathrm{Tel}(\cC)$ as a chain complex of free $\mathbb{F}_2[U_1,\cdots,U_n]$-modules. The following lemma is then almost obvious and thus its proof is omitted.

\begin{lem}
    Let $\cC,\cD$ be $n$-dimensional quasimodules. Then any map $f:\cC\rightarrow \cD$ induces an $\mathbb{F}_2[U_1,\cdots,U_n]$-linear function
    \[
    \mathrm{Tel}(f):\mathrm{Tel}(\cC)\rightarrow \mathrm{Tel}(\cD),
    \]
    which is a chain map if $f$ is regular. If $f$ is a quasi-isomorphism (or a homotopy equivalence), then $\mathrm{Tel}(f)$ is also a quasi-isomorphism (or a homotopy equivalence). If two maps $f,g:\cC\rightarrow \cD$ are homotopic, then $\mathrm{Tel}(f)$ and $\mathrm{Tel}(g)$ are also homotopic. 
\end{lem}
\begin{proof}
    The proof is straightforward and thus left to the reader.
\end{proof}
Summarizing, we have the following.
\begin{lem}
    The telescope construction defines a functor 
    \begin{align*}
        \Tel: \QMod_n \ra \mathrm{Kom}_{\F[U_1, \hdots, U_n]}.
    \end{align*}
\end{lem}
\begin{proof}
    This follows immediately from the previous lemmas.
\end{proof}

We can also endow the homology of an $n$-dimensional quasimodule with a canonical structure of a module over $\mathbb{F}_2[U_1,\cdots,U_n]$ as follows. Write $\cC=(C,D^\epsilon)$. For each $i\in \{1,\cdots,n\}$, we denote the $i$th unit vector by $e_i (\in \mathbb{E}^n)$. Then the chain endomorphisms
\[
D^{e_1},\cdots,D^{e_n}:(C,D^0)\rightarrow (C,D^0)
\]
are pairwise homotopy-commuting due to structure equations, which means that the homology-level maps
\[
(D^{e_1})_\ast,\cdots,(D^{e_n})_\ast:H_\ast(\cC)\rightarrow H_\ast(\cC)
\]
are pairwise commuting. In particular, we obtain an $\mathbb{F}_2[U_1,\cdots,U_n]$-module structure on $H_\ast(\cC)$.

\begin{ex}
    As a simple example to keep in mind, consider the 1-dimensional quasimodule $$\cQ = (\hdots \xra{U} \F[U]\langle x_{-2} \rangle \xra{U} \F[U]\langle x_{-1} \rangle\xra{U} \F[U]\langle x_0 \rangle).$$ Here, we regard $\F[U]\langle x_i \rangle$ as a $\F_2$-complex with trivial differential. Recall that telescope of this quasimodule is generated by pairs $(x_i, y_j)$, with differential $\partial (x_i, y_j) = (y_j + Uy_{j+1}, 0)$. The complex $\Tel(\cQ)$ has the structure of an $\F[U]$-complex, given by $U\cdot(x_i,y_j) = (x_{i-1}, y_{j-1}).$ The homology, $H_*(\Tel(\cQ))$, is generated as a vector space by the classes $\{[x_i,0]\}_{i \le 0}$ and the relations imposed on the homology are precisely $[Ux_{i+1}, 0]=[x_{i}, 0]$. Hence, $H_*(\Tel(\cQ))$ is a free $\F[U]$-module generated by the class $[x_0, 0]$.  Hence, in this case, $H_*(\Tel(\cQ))\cong \F[U]$ as $\F[U]$-modules. 
\end{ex}

\begin{lem} \label{lem:same_homology}
    Let $\cC=(C,D^\epsilon)$ be an $n$-dimensional quasimodule; note that this gives $H_\ast(\cC)$ an $\mathbb{F}_2[U_1,\cdots,U_n]$-module structure and $\mathrm{Tel}(\cC)$ a structure of a chain complex of free $\mathbb{F}_2[U_1,\cdots,U_n]$-modules. Then $H_\ast(\mathrm{Tel}(\cC))$ is canonically isomorphic to $H_\ast(\cC)$ as $\mathbb{F}_2[U_1,\cdots,U_n]$-modules.
\end{lem}

\begin{proof}
    By regarding $\Tel(\cC)$ as a 1-dimensional quasimodule of $(n-1)$-dimensional quasimodules, it suffices to prove the case that $n=1$. Recall that we have an $\mathbb{F}_2$-linear canonical inclusion 
    \[
    i:C\rightarrow \mathrm{Tel}(\cC)
    \]
    that is a quasi-isomorphism. Consider the following {(non-commutative) diagram}
    \begin{align*}
        \begin{tikzcd}[ampersand replacement = \&]
            \cC \ar[r,"i"] \ar[d,"f"] \& \Tel(\cC) \ar[d,"U"]\\
            \cC \ar[r,"i"] \& \Tel(\cC).
        \end{tikzcd}
    \end{align*}
    Indeed, the clockwise composition takes an element $x$ to $U(x_0, 0) = (x_{-1}, 0)$, whereas the counter-clockwise composition takes $x$ to $((f(x))_0, 0)$. Define $H: \cC \ra \Tel(\cC)$ to be the canonical inclusion of $C$ into $\cC^2_{-1}$. It is easy to check that $H$ provides the homotopy between the two compositions in the diagram above. 
    
    Hence, we have the following 2-dimensional hypercube
    \[
    \xymatrix{
    C \ar[r]^{i}\ar[d]_{D^{e_1}} \ar[dr]^{H} & \mathrm{Tel}(\cC)\ar[d]^{U_1} \\
    C\ar[r]_{i} & \mathrm{Tel}(\cC).
    }
    \]
    Thus, after taking homology, we get the following commutative square:
    \[
    \xymatrix{
    H_\ast(C) \ar[r]^{i_\ast}\ar[d]_{U_1} & H_\ast(\mathrm{Tel}(\cC))\ar[d]^{U_1 } \\
    H_\ast(C)\ar[r]_{i_\ast} & H_\ast(\mathrm{Tel}(\cC))
    }
    \]
    Hence the isomorphism $i_\ast:H_\ast(C)\rightarrow H_\ast(\mathrm{Tel}(\cC))$ is $\mathbb{F}_2[U_1,\cdots,U_n]$-linear. Therefore $H_\ast(C)\cong H_\ast(\mathrm{Tel}(\cC))$ as $\mathbb{F}_2[U_1,\cdots,U_n]$-modules.
\end{proof}

Our quasimodules will often come with the {additional} structure of a grading, which we formally define here. Note that we will sometimes {refer to} $n$-quasimodules {as} $\F[U_1,\cdots,U_n]$-quasimodules as well. Also, in the case $n=2$, we {refer to} 2-quasimodules {as} $\F[U,V]$-quasimodules.

\begin{defn}
    A \emph{graded} $\F[U_1, \hdots, U_n]$-\emph{quasimodule} is a quasimodule with a $\Z^n$-grading, $\gr(x) =(\gr_1(x), \hdots, \gr_n(x))$ where the variable $U_i$ shifts $\gr_i$ by $-2$ and preserves all other gradings. 
\end{defn}

\subsection{Hypercubization}

We now consider the reverse question of constructing quasimodules from $\F[U_1, \hdots, U_n]$-complexes.

\begin{defn}
The \emph{hypercubization functor},
\begin{align*}
    \Hyp: \mathrm{Kom}_{\F[U_1, \hdots, U_n]} \ra \QMod_n,
\end{align*}
assigns to an $\F[U_1, \hdots, U_n]$-complex $C = (C, \partial)$ the {$n$}-dimensional quasimodule $\Hyp(C) := (C, D^\ep)$ where 
\begin{align*}
    D^0 = \partial, \quad D^{e_i} = U_i\cdot \id,
\end{align*}
for $e_i$ a unit basis vector and to a morphism $f: C_0 \ra C_1$ the map with length 1-arrows given by
\begin{align}
    \begin{tikzcd}[ampersand replacement =\&]
        C_0 \ar[r,"f"] \ar[d,"U_i"] \& C_1 \ar[d,"U_i"] \\
        C_0 \ar[r,"f"] \& C_1.
    \end{tikzcd}
\end{align}
All higher arrows are taken to be zero.

There is also a notion of \emph{partial hypercubization}: given a {subset} $S \sub \{U_1, \hdots, U_n\}$ the {functor}
\begin{align*}
    \Hyp_S: \mathrm{Kom}_{\F[U_1, \hdots,U_n]} \ra \QMod_{|S|}.
\end{align*}
is induced by the forgetful functor $\F[U_1, \hdots, U_n]\mathrm{-mod} \ra \F[S]\mathrm{-mod}$.
\end{defn}

\begin{lem}\label{lem:telhyp-canonical-projection-QI}
    Let $C$ be a chain complex over $\F[U_1, \hdots, U_n]$. Then, the canonical projection
    \begin{align*}
        p: \Tel\Hyp(C) \ra C
    \end{align*}
    is an $\F[U_1, \hdots, U_n]$-linear quasi-isomorphism.
\end{lem}
\begin{proof}
    As usual, we will prove the case that $n = 1$ and then induct. That this map is a {homotopy equivalence of the underlying $\F$-complexes} follows from \Cref{lem:ray_contraction_1-dim}. If $(x_{-i}, y_{-j})$ is an element of $\Tel\Hyp(C)$, then we can verify:
    \begin{align*}
        U \cdot p(x_{-i}, y_{-j}) = U \cdot U^{i} x = U^{i+1} x\\
        p(U \cdot (x_{-i}, y_{-j}) = p((x_{-(i+1)}, {y_{-(j+1)}})) = U^{i+1}x.
    \end{align*}
    The lemma follows.
\end{proof}
\begin{rem}
    We emphasize that the obvious inclusion of $C$ into $\Tel\Hyp(C)$ is not $\F[U_1, \hdots, U_n]$-linear; see the proof of \Cref{lem:same_homology}.
\end{rem}

\subsubsection*{The Tel-Hyp adjunction}

The functors $\Tel$ and $\Hyp$ are related via the following natural transformations.
\begin{align*}
    \eta: \id \ra \Hyp \circ \Tel \quad \epsilon: \Tel \circ \Hyp \ra \id.
\end{align*}
The transformation $\eta$ is induced by the canonical inclusion map $M \ra \Tel(M)$ constructed in \Cref{lem:ray_contraction_1-dim}; likewise, $\epsilon$ is induced by the canonical projection $\Tel(M) \ra M$. We will construct these transformations explicitly in the proof of \Cref{lem: ep and eta are natural}.

For a general framework which will allow us to apply induction on dimensions, we choose any additive category $\mathcal{A}$ that is enriched over the category $\mathrm{Mod}_{\mathbb{F}_2}$ of $\mathbb{F}_2$-vector spaces and has finite limits, finite colimits, and countable direct sums. 

\begin{defn}\label{def:A qmods}
    Given any integer $n\ge 0$, let $\mathrm{QMod}_n(\mathcal{A})$ be the category of $n$-dimensional quasimodules in $\cA$, i.e. the category consisting of hypercubes whose vertices are objects of $\cA$ and whose edges are morphisms of $\mathcal{A}$ (note that the hypercube structure equations make sense in any additive category). Furthermore, we define the category $\mathcal{A}_{\mathbb{Z}_+^n}$ as follows:
\begin{itemize}
    \item objects of $\mathcal{A}_{\mathbb{Z}_+^n}$ are objects of $\mathcal{A}$ together with an action of the commutative monoid $\mathbb{Z}_+^n$;
    \item morphisms of $\mathcal{A}_{\mathbb{Z}_+^n}$ are $\mathbb{Z}_+^n$-equivariant morphisms in $\mathcal{A}$.
\end{itemize}
\end{defn}

Exactly as in the previous section, we have well defined functors:
\[
\mathrm{Hyp}:\mathcal{A}_{\mathbb{Z}_+^n}\rightleftharpoons \mathrm{QMod}_n(\mathcal{A}):\mathrm{Tel}.
\]
Furthermore, there are natural notions of partial hypercubizations and partial telescopes: For any subset $S\subset \{1,\cdots,n\}$, we have well-defined functors
\[
\begin{split}
    \mathrm{Hyp}_S:\, &\mathcal{A}_{\mathbb{Z}_+^n}\rightarrow \mathrm{QMod}_{|S|}\left(\mathcal{A}_{\mathbb{Z}_+^{n-|S|}}\right), \\
    \mathrm{Tel}_S:\, &\mathrm{QMod}_n(\mathcal{A}) \rightarrow \mathrm{QMod}_{n-|S|} \left(\mathcal{A}_{\mathbb{Z}_+^{|S|}}\right).
\end{split}
\]
The following observations are trivial for any disjoint subsets $A,B\subset\{1,\cdots,n\}$:
\[
\mathrm{Hyp}_A\mathrm{Hyp}_B = \mathrm{Hyp}_{A\sqcup B},\quad \mathrm{Tel}_A\mathrm{Tel}_B=\mathrm{Tel}_{A\sqcup B}.
\]
Furthermore, the natural transformations $\epsilon$ and $\eta$ are still well-defined over any $\mathcal{A}$.

\begin{rem}
    It is straightforward to see that $\mathcal{A}_{\mathbb{Z}_+^n}$ and $\mathrm{QMod}_n(\mathcal{A})$ also have finite limits, finite colimits, and countable direct sums.
\end{rem}

\begin{ex}
    We will often work in the following setting: if $\mathcal{A}=\mathrm{Mod}_{\mathbb{F}_2}$, then we have the following canonical isomorphisms of categories: $\mathcal{A}_{\mathbb{Z}_+^n}\simeq \mathrm{Mod}_{\mathbb{F}_2[U_1,\cdots,U_n]}$ and $\mathrm{QMod}_n(\mathcal{A})\simeq \mathrm{QMod}_n$.
\end{ex}
    
Let $A$ be a (possibly multi-graded) $\F[U_1, \hdots, U_n]$-algebra. Let $\dgMod_A$ be the dg-category of (multi-graded) differential graded $A$-modules; further, let $\dgMod_A^h$ be the associated homotopy category. The functors $\Tel$ and $\Hyp$ are natural and preserve homotopies, and therefore descend to functors between the homotopy categories: 
\[
\mathrm{Hyp}:\mathrm{dgMod}_A^h\rightarrow \mathrm{QMod}^h_n,\quad \mathrm{Tel}:\mathrm{QMod}^h_n\rightarrow \mathrm{dgMod}^h_A.
\]
We will prove this in a much more general setting, and without passing to homotopy categories, as it turns out that the hypercubization-telescope adjunction holds in a strict setting.

\begin{lem}\label{lem: ep and eta are natural}
    There are natural transformations
    \begin{align*}
    \eta: \id \ra \Hyp \circ \Tel, \quad \epsilon: \Tel \circ \Hyp \ra \id.
\end{align*}
    relating $\Tel$ and $\Hyp$.
\end{lem}
    \begin{proof}
    It suffices to consider $\mathrm{QMod}^h_n$ for $n = 1$, since we may view an $n$-dimensional hypercube as a 1-dimensional hypercube of $(n-1)$-dimensional hypercubes. First, we consider 
    \begin{align*}
        \eta: \id \ra \Hyp \circ \Tel.
    \end{align*}
    Fix a 1-dimensional quasimodule $N = (C, f_U)$. $\Hyp(\Tel(N))$ is the 1-dimensional quasimodule
    \begin{align*}
        (\Tel(N), U\cdot \id) = (\hdots \xra{U} \Tel(N) \xrightarrow{U} \Tel(N)),
    \end{align*}
    where $U$ denotes the shifting action. We define $\eta_N: N \ra \Hyp\Tel(N)$ on a slice
    \begin{align*}
        \begin{tikzcd}[ampersand replacement = \&]
            C \ar[r,"i"]\ar[dr,dashed, "H" ]\ar[d,"f_U"] \& \Tel(N) \ar[d,"U"] \\
            C \ar[r,"i"] \& \Tel(N)
        \end{tikzcd}
    \end{align*}
    just as in the proof of \Cref{lem:ray_contraction_1-dim}; $i$ is the canonical inclusion of $C$ into $C^1_0 \subset \Tel(N) := \bigoplus_{i \le 0}C_i^1 \oplus \bigoplus_{i \le -1} C_i^2$ while $H$ is the inclusion into $C_{-1}^2$. As before, we write $(x_i, y_j)\in C_i^1 \oplus C_j^2$ for an element of the telescope. Given a second quasimodule $N' = (D, g_U)$ and a morphism $(\phi, h): N \ra N'$, consider the diagram 
    \begin{align*}
        \begin{tikzcd}[ampersand replacement = \&]
            N \ar[r,"\eta_N"] \ar[d,"(\phi{,}h)"] \& \Hyp \Tel(N) \ar[d,"\Hyp\Tel(\phi{,} h)"] \\
            N' \ar[r,"\eta_{N'}"]  \& \Hyp \Tel(N').
        \end{tikzcd}
    \end{align*}
    It remains to verify that the compositions $\Hyp \Tel(\phi, h) \circ \eta_N$ and ${\eta_{N'}} \circ {(\phi,h)}$ agree. These two compositions are obtained by compressing the following hyperboxes
    \begin{align*}
        \begin{tikzcd}[ampersand replacement = \&]
            N \ar[r,"i_N"]\ar[d,"f_U"]\ar[dr,"H", dashed] \& \Tel(N)\ar[d,"U"] \ar[r,"\Tel(\phi{,}h)"] \& \Tel(N')\ar[d,"U"] \& N \ar[r,"\phi"]\ar[d,"f_U"]\ar[dr,"h",dashed] \& N'\ar[d,"g_U"] \ar[r,"i_{N'}"]\ar[dr,"K"]\& \Tel(N')\ar[d,"U"] \\
            N \ar[r,"i_N"] \& \Tel(N) \ar[r,"\Tel(\phi{,}h)"] \& \Tel(N') \& N \ar[r,"\phi"] \& N' \ar[r,"i_{N'}"]\& \Tel(N').
        \end{tikzcd}
    \end{align*}
    Checking that these compositions are equal amounts to verifying that 
    \begin{align*}
        \Tel(\phi,h) \circ i_N + i_{N'} \circ \phi = 0 \\
        \Tel(\phi, h) \circ H + i_{N'} \circ h + K \circ \phi = 0.
    \end{align*}
    The first equation is clear; we check the second directly: 
    \begin{align*}
        (\Tel(\phi, h) \circ H + i_{N'} \circ h + K \circ \phi)(x) & =  \Tel(\phi, h)(0, x_{-1}) + (h(x)_0, 0) + (0, \phi(x)_{-1})\\
        &= (h(x)_0, \phi(x)_{-1}) + (h(x)_0, 0) + (0, \phi(x)_{-1})\\
        &= 0.
    \end{align*}
    Likewise, we define $\epsilon_M: \Tel \Hyp(M) \ra M$ to be the map $p$ constructed in \Cref{lem:ray_contraction_1-dim}. Given a morphism $\psi: M \ra M'$, the morphism $\Hyp(\psi)$ is given by the (strictly commutative) diagram  
    \begin{align*}
        \begin{tikzcd}[ampersand replacement = \&]
            M \ar[r,"U"]\ar[d,"\psi"] \& M \ar[d,"\psi"] \\
            M' \ar[r,"U"] \& M'
        \end{tikzcd}
    \end{align*}
    and therefore the induced map on the telescopes is given by $(x_n, y_m) \mapsto (\psi(x)_n, \psi(y)_m)$. Hence, we have that
    \begin{align*}
        \epsilon_{M'} \circ \psi(x_n, y_m) & = \epsilon_{M'}(\psi(x)_n, \psi(y)_m)\\
        & = U^m \psi(y)\\
        & = \psi(U^m y )\\
        & = \psi(\epsilon_M(x_n, y_m)).
    \end{align*}
    This concludes the proof.
\end{proof}

\begin{lem}
    Let $C=(C,\partial)$ be a projective chain complex over $\mathbb{F}_2[U_1,\hdots,U_n]$. Then there exists a quasi-isomorphism from $C$ to $\mathrm{Tel}(\Hyp(C))$.
\end{lem}
\begin{proof}
    Consider the chain complex $C_0 = \mathbb{F}_2[U_1,\hdots,U_n]$ with zero differential. Then, by construction, we have
    \[
    \mathrm{Tel}(\Hyp(C)) = C\otimes_{\mathbb{F}_2[U_1,\cdots,U_n]} \mathrm{Tel}(\Hyp(C_0)),
    \]
    so we only have to prove the lemma for the case $C=C_0$. To do this, we use \Cref{lem:same_homology} to see that $H_\ast(\mathrm{Tel}(\Hyp(C_0)))$ is isomorphic to $H_\ast(\Hyp(C_0))$, which is by construction isomorphic to $\mathbb{F}_2[U_1,\cdots,U_n]$, so it is generated by a single homology class. Choose any cycle $c$ of $\mathrm{Tel}(\Hyp(C_0))$ whose homology class generates $H_\ast(\mathrm{Tel}(\Hyp(C_0)))$. Then the induced map
    \[
    C_0 \rightarrow \mathrm{Tel}(\Hyp(C)),\quad 1\mapsto c
    \]
    is the desired $\mathbb{F}_2[U_1,\cdots,U_n]$-linear quasi-isomorphism. 
\end{proof}

These natural transformations behave well with respect to partial hypercubization and telescoping.

\begin{lem}\label{lem: unit counit tensor}
    Let $A$ and $B$ be disjoint subsets of $\{U_1, \hdots, U_n\}$. Then, there is a natural isomorphism $\alpha_{A, B}: \Hyp_A \Tel_A \Hyp_B \Tel_B \ra \Hyp_{A\sqcup B}\circ \Tel_{A\sqcup B}$, so that the following diagram commutes.
    \begin{align*}
        \begin{tikzcd}[ampersand replacement = \&]
            \Hyp_A \Tel_A \Hyp_B \Tel_B \ar[rr,"\alpha_{A, B}"] \& \& \Hyp_{A\sqcup B} \Tel_{A\sqcup B} \\
            \& \id \ar[ul,"\eta_A \otimes \eta_B"]\ar[ur,"\eta_{A\sqcup B}"]  \\
            \Tel_A\Hyp_A \Tel_B\Hyp_B  \ar[rr,"\beta_{A, B}"] \ar[ur,"\epsilon_A \otimes \epsilon_B"] \& \& \Tel_{A\sqcup B} \Hyp_{A\sqcup B} \ar[ul,"\epsilon_{A\sqcup B}"]
        \end{tikzcd}
    \end{align*}
    Here, the symbol $\otimes$ denotes horizontal compositions of natural transformations, i.e. if $f:F\rightarrow G$ and $f':F'\rightarrow G'$ are natural transformations, then $f\otimes f'$ is a natural transformation from $FF'$ to $GG'$.
\end{lem}
\begin{proof}
    From the definition of hypercubizations and (partial) telescopes, it is clear that for any quasimodule $M$, $\mathrm{Hyp}_A\mathrm{Tel}_A\mathrm{Hyp}_B\mathrm{Tel}_B (M)$ is canonically isomorphic to $\mathrm{Hyp}_{A\sqcup B}\mathrm{Tel}_{A\sqcup B}(M)$. By taking $\alpha_{A,B}$ to be this canonical identification, the commutativity of the given diagram becomes obvious.
\end{proof}

\begin{lem}\label{lem: tel-hyp adjunction}
    The functor $\Tel$ is left adjoint to $\Hyp$; i.e. $\Tel \dashv \mathrm{Hyp}$.
\end{lem}
\begin{proof}
    We show that ${\eta}$ and ${\epsilon}$ are the unit and counit of adjunction respectively. Fix an $n$-dimensional quasimodule $M$ and consider the composition.
    \begin{align*}
        \Tel(M) \xra{\Tel \eta} \Tel\Hyp\Tel(M) \xra{\epsilon \Tel} \Tel(M).
    \end{align*}     
    The case of 1-dimensional quasimodules follows exactly as in \Cref{lem:ray_contraction_1-dim}. In the general case, we identify an $n$-dimensional {quasimodule} $M$ as a ray of $(n-1)$-dimensional rays $(N \xleftarrow{f} N \xleftarrow{f}\hdots)$. It is straightforward (though tedious) to verify $\epsilon \Tel \circ \Tel \eta= \id$ on the level of elements, though, there is a conceptually simpler reason as well. The module $\Tel\Hyp\Tel(M)$ can be obtained as the total telescope\footnote{While the 2-dimensional array that we are using here is not exactly a (negative) 2-ray, its total telescope is still well-defined. Since this array is preserved by rightward shifts but not upward shifts, the total telescope only inherits an action of $U$ via right shift, just like 1-dimensional quasimodules.} of the $2$-dimensional array
    \begin{align*}
        \Tel\left(
        \begin{tikzcd}[ampersand replacement = \&]
             \& \& \&  \iddots \\
            {} \& {} \& N \ar[d,"\id"]\& \hdots \ar[l,"f"]\\
            {} \& N\ar[d,"\id"]\&N\ar[d,"\id"] \ar[l,"f"]\& \hdots  \ar[l,"f"]\\
            N \& N \ar[l,"f"] \&N\ar[l,"f"] \& \hdots \ar[l,"f"]
        \end{tikzcd} 
        \right).
    \end{align*}
    The total telescope can be obtained in two ways: one by collapsing this configuration first vertically and then horizontally, and the other by collapsing horizontally and then vertically. In the first case, we obtain 
    \begin{align*}
        \Tel\left(
        \Tel(N) \xleftarrow{f} \Tel\left( \vcenter{\xymatrix{
            N \ar[d]\\
            N
        }}\right)\xleftarrow{f} \Tel\left(\vcenter{ \xymatrix{
            N \ar[d]\\
            N \ar[d]\\
            N
        }}\right)\xleftarrow{f} \hdots
        \right)
    \end{align*}
This telescope has a canonical projection to $\Tel(M) = \Tel(N \xleftarrow{f} N \xleftarrow{f} \hdots)$, which is precisely $\epsilon \Tel$. In the second case, we obtain 
\begin{align*}
    \Tel\left(
    \begin{tikzcd}
            \vdots \ar[d,"U"] \\
            \Tel(N\xleftarrow{f} N\xleftarrow{f} \hdots) \ar[d,"U"] \\
            \Tel(N\xleftarrow{f} N\xleftarrow{f} \hdots) \\
    \end{tikzcd}
    \right).
\end{align*}
From this point of view, we have a canonical inclusion of $\Tel(M) = \Tel(N\xleftarrow{f} N\xleftarrow{f} \hdots)$ into $\Tel\Hyp\Tel(M)$ which is $\Tel \eta$. The composition of these two morphisms is the identity. 

In the other direction, fix a {dg}-module $M$ and consider the composition
\begin{align*}
    \Hyp(M) \xra{\Hyp \eta} \Hyp \Tel \Hyp(M) \xra{\epsilon \Hyp} \Hyp(M).
\end{align*}
This composition has slices given by the compression of the hyperbox 
\begin{align*}
    \begin{tikzcd}[ampersand replacement = \&]
        M \ar[d,"U"] \ar[r,"i"]\ar[rd,"H", dashed] \& \Tel \Hyp(M)\ar[r,"p"]\ar[d,"U"] \&  M \ar[d,"U"]\\
        M \ar[r,"i"] \& \Tel \Hyp(M)\ar[r,"p"] \&  M 
    \end{tikzcd}.
\end{align*}
The compression is clearly the identity, as $p \circ i = \id_M$ and $p \circ H = 0$.

Assume now that $M$ is a dg-module over $\F[U_1, \hdots, U_n]$; let $A = \{U_1\}$ and let $B = \{U_2, \hdots, U_n\}$. Let us assume that the compositions
\begin{align*}
    \Hyp_S(M) \xra{\Hyp_S \eta} \Hyp_S \Tel_S \Hyp_S(M) \xra{\epsilon \Hyp_S} \Hyp_S(M)
\end{align*}
for $S \in \{A, B\}$ are each the identity. It remains to show that the composition
\begin{align*}
    \Hyp_{A}\Hyp_B(M) \xra{\Hyp_{A\cup B} \eta} \Hyp_{A}\Hyp_B \Tel_A\Tel_B \Hyp_{A}\Hyp_B(M) \xra{\epsilon \Hyp_{A\cup B}} \Hyp_{A}\Hyp_B(M)
\end{align*}
is the identity; recall that $\Hyp_A \Hyp_B \cong \Hyp_{A\sqcup B} = \Hyp$ (and similarly for $\Tel)$. Consider the following commutative diagram.
\begin{align*}
    \begin{tikzcd}[ampersand replacement = \&, column sep = -.5 cm, row sep = 1.5cm]
        {} \& \& \Hyp_A  \Hyp_B \Tel_A \Tel_B \Hyp_A \Hyp_B \ar[drr,"\Hyp\epsilon" description]\ar[d,"\alpha" description] \& \& {} \\
        \Hyp_A \Hyp_B \ar[rr,"\eta\Hyp_A\eta\Hyp_B"description]\ar[dr,"\eta\Hyp"description]\ar[urr,"\eta\Hyp"]
        \& \& 
        \Hyp_A \Tel_A \Hyp_A \Hyp_B \Tel_B \Hyp_B\ar[rr,"\Hyp_A\epsilon_A\Hyp_B\epsilon_B"description]\ar[dr,"\Hyp_A\epsilon_A\Hyp_B\Tel_B\Hyp_B" description] 
        \& \& 
        \Hyp_A \Hyp_B\\
        {} \& \Hyp_A \Tel_A \Hyp_A \Hyp_B \ar[ur,"\Hyp_A\Tel_A\Hyp_A\eta_B\Hyp_B" description]\ar[dr,"\Hyp_A\epsilon_A\Hyp_B"description]
        \&  \& \Hyp_A \Hyp_B \Tel_B \Hyp_B\ar[ur,"\Hyp\epsilon_B"description] \& {}\\
        {} \& {} \& \Hyp_A \Hyp_B\ar[ur,"\Hyp_A\eta_B\Hyp_B"description] \& {} \& {}\\
    \end{tikzcd}
\end{align*}
Commutativity of the topmost triangles follows from \Cref{lem: unit counit tensor}; the central rectangle commutes by naturality; commutativity of the remaining two triangles follows from the fact that taking partial telescopes and hypercubes along disjoint subsets of $\{1, \hdots, n\}$ can be done in any order or done simultaneously. Commutativity around the edges of the diagram implies that:
\begin{align*}
    (\epsilon \Hyp) \circ (\Hyp \eta)= (\Hyp_A \Hyp_B \epsilon_B) \circ (\Hyp_A \eta_B \Hyp_B) \circ (\Hyp_A \epsilon_A \Hyp_B) \circ (\eta_A \Hyp_A \Hyp_B).
\end{align*}
But, $(\Hyp_A \Hyp_B \epsilon_B) \circ (\Hyp_A \eta_B \Hyp_B) = \Hyp_A(\id)$ and  $(\Hyp_A \epsilon_A \Hyp_B) \circ (\eta_A \Hyp_A \Hyp_B) = (\id)\Hyp_B$ by assumption. The lemma then follows. 
\end{proof}
\subsection{A concrete description of $(\Tel \Hyp)(C)$} Henceforth, we will specialize to the following situation. Let $C$ be a free, finitely generated $\F[U,V]$-complex. The chain complex underlying $(\Tel \Hyp)(C)$ is
\begin{align*}
    \bigoplus_{j \ge 0}\left( \left(\bigoplus_{i \ge 0} C_{i,j} \oplus C_{i,j} \right)\oplus \left(\bigoplus_{i \ge 0} C_{i,j} \oplus C_{i,j}\right)\right)
\end{align*}
with differential given by 
\begin{align*}
    \partial_\Tel = \begin{bmatrix}
        \partial & 0 & 0 & 0 \\
        \id[1,0] + U[0,0] & \partial & 0 & 0 \\
        \id[0,1] + V[0,0] & 0 & \partial & 0 \\
        0 & \id[0,1] +V[0,0] & \id[1,0]+U[0,0] & \partial \\
    \end{bmatrix},
\end{align*}
where $U[m,s]: C_{i,j} \ra C_{i+m, {j+s}}$ is the map taking $1 \mapsto U$. The maps $V[m, s]$ and $\id[m, s]$ are defined similarly. 

To get a concrete handle on this complex, fix a cell structure $\frak{R}$ for $\R_{\ge 0}\times \R_{\ge 0}$ with vertices $\Z_{\ge 0}\times \Z_{\ge 0}$, vertical edges $\{i\}\times [j,j+1]$, horizontal edges, $[i,i+1]\times \{j\}$, and faces $[i,i+1]\times [j,j+1]$. This complex has the structure of an $\F[U, V]$-{module}, given by translation. In light of this, we think of this space as being generated over $\F[U,V]$ by a single vertex $w$, a horizontal edge $e_h$, a vertical edge $e_v$, and a face $F$.
\begin{align*}
    \frak{R} = \begin{tikzcd}[ampersand replacement = \&, row sep = small, column sep = small]
        \vdots 
        \ar[dd]
        \&
        \&
            \vdots
            \ar[dd]
            \& 
            \&
                \vdots
                \ar[dd]
                \& 
                \&
                    \\
                    \\
        V^2 w 
        \ar[dd, "Ve_v" description]
        \&
        \&
            UV^2 w
            \ar[dd, "UVe_v" description]
            \ar[ll, "V^2e_h" description]
            \&
            \&
                U^2V^2 {w}
                \ar[dd, "U^2Ve_v" description]
                \ar[ll, "UV^2e_h" description]
                \&
                \&
                    \hdots
                    \ar[ll]
                    \\
                    \& VF \& \& UVF
                    \\
        Vw
        \ar[dd, "e_v" description]
        \&
        \&
            UVw
            \ar[dd, "Ue_v" description]
            \ar[ll, "Ve_h" description]
            \&
            \&
                U^2Vw
                \ar[dd, "U^2e_v" description]
                \ar[ll, "UVe_h" description]
                \&
                \&
                    \hdots
                    \ar[ll]
                    \\
                    \& F \& \& UF
                    \\
        w
        \&
        \&
            U w
            \ar[ll, "e_h" description]
            \&
            \&
                U^2 w
                \ar[ll, "Ue_h" description]
                \&
                \&
                    \hdots
                    \ar[ll]
    \end{tikzcd}.
\end{align*}
Given an $\F[U,V]$-complex $C$, $(\Tel \Hyp)(C)$ can be described as the complex generated by the cells of $\frak{R}$ labeled by elements of $C$. We will write elements of $(\Tel \Hyp)(C)$ as linear combinations of terms of the form $r \otimes x| \sigma$, where $\sigma$ is one of $\{w, e_h, e_v, F\}$, $r$ is a monomial in $\F[U,V]$ which we think of as specifying the horizontal and vertical position of $\sigma$, and $x$ is an element of $C$ labeling the cell. In this language, we see that the telescope differential $\tilde{\partial}$ is essentially given by the cellular boundary map on $\frak R$, {defined} 
for $x \in C$:
\begin{align*}
    \tilde{\partial}(1 \otimes x | F) &= 1 \otimes \partial x | F + (V \otimes x + 1 \otimes V x )| e_h + (U \otimes x + 1 \otimes U x) | e_v\\
    \tilde{\partial}(1 \otimes x | e_v) &= 1 \otimes \partial x|e_v + (V \otimes x + 1 \otimes Vx) | w \\
    \tilde{\partial}(1 \otimes x | e_h) &= 1 \otimes \partial x|e_h + (U \otimes x + 1 \otimes U x) | w \\
    \tilde{\partial}(1 \otimes x | w) &= 1 \otimes \partial x |w
\end{align*}

\subsection{The Derived Category}

For the remainder of this section, we will specialize to the category $\QMod_2$, endowed with some additional structure.

Consider again the bigraded commutative ring $\cS = \mathbb{F}_2[U,V]$ where $\deg(U)=(-2,0)$ and $\deg(V)=(0,-2)$. Observe that, given a bigrading, one can consider the \emph{collapsed grading}, which is defined as the half of the sum of the two components of the given bigrading, i.e.
\[
\deg(x)=(a,b)\quad \rightarrow\quad \deg_c(x)=(a+b).
\]
One can consider the category $\mathrm{Mod}_\cS$ of bigraded $\cS$-modules (i.e. bigraded dg $\cS$-modules with zero differential), which is clearly an abelian category. It thus defines the unbounded derived category of $\cS$-modules, i.e.
\[
\mathcal{D}^\infty\mathrm{Mod}_\cS := (\text{quasi-isomorphisms})^{-1}\mathrm{Kom}^h_\cS,
\]
where $\mathrm{Kom}^h_\cS$ denotes the homotopy category of chain complexes of bigraded $\cS$-modules. Note that objects of $\mathcal{D}^\infty\mathrm{Mod}_\cS$ are chain complexes of the form
\[
\mathcal{M} = [\cdots \rightarrow M_{1}\xrightarrow{d_0} M_0 \xrightarrow{d_{-1}} M_{-1} \rightarrow \cdots]
\]
where each $M_i$ is a bigraded $\cS$-module and each $d_i$ is a homogeneous $\cS$-linear map.

On the other hand, we have another derived category, namely
\[
\mathcal{D}^\infty_\cS := (\text{quasi-isomorphisms})^{-1}\mathrm{dgMod}^h_\cS,
\]
where $\mathrm{dgMod}^h_\cS$ denotes the dg-category of bigraded dg $\cS$-modules, where morphisms are defined up to homotopy. Observe that we have the \emph{totalization functor}
\[
\mathrm{Tot}:\mathcal{D}^\infty\mathrm{Mod}_\cS\rightarrow \mathcal{D}^\infty_\cS,
\]
defined as 
\[
\mathrm{Tot}\left(\mathcal{M}\right) = \left( \bigoplus_{i\in\mathbb{Z}} M_i \left[ (i,i+\mathbf{n}_i )\right],\partial_{tot} \right),\quad \mathbf{n}_i = \begin{cases}
    \sum_{j=0}^{i-1} \deg d_j &\text{if}\quad i\ge 0,\\
    \sum_{j=i}^{-1}\deg d_j &\text{if}\quad i<0,
\end{cases}
\quad \partial_{tot}=\sum_{i\in\mathbb{Z}} \partial_i.
\]
However, there seems to not exist any reasonable functor in the opposite direction.

We start by observing the invertibility of certain quasi-isomorphisms in $\mathrm{Kom}^h_\cS$, which is a wide subcategory of $\mathcal{D}^\infty \mathrm{Mod}_\cS$.

\begin{lem} \label{lem: invertibility in complexes}
    Let $C,D$ be projective chain complexes of bigraded $\cS$-modules. Then the inclusion
    \[
    [C,D]_{\mathrm{Kom}^h_\cS}\rightarrow [C,D]_{\mathcal{D}^\infty\mathrm{Mod}_\cS}
    \]
    is bijective.
\end{lem}
\begin{proof}
    It suffices to prove that any quasi-isomorphism between projective chain complexes of bigraded $\cS$-modules admit a homotopy inverse, which is equivalent to saying that any acyclic projective chain complex of bigraded $\cS$-modules is contractible. This follows from the fact that the abelian category $\mathrm{Mod}_\cS$ has finite projective dimension, i.e. any bigraded projective $\cS$-module admits a finite-length projective resolution, which is a simple corollary of {the} Hilbert syzygy theorem.
\end{proof}

Our goal is to prove the following technical lemma.

\begin{lem} \label{lem:invertibility for fg free}
    Let $M,N$ be bigraded dg $\cS$-modules, where $M$ is finitely generated and free. Then the inclusion
    \[
    [M,N]_{\mathrm{dgMod}^h_\cS}\rightarrow [M,N]_{\mathcal{D}^\infty_\cS}
    \]
    is bijective.
\end{lem}

We start by showing that every finitely generated bigraded free dg $\cS$-module is an iterated mapping cone of rank 1 free $\cS$-modules. 
\begin{lem}\label{lem: fg free implies perfect}
    Let $M$ be a bigraded finitely generated free dg $\cS$-module. Then there exist bigraded finitely generated free dg $\cS$-modules $N$ and $L$, and a homogeneous dg $\cS$-module morphism $f:N\rightarrow L$, such that:
    \begin{itemize}
        \item $M$ is isomorphic to the mapping cone of $f$;
        \item $L$ has rank $1$ (and thus $\mathrm{rank }(N)=\mathrm{rank }(M)-1$).
    \end{itemize}
\end{lem}
\begin{proof}
    It suffices to find a cycle in $M$ that is contained in some basis of $M$. We will assume, without any loss of generality, that $M$ is \emph{reduced}, i.e. the induced differential of $M\otimes_\cS \cS/(U,V)$ is zero. Choose a free homogeneous basis $B$ of $M$ as well as an element $x\in B$ whose total degree is maximal among elements of $B$. 
    
    If $x$ is a cycle we are done. If not, since $U$, $V$, and $\partial$ have collapsed degree $-2$, we deduce that
    \[
    \partial x = Uy+Vz,\quad y,z\in \mathrm{Span}_{\mathbb{F}_2}(B)\text{ not both zero}.
    \]
    Since $\deg_c(y)$ and $\deg_c(z)$ are also maximal among all elements of $B$ (and thus among all nonzero elements of $M$), we may also write
    \[
    \partial y=Ua+Vb,\quad \partial z=Uc+Vd,\quad a,b,c,d\in \mathrm{Span}_{\mathbb{F}_2}(B).
    \]
    Since we have
    \[
        0 = \partial^2 x = U\cdot \partial y+V\cdot \partial z = U^2 a + UV(b+c) + V^2 d,
    \]
    we see that $a=d=0$ and $b=c$, i.e. we may instead write
    \[
    \partial y = Vw,\quad \partial z = Uw,\quad w=b=c.
    \]
    We divide into several cases.
    \begin{itemize}
    \item If $y=0$ (and therefore, $z\neq 0$ by assumption) and $w=0$, then $z$ is a nonzero cycle that is an $\mathbb{F}_2$-linear combination of elements of $B$. Since we may modify $B$ to another basis that contains $z$, we are done.
    \item If $y\neq 0$ and $w=0$, then $y$ is a nonzero cycle that is an $\mathbb{F}_2$-linear combination of elements of $B$, so we are also done.
    \item If $y\neq 0$ and $w\neq 0$, then $w$ is a nonzero cycle that is an $\mathbb{F}_2$-linear combination of elements of $B$, so we are also done.
    \end{itemize}
    The lemma is thus proven.
\end{proof}

In the previous sections, we considered the functors $\Tel$ and $\Hyp$; these descend to functors on the homotopy categories, which we will again call $\Tel$ and $\Hyp$.

\begin{lem} \label{lem: counit is invertible for perfects}
    For any finitely generated free bigraded dg $\cS$-module $M$, the quasi-isomorphism
    \[
    \epsilon(M):\mathrm{Tel}(\mathrm{Hyp}(M))\rightarrow M
    \]
    is a homotopy equivalence.
\end{lem}
\begin{proof}
    It is straightforward to check that both $\mathrm{Tel}$ and $\mathrm{Hyp}$ preserve iterated mapping cones. Hence it follows from \Cref{lem: fg free implies perfect} that we only have to show the lemma in the case $M=\cS$. In this case, we may write
    \[
    \mathrm{Tel}(\mathrm{Hyp}(\cS)) = \left((\cS x\oplus \cS e_h \oplus \cS e_v \oplus \cS F)\otimes_{\mathbb{F}_2} \cS,\partial\right),
    \]
    where $\cS$ acts by left multiplication and $\partial$ is given as follows:
    \[
    \begin{split}
    \partial(rx\otimes s) &= 0,\\
    \partial(re_h \otimes s) &= Urx\otimes s+rx\otimes {Us},\\
    \partial(re_v\otimes {s}) &= Vrx\otimes s+rx\otimes Vs,\\
    \partial(rF\otimes s) &= Ure_v \otimes s + re_v \otimes Us + Vre_h \otimes s + re_h \otimes Vs.
    \end{split}
    \]
    Consider the following submodules:
    \[
    C_1 = \cS x\otimes \cS,\quad C_2 = \cS e_h\otimes \cS \oplus \cS e_v\otimes \cS,\quad C_3 = \cS F\otimes \cS\subset \mathrm{Tel}(\mathrm{Hyp}(\cS)).
    \]
    Each of these submodules is a subcomplex (i.e., are closed under $\partial$), and the induced differentials are all identically zero. In fact, we may write
    \[
    \mathrm{Tel}(\mathrm{Hyp}(\cS)) = \mathrm{Tot}\, \mathcal{C},\quad \mathcal{C}=\left[ 0 \rightarrow C_3 \rightarrow C_2 \rightarrow C_1 \rightarrow 0 \right],
    \]
    for some homogeneous $\cS$-linear maps $C_3 \rightarrow C_2$ and $C_2 \rightarrow C_1$. Furthermore, by considering the map
    \[
    {C_1} = \cS x\otimes \cS\rightarrow \cS, \quad rx\otimes s\mapsto rs,
    \]
    which induces a quasi-isomorphism $T:\mathcal{C}\rightarrow \cS[0]$, it is straightforward to see that
    \[
    \mathrm{Tot}\,T = \epsilon(\cS).
    \]
    Observe that $T$ admits a homotopy inverse (say $T^{-1}$) by \Cref{lem: invertibility in complexes}. Then $\mathrm{Tot}\,T^{-1}$ is the homotopy inverse of $\epsilon(\cS)$. Therefore $\epsilon(\cS)$ is a homotopy equivalence, as desired.
\end{proof}

We also need invertibility of quasi-isomorphisms between bigraded $\mathbb{F}_2[U,V]$-quasimodules.
\begin{lem} \label{lem: invertibility of hypercube qis}
    Every homogeneous quasi-isomorphism between  bigraded $\mathbb{F}_2[U,V]$-quasimodules admits a homotopy inverse.
\end{lem}
\begin{proof}
    Let $\cM = (M, f)$ and $\cN = (N,g)$ be $\mathbb{F}_2[U,V]$-quasimodules and let $\Phi: \cM \ra \cN$ be a quasi-isomorphism. The length $1$ components of $\Phi$ are quasi-isomorphisms $\phi: M \ra {N}$. Since $M$ is a $\F$-chain complex, $\phi$ is a homotopy equivalence, with homotopy inverse $\psi$. Let us fix homotopies $A$ and $B$ satisfying 
    \begin{align}\label{eqn:A and B commutators}
        \psi \phi + \id_M = [\partial, A] \\
        \phi \psi + \id_N = [\partial, B].
    \end{align}
    The morphisms $f$, $g$, and $\psi$ define a pre-hypercube which we claim can be filled to produce a morphism of hypercubes $\Psi: \cN \ra \cM$ which is a homotopy inverse to $\Phi$. We do so iteratively. Consider a face 
    \begin{align}\label{eqn:face composition}
        \begin{tikzcd}[ampersand replacement = \&]
            M \ar[d,"f"] \ar[r,"\phi"] \ar[dr,"h"] \& N \ar[d,"g"] \ar[r,"\psi"] \ar[dr,"k",dashed] \& M \ar[d,"f"] \\
            M \ar[r,"\phi"] \& N \ar[r,"\psi"]\& M.
        \end{tikzcd}
    \end{align}
    Here, $h$ is a homotopy satisfying 
    \begin{align}\label{eqn:h commutator}
        \phi f + g\phi = [\partial, h].
    \end{align} 
    We first must define a homotopy $k: N \ra M$ which satisfies $f\psi + \psi g  = [\partial, k]$. There are many such choices: simplest choice is to take
    \begin{align*}
        k= A f\psi + \psi h \psi + \psi g B.
    \end{align*}
    Though, in order to satisfy the higher hypercube relations, it is better to define
    \begin{align*}
        k:= A f\psi + \psi h \psi + \psi g B + f(\psi B + A \psi).
    \end{align*}
    It is straightforward to check that 
    \begin{align}\label{eqn:k commutator}
        [\partial, k] = f\psi + \psi g.
    \end{align}
    A nullhomotopy of the sum of the compression of \eqref{eqn:face composition} and the identity hypercube is a diagram
    \begin{align*}
    \begin{tikzcd}[column sep={3.2cm,between origins},row sep={1.5cm,between origins},labels=description,ampersand replacement = \&]
        M
        	\ar[dd, swap,"f"]
        	\ar[dr,  "\id"]
        	\ar[rr, " \psi\phi+\id"]
        	\ar[ddrr,dashed,near end,"\psi h+k \phi "]
        	\ar[dddrrr,dotted,sloped,"P"]
        \&\&[-.8cm]
        M
        	\ar[dd, "f"]
        	\ar[dr,"\id"]
        \&
        \\
        \&M
        \&\&
        M
        	\ar[dd, "f"]
        	\ar[from=ulll,dashed,crossing over, "A"]
        	\ar[from=ll,crossing over, "0"]
        \\[2cm]
        M
        	\ar[rr, "\psi \phi + \id"]
        	\ar[dr,"\id"]
        	\ar[drrr,dashed, "A"]
        \&\&M
        	\ar[dr, " \id"]	
        \&
        \\
        \&
        M
        	\ar[rr, "0"]
        	\ar[from =uu, crossing over,"f"]
        	\&\&
        M
    \end{tikzcd}.
    \end{align*}
    where the map $P$ satisfies the relation
    \begin{align*}
        [\partial, P] = Af + fA + \psi h + k\phi.
    \end{align*}
    We choose the length 2 map
    \begin{align*}
        P = AfA + \psi hA + k\phi A + fA^2 + kB\phi.
    \end{align*}
    It is straightforward to verify that $P$ satisfies the above relation. Next, we need a nullhomotopy of the compression of the hyperbox
    \begin{align*}
        \begin{tikzcd}[ampersand replacement = \&]
            N \ar[d,"g"] \ar[r,"\psi"] \ar[dr,"k"] \& M \ar[d,"f"] \ar[r,"\phi"] \ar[dr,"h"] \& N \ar[d,"{g}"] \\
            N \ar[r,"\psi"] \& M \ar[r,"\phi"]\& N,
        \end{tikzcd}
    \end{align*}
    which amounts to constructing a degree 2 map $Q: N \ra N$, satisfying
    \begin{align}\label{eqn: Q hypercube relation}
        [\partial, Q] = gB + Bg + h \psi + \phi k.
    \end{align}
    Indeed, we choose 
    \begin{align*}
        Q = BgB + Bh\psi + h A \psi + h \psi B + \phi A k + B \phi k + gB^2.
    \end{align*}
    One can verify that 
    \begin{align*}
        [\partial, Q] = (gB + Bg + h \psi + \phi k) + [g, \phi A \psi + B \phi \psi].
    \end{align*}
    The hypercube relation is not quite satisfied, but notice that $\phi A \psi + B \phi \psi$ is a chain map:
    \[
    [\partial,\phi A \psi + B \phi \psi] = \phi(\mathrm{id}+\psi\phi)\psi + (\mathrm{id}+\phi\psi)\phi\psi = 0.
    \]
    In particular, we can redefine
    \begin{align*}
        B := B + \phi A \psi + B \phi \psi.
    \end{align*}
    It then follows that Equation \eqref{eqn: Q hypercube relation} is satisfied, and since we have only modified $B$ by a chain map, this modification will have no effect on the previous relations. 

    The result then follows by induction, as we can (as usual) view $\cM$ and $\cN$ as 1-dimensional quasimodules of $(n-1)$-dimensional quasimodules.
\end{proof}

\Cref{lem: invertibility of hypercube qis} allows us to formulate the adjunction $\mathrm{Tel}\dashv\mathrm{Hyp}$ in the derived category.

\begin{cor} \label{cor: derived tel-hyp adjunction}
    The functors $\mathrm{Tel}$ and $\mathrm{Hyp}$ induce the following functors:
    \[
    \mathrm{Tel}:\mathrm{QMod}_2^h\rightarrow \mathcal{D}^\infty_\cS,\quad \mathrm{Hyp}:\mathcal{D}^\infty_\cS\rightarrow \mathrm{QMod}_2^h.
    \]
    Moreover, the adjunction $\mathrm{Tel}\dashv \mathrm{Hyp}$ continues to hold.
\end{cor}
\begin{proof}
    The well-definedness of $\mathrm{Tel}$ is obvious; the well-definedness of $\mathrm{Hyp}$ follows from \Cref{lem: invertibility of hypercube qis}. Since the unit and counit can be given {in} the same way as in the proof of \Cref{lem: tel-hyp adjunction}, the corollary follows.
\end{proof} 

Finally, we need a small categorical lemma.

\begin{lem} \label{lem: categorical lemma}
    Let $C,D$ be categories and $F:C\rightarrow D$ and $G:D\rightarrow C$ be functors, where $F$ is left adjoint to $G$. Given objects $A,B$ of $D$, suppose that the {counit} $\epsilon:FG\rightarrow\mathrm{id}$ for the given adjunction is invertible at $A$, i.e. $\epsilon(A)$ is an isomorphism. Then the map
    \[
    G:[A,B]_D\rightarrow [G(A),G(B)]_C
    \]
    is bijective.
\end{lem}
\begin{proof}
    The adjunction
    \[
    \Phi:[X,G(Y)]_C \xrightarrow{\simeq} [F(X),Y]_D
    \]
    is natural, i.e. we have $\Phi(G(g)f)=g\Phi(f)$ for any objects $X$ of $C$, $Y,Z$ of $D$, and morphisms $f\in [X,G(Y)]_C$ and $g\in [Y,Z]_D$. Since $\epsilon_B=\Phi(\mathrm{id})$, we get $\Phi(G(g))=g\epsilon$. Thus the composition
    \[
    [A,B]_D \xrightarrow{G} [G(A),G(B)]_C \xrightarrow{\Phi}[FG(A),B]_D
    \]
    is the same as $g\mapsto g\epsilon(A)$. Since $\epsilon(A)$ and $\Phi$ are both invertible, the lemma follows.
\end{proof}

We are now ready to prove \Cref{lem:invertibility for fg free}.

\begin{proof}[Proof of \Cref{lem:invertibility for fg free}]
    We only need to prove the surjectivity. Choose any morphism $f\in [M,N]_{\mathcal{D}^\infty_\cS}$. Then $f$ is the formal composition of a zig-zag of morphisms, where wrong way maps are quasi-isomorphisms. More precisely, there exist bigraded dg $A$-modules $M=M_0,M_1,\cdots,M_n = N$ such that:
    \begin{itemize}
        \item For each $i$, we have a morphism $f_i:M_{i-1} \rightarrow M_i$ that is either a homogeneous dg morphism or the formal inverse of a homogeneous quasi-isomorphism;
        \item $f = f_n\circ\cdots\circ f_1$.
    \end{itemize}
    By \Cref{lem: invertibility of hypercube qis}, we see that for each $i=0,\cdots,n$, the morphism $\mathrm{Hyp}(f_i)\in [\mathrm{Hyp}(M_{i-1}),\mathrm{Hyp}(M_i)]_{\mathrm{QMod}_2^h}$ is well-defined. By composing them, we get a well-defined morphism
    \[
    \mathrm{Hyp}(f) := \mathrm{Hyp}(f_n)\circ\cdots\circ\mathrm{Hyp}(f_1)\in [\mathrm{Hyp}(M),\mathrm{Hyp}(N)]_{\mathrm{QMod}_2^h}.
    \]
    Since $M$ is finitely generated and free, we see from \Cref{lem: counit is invertible for perfects,lem: tel-hyp adjunction,lem: categorical lemma} that there exists a unique morphism $f_0 \in [M,N]_{\mathrm{dgMod}^h_\cS}$ such that $\mathrm{Hyp}(f_0)=\mathrm{Hyp}(f)$. By applying \Cref{lem: categorical lemma} again for the derived Tel-Hyp adjunction as in \Cref{cor: derived tel-hyp adjunction}, we see that the image $\Psi(f_0)$ of $f_0$ under the inclusion
    \[
    \Psi:[M,N]_{\mathrm{dgMod}^h_\cS}\rightarrow [M,N]_{\mathcal{D}^\infty_\cS}
    \]
    is the unique morphism in $[M,N]_{\mathcal{D}^\infty_\cS}$ which is mapped to $\mathrm{Hyp}(f)$ under the functor $\mathrm{Hyp}$. Since the given morphism $f$ is obviously a morphism that satisfies this property, we deduce that $\Psi(f_0)=f$. The lemma follows.
\end{proof}

\subsection{Truncated Hypercubes} We now briefly discuss hypercubes of $\cR$-modules. Let $C$ be a free, finitely generated $\cR$-complex. 
As usual, its total telescope $\Tel(\Hyp(C))$ is a $\F[U,V]$-complex, and we can consider its $(UV) = 0$ truncation. Given an $\cR$-complex $C$, we will use the notation 
\begin{align*}
    (\Tel \Hyp)_\cR(C):= {\Tel(\Hyp(C))} \otimes_{\F[U,V]}\cR. 
\end{align*}
Here, we remind the reader that the $\cR$ action on $(\Tel \Hyp)_\cR(C)$ does not actually interact with the $\cR$ action on $C$; we regard the vertices $C$ of $\Hyp(C)$ as $\F$-vector spaces, and the $\cR$-action on the total telescope is given by the two natural shift maps.

Given an $\F[U, V]$-complex $C$, the canonical inclusion $C \ra \Tel(\Hyp(C))$ is a quasi-isomorphism; this is no longer the case when we truncate to and work with $\cR$-complexes. Even in the case that $C = \cR$ with trivial differential, $(\Tel \Hyp)_\cR(C)$ is quasi-isomorphic to $\cR \oplus \cR$.

Throughout, we will define $\cV$ to be a 2-dimensional bigraded $\F$-vector space, whose generators lie in Alexander grading 0 and Maslov gradings $\pm 1/2$. 

\begin{defn}\label{def:doubling functor}  
    Let $\cD: \Kom(\cR) \ra \Kom(\cR)$ be the \emph{doubling functor}, which on objects $C \in \Kom(\cR)$ is defined by $\cD(C):= C \otimes_\F \cV$ and on morphisms $f: C \ra D$ is defined $\cD(f):= f \otimes_\F \id_{\cV}$.
\end{defn}

When the grading is irrelevant, we will at times simply write $C\oplus C$, rather than $C \otimes \cV$. 
The main result of this {section} is the following:

\begin{lem}\label{lem: main appendix lem}
    There is a natural transformation 
    \begin{align*}
        \frak F: \cD \ra (\Tel \Hyp)_\cR,
    \end{align*}
    with the property that if $C$ is a finitely generated, free $\cR$-complex then 
    \begin{align*}
        \frak F_C: \cD(C) \ra (\Tel \Hyp)_\cR(C)
    \end{align*}
    is a homotopy equivalence.
\end{lem}

\begin{lem} \label{lem: functoriality of FC}
    For $C$ an $\cR$-chain complex, define a morphism $\frak F_C$ by
    \begin{align*}
        \mathfrak{F}_C:\cD(C) \cong C \oplus C \ra {(\Tel\Hyp)}_\cR(C), (x, y) \mapsto 1 \otimes x | w + 1 \otimes U y |e_v + V \otimes y| e_h.
    \end{align*}
    Then, $\frak F: \cD \ra {(\Tel\Hyp)}_\cR$ is a natural transformation.
\end{lem}
\begin{proof}
    If $f: C \ra D$ is a chain map, then the induced map $(\Tel \Hyp)_\cR(f): (\Tel \Hyp)_\cR(C) \ra (\Tel \Hyp)_\cR(D)$ is given by 
    \begin{align*}
        1 \otimes x | \sigma \mapsto 1 \otimes f(x) | \sigma.
    \end{align*}
    We can easily verify:
    \begin{align*}
        {(\Tel\Hyp)}_\cR(f) \circ \frak F_C(x,y) &= {(\Tel\Hyp)}_\cR(f)(1 \otimes x | w + 1 \otimes U y |e_v + V \otimes y| e_h)\\
        & = 1 \otimes f(x) | w + 1 \otimes U f(y) |e_v + V \otimes f(y)| e_h\\
        & = \frak F_D \circ\cD(f)(x, y). 
    \end{align*}
\end{proof}

The second claim of \Cref{lem: main appendix lem}  will be broken into several steps. Our strategy will be to induct on the rank of $C$. Hence, we first consider the case that $C = \cR$ with trivial differential. For simplicity, we will use the operations $U^{-1}$ and $V^{-1}$ on $\cR$, defined as follows.

\begin{defn}
    Let $U^{-1}$ be the $\mathbb{F}_2$-linear endomorphism of $\cR$ defined by:
    \[
    U^{-1}(V^n)=0 \quad \text{for all} \quad n\ge 0,\quad U^{-1}(U^n)=U^{n-1} \quad \text{for all}\quad n>0.
    \]
    Similarly, let $V^{-1}$ {be} the $\mathbb{F}_2$-linear endomorphism of $\cR$ defined by:
    \[
    V^{-1}(U^n)=0 \quad \text{for all} \quad n\ge 0,\quad V^{-1}(V^n)=V^{n-1} \quad \text{for all}\quad n>0.
    \]
    We also define $U^{-n}$ and $V^{-n}$ for integers $n>0$ as endomorphisms $(U^{-1})^n$ and $(V^{-1})^n$, respectively.
\end{defn}

\begin{lem} \label{lem: Tel(R) is R+R}
    When $C = \cR$, the truncated telescope $(\Tel \Hyp)_\cR(C)$ is homotopy equivalent to $\cR\otimes_\F \cV$.
\end{lem}
\begin{proof}
    Consider the map
    \begin{align*}
        f: \cR \oplus \cR \ra (\Tel \Hyp)_\cR(C), (r, s) \mapsto r \otimes 1 | w + U s \otimes 1 | e_v + s \otimes V | e_h.
    \end{align*}
    This is a chain map; indeed,
    \begin{align*}
        \tilde{\partial} f(r, s) & =\tilde{\partial}(r \otimes 1| w + U s \otimes 1 | e_v + s \otimes V | e_h) \\
        & = U (V s\otimes 1 + s \otimes V) | w + (U s\otimes V + s \otimes UV) | w \\
        & = (U s\otimes V + U s\otimes V)|w \\
        & = 0 \\
        & = f(\partial (r, s)).
    \end{align*}
    We define a map, $g: (\Tel \Hyp)_\cR(C) \ra \cR \oplus \cR$, in the opposite direction as follows:
    \begin{align*}
        g(1\otimes r | w) & =  (r, 0) \\
        g(1\otimes r|e_v) & = (0, U^{-1}r) \\
        g(1\otimes r| e_h) & = (0, V^{-1}r) \\
        g(1\otimes r | F) & = 0 
    \end{align*}
    extending $\cR$-linearly in the first factor. We emphasize that $\cR$ acts on the left.

    It will often be helpful to express $r = r_0 + r_U + r_V$, where $r_0$ is a constant, $r_U \in U \cdot \cR$, and $r_V \in V \cdot \cR$. We may easily check that $g$ is a chain map. The only non-trivial computation involves the faces:
    \begin{align*}
        g(\tilde{\partial}(1\otimes r | F))) & = g((V \otimes r + 1 \otimes V r )| e_h + (U \otimes r + 1 \otimes U r) | e_v) \\
        & = V(0,V^{-1}r)+ (0,V^{-1}Vr) + U(0, U^{-1}r) + (0, U^{-1}Ur) \\
        & = V(0,r_V)+ (0,r_0+Vr_V) + U(0, r_U) + (0, r_0+Ur_U)\\
        & = 0.
    \end{align*}
    The composition $g \circ \tilde{\partial}$ is trivial on edges and vertices, so $g$ is indeed a chain map. 
    
    Note that $g$ is a left inverse for $f$. For $(r, s) \in \cR \oplus \cR$, we have 
    \begin{align*}
        g(f(r, s)) & = g( r \otimes 1 | w + U s \otimes 1 | e_v + s \otimes V | e_h)\\
        & = (r, 0) + (0,0) + (0, s)\\
        & = (r, s).
    \end{align*}
    Of course, $f$ is only a homotopy right inverse for $g$. Define a homotopy $H: (\Tel \Hyp)_\cR(C) \ra (\Tel \Hyp)_\cR(C)$ by
    \begin{align*}
        H(1 \otimes r |w) &= \sum_{i=0}^{\deg(r_U)-1} U^i \otimes U^{-i-1} r | e_h + \sum_{i=0}^{\deg(r_V)-1} V^i \otimes V^{-i-1} r | e_v \\
        H(1 \otimes r | e_v) &= \sum_{i=0}^{\deg(r_U)-1} U^i \otimes U^{-i-1} r | F \\
        H(1 \otimes r | e_h) & = \sum_{i=0}^{\deg(r_V)-2} V^i \otimes V^{-i-1} r| F\\
        H(1 \otimes r | F) & = 0,
    \end{align*}
    extending $\cR$-linearly. Let us compute: for a term $1 \otimes r |w$, we have,
    \begin{align}\label{eqn:1+gf w}
        (1 + f \circ g)(1 \otimes r |w) & = 1 \otimes r |w + r \otimes 1|w,
    \end{align}
    and 
    \begin{align}\label{eqn:dh+hd w}
        (H\partial + \partial H)(1 \otimes r|w) &= \partial\left( \sum_{i=0}^{\deg(r_U)-1} U^i \otimes U^{-i-1} r | e_h + \sum_{i=0}^{\deg(r_V)-1} V^i \otimes V^{-i-1} r | e_v \right).
    \end{align}
    Computing the sum on the right is slightly subtle when $r$ is a general polynomial in $\cR$, so we break it into cases. First, note that when $r = 1$, we have that $(1 + f \circ g)(1 \otimes r |w) = 1 \otimes 1 |w + 1 \otimes 1|w = 0$. Also, the sums in \Cref{eqn:dh+hd w} all vanish, since $U^{-1}$ and $V^{-1}$ act by zero on constant polynomials. When $r = U^n$, \Cref{eqn:1+gf w} becomes $1 \otimes U^n | w + U^n \otimes 1 |w$, while \Cref{eqn:dh+hd w} is much more interesting: 
    \begin{align*}
        \partial\left( \sum_{i=0}^{n-1} U^i \otimes U^{-i-1} r | e_h + \sum_{i=0}^{-1} V^i \otimes V^{-i-1} r | e_v \right) & = \sum_{i=0}^{n-1} U^i \partial (1 \otimes U^{-i-1} U^n | e_h) + 0 \\
        & = \sum_{i=0}^{n-1} (U^{i+1}\otimes U^{n-i-1} + U^i \otimes U^{n-i})| w.
    \end{align*}
    This sum telescopes, and the only terms which remain are $1 \otimes U^n|{w}$ and $U^n \otimes 1 | {w}$. Hence, the contributions from \Cref{eqn:1+gf w} and \Cref{eqn:dh+hd w} cancel. The computation when $r=V^n$ is similar. 

    Terms of the form $1 \otimes r | e_h$ {contribute}:
    \begin{align}\label{eqn:1+gf e_h}
        (1 + f \circ g)(1 \otimes r |e_h) & = 1 \otimes r|e_h + U V^{-1} r \otimes 1 | e_v + V^{-1} r \otimes V | e_h
    \end{align}
    and 
    \begin{align}\label{eqn:dh+hd e_h}
        (H\partial + \partial H)(1 \otimes r|e_h) = H(U \otimes r|w + 1 \otimes U r | w) +  \partial\left(\sum_{i=0}^{\deg(r_V)-2} V^i \otimes V^{-i-1} r | F \right)\\
         =  \left(\sum_{i=0}^{\deg(r_U)-1} U^{i+1} \otimes U^{-i-1} r | e_h + U \otimes V^{-1}r|e_v \right) + \left( \sum_{i=0}^{\deg(r_U)-1} U^i \otimes U^{-i-1} Ur | e_h + 0\right ) \\+ \sum_{i=0}^{\deg(r_V)-2} V^i\partial(1 \otimes V^{-i-1} r | F)
    \end{align}
    Note, that expressions such as $U V^{-1} r \otimes 1$ cannot immediately be simplified; this expression zero unless $r = V$, in which case, it is equal to $U \otimes 1$! 

    We again break this computation into cases. When $r = 1$, \Cref{eqn:1+gf e_h} yields $1 \otimes 1 | e_h$. The first and third sums in \Cref{eqn:dh+hd e_h} are trivial, whereas the second sum has a single term, $1 \otimes 1 | e_h$. Hence, the sum is zero. When $r = U^n$, \Cref{eqn:1+gf e_h} is equal to $1 \otimes U^n|e_h$. \Cref{eqn:dh+hd e_h} is equal to 
    \begin{align*}
         \left(\sum_{i=0}^{n-1} U^{i+1} \otimes U^{n-i-1} | e_h + 0 \right) + \left( \sum_{i=0}^{{n}} U^i \otimes U^{n-i} | e_h \right )  + 0,
    \end{align*}
    which again telescopes, and only term which survives is $1 \otimes U^n | e_h$. When $r = V^n$, \Cref{eqn:1+gf e_h} reads $1 \otimes V^n | e_h + UV^{n-1}\otimes 1| e_v + V^{n-1} \otimes V |e_h$ (again, the second term is zero unless $n = 1$). \Cref{eqn:dh+hd e_h} becomes 
    \begin{align*}
        \left(0 + U \otimes V^{n-1}|e_v \right) + 0 + \sum_{i=0}^{n-2} V^i\partial(1 \otimes V^{n-i-1} | F) \\= U \otimes V^{n-1}|e_v + \sum_{i=0}^{n-2}V^iU \otimes V^{n-i-1}|e_v + V^i \otimes U V^{n-i-1}|e_v + V^{i+1} \otimes V^{n-i-1}|e_h + V^i \otimes V^{n-i}|e_h).
    \end{align*}
    When $n = 1$, \Cref{eqn:1+gf e_h} is equal to $1 \otimes V| e_h + U \otimes 1| e_v + 1 \otimes V| e_h,$ while \Cref{eqn:dh+hd e_h} is equal to $U \otimes 1| e_v$, so their sum is zero. When $n > 1$, \Cref{eqn:1+gf e_h} is equal to $1 \otimes V^n |e_h + V^{n-1} \otimes V|e_h$ and \Cref{eqn:dh+hd e_h} is equal to $U \otimes V^{n-1}|e_v + U \otimes V^{n-1}|e_v + 1 \otimes V^n |e_h + V^{n-1} \otimes V |{e_h}$ (all other terms cancel in pairs). The verification for terms of the form $1 \otimes r {|} e_v$ is similar, so we leave it to the reader. 

    Finally, we consider the faces. We have 
    \begin{align}\label{eqn:1+gf F}
        (1 + f \circ g)(1 \otimes r |F) & = 1 \otimes r|F,
    \end{align}
    and 
    \begin{align}\label{eqn:dh+hd F}
        (H\partial + \partial H)(1 \otimes r|F) = H(U \otimes r | e_v + 1 \otimes U r |e_v + V \otimes r | e_h + 1 \otimes V r |e_h) + 0\\
        = U \left(\sum_{i=0}^{\deg(r_U)-1} U^i \otimes U^{-i-1} r | F \right) + \left(\sum_{i=0}^{{\deg((Ur)_U)-1}} U^i \otimes U^{-i-1} U r | F \right) + \\V \left(\sum_{i=0}^{\deg(r_V)-2} V^i \otimes V^{-i-1} r | F \right) + \left(\sum_{i=0}^{{\deg((Vr)_V)-2}} V^i \otimes V^{-i-1} V    r | F \right)
    \end{align}
    When $r = 1$, \Cref{eqn:1+gf F} is equal to $1 \otimes r|F$ as is \Cref{eqn:dh+hd F} (only the second sum has nontrivial terms.) When $r = U^m$, the second two sums are trivial, and the first two sums telescope leaving $1 \otimes U^m | F$; $r = V^m$ is similar. Therefore, we have that $f \circ g$ is homotopic to the identity, proving the claim.
\end{proof} 

Before proceeding to the inductive step, we state a standard fact of homological algebra.

\begin{lem}\label{lem:natural transf}
Consider a homotopy-commutative diagram
\begin{align*}
    \begin{tikzcd}[ampersand replacement = \&]
        C_1 \ar[r,"F"] \ar[d, shift right,"\phi_1" left] \& D_1 \ar[d, shift right, "\phi_2" left] \\
        C_2 \ar[r,"G"] \ar[u, shift right, "\psi_1" right] \& D_2 \ar[u, shift right, "\psi_2" right]
    \end{tikzcd}
\end{align*}
where $F$ and $G$ are chain maps and $\phi_i$ and $\psi_i$ are homotopy inverses to each other. Then, there exists a homotopy equivalence $\Cone(F) \ra \Cone(G)$.\footnote{This homotopy equivalence is not uniquely determined up to homotopy; in order to achieve uniqueness, we need to specify a homotopy class of the commutation homotopy.} 
\end{lem}

We would like {to} decompose $C$ into smaller pieces. 

\begin{defn}{\cite[Definitions 2.7, 2.12, 3.6]{popovic2023link}}
    A \emph{zero complex} is {a} complex over $\cR$ with two generators $x$ and $y$ so that $\partial x = y$ and $\partial y = 0$.

    A \emph{snake complex} is a finitely generated, free complex over $\cR$ with a basis $B = \{x_0, \hdots, x_n\}$ with arrows between $x_i$ and $x_{i+1}$ which alternate between horizontal and vertical.

    A \emph{local system} is a finitely generated complex $C$ over $\cR$ which has no arrows of length zero, is indecomposable as a chain complex over $\cR$, has torsion homology (i.e. $H_*(C/U)$ is a torsion $\F[V]$-module and $H_*(C/V)$ is a torsion $\F[U]$-module), and admits a simplified decomposition. 
\end{defn}

According to \cite[Theorem 4.1]{popovic2023link}, any finitely generated free $\cR$-complex is isomorphic to a direct sum of \emph{snake complexes}, \emph{local systems}, and \emph{zero complexes}, which are uniquely determined up to permutation. 

\begin{defn} \label{defn: standard complex}
A finitely generated free $\cR$-complex is a \emph{standard complex} if it is isomorphic to a direct sum of finitely many snake complexes and finitely many local systems.
\end{defn}

Using \Cref{defn: standard complex}, \cite[Theorem 4.1]{popovic2023link} can be rephrased as follows: for any finitely generated free $\cR$-complex $C$, there exists {a} subcomplex $\mathrm{Std}(C)$, determined uniquely up to isomorphism, which is standard and the inclusion $\mathrm{Std}(C)\hookrightarrow C$ is a homotopy equivalence. 

\begin{defn}
    We define $\mathrm{Std}(C)$ as the \emph{standardization} of $C$. Furthermore, we denote the rank of $\mathrm{Std}(C)$ as a free $\cR$-module as $\mathrm{rank}_s(C)$, the \emph{standard rank} of $C$.
\end{defn}

Note that, by \cite[Theorem 1.1]{popovic2023link}, the isomorphism class of $\mathrm{Std}(C)$, and thus the standard rank of $C$, depends only on the homotopy equivalence class of $C$. It is thus clear that $\mathrm{rank}_s(C)=0$ if and only if $C$ is acyclic. Furthermore, we have an inequality
\[
\mathrm{rank}_s(C) \le \mathrm{rank}(C),
\]
where the equality holds if and only if $C$ is standard.

\begin{lem} \label{lem: complexes are mapping cones}
    Let $C$ be a finitely generated free $\cR$-complex with $\mathrm{rank}_s(C)\ge 2$. Then there exist finitely generated free $\cR$-complexes $C_1$ and $C_2$ and a chain map $f:C_1\rightarrow C_2$ such that the following conditions are satisfied.
    \begin{itemize}
        \item $C$ is homotopy equivalent to the mapping cone of $f$;
        \item $\mathrm{rank}_s(C_i)<\mathrm{rank}_s(C)$ for $i=1,2$.
    \end{itemize}
\end{lem}
\begin{proof}
    We may assume, without any loss of generality, that $C$ is standard. Then $C$ is a finite direct sum of snake complexes and local systems. If there {is} more than one direct {summand}, we may simply take $C_1$ as {one of the summands}, $C_2$ as the direct sum of all other summands, and $f$ as the zero map. Thus it suffices to consider the case when $C$ is either a snake complex or a local system.

    Suppose first that $C$ is a snake complex. Then we have
    \[
    \mathrm{rank}(C) = \mathrm{rank}_s(C) > 1,
    \]
    so $C$ is not a trivial snake complex, i.e. it has a nontrivial differential. A snake complex has a basis $B = \{x_0, \hdots, x_n\}$ with arrows connecting $x_i$ and $x_{i+1}$ alternating between vertical and horizontal. There must exist some basis element $x_i$ so that $\partial(x_i) = U^mx_{i+1} + V^nx_{i-1}$ or $\partial(x_i) = U^mx_{i-1} + V^nx_{i+1}$ (here, we allow either $x_{i\pm 1}$ to be zero). By definition of snake complexes, this element $x_i$ cannot appear in the differential of any other basis element. Therefore, we may take $C_2$ to be the complex generated by $B\smallsetminus \{x_i\}$, so that $C_1 := C/C_2$ is isomorphic to $\cR$, generated by $x_i$. This represents $C$ as the mapping cone of some chain map $f:C_1 \rightarrow C_2$. The complex $C_1$ is standard, so we have 
    \[
    \mathrm{rank}_s(C_1) = \mathrm{rank}(C_1) = 1 < \mathrm{rank}_s(C).
    \]
    Also, we have
    \[
    \mathrm{rank}_s(C_2) \le \mathrm{rank}(C_2) < \mathrm{rank}(C) = \mathrm{rank}_s(C).
    \]
    Hence the lemma holds in this case.

    In the case that $C$ is a local system the argument is nearly identical.
\end{proof}

Before turning into homotopy equivalences, we establish that $\frak{F}_C$ is a quasi-isomorphism.

\begin{lem} \label{lem: FC is a qis}
    For any {finitely} generated free $\cR$-complex $C$, the chain map $\mathfrak{F}_C$ is a quasi-isomorphism.
\end{lem}
\begin{proof}
    If $\mathrm{rank}_s(C)=1$, then $C$ is homotopy equivalent to $\mathcal{R}$ (with zero differential), which we write as $C_0$. Choose homotopy equivalences
    \[
    f:C\rightarrow C_0,\quad g:C_0 \rightarrow C
    \]
    which are homotopy inverses to each other. Then it follows from \Cref{lem: functoriality of FC} that the maps
    \[
    (\Tel \Hyp)_\cR(f):(\Tel \Hyp)_\cR(C)\rightarrow (\Tel \Hyp)_\cR(C_0),\] 
    \[
    (\Tel \Hyp)_\cR(g):(\Tel \Hyp)_\cR(C_0)\rightarrow (\Tel \Hyp)_\cR(C)
    \]
    are also homotopy inverses to each other, and thus are homotopy equivalences. 
    
    Recall from \Cref{lem: Tel(R) is R+R} that the map
    \[
    \mathfrak{F}_{C_0}:C_0\oplus C_0 \rightarrow (\Tel \Hyp)_\cR(C_0)
    \]
    is a homotopy equivalence. By applying \Cref{lem: functoriality of FC}, we obtain a commutative diagram:
    \begin{align*}
        \begin{tikzcd}[ampersand replacement = \&, column sep = huge]
            C_0 \oplus C_0 \ar[r,"g\oplus g"] \ar[d,"\frak F_{C_0}"] \& C \oplus C \ar[d,"\frak F_{C}"]\\
            (\Tel \Hyp)_\cR(C_0) \ar[r,"(\Tel \Hyp)_\cR(g)"] \& (\Tel \Hyp)_\cR(C)
        \end{tikzcd}
    \end{align*}
    Hence we see that $\mathfrak{F}_C$ is a homotopy equivalence, which proves the lemma in the case $\mathrm{rank}_s(C)=1$.
    
    Suppose that, for some integer $n>1$, the lemma holds whenever $\mathrm{rank}_s(C)<n$. Let $C$ be a finitely generated free $\cR$-complex with $\mathrm{rank}_s(C)=n$. By \Cref{lem: complexes are mapping cones}, there exist finitely generated free $\cR$-complexes $C_1$ and $C_2$, together with a chain map $f:C_1\rightarrow C_2$, such that $C$ is isomorphic to the mapping cone of $f$ and $\mathrm{rank}_s(C_i)<n$ for $i=1,2$. Note that this gives a short exact sequence
    \[
    0 \rightarrow C_2 \rightarrow C \rightarrow C_1 \rightarrow 0
    \]
    of chain complexes and chain maps. Then \Cref{lem:natural transf} and \Cref{lem: functoriality of FC} imply that $(\Tel\Hyp)_{\cR}(C)$ is isomorphic to the mapping cone of $(\Tel\Hyp)_{\cR}(f):(\Tel \Hyp)_\cR(C_1)\rightarrow (\Tel \Hyp)_\cR(C_2)$. Then, by \Cref{lem: functoriality of FC}, we get the following commutative diagram, whose arrows are chain maps and rows are short exact sequences of chain complexes.
    \[
    \xymatrix{
    0\ar[r] & C_2 \oplus C_2\ar[r]\ar[d]^{\mathfrak{F}_{C_2}} & C \oplus C \ar[r]\ar[d]^{\mathfrak{F}_C} & C_1 \oplus C_1 \ar[r]\ar[d]^{{\frak F_{C_1}}} & 0 \\
    0\ar[r] & (\Tel \Hyp)_\cR(C_2) \ar[r] & (\Tel \Hyp)_\cR(C)\ar[r] & (\Tel \Hyp)_\cR(C_1)\ar[r] & 0
    }
    \]
    Our induction hypothesis implies that $\mathfrak{F}_{C_1}$ and $\mathfrak{F}_{C_2}$ are quasi-isomorphisms. Hence, a standard five-lemma argument shows that $\mathfrak{F}_C$ is also a quasi-isomorphism. The lemma follows.
\end{proof}

We recall the following fact in homological algebra.
\begin{lem} \label{lem: qis are invertible in bounded complex}
    Let $R$ be a commutative ring and $C,D$ be {bounded} finitely generated free $R$-complexes. Then every quasi-isomorphism $f:C\rightarrow D$ is a homotopy equivalence.
\end{lem}

Now we can prove \Cref{lem: main appendix lem}.

\begin{proof}[Proof of \Cref{lem: main appendix lem}]
    We follow the induction argument in the proof of \Cref{lem: FC is a qis}. The lemma is true when $\mathrm{rank}_s(C)=1$ by \Cref{lem: Tel(R) is R+R}. Suppose that, for some integer $n>1$, the lemma holds whenever $\mathrm{rank}_s(C)<n$, and let $C$ be a finitely generated free $\cR$-complex with $\mathrm{rank}_s(C)=n$. Then, by applying \Cref{lem: complexes are mapping cones}, we may write $C$ as the mapping cone of some chain map $f:C_1 \ra C_2$ for finitely generated, free $\cR$-complexes $C_1$ and $C_2$ such that $0<\mathrm{rank}_s(C_i)<n$ for $i=1,2$. It then follows that $(\Tel\Hyp)_{\cR}(C)$ is the mapping cone of $(\Tel\Hyp)_{\cR}(f):(\Tel\Hyp)_{\cR}(C_1)\rightarrow (\Tel\Hyp)_{\cR}(C_2)$. Our induction hypothesis gives homotopy equivalences
    \[
    G_i:(\Tel \Hyp)_\cR(C_i)\rightarrow C_i \oplus C_i,\quad i=1,2,
    \]
    which are homotopy inverses to $\mathfrak{F}_{C_i}$. Consider the chain map
    \[
    h = G_2 \circ (\Tel \Hyp)_\cR(f) \circ \mathfrak{F}_{C_1}.
    \]
    It follows from \Cref{lem:natural transf} that the mapping cone of {$(\Tel\Hyp)_\cR(f)$} is homotopy equivalent to the mapping cone of $h$. Hence there exist homotopy equivalences
    \[
    F:(\Tel \Hyp)_\cR(C) \rightarrow \mathrm{Cone}(h),\quad G:\mathrm{Cone}(h)\rightarrow (\Tel \Hyp)_\cR(C),
    \]
    which are homotopy inverses to {each} other. 
    
    Since the domain and codomain of $h$ are finitely generated free $\cR$-complexes, $\mathrm{Cone}(h)$ is also a finitely generated free $\cR$-complex. Consider the chain map
    \[
    F \circ \mathfrak{F}_C:C\oplus C\rightarrow \mathrm{Cone}(h).
    \]
    By \Cref{lem: FC is a qis}, we know that this map is a quasi-isomorphism. Since $C\oplus C$ and $\mathrm{Cone}(h)$ are both finitely generated free $\cR$-complexes, it follows from \Cref{lem: qis are invertible in bounded complex} that $F\circ \mathfrak{F}_C$ is a homotopy equivalence. Since 
    \[
    \mathfrak{F}_C \sim (G\circ F)\circ \mathfrak{F}_C = G\circ (F\circ \mathfrak{F}_C),
    \]
    we deduce that $\mathfrak{F}_C$ is a homotopy equivalence. The lemma follows.
\end{proof}


\subsection{Hypercubes and Bordered Modules}

In addition to hypercubes of chain complexes, it will be useful to consider hypercubes of type A and D structures. Here, we recall the formalism of \cite{LOT_spectral_sec_II}.

\begin{defn}{\cite[Definition 2.5]{LOT_spectral_sec_II}}
    An $n$-dimensional pre-hypercube $H$ of type D structures (over $\mathcal{A}(T^2)$) is a collection $(M^\eta)_{\eta \in \E^n}$, together with elements
    \[
    D^\epsilon_{\epsilon_0}\in \mathrm{Mor}^{\mathcal{A}(T^2)}(M^{\epsilon_0},M^{\epsilon_0+\epsilon})
    \]
    for $(\epsilon_0,\epsilon)\in \E^n \times \E^n$ so that $\epsilon_0+\epsilon \in \E^n$. We say that $H$ is a hypercube if the following conditions are satisfied:
    \begin{itemize}
        \item For each $\epsilon\in \E^n$, the type D endomorphism $D^\epsilon_0$ of $M^\epsilon$ coincides with the differential map of $M^\epsilon$;
        \item $\sum_{\substack{\epsilon' \in \E^n,\\ \epsilon' \le \epsilon}} D^{\epsilon - \epsilon'}_{\epsilon_0 + \epsilon'}\circ D^{\epsilon'}_{\epsilon_0} = 0$;
    \end{itemize}
\end{defn}

\begin{defn}{\cite[Definition 2.2]{LOT_spectral_sec_II}}
    An $n$-dimensional pre-hypercube $H$ of type A structures (over $\mathcal{A}(T^2)$) is a collection $(N^\epsilon)_{\epsilon \in \E^n}$, together with elements
    \[
    F^\epsilon_{\epsilon_0}\in \mathrm{Mor}_{\mathcal{A}(T^2)}(N^{\epsilon_0},N^{\epsilon_0+\epsilon})
    \]
    for $(\epsilon_0,\epsilon)\in \E^n \times \E^n$ so that $\epsilon_0+\epsilon \in \E^n$. We say that $H$ is a hypercube if the following conditions are satisfied:
    \begin{itemize}
        \item For each $\epsilon\in \E^n$, the type A endomorphism $F^\epsilon_0$ of $N^\epsilon$ coincides with the structure maps of $N^\epsilon$;\footnote{For any type A structure $N$, the structure maps $\{m_i\}_{i\ge 0}$ define a canonical type A endomorphism of $N\otimes \cT^*(\cA)$ which squares to zero, where $\cT^*(\cA)$ is the tensor algebra. This should be seen as the ``type A'' analogue for the differential endomorphism of chain complexes and type D structures.}
        \item $\sum_{\substack{\epsilon' \in \E^n,\\ \epsilon' \le \epsilon}} F^{\epsilon - \epsilon'}_{\epsilon_0 + \epsilon'}\circ F^{\epsilon'}_{\epsilon_0} = 0$;
    \end{itemize}
\end{defn}

\begin{defn}{\cite[Definition 2.9]{LOT_spectral_sec_II}}
    Let $\mathcal{N}=(N^\epsilon,F^\epsilon_{\epsilon_0})$ be an $k$-dimensional hypercube of type A structures and $\mathcal{M}=(M^\epsilon,D^\epsilon_{\epsilon_0})$ be an $\ell$-dimensional hypercube of type D structures (at least one of which is bounded). Consider the partially defined $(k+\ell)$-dimensional hypercube $(\mathcal{N}\boxtimes \mathcal{M})_{pre}=(C^{\epsilon,\epsilon'},D^{\epsilon,\epsilon'}_{\epsilon_0,\epsilon'_0})$, given as follows:
    \begin{itemize}
        \item $C^{\epsilon,\epsilon'} = N^\epsilon\boxtimes M^{\epsilon'}$ for all $\epsilon\in\E^k$ and $\epsilon'\in\E^\ell$;
        \item $D^{\epsilon,\epsilon'}_{\epsilon_0,0} = F^\epsilon_{\epsilon_0}\boxtimes \mathrm{id}$ for any $\epsilon,\epsilon',\epsilon_0\in\E^k$ with $\epsilon+\epsilon_0\in\E^k$;
        \item $D^{\epsilon,\epsilon'}_{0,\epsilon'_0} = \mathrm{id}\boxtimes D^{\epsilon'}_{\epsilon'_0}$ for any $\epsilon,\epsilon',\epsilon'_0\in \E^\ell$ with $\epsilon'+\epsilon'_0\in \E^\ell$.
    \end{itemize}
    Then $(\mathcal{N}\boxtimes \mathcal{M})_{pre}$ admits a canonical completion to a well-defined $(k+\ell)$-dimensional hypercube $\mathcal{N}\boxtimes \mathcal{M}$ by the maps
    \begin{align*}
        D^{\epsilon, \epsilon'}_{\epsilon_0, \epsilon'_0} = \large\sum_{\ell=0}^\infty\, \large\sum_{\epsilon_0 < \hdots < \epsilon_\ell = \ep'} \begin{tikzcd}[ampersand replacement = \&, column sep = huge]
            \ar[dddddddd,dashed]\&\ar[d,dashed] \\
            \& \delta^{M^{\epsilon_0'}}\ar[d,dashed]\ar[lddddddd,Rightarrow, bend right = 15] \\
            \& D_{\epsilon'_0}^{\epsilon'_1 - \epsilon_0'} \ar[d,dashed] \ar[ldddddd, bend right = 15]\\
            \& \delta^{M^{\epsilon'_1}}\ar[d,dashed]\ar[lddddd,Rightarrow, bend right = 15] \\
            \& \vdots\ar[d,dashed] \\
            \& \delta^{M^{\epsilon'_{\ell-1}}}\ar[d,dashed]\ar[lddd,Rightarrow, bend right = 15] \\
            \& D_{\epsilon'_{n-1}}^{\epsilon'_\ell- \epsilon_{\ell-1}'} \ar[d,dashed] \ar[ldd,bend right = 15] \\
            \& \delta^{M^{\epsilon'_\ell}}\ar[dd,dashed]\ar[ld, Rightarrow, bend right = 15]  \\
           F_{\epsilon_0}^\epsilon \ar[d,dashed] \&  \\
           {}  \& {}
        \end{tikzcd}
    \end{align*}
    following \cite[Equation 2.7]{LOT_spectral_sec_II}.
\end{defn}

\section{Hypercubes of attaching curves}\label{sec: attaching_curves}

\subsection{Hyperboxes of attaching curves}

In this section, we review the analogues of hypercubes of chain complexes in the Fukaya category. This discussion closely follows that of \cite{HHSZ_inv_nat,HHSZ_surgery_exact_invol}. See also \cite{manolescu2024heegaardfloerhomologyinteger,LOT_spectral_sec_II} for closely related notions. 

\begin{defn}
    An \emph{empty $n$-dimensional hypercube of beta (or alpha) attaching curves} is a collection $\cL_\beta = \{\bm \beta^\varepsilon\}_{\varepsilon \in \E^n}$ of attaching curves on a surface $\Sigma$ indexed by points of $\E^n$.
\end{defn}
Such objects can naturally be equipped with $\SpinC$ structures. To do so, we set up some notation. Define $D_m$ to be the 2-disk with $m$ punctures, labeled by $v_1, \hdots, v_{m}$. Label the portion of the boundary between $v_i$ and $v_{i+1}$ by $e_i$. Fix an $n$-dimensional hypercube of beta (or alpha) attaching curves, $\cL_\beta = \{\bm\beta^\varepsilon\}_{\varepsilon \in \E^n}$, and let $U_\varepsilon$ be the handlebody specified by $\bm\beta^\varepsilon \in \cL_\beta$. For each sequence $\varepsilon_1 < \hdots < \varepsilon_m$ in $\E^n$ with $m > 1$, there is an associated 4-manifold, $X_{\varepsilon_1, \hdots, \varepsilon_m}$, which is defined to be
\begin{align*}
    X_{\varepsilon_1, \hdots, \varepsilon_m} = (\Sigma \times D_m) \cup_{\Sigma\times e_1}  (e_1 \times U_{\varepsilon_1}) \cup_{\Sigma\times e_1}  \hdots \cup_{\Sigma \times e_{m}} (e_{m} \times U_{\varepsilon_{m}}).
\end{align*}
Every way of parenthesizing $\ep_1, \hdots, \ep_m$ induces a decomposition of $X_{\varepsilon_1, \hdots, \varepsilon_m}.$
Note that when $m = 2$, $X_{\ep_1, \ep_2}$ is diffeomorphic to the product $Y_{\ep_1, \ep_2}\times [0,1]$.

\begin{defn}
    A \emph{hypercube of $\SpinC$-structures for $\cL_\beta$} consists of a collection of $\SpinC$ structures 
    \begin{align*}
        \frak S_{\varepsilon_1, \hdots, \varepsilon_m} \sub \SpinC(X_{\varepsilon_1, \hdots, \varepsilon_m})
    \end{align*}
    for each sequence $\varepsilon_1 < \hdots < \varepsilon_m$ in $\E^n$ with $m > 1$ which satisfies the following compatibility relations:
    \begin{enumerate}
        \item If $1 \le i < j \le m$, then $\frak S_{\varepsilon_1, \hdots, \varepsilon_m}$ is closed under the action of $\delta^1 H^1(Y_{\varepsilon_i, \varepsilon_j}).$
        \item If $\varepsilon_1< \hdots <\varepsilon_m$ is a sequence in $\E^n$ and $\varepsilon_{i_1}< \hdots <\varepsilon_{i_j}$ is a subsequence where $1 \le i_1 < \hdots < i_j \le m$, then $\frak S_{\varepsilon_{i_1}, \hdots,\varepsilon_{i_j}}$ is the image of
        $\frak S_{\varepsilon_1, \hdots, \varepsilon_m}$ under the restriction map 
        \begin{align*}
            \SpinC(X_{\varepsilon_1, \hdots, \varepsilon_m})
            \ra \SpinC(X_{\varepsilon_{i_1}, \hdots,\varepsilon_{i_j}}).
        \end{align*}
    \end{enumerate}
\end{defn}

\begin{rem}
    Let $\frs_X \in \frak S_{\varepsilon_1, \hdots, \varepsilon_m}$ and suppose $X_{\varepsilon_1, \hdots, \varepsilon_m}$ is cut along $Y_{\varepsilon_i, \varepsilon_j}$ into pieces $W$ and $Z$, then we obtain $\SpinC$ structures $\frs_W$ and $\frs_Z$ on $W$ and $Z$ respectively by restricting $\frs_X$. The first criterion in the definition above ensures that all $\SpinC$ structures on $X_{\varepsilon_1, \hdots, \varepsilon_m}$ obtained by gluing $\frs_W$ and $\frs_Z$ are contained in $\frak S_{\varepsilon_1, \hdots, \varepsilon_m}.$ It will often be the case that the groups $\delta^1 H^1(Y_{\varepsilon_i,\varepsilon_j})$ act trivially on $H^2(X_{\varepsilon_1, \hdots, \varepsilon_m})$. In this situation, a $\SpinC$ structure on $X_{\varepsilon_1, \hdots, \varepsilon_m}$ is uniquely determined by a choice of $\SpinC$ structures on any decomposition of $X_{\varepsilon_1, \hdots, \varepsilon_m}$ into 3-ended components, $X_{\varepsilon_i, \varepsilon_j, \varepsilon_k}$.
\end{rem}

\begin{defn}
    An $n$-dimensional \emph{hypercube of beta attaching curves} on $(\Sigma, \bm w)$ is a collection of attaching curves $\bm \beta^\varepsilon$ for $\varepsilon \in \E^n$ together with distinguished elements 
    \begin{align*}
        \Theta_{\beta^\varepsilon,\beta^{\varepsilon'}} \in \CF(\beta^\varepsilon,\beta^{\varepsilon'})
    \end{align*}
    for every $\varepsilon < \varepsilon'.$ These distinguished elements must satisfy the following structure relation:
    \begin{align}\label{eqn:hyper_cube_attaching_curves}
        \sum_{\ep< \ep_1 < \hdots < \ep_k < \ep'} f_{\beta^\ep, \beta^{\ep_1}, \hdots, \beta^{\ep_k}, \beta^{\ep'}}(\Theta_{\beta^\ep, \beta^{\ep_1}}, \hdots, \Theta_{\beta^{\ep_k}, \beta^{\ep'}}) = 0,
    \end{align}
    where the map $f_{\beta^\ep, \beta^{\ep_1}, \hdots, \beta^{\ep_k}, \beta^{\ep'}}$ counts $(2+k)$-gons. Here, the sum is taken over $\SpinC$-structures according to a hypercube of $\SpinC$-structures. Hypercubes of alpha attaching curves are similar.
\end{defn}

Given hypercubes $\cL_\alpha = (\bm \alpha^\nu)_{\nu \in \E^n}$ and $\cL_\beta = (\bm \beta^\ep)_{\ep \in \E^m}$, we can consider the Heegaard multi-diagram consisting of all $(n + m)$ curves. To this diagram, there is a naturally associated $(n+m)$-dimensional hypercube of chain complexes, denoted $\CF(\cL_\alpha, \cL_\beta)$. The chain complex at vertex $(\nu, \ep)\in \E^{n+m}$ is taken to be $\CF(\bm \alpha^\nu, \bm \beta^\ep)$. 

\begin{defn}
    If $\cL_\alpha = \{\bm \alpha^\eta\}_{\eta \in \E^n}$ and $\cL_\beta = \{\bm \beta^\varepsilon\}_{\varepsilon \in \E^m}$ are hypercubes of attaching curves, then a \emph{hypercube of $\SpinC$-structures for the pair $(\cL_\alpha, \cL_\beta)$} consists of the following data: for every pair of sequences $\eta_1 < \hdots < \eta_k$ in $\E^n$ and $\varepsilon_1 < \hdots < \varepsilon_\ell$ in $\E^m$ with $k + \ell > 1$ a collection of $\SpinC$-structures $\frak S_{\eta_k, \hdots, \eta_1, \varepsilon_1, \hdots, \varepsilon_\ell} \sub \SpinC(X_{\eta_k, \hdots, \eta_1, \varepsilon_1, \hdots, \varepsilon_\ell})$ satisfying compatibility conditions analogous to those defined above.
\end{defn}

Let $(\cL_\alpha, \frak S_\alpha) = (\bm \alpha^\eta, \frak S_{\eta_1, \hdots, \eta_n})$ and $(\cL_\beta, \frak S_\beta) =  (\bm \beta^\varepsilon, \frak S_{\varepsilon_1, \hdots, \varepsilon_m})$ be $\SpinC$-hypercubes of attaching curves. For brevity, let $X_{\alpha} = X_{\eta_n, \hdots, \eta_1}$, $X_{\beta} = X_{\varepsilon_1, \hdots, \varepsilon_m}$, and $X_{\alpha, \beta} =  X_{\eta_n, \hdots, \eta_1, \varepsilon_1, \hdots, \varepsilon_m}$. The \emph{hypercube of $\SpinC$ structures generated by $\frak S_\alpha$ and $\frak S_\beta$} is the minimal collection 
\begin{align*}
    \frak S_{\alpha, \beta} \sub \SpinC(X_{\alpha, \beta}),
\end{align*}
which forms a hypercube of $\SpinC$ structures for the pair $(\cL_\alpha, \cL_\beta)$. 
    
\begin{defn}
    Given hypercubes of attaching curves $\cL_\alpha = \{\bm \alpha^\eta\}_{\eta \in \E^n}$ and $\cL_\beta = \{\bm \beta^\varepsilon\}_{\varepsilon \in \E^m}$, together with a hypercube of $\SpinC$ structures $\frak S_{\eta_k, \hdots, \eta_1, \varepsilon_1, \hdots, \varepsilon_m}$ extending the hypercubes of $\SpinC$ structures on $\cL_\alpha$ and $\cL_\beta$, then there is an $(n+m)$-dimensional hypercube of chain complexes given as follows: given sequences $\nu_1 < \hdots < \nu_i$ and $\ep_1 < \hdots < \ep_j$ we define a map
\begin{align*}
    f^{\beta^{\ep_1} \ra \hdots \ra \beta^{\ep_j}}_{\alpha^{\nu_1} \ra \hdots \ra \alpha^{\nu_i}}: \CF(\bm\alpha^{\nu_1}, \bm \beta^{\ep_1}) \ra \CF(\bm\alpha^{\nu_i}, \bm \beta^{\ep_j})
\end{align*}
by taking $f^{\beta^{\ep_1} \ra \hdots \ra \beta^{\ep_j}}_{\alpha^{\nu_1} \ra \hdots \ra \alpha^{\nu_i}}(\bm x)$ to be
\begin{align*}
     f_{\alpha^{\nu_i}, \hdots, \alpha^{\nu_1},\beta^{\ep_1}, \hdots,\beta^{\ep_j},\frak S_{\eta_k, \hdots, \eta_1, \varepsilon_1, \hdots, \varepsilon_m}}(\Theta_{\alpha^{\nu_i},\alpha^{\nu_{i-1}}}, \hdots, \Theta_{\alpha^{\nu_2},\alpha^{\nu_{1}}}, \bm x, \Theta_{\beta^{\ep_1},\beta^{\ep_{2}}}, \hdots, \Theta_{\beta^{\ep_{j-1}},\beta^{\ep_j}}).
\end{align*}
Define the map in the hypercube from $(\nu, \ep)$ to $(\nu', \ep')$ by
\begin{align*}
    F^{\beta^\ep \ra \beta^{\ep'}}_{\alpha^\nu\ra\alpha^{\nu'}} = \sum_{\substack{\nu < \nu_1 < \hdots < \nu_i < \nu'\\ \ep < \ep_1 < \hdots < \ep_j < \ep'}} f^{\beta^\ep \ra \beta^{\ep_1} \ra \hdots \ra \beta^{\ep_j}\ra \beta^{\ep'}}_{\alpha^\nu \ra \alpha^{\nu_1} \ra \hdots \ra \alpha^{\nu_i}\ra \alpha^{\nu'}}.
\end{align*}
\end{defn}

Throughout, we assume that $\cL_\alpha$ and $\cL_\beta$ are hypercubes of attaching curves such that the associated Heegaard multi-diagram is weakly admissible. It then follows from the structure equation for hypercubes of attaching curves (\Cref{eqn:hyper_cube_attaching_curves}) that $\CF(\cL_\alpha, \cL_\beta)$ is an $(n+m)$-dimensional hypercube of chain complexes. It may be necessary to work over the ring of formal power series, $\F[[U]]$, for the above expressions to make sense.

\subsection{Special Heegaard Multi-diagrams}

In this section, we define several particularly useful classes of Heegaard diagrams.

\begin{defn}{\cite[Section 12]{HHSZ_surgery_exact_invol}}\label{def:surgery_diagram}
    Let $K$ be a knot in $S^3$. We say that a doubly pointed Heegaard quadruple $(\Sigma, \bm \alpha_2, \bm \alpha_1, \bm \beta, w, z)$ is an \emph{$n$-surgery diagram} for $(S^3, K)$ if the following criteria are satisfied:
    \begin{enumerate}
        \item The $\bm \alpha_i$ satisfy
        \begin{align*}
            \bm \alpha_1 = \bm \alpha\cup \{a_1\}, \quad \bm \alpha_2 = \bm \alpha'\cup \{a_2\} 
        \end{align*}
        where $\bm \alpha$ and $\bm \alpha'$ are pairwise Hamiltonian translates of each other so that
        \begin{align*}
            |\alpha_i\cap \alpha_j'| = 2\delta_{ij}.
        \end{align*}
        \item There is a once punctured torus $F \sub \Sigma$ which contains $a_1$ and $a_2$ and is disjoint from $\bm \alpha$ and $\bm \alpha'$. $F$ is called the surgery region. 
        \item The diagram $(\Sigma, \bm \alpha_1, \bm\beta, w, z)$ represents a doubly based unknot in $S^3_n(K)$. Moreover, $w$ and $z$ are adjacent in the same region of the Heegaard diagram. 
        \item The diagram $(\Sigma, \bm \alpha_2, \bm\beta, w, z)$ represents $K$ in $S^3$. 
        \item The curves $a_1$ and $a_2$ are oriented such that
        \begin{align*}
            \#(a_1\cap a_2) = 1.
        \end{align*}
    \end{enumerate}
\end{defn}

It is straightforward to verify that for the torsion $\SpinC$ structure $\frs_0 \in \SpinC(X_{\alpha_2, \alpha_1})$, the alpha curves in a surgery diagram form a hypercube of attaching curves.

\begin{prop}
    The diagram 
    \begin{align}
    (\cL_\alpha, \frs_0) = \begin{tikzcd}[ampersand replacement = \&, column sep = large, row sep = large]
        \bm \alpha_2 \ar[r,"\Theta_{\alpha_{2},\alpha_1}"] \& \bm \alpha_1 
    \end{tikzcd}
    \end{align}
is a hypercube of attaching curves. 
\end{prop}
\begin{proof}
    The diagram $\cH(\bm \alpha_1,\bm \alpha_2)$ is a Heegaard diagram for $\#^{g-1}(S^1\times S^2)$. The generator $\Theta_{\alpha_{2},\alpha_1}$ is the distinguished intersection point $\theta_1^+ \times \hdots \times \theta_{g-1}^+ \times (a_1 \cap a_2)$. The only structure relation to verify is that $\Theta_{\alpha_{2},\alpha_1}$ is a cycle, which is clearly the case.
\end{proof}

There is another system of attaching curves which will be considered.

\begin{defn}\label{def:framing_diagram}
    A Heegaard diagram $(\Sigma, \bm \beta_0, \hdots, \bm\beta_s, w, z)$ is an $m$-\emph{framing multi-diagram} if the following criteria are satisfied:
    \begin{enumerate}
        \item The attaching set $\bm \beta_i$ contains a distinguished curve $b_i$ such that $\hat{\bm\beta}_j := \bm \beta_j \smallsetminus b_j$ is obtained from $\hat{\bm \beta}_i :=\bm \beta_i \smallsetminus{b_i}$ by a small Hamiltonian isotopy for all $i \neq j$; furthermore, we assume that $|(\hat{\bm \beta}_i)_k \cap (\hat{\bm \beta}_j)_\ell| = 2\delta_{k, \ell}$.
        \item There is a once punctured torus $F \sub \Sigma$ which contains $b_0, \hdots, b_s$ and is disjoint from $\bm \beta_i \smallsetminus{b_i}$ for all $i$. We call $F$ the framing region. 
        \item The curves $b_0, \hdots, b_s$ are oriented such that 
        \begin{align*}
            \#(b_i \cap b_{i+1}) = m 
        \end{align*}
        for all $i \in \{0,\hdots,s-1\}$.
    \end{enumerate}
\end{defn}

It is not hard to see that the Heegaard diagrams $\cH(\bm\beta_i, \bm\beta_{i+k})$ represent $\#^{g-1}(S^1 \times S^2)\#(-L(km,1))$. In \Cref{sec:main_hypercube}, we will verify that these beta curves fit naturally into a hypercube of attaching curves.

Let us establish some notation. Let $D(m, 1)$ denote the disk bundle over $S^2$ with Euler number $m$. $D(m, 1)$ has a distinguished $\SpinC$-structure $\frt_0$ which is characterized by the property that 
\begin{align*}
    \langle c_1(\frt_0), [S^2]\rangle = \pm m.
\end{align*}
Let $\frs_0$ be the restriction of this $\SpinC$ structure to $-L(m, 1)$. Furthermore, let $\frs_{k}^\pm$ be those $\SpinC$ structures which satisfy $\langle c_1(\frs_{k}^\pm), S\rangle = \pm (2k+1)m$ for $k \ge 0$. Let $\frs_{\pm1}$ be the $\SpinC$ structures on $-L(m, 1)$ with the property that $\frs_0 - \frs_{\pm 1} = \pm 1 \in H^2(L(m, 1);\Z)$. For the time being, consider the case that $\bm\beta = \emptyset$, so that $\bm \beta_i = b_i$. In this diagram, $\cH(\beta_i, \beta_{i+k})$ represents the lens space $-L(km, 1)$. Moreover, every intersection point corresponds to a generator of $\HFh(-L(km, 1))$. Let $\Theta^0_{\beta^i, \beta^{i+k}}$ denote the unique intersection point which supports $\HFh(L(km, 1),\frs_0)$ and let $\Theta^\pm_{\beta^i, \beta^{i+k}}$ denote the two intersection points adjacent to $\Theta^0_{\beta^i, \beta^{i+k}}$ (which generate $\HFh(L(km, 1),\frs_{\pm1})$).

\begin{defn}\cite[Definition 6.1]{HHSZ_inv_nat}
    Let $\cD_0 = (\Sigma_0, \beta_0, \hdots, \beta_s, \bm w)$ be a weakly admissible Heegaard multi-diagram. We say that $(\cD_0, \frak G_0)$ is \emph{algebraically rigid} if $\partial = 0$ on $\widehat{CF}(\Sigma_0, \beta_i, \beta_{i+1}, \frs_{i,i+1})$ where $\frs_{i,i+1}$ is the restriction of $\frak G_0$ to $Y_{i, i+1}$.
\end{defn}

In particular, $m$-framing multi-diagrams are algebraically rigid (away from the framing region, $(\Sigma, \bm\beta_i, \bm\beta_j)$ is the standard diagram for a connected sum of $S^1 \times S^2$). In this situation, we have the following analogue of \cite[Proposition 6.5]{HHSZ_inv_nat}. Following \cite{HHSZ_inv_nat}, we write $D_s$ for the disk with $s$ boundary puncture points and $K_n$ for Stasheff's associahedron. 

\begin{lem}\label{lem:multi_stabilization}
    Let $\cD = (\Sigma, \bm \delta_1, \hdots, \bm \delta_s, \bm w)$ be a Heegaard multi-diagram and let $\cD_0=(\Sigma_0, \bm \gamma_1, \hdots, \bm \gamma_s, \bm w)$ be an $m$-framing multi-diagram of genus 1. Fix $\frt \in \SpinC(X_{\delta_1, \hdots, \delta_s})$ and $\frt_0 \in \SpinC(X_{\gamma_1,\hdots, \gamma_s})$ and stratified families of almost complex structures $\{J_x\}_{x\in K_{s-1}}$ and $\{J_x^0\}_{x\in K_{s-1}}$ on $\Sigma \times D_n$ and $\Sigma_0 \times D_s$ for counting holomorphic $s$-gons. Let $\theta_1, \hdots, \theta_{s-1}$ be homogeneous classes of $\widehat{CF}(\Sigma_0,  \bm \gamma_1,  \bm \gamma_2, \frs_{1,2}),\hdots, \widehat{CF}(\Sigma_0,  \bm \gamma_{s-1},  \bm \gamma_s, \frs_{s-1,s})$. Assume that 
    \begin{align*}
        \bm y = f_{\cD_0, \frt_0; J^0}( \theta_1, \hdots, \theta_{s-1})
    \end{align*}
    is nonzero. Then,
    \begin{align*}
        f_{\cD\# \cD_0, \frt\# \frt_0; J \wedge J^0}(\bm x_1 \times \theta_1, \hdots, \bm x_{s-1}\times \theta_{s-1}) = f_{\cD, \frt; J }(\bm x_1 , \hdots, \bm x_{s-1})\otimes \bm y.
    \end{align*}
\end{lem}

\begin{proof}
    The proof of \cite[Proposition 6.5]{HHSZ_inv_nat} holds in this case as well, as we now sketch. The moduli space of $s$-gons in $\cD \# \cD_0$ can be identified with a fibered product of $s$-gons in $\cD$ and $ \cD_0$. An index argument shows that if $\psi \wedge \psi_0 \in \pi_2(\theta_1, \hdots, \theta_{s-1}, \bm y_0)$ for some $\bm y_0$, the unconstrained moduli space $\cM(\psi)$ is zero dimensional, so it suffices to understand the moduli space $\cM_k(\theta_1, \hdots, \theta_{s-1}, \bm y_0)$ of curves in a class $\psi_0$ with multiplicity $k$ at the connected sum point. 

    We consider the evaluation map 
    \begin{align*}
        \mathrm{ev}:\cM_k(\theta_1, \hdots, \theta_{s-1}, \bm y_0) \ra D^k_s \times K_{s-1}.
    \end{align*}
    Choose a path $\eta: [0, \infty) \ra D_s^k \times K_{s-1}$ which limits to a tree $T$ in $\partial K_{s-1}$ and a boundary puncture in $D_s$ in each factor, approaching a fixed $\bm d' \in ((0,1) \times \R)^k$, modulo translation. The 1-dimensional moduli space of curves with $\mathrm{ev}(u) \in \im \eta$ has ends corresponding to strip breaking at finite times $t$ and to ends appearing as $t \ra \infty$. Ends of the first kind cancel in pairs, due to the algebraic rigidity assumption. Ends of the second kind are of the form 
    \begin{align*}
        \left( \coprod_{\substack{\phi \in \pi_2(\theta_1, \theta_1)\\ \mu(\phi) = 2k}} \cM(\phi, \bm d') \right) \times \left( \coprod_{\substack{\psi_0 \in \pi_2(\theta_1, \hdots, \theta_{n-1}, \bm y_0)\\ \mu(\psi_0) = 3-n\\ n_{\bm w_0}(\psi_0) = 0} } \cM_{J_T^0}(\psi_0) \right).
    \end{align*}
    The first factor has odd cardinality by \cite[Equation (31)]{zemke_duality_mapping_tori} and the cardinality of the second is exactly the coefficient of $\bm y_0$ in $f_{\cD_0, \frs_0 ;J_T^0}(\theta_1, \hdots, \theta_{n-1})$, proving the claim. Note that we do not have any extra terms consisting of generators of higher degrees, unlike the original formulation of \cite[Proposition 6.5]{HHSZ_inv_nat}, as $\cD_0$ is a genus 1 diagram and thus no such terms can exist; there is a unique intersection point in $\cD_0$ for each of its $\mathrm{Spin}^c$ structures.
\end{proof}

\section{The Framing hypercube}\label{sec:main_hypercube}
Let us fix an $m$-framing quadruple diagram $\cD = (\Sigma, \beta_0, \beta_1, \beta_2, \beta_3)$. As is typical in Heegaard Floer theory, we consider the associated 4-manifold
\begin{align*}
    X_{0,1,2,3} = (\Sigma \times \Box) \cup (e_0 \times U_0) \cup (e_1 \times U_1) \cup (e_2 \times U_2)\cup (e_3 \times U_3).
\end{align*}
There is a map
\begin{align}\label{eqn:tri_to_spinC}
    \frs_z: \pi_2(\bm x, \bm y, \bm z, \bm u) \ra \SpinC(X_{0,1,2,3}),
\end{align}
which partitions homotopy classes of rectangles into $\SpinC$-structures. $X_{0,1,2,3}$ has two natural decompositions:
\begin{align*}
    X_{\beta_0,\beta_1,\beta_2,\beta_3} = X_{0,1,2} \cup_{Y_{0,2}} X_{0,2,3} = X_{0,1,3} \cup_{Y_{1,3}} X_{1,2,3},
\end{align*}
as well as maps
\begin{align*}
    \frs_z: \pi_2(\bm x, \bm y, \bm z) \ra \SpinC(X_{i,j,k}),
\end{align*}
partitioning homotopy classes of triangles for the various subdiagrams obtained by forgetting a set of curves. In general, $\SpinC$-structures on $X_{0,1,2,3}$ are not uniquely determined by their restrictions to $X_{i,j,k}$. However, when there are consecutive tuples of attaching curves which span the same subspace of $H_1(\Sigma;\Z)$, $\delta H^1(Y_{0,2}) + \delta H^1(Y_{1,3})$ vanishes. Therefore, in these cases, there is no ambiguity in gluing $\SpinC$ structures. See \cite[Remark 2.4]{os_holotri}; in other words, the restriction maps 
\begin{align*}
\begin{tikzcd}[ampersand replacement = \&, column sep = small]
    \&
    \SpinC(X_{\beta_0,\beta_1,\beta_2,\beta_3}) 
    \ar[dr] 
    \ar[dl] 
    \&
    \\    
    \SpinC(X_{\beta_0,\beta_1,\beta_2})\times \SpinC(X_{\beta_0,\beta_2,\beta_3})
    \&
    \&
    \SpinC(X_{\beta_0,\beta_1,\beta_3})\times \SpinC(X_{\beta_1,\beta_2,\beta_3})
\end{tikzcd}
\end{align*}
are injective.

The primary goal of this section is to construct a hypercube from these beta curves. The hypercube relations will follow from a model computation in the genus 1 case. When the genus of $\Sigma$ is one, the boundary components $Y_{0,1}$, $Y_{1,2}$, and $Y_{2,3}$ are diffeomorphic to $-L(m, 1)$ and $Y_{0,3}$ is diffeomorphic to $-L(3m,1)$. Also note that $Y_{0,2}$ and $Y_{1,3}$ are diffeomorphic to $-L(2m,1)$.  These will be represented as $m$, $2m$, and $3m$ surgery on the unknot respectively. 

The first step is to understand the $\SpinC$ structures on $X_{0,1,2,3}$ and how they restrict to the various $X_{i,j,k}$ and $Y_{i,j}$. The homologies of these manifolds, as well as the maps induced by the various restrictions, are shown below.
\begin{align*}
    \begin{tikzcd}[ampersand replacement = \&, column sep = small, row sep = small]
        \& \& X_{0,1,2,3} \ar[dl]\ar[dr]\& \& \\
        \& X_{0,1,2}\ar[dl]\ar[d]\ar[dr] \&  \& X_{0,2,3}\ar[dl]\ar[d]\ar[dr] \&\\
        Y_{0,1} \& Y_{1,2} \& Y_{0,2} \& Y_{0,3} \& Y_{2,3}
    \end{tikzcd}
\end{align*}
\begin{align}\label{diag:SpinC_restrictions}
    \begin{tikzcd}[ampersand replacement = \&, column sep = large, row sep = large]
        \& \& 
        \Z \oplus \Z\oplus \Z/m 
        \ar[dl,"(a{,}b{,}c)\mapsto (a{,}c)" description]
        \ar[dr,"(a{,}b{,}c)\mapsto (a-2b{,}b+c)" description]
        \& \& \\
        \& 
        \Z\oplus \Z/m 
        \ar[dl,"a+b" description]
        \ar[d,"b" description]
        \ar[dr,"a+2b" description] 
        \&  \& \
        \Z\oplus \Z/m
        \ar[dl,"a+2b" description]
        \ar[d,"a+3b" description]
        \ar[dr,"b" description] \&\\
        \Z/m \& \Z/m \& \Z/2m \& \Z/3m \& \Z/m
    \end{tikzcd}
\end{align}
Since $\delta H^1(Y_{0,2}) + \delta H^1(Y_{1,3}) = 0$, $\SpinC$ structures on $X_{0,1,2,3}$ are determined by their restrictions to either $X_{0,1,2}\cup X_{0,2,3}$ or $X_{0,1,3}\cup X_{1,2,3}$. 

\begin{figure}
    \centering
    \includegraphics[width=0.3\linewidth]{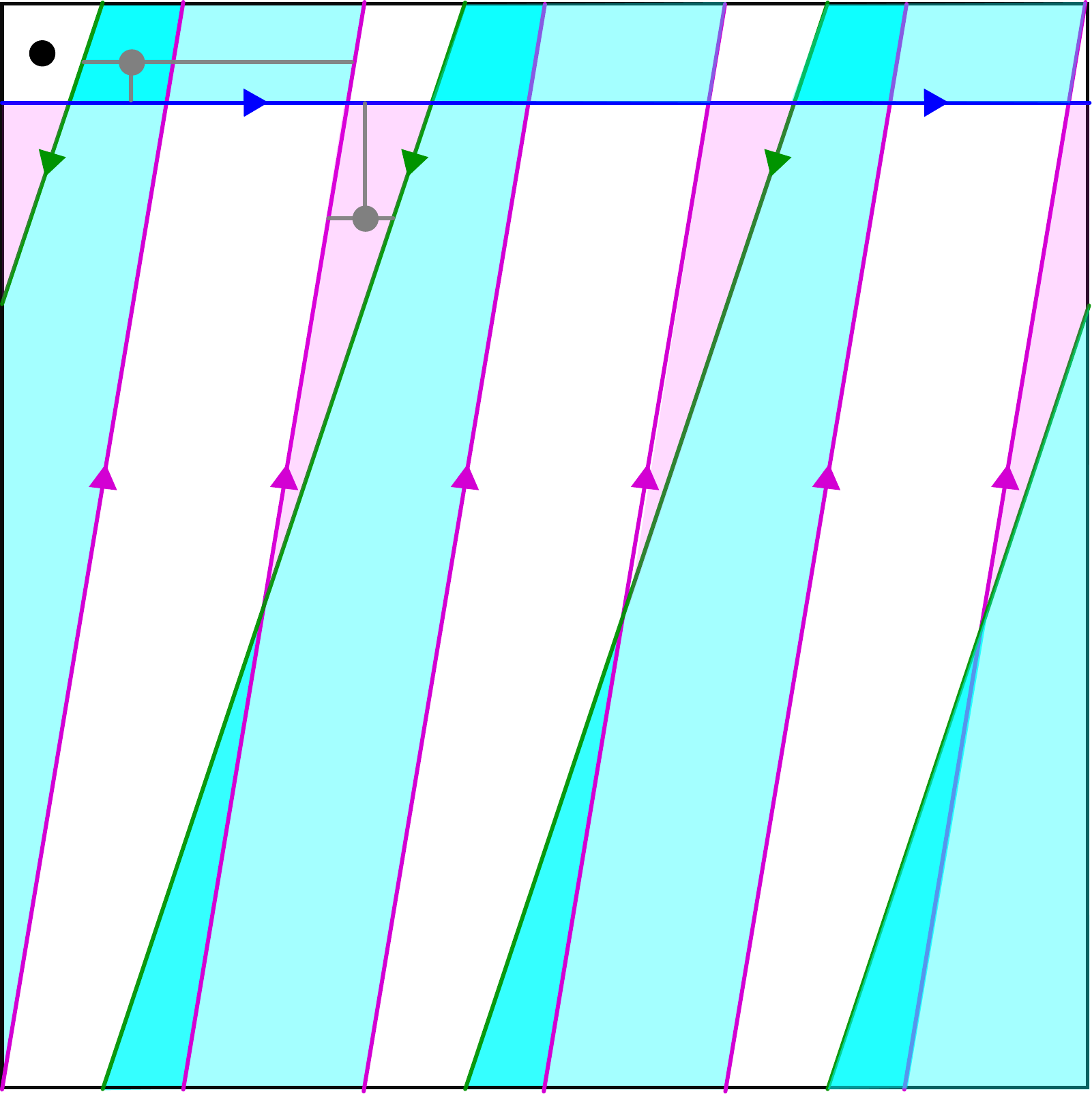}
    \caption{The triply periodic domain, $P$.}
    \label{fig:triple_periodic}
\end{figure}

Going forward, we assume that $m$ is odd. In this case, $H^2(X_{i,j,k};\Z)$ has no two torsion, so any $\frs \in \SpinC(X_{i,j,k})$ can be determined by its Chern number and its restriction to its boundary components. The Chern numbers of the $\SpinC$-structure associated to a triangle $\phi$ by \Cref{eqn:tri_to_spinC} can be computed by the formula given in \cite[Section 6]{os_holotri}. Fix a basis for $H_1(\Sigma;\Z)$, $\mu$ and $\lambda$, so that $\beta_i = \lambda + i\cdot \mu$. The kernel of the map 
\begin{align*}
    \Z \langle \beta_0, \beta_1, \beta_2 \rangle \ra H_1(\Sigma;\Z),
\end{align*}
is free of rank one, generated by the class $\beta_0 - 2\beta_1 + \beta_2.$ Therefore, this class bounds a 2-chain $P$, called a triply periodic domain, which has multiplicity zero on the region containing the basepoint. See \Cref{fig:triple_periodic}. To $P$, Ozsv\'{a}th and Szab\'{o} associate an element of $H_2(X;\Z)$, as follows: the chain $P$ can be represented by a map
\begin{align*}
    \Phi: F \ra \Sigma,
\end{align*}
where $F$ is a surface with boundary which misses the basepoint. Since $F$ has boundary contained in the set of attaching circles, $\Phi$ can be extended to a map 
\begin{align*}
    \hat{\Phi}: \hat{F} \ra  X_{\beta_0,\beta_1,\beta_2} 
\end{align*}
by capping $F \sub \Sigma \sub X_{\beta_0,\beta_1, \beta_2}$ with the compressing disks given by the attaching curves in the various $U_i$ handlebodies. Denote this surface $H(P)$. According to \cite[Proposition 6.3]{os_holotri}, the $\SpinC$ structure of a Whitney triangle $\psi$ is determined by the formula:
\begin{align}\label{eqn:chern_number}
\langle c_1(\frs_z(\psi)), H(P) \rangle = \hat{\chi}(P) + \#(\partial P) - 2n_x(P) + 2\sigma(\psi, P),
\end{align}
where $\hat{\chi}(P)$ is the Euler measure of $P$ and $\sigma(\psi, P)$ is its \emph{dual spider number}. We refer the reader to the text preceding Proposition 6.3 of \cite{os_holotri} for the definitions of these objects\footnote{We note that there is a sign error \cite{os_holotri} Prop 6.3}. It is straightforward to verify
\begin{align*}
    \hat{\chi}(P) = 0\\
    \#(\partial P) = 4\\
    n_x(P) = 0.
\end{align*}

There is a small triangle in $\psi_0 \in \pi_2(\Theta_{\beta_0,\beta_1}^0, \Theta_{\beta_1,\beta_2}^0, \Theta_{\beta_0,\beta_2}^0)$ which does not cover the basepoint. One computes that $\sigma(\psi_0,P) = -2$, and hence, 
\begin{align*}
    \langle c_1(\frs_z(\psi_0)), H(P) \rangle = 0.
\end{align*}
By fixing $\frs_0 = \frs_z(\psi_0)$, we will write elements of $\SpinC(X_{i,j,k})$ as $(a, [b]) \in \Z \oplus \Z/m$, according to the identifications in \Cref{diag:SpinC_restrictions}.

Let $\Theta_{\beta_i,\beta_j}^\pm$ be the unique intersection points in $(\Sigma, \beta_i, \beta_j)$ which carry $\HFh(-L(km, 1),[\pm 1])$ and similarly let $\Theta_{\beta_i,\beta_j}^0$ be the unique intersection point in $(\Sigma, \beta_i, \beta_j)$ which carries $\HFh(-L(km, 1),[0])$. 

\begin{lem}\label{lem:triangle_count}
Triangles $\phi_k \in \pi_2(\Theta_{\beta_0,\beta_1}^\bullet, \Theta_{\beta_1,\beta_2}^\circ, \Theta_{\beta_0,\beta_2}^\ast)$ for $\bullet,\circ,\ast \in \{0,+,-\}$ are characterized as in \Cref{tab:triangles}.
    \begin{table}[h!]
    \centering
    \begin{tabular}{c|c|c|c}
         & $\langle c_1(\frs_z(\psi_k)), H(P) \rangle$ & $\frs_z(\psi_k)$ & $n_w(\phi)$
        \\ \hline
        $\psi_0 \in \pi_2(\Theta_{\beta_0,\beta_1}^0, \Theta_{\beta_1,\beta_2}^0, \Theta_{\beta_0,\beta_2}^0)$ & $0$ & $(0,[0])$ &$ 0 $
        \\ 
        $\psi_k^{+,0,+} \in \pi_2(\Theta_{\beta_0,\beta_1}^+, \Theta_{\beta_1,\beta_2}^0, \Theta_{\beta_0,\beta_2}^+)$ & $2 + 4km$ & $(1 + 2km,[0])$ &$k(mk+1)$ 
        \\ 
        $\psi_k^{0,+,+} \in\pi_2(\Theta_{\beta_0,\beta_1}^0, \Theta_{\beta_1,\beta_2}^+, \Theta_{\beta_0,\beta_2}^+)$ & $-2 - 4km$ & $(-1 - 2km,[1])$ & $k(mk+1)$ 
        \\ 
        $\psi_k^{-,0,-} \in\pi_2(\Theta_{\beta_0,\beta_1}^-, \Theta_{\beta_1,\beta_2}^0, \Theta_{\beta_0,\beta_2}^-)$ & $-2 - 4km$ &$(-1 - 2km,[0])$ & $k(mk+1)$ 
        \\ 
        $\psi_k^{0,-,-} \in\pi_2(\Theta_{\beta_0,\beta_1}^0, \Theta_{\beta_1,\beta_2}^-, \Theta_{\beta_0,\beta_2}^-)$ & $2 + 4km$ & $(1 + 2km,[-1])$ & $k(mk+1)$ 
        \\ 
        $\psi_k^{+,-,0} \in\pi_2(\Theta_{\beta_0,\beta_1}^+, \Theta_{\beta_1,\beta_2}^-, \Theta_{\beta_0,\beta_2}^0)$ & $4 + 4km$ & $(2 + 2km,[-1])$ & $k(mk+2)+1$ 
        \\ 
        $\psi_k^{-,+,0} \in\pi_2(\Theta_{\beta_0,\beta_1}^-, \Theta_{\beta_1,\beta_2}^+, \Theta_{\beta_0,\beta_2}^0)$ & $-4 - 4km$ & $(-2 - 2km,[1])$ & $k(mk+2)+1$ 
        \\ 
    \end{tabular}
    \caption{An enumeration of the triangles appearing in the genus 1 diagram $\cD$.}
    \label{tab:triangles}
\end{table}
\end{lem}
\begin{figure}
    \centering
    \includegraphics[width=\linewidth]{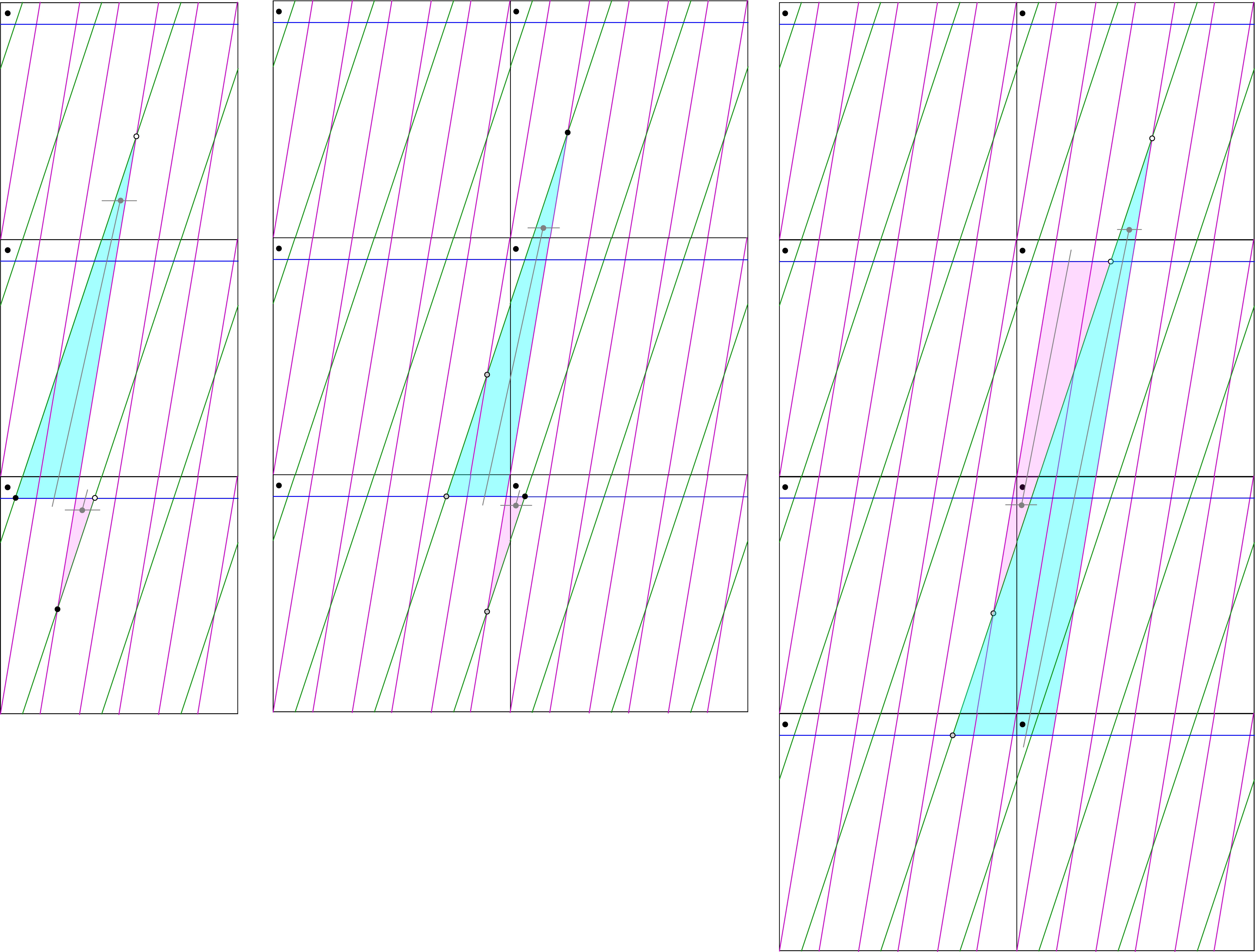}
    \caption{Six of the small triangles which contribute to the counts in \Cref{lem:triangle_count} equipped with dual spider graphs.}
    \label{fig:spider}
\end{figure}
\begin{proof}
This computation can be tediously carried out in the universal cover. The six small triangles which generate $\pi_2(\Theta_{\beta_0,\beta_1}^\bullet, \Theta_{\beta_1,\beta_2}^\circ, \Theta_{\beta_0,\beta_2}^\ast)$ are shown in \Cref{fig:spider}. Those triangles are decorated with \emph{dual spider graphs} which are used to compute the Chern numbers. 

For instance, the triangle $\psi_0^{+,0,+}$ has dual spider number $-1$. When $k \le 0$, we have that $\psi_k^{+,0,+} \sub \psi_{k+1}^{+,0,+}$ and $\sigma(\psi_{k+1}^{+,0,+}, P)-\sigma(\psi_{k}^{+,0,+}, P) = 2m$, so inductively we have
\begin{align*}
    \sigma(\psi_{k}^{+,0,+}, P) = -1 + 2mk.
\end{align*}
By \Cref{eqn:chern_number}, it follows that 
\begin{align*}
    \langle c_1(\frs_z(\psi_k^{+,0,+})),H(P)\rangle &= 4 + 2( -1 + 2mk) \\
    & = 2 + 4mk.
\end{align*}
By restricting $\frs_z(\psi_k^{+,0,+})$ to its boundary, we obtain the $\SpinC$ structure $[1]\amalg [0]\amalg [1]$. The restriction maps are given as in \Cref{diag:SpinC_restrictions}, forcing $\frs_z(\psi_k^{+,0,+}) = (1+2km,[0])$. The other computations are similar.
\end{proof}

\begin{rem}
    Thus far, we have completely ignored the fourth attaching circle, $\beta_3$. Counts of triangles between the curves $\{\beta_i\}_{i=1,2,3}$ are identical to those for $\{\beta_i\}_{i=0,1,2}$ since they differ by a reparametrization of the torus. Furthermore, in the universal cover, it is not hard to see that in the diagram containing all four sets of curves, $(\Sigma, \beta_0, \beta_1, \beta_2, \beta_3)$, there are \emph{no} rigid rectangles with the correct boundary conditions; since the slopes of each successive curve increase, it is easy to see that any rectangle with the appropriate boundary conditions cannot be convex. Hence all higher polygon counting maps vanish.
\end{rem}

We organize the $\SpinC$ structures appearing in \Cref{tab:triangles} into the following collections:
\begin{align*}
    \frak S_{k}^{0,+} &= \{ (1 + 2mk, [0]), (-1-2mk,[1])\} \sub \SpinC(X_{0,1,2}) \\
    \frak S_{k}^{0,-} &= \{(-1-2mk,[0]), (1+2mk,[-1])\} \sub \SpinC(X_{0,1,2}) \\
    \frak S_{k}^{+,-} &= \{(2+2mk,[-1]), (-2-2mk,[1])\} \sub \SpinC(X_{0,1,2}).
\end{align*}
Since the restriction map 
\begin{align*}
    \SpinC(X_{0,1,2,3}) \ra \SpinC(X_{0,1,2}) \times \SpinC(X_{0,2,3})
\end{align*}
is injective, any subset
\begin{align}\label{eqn:0123_SpinC_cubes}
    \frak S_{0,1,2,3} \sub \coprod_{\bullet, \circ, \ast \in \{0,-,+\}}  \frak S_{k}^{\bullet,\circ} \times  \frak S_{k}^{\circ,\ast}
\end{align}
forms a hypercube of $\SpinC$-structures on $X_{0,1,2,3}$.

\begin{prop}\label{prop:main_beta_hypercube}
    The diagram 
    \begin{align}\label{fig:framing_hyper_cube}
        \begin{tikzcd}[column sep={3.2cm,between origins},row sep={1.5cm,between origins},labels=description,ampersand replacement = \&]
        \bm \beta_0
        	\ar[dd, swap,"\Theta_{\beta_0,\beta_1}^{-}"]
        	\ar[dr,  "\Theta_{\beta_0,\beta_1}^{0}"]
        	\ar[rr, " \Theta_{\beta_0,\beta_1}^{+}"]
        \&\&[-.8cm]
        \bm \beta_1
        	\ar[dd, "\Theta_{\beta_1,\beta_2}^{-}"]
        	\ar[dr,"\Theta_{\beta_1,\beta_2}^{0}"]
        \&
        \\
        \&\bm \beta_1
        \&\&
        \bm \beta_2
        	\ar[dd, "\Theta_{\beta_2, \beta_3}^-"]
        	\ar[from=ll,crossing over, "\Theta_{\beta_1, \beta_2}^+"]
        \\[2cm]
        \bm \beta_1
        	\ar[rr, "\Theta_{\beta_1,\beta_2}^{+}"]
        	\ar[dr,"\Theta_{\beta_1, \beta_2}^0"]
        \&\&\bm \beta_2
        	\ar[dr, " \Theta_{\beta_2, \beta_3}^0"]	
        \&
        \\
        \&
        \bm \beta_2
        	\ar[rr, "\Theta_{\beta_2, \beta_3}^+"]
        	\ar[from =uu, crossing over,"\Theta_{\beta_1, \beta_2}^-"]
        	\&\&
        \bm \beta_3
    \end{tikzcd}
    \end{align}
    is a hypercube of attaching curves for any hypercube of $\SpinC$ structures as in \Cref{eqn:0123_SpinC_cubes}.
\end{prop}
\begin{proof}
    It follows immediately from the count of triangles in \Cref{lem:triangle_count} that each 2-face of the hypercube strictly commutes. For example, the triangle $\psi_k^{+,0,+} \in \pi_2(\Theta_{\beta_0,\beta_1}^+, \Theta_{\beta_1,\beta_2}^0, \Theta_{\beta_0,\beta_2}^+)$ in $\SpinC$ structure $(1 + 2km,[0])$ is canceled by the triangle $\psi_k^{0,+,+} \in \pi_2(\Theta_{\beta_0,\beta_1}^0, \Theta_{\beta_1,\beta_2}^+, \Theta_{\beta_0,\beta_2}^+)$ with $\SpinC$ structure $(-1 - 2km,[1])$. 

    Since the faces strictly commute, we can take all length $n$ arrows for $n \ge 2$ to be identically zero. The higher hypercube relations are therefore trivially satisfied.

    The case that $g > 1$ now follows from \Cref{lem:multi_stabilization}.
\end{proof}

\section{Telescopes and large surgeries}\label{sec: Telescopes and large surgeries}
A fundamental result in Heegaard Floer theory is the \emph{large surgery formula}.

\begin{thm}{\cite{os_knotinvts,rasmussen_knotcompl}}
    Given any knot $K \subset S^3$ and integer $N>0$, consider the canonical genus zero link cobordism
    \[
    D_{K,N}:S^3_N(K)\rightarrow (S^3,K),
    \]
    which, for each $s\in \mathbb{Z}$, induces a map 
    \[
    F_{D_{K,N},[s]}:\CF^-(S^3 _N(K),[s])\rightarrow \CFK_{\F[U,V]}(S^3,K,s),
    \]
    where $\CFK_{\F[U,V]}(S^3,K,s)$ denotes the summand of $\CFK_{\cR}(S^3,K)$ in Alexander grading $s$. Then, whenever $N>g_3(K)+|s|$, the map $F_{D_{K,N},[s]}$ is a homotopy equivalence whose degree shift is $\frac{1-N}{4}$, where its domain is given the absolute $\mathbb{Q}$-grading and codomain is given the Maslov grading.
\end{thm}

Roughly, the idea is to study the map induced by the cobordism $(S^3,K) \ra S^3_N(K)$ given by attaching an $N$-framed 2-handle along $K$. We think of this as a link cobordism where $K$ is capped off by the core of the 2-handle. This map can be computed by choosing a surgery diagram $\cD = (\Sigma, \bm \alpha_1, \bm \alpha_2, \bm \beta, w, z)$ for $K$ (in the sense of \Cref{def:surgery_diagram}) and proving that the map given by counting triangles sets up a bijection between intersection points representing generators of $\CFK_\cR(K, s)$ and intersection points of $\CFh(S^3_n(K),[s])$.

In this section, we show that this process can be reversed; namely, given $\{\CFh(S^3_n(K))\}_{n \in \N}$, one can recover $\CFK_\cR(K)$. Intuitively, in identifying $\CFK_\cR(K,[s])$ with $\CFh(S^3_n(K),[s]),$ one sacrifices the $\F[U, V]$-module structure on $\CFK_\cR(K,[s])$; while we cannot equip $\CFh(S^3_n(K),[s])$ with a $\F[U,V]$-module structure, it turns out we \emph{can} equip $\bigoplus_{n\in \N,s\in \Z} \CFh(S^3_n(K),[s])$ with a $\F[U,V]$-module structure by considering various cobordisms 
\begin{align*}
    (S^3_n(K),[s]) \ra (S^3_{n+m}(K),[s \pm 1]).
\end{align*}

We will show that the collection $\{\CFh(S^3_n(K),[s])\}_{n,s}$ can naturally be assembled into a mapping telescope, $\Tel(K)$. Furthermore, this telescope is compatible with certain cobordism maps which will allow us to define a pair of endomorphisms of $\Tel(K)$ which \emph{up to homotopy} behave like an $\F[U,V]$-action on the telescope. More precisely, we show that these endomorphisms of $\Tel(K)$ give rise to a quasimodule structure; the total telescope of this quasimodule has an honest $\F[U,V]$-action, and, essentially by the nature of the construction, the resulting object is a model for $\CFK_\cR(K).$ Keeping in mind this bird's eye view, we turn to the construction of the surgery telescopes. 

\subsection{The Surgery Telescope}

Throughout this section, fix a knot $K \sub S^3$. Consider the three-component link in \Cref{fig:framing_cobordism}. Let $\cD = (\Sigma, \bm \alpha\cup \alpha_K, \bm \beta \cup \lambda,w, z)$ be a doubly pointed Heegaard diagram for $K$ with the following properties:
\begin{enumerate}
    \item $\cD_{K, \lambda} = (\Sigma, \bm \alpha\cup \alpha_K, \bm \beta \cup \lambda,w, z)$ is a meridional Heegaard diagram for $K$, i.e. $w$ and $z$ are on opposite sides of $\alpha_K$, and $\alpha_K$ is a meridian for $K$.
    \item $\cD_\emptyset = (\Sigma, \bm \alpha\cup \alpha_K, \bm \beta,w, z)$ is a doubly pointed Heegaard diagram for $((S^1\times S^2)\smallsetminus \nu(\lambda), K) $, where $\lambda$ is the curve shown in \Cref{fig:framing_cobordism}.
\end{enumerate}

\begin{figure}[h]
        \def\svgwidth{.8\linewidth}
\begingroup%
  \makeatletter%
  \providecommand\color[2][]{%
    \errmessage{(Inkscape) Color is used for the text in Inkscape, but the package 'color.sty' is not loaded}%
    \renewcommand\color[2][]{}%
  }%
  \providecommand\transparent[1]{%
    \errmessage{(Inkscape) Transparency is used (non-zero) for the text in Inkscape, but the package 'transparent.sty' is not loaded}%
    \renewcommand\transparent[1]{}%
  }%
  \providecommand\rotatebox[2]{#2}%
  \newcommand*\fsize{\dimexpr\f@size pt\relax}%
  \newcommand*\lineheight[1]{\fontsize{\fsize}{#1\fsize}\selectfont}%
  \ifx\svgwidth\undefined%
    \setlength{\unitlength}{623.62204724bp}%
    \ifx\svgscale\undefined%
      \relax%
    \else%
      \setlength{\unitlength}{\unitlength * \real{\svgscale}}%
    \fi%
  \else%
    \setlength{\unitlength}{\svgwidth}%
  \fi%
  \global\let\svgwidth\undefined%
  \global\let\svgscale\undefined%
  \makeatother%
  \begin{picture}(1,0.34545455)%
    \lineheight{1}%
    \setlength\tabcolsep{0pt}%
    \put(0,0){\includegraphics[width=\unitlength,page=1]{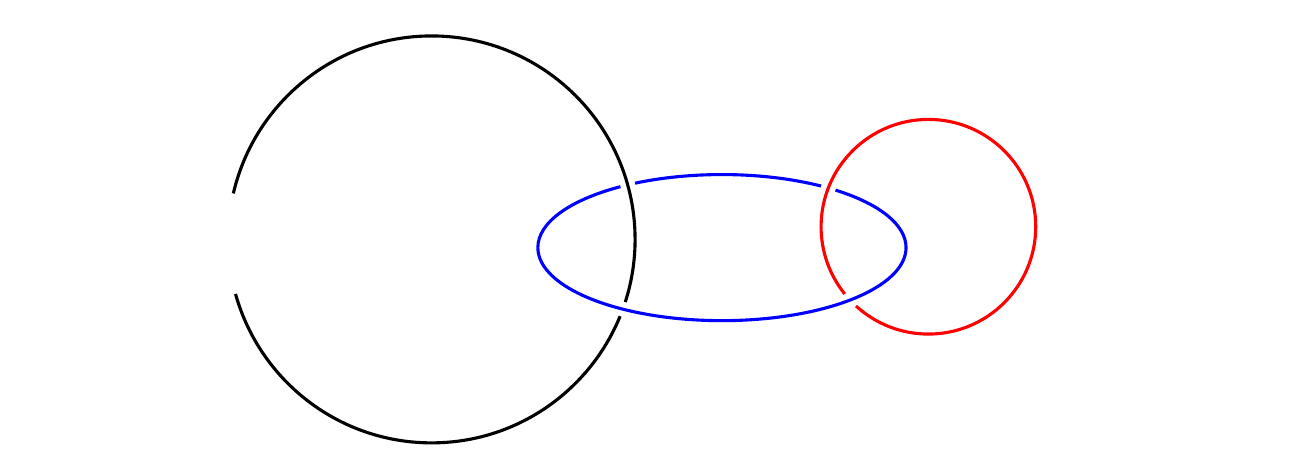}}%
    \put(0.17951122,0.14430473){\color[rgb]{0,0,0}\makebox(0,0)[t]{\lineheight{1.25}\smash{\begin{tabular}[t]{c}{$K$}\end{tabular}}}}%
    \put(0,0){\includegraphics[width=\unitlength,page=2]{framing_cobordism.pdf}}%
    \put(0.71568636,0.27101383){\color[rgb]{0,0,0}\makebox(0,0)[t]{\lineheight{1.25}\smash{\begin{tabular}[t]{c}{$\lambda$}\end{tabular}}}}%
    \put(0.55748327,0.06737803){\color[rgb]{0,0,0}\makebox(0,0)[t]{\lineheight{1.25}\smash{\begin{tabular}[t]{c}{$0$}\end{tabular}}}}%
  \end{picture}%
\endgroup%

            \caption{The framed link used in the definition of $\cD_{K, \lambda}$.}
        \label{fig:framing_cobordism}
        \end{figure}

Define $\cD_{N,m}$ to be the diagram obtained from $\cD_{K,\lambda}$ by replacing $\alpha_K$ with an $N$-framed longitude of $K$, $\alpha_N$, and replacing $\lambda$ with a curve, $\beta_m$, representing an $m$-framed longitude of $\lambda$. By construction, $\cD_{N,m}$ represents the 3-manifold $S^3_{N+m}(K)$. Moreover, $(\Sigma, \bm \alpha\cup \alpha_K, \bm \beta\cup \beta_k, w, z)$ is a surgery diagram for all $k$ and $(\Sigma, \bm \beta\cup \beta_{0}, \bm \beta\cup \beta_{km},\bm \beta\cup \beta_{(k+1)m}, \bm\beta\cup \beta_{(k+2)m}, w, z)$ is an $m$-framing quadruple diagram (cf. \Cref{def:surgery_diagram} and \Cref{def:framing_diagram}). 

As discussed in Sections \ref{sec:tel_hyp} and \ref{sec:main_hypercube}, each $\CF(\bm \beta\cup \beta_{km}, \bm \beta \cup \beta_{(k+1)m})$ represents $-L(m,1)$, and has three distinguished generators, which we denote $\Theta^0$, $\Theta^+$, and $\Theta^-$ (we will drop $k$ and $m$ from the notation for brevity). Consider the positive ray of attaching curves determined by the $\Theta^0$ cycles:
\begin{align*}
    \frak T_{\beta}:= \bm\beta\cup\beta_{0} \xra{\Theta^0} \bm\beta\cup\beta_{m} \xra{\Theta^0} \bm\beta\cup\beta_{2m} \ra \hdots
\end{align*}
Implicitly, each arrow $\bm\beta\cup\beta_{km} \xra{\Theta^0} \bm\beta\cup\beta_{(k+1)m}$ is equipped with the $\SpinC$ structure $[0] \in \SpinC(X_{\beta^{km}, \beta^{(k+1)m}})\cong \SpinC(Y_{\beta^{km}, \beta^{(k+1)m}}) = \Z/m$. Consider the pair $(\bm \alpha \cup \alpha_N, \bm \beta \cup \beta_{km})$ and let $\frak S_{N,km,s}$ be the hypercube of $\SpinC$ structures generated by $([s],[0]) \in \SpinC(S^3_{N+km}(K))\times \SpinC(-L(m,1))$. This data specifies a collection of morphisms 
\begin{align*}
    \CF(\bm \alpha \cup \alpha_N, \bm\beta\cup\beta_{km},[s]) \xra{F^{\beta_{m} \ra \beta_{(k+1)m}}} \CF(\bm \alpha \cup \alpha_N, \bm\beta\cup\beta_{(k+1)m}) \rightarrow \CF(\bm \alpha \cup \alpha_N, \bm\beta\cup\beta_{(k+1)m},[s]),
\end{align*}
where the second arrow is projection. Denote this composition by $\frak F^0_{k}$. These maps clearly fit together to form a ray of attaching complexes. Define $\Tel(K, s):=\Tel(\CF(\bm \alpha \cup \alpha_N,  \frak T_\beta, s))$ to be the mapping telescope of this ray, and $\Tel(K) := \bigoplus_s \Tel(K, s)$. 

\begin{prop}\label{prop:TelK_is_large_surgery}
    For sufficiently large $N$ and $m$, the maps 
    \begin{align*}
         \CF(\bm \alpha \cup \alpha_N, \bm\beta\cup\beta_{km},[s]) \xra{\frak F^0_k} \CF(\bm \alpha \cup \alpha_N, \bm\beta\cup\beta_{(k+1)m},[s])
    \end{align*}
    are homotopy equivalences. In particular, $\Tel(K, s)$ is homotopy equivalent to the $\CF(S^3_{N+km}(K), [s])$. Further, it follows that the inclusion of $\CF(S^3_{N+km}(K), [s])$ into $\Tel(K, s)$ is a homotopy equivalence (see \ref{sec:tel_hyp}).
\end{prop}
\begin{proof}
    To make things more concrete, recall that $\CF(\bm \alpha \cup \alpha_N, \bm\beta\cup\beta_{km})$ is a model for $\CF(S^3_{N+km}(K),[s])$. The triangle counting map 
    \begin{align*}
        \CF(\bm \alpha \cup \alpha_N, \bm\beta\cup\beta_{km},[s]) \xra{F^{\beta_{km} \ra \beta_{(k+1)m}}} \CF(\bm \alpha \cup \alpha_N, \bm\beta\cup\beta_{(k+1)m})
    \end{align*}
    is the cobordism map for the 2-handle attachment cobordism $S^3_{N+km}(K)\sqcup L(m,1)\rightarrow S^3_{N+(k+1)m}(K)$, where we evaluate at the generator of $\widehat{CF}(L(m,1))$ for a fixed $\mathrm{Spin}^c$ structure on $L(m,1)$. We refer to this cobordism as the \emph{change of framing} cobordism.  
    
    Let $\Gamma_{r, [s]}: \CFh(S^3_r(K),[s]) \ra \CFK_\cR(K,s)$ be the large surgery isomorphism. This map can be realized as a cobordism map as follows: Consider the link cobordism
    \begin{align*}
        (S^3, K) \xra{(X_r(K),C)} (S^3_r(K), \emptyset),
    \end{align*}
    where $X_r(K)$ is the cobordism given by attaching an $r$-framed 2-handle along $K$ and $C$ is the core of the 2-handle. The reverse of this cobordism realizes the large surgery isomorphism. Moreover, this cobordism commutes with the change of framing cobordism. This means we have a homotopy commutative square:
    \begin{align*}
        \begin{tikzcd}[ampersand replacement = \&, column sep = large, row sep = large]
        \CFh(S^3_{N+km}(K),[s])
        \ar[r,"\frak F^0_k"]
        \ar[d,"\Gamma_{N+km,[s]}"]
        \& \CFh(S^3_{N+(k+1)m}(K),[s])
        \ar[d,"\Gamma_{N+(k+1)m,[s]}"]\\
        \CFK_\cR(K, s)
        \ar[r,"\frak F^0_k"] \&
        \CFK_\cR(K, s).
        \end{tikzcd}
    \end{align*}
    The vertical arrows are the large surgery isomorphisms, and bottom arrow is the map obtained by pairing $\bm \alpha_K$ and $\frak T_\beta$ and taking the Alexander grading preserving part of the map
\begin{align*}
    \CF(\bm \alpha \cup \alpha_K, \bm\beta\cup\beta_{km},s) \xra{F^{\beta_{m} \ra \beta_{(k+1)m}}} \CF(\bm \alpha \cup \alpha_K, \bm\beta\cup\beta_{(k+1)m}) \rightarrow \CF(\bm \alpha \cup \alpha_K, \bm\beta\cup\beta_{(k+1)m},s).
\end{align*}
    This map is a homotopy equivalence by the proof of \cite[Lemma 3.13]{guthkang2024invariantsplittingprinciples}. Hence, the upper horizontal arrow is a homotopy equivalence as well.
\end{proof}

We note that while the various diagrams utilized in defining $\Tel(K)$ have been dropped from the notation, the additional structure it enjoys does a priori depend on those choices. 

\subsection{Homotopy $U$ and $V$ Actions}

Our next goal is to equip $\Tel(K)$ with a module structure. All of what follows is a consequence of the existence of the hypercube appearing in \Cref{prop:main_beta_hypercube}. Consider one of the 2-faces in that hypercube of the form 
\begin{align}\label{diag:htpyV_diagram}
    \cL_\frV = \begin{tikzcd}[ampersand replacement = \&, column sep = large, row sep = large]
        \bm \beta \cup \beta_{0} 
        \ar[r,"\Theta^0"]
        \ar[d,"\Theta^+"] 
        \&
        \bm \beta \cup\beta_{1} 
        \ar[d,"\Theta^+"]
        \\
        \bm \beta \cup\beta_{1} 
        \ar[r,"\Theta^0"]
        \&
        \bm \beta \cup \beta_{2},
    \end{tikzcd}
\end{align}
equipped with the hypercube of $\SpinC$-structures $\frak S_0^{0,+}$. This hypercube is precisely the data of a (regular) map of rays $\frak T_\beta \ra \frak T_\beta$. Consider now the pair $(\bm \alpha \cup \alpha_N, \cL_\frV)$, and let $\frak S_{N, k,s}^{0,+}$ be the hypercube of $\SpinC$ structures generated by $\{[s]\} \times \frak S_0^{0,+} \sub \SpinC(S^3_{N+km}(K))\times \SpinC(X_{\beta_0, \beta_1, \beta_2})$. This pairing specifies a collection of diagrams of the form
\begin{align*}
    \begin{tikzcd}[ampersand replacement = \&, column sep = huge, row sep = huge]
        \CF(\bm \alpha \cup \alpha_N, \bm\beta\cup\beta_{km},[s]) 
        \ar[rr,"\frak F^0_k"]
        \ar[dd,"F^{\beta_{km} \ra \beta_{(k+1)m}}" left]
        \ar[ddrr,"F^{\beta_{km} \ra \beta_{(k+1)m}\ra \beta_{(k+2)m}}",dashed]
        \& \&
        \CF(\bm \alpha \cup \alpha_N, \bm\beta\cup\beta_{(k+1)m},[s]) 
        \ar[dd,"F^{\beta_{(k+1)m} \ra \beta_{(k+2)m}}" right]
        \\
        \\
        \CF(\bm \alpha \cup \alpha_N, \bm\beta\cup\beta_{(k+1)m},[s+1]) 
        \ar[rr,"\frak F^0_{k+1}"]
        \& \&
        \CF(\bm \alpha \cup \alpha_N, \bm\beta\cup\beta_{(k+2)m},[s+1]).
    \end{tikzcd}
\end{align*}
Here we have implicitly composed the vertical and diagonal arrows with the projection
\begin{align*}
    \CF(\bm \alpha \cup \alpha_N, \bm\beta\cup\beta_{km}) \ra \CF(\bm \alpha \cup \alpha_N, \bm\beta\cup\beta_{km},[s+1]).
\end{align*}
Henceforth, these vertical compositions will be denoted $\frak F^+_k$, and the diagonal compositions $\frak H_k^{0,+}$. Clearly, these diagrams glue together to define maps 
\begin{align*}
    \frV_s: \Tel(K, s)= \Tel(\bm \alpha \cup \alpha_N, \frak T_\beta, s) \ra \Tel(\bm \alpha \cup \alpha_N, \frak T_\beta, s+1) = \Tel(K, s+1).
\end{align*}
Summing over all $s$, we define
\begin{align*}
    \frV: \Tel(K) \ra \Tel(K)
\end{align*}
In exactly the same way, we may consider the hypercube
\begin{align*}
    \cL_\frU = \begin{tikzcd}[ampersand replacement = \&, column sep = large, row sep = large]
        \bm \beta \cup \beta_{0} 
        \ar[r,"\Theta^0"]
        \ar[d,"\Theta^-"] 
        \ar[dr,"\Theta^*"]
        \&
        \bm \beta \cup\beta_{1} 
        \ar[d,"\Theta^-"]
        \\
        \bm \beta \cup\beta_{1} 
        \ar[r,"\Theta^0"]
        \&
        \bm \beta \cup \beta_{2},
    \end{tikzcd}
\end{align*}
and $\SpinC$ structures $\frak S^{0,-}_0$ to construct maps $\frU_s: \Tel(K, s) \ra \Tel(K, s-1)$, and we write $$\frU: \Tel(K) \ra \Tel(K)$$ for their sum.

\begin{prop}\label{prop:UV commute}
The endomorphisms $\frU$ and $\frV$ constructed above commute up to homotopy. More specifically, for any $s \in \Z$, there is a natural map $\frak W_s: \Tel(K, s) \ra \Tel(K, s)$ which makes the following diagram
    \begin{align}\label{diag:UV_commutation}
    \begin{tikzcd}[ampersand replacement = \&, column sep = large, row sep = large]
        \Tel(K,[s]) 
        \ar[r,"\frU_s"]
        \ar[d,"\frV_{s}"]
        \ar[dr,"\frak W_{s}", dashed]
        \&
        \Tel(K,[s-1])
        \ar[d,"\frV_{s-1}"]
        \\
        \Tel(K,[s+1]) 
        \ar[r,"\frU_{s+1}"]
        \&
        \Tel(K,[s]).
    \end{tikzcd}
\end{align}
a hypercube of chain complexes.
\end{prop}
\begin{proof}
    It should be no surprise that this diagram is produced by pairing $\bm \alpha \cup \alpha_N$ with the full hypercube from \Cref{prop:main_beta_hypercube}. Let $\cL_\beta$ denote the full hypercube and equip it with the hypercube of $\SpinC$ structures 
    \begin{align*}
        \frak S_0 = \{\frak S_0^{0,+},\frak S_0^{0,-},\frak S_0^{+,-} \}.
    \end{align*}
    and let $\frak S_{N,k,s}$ be the hypercube of $\SpinC$ structures generated by $\{[s]\} \times \frak S_0 \sub \SpinC(S^3_{N+km}(K))\times \SpinC(X_{\beta_0, \beta_1, \beta_2, \beta_3})$. The pairing $(\bm \alpha \cup \alpha_N, \cL_\beta)$ then produces a 3-dimensional hypercube extending \Cref{diag:htpyV_diagram}. There are six faces: four are the 2-dimensional hypercubes defining the actions of $\frU_s$ and $\frV_s$. The front and back faces are new (and come from the back and front face of \Cref{prop:main_beta_hypercube}). Taking the telescope of the ray formed by these hypercubes gives exactly \Cref{diag:UV_commutation}.
\end{proof}

It follows from \Cref{prop:UV commute} that we may form the 2-dimensional quasimodule shown in \Cref{diag:htpy_UV_ray} whose vertices are all identified with $\Tel(K)$, with horizontal arrows given by $\frU$, vertical arrows given by $\frV$, and diagonal arrows given by $\frak W$. 

\begin{align}\label{diag:htpy_UV_ray}
    \begin{tikzcd}[ampersand replacement = \&,column sep = huge, row sep = huge]
    \vdots 
    \ar[d,"\frV"]
    \&
    \vdots
    \ar[d,"\frV"]
    \& 
    \vdots
    \ar[d,"\frV"]
    \&
    \\
    \Tel(K) 
    \ar[d,"\frV"]
    \& 
    \Tel(K) 
    \ar[d,"\frV"]
    \ar[l,"\frU"]
    \ar[dl,"\frak W", dashed]
    \& 
    \Tel(K) 
    \ar[d,"\frV"]
    \ar[l,"\frU"]
    \ar[dl,"\frak W", dashed]
    \& 
    \hdots
    \ar[l,"\frU"]
    \\
    \Tel(K) 
    \ar[d,"\frV"]
    \& 
    \Tel(K) 
    \ar[d,"\frV"]
    \ar[l,"\frU"]
    \ar[dl,"\frak W", dashed]
    \& 
    \Tel(K) 
    \ar[d,"\frV"]
    \ar[l,"\frU"]
    \ar[dl,"\frak W", dashed]
    \& 
    \hdots
    \ar[l,"\frU"]
    \\
    \Tel(K) 
    \& 
    \Tel(K) 
    \ar[l,"\frU"]
    \& 
    \Tel(K) 
    \ar[l,"\frU"]
    \& 
    \hdots
    \ar[l,"\frU"]
    \end{tikzcd}
\end{align}
The total telescope of this quasimodule is an honest $\F[\frU, \frV]$-module, which will be denoted $\frak{CFK}(K)$. According to \Cref{lem:same_homology}, this complex is quasi-isomorphic to $\Tel(K)$ as a vector space, which is in turn isomorphic to infinitely many copies of large surgery on $K$, much like $\CFK(S^3, K)$ itself. Our next goal is to relate the module actions on these two complexes.

\subsection{A second module structure on \texttt{$\CFK_\cR(S^3,K)$}}

The entire construction in the previous section amounted to studying the hypercube obtained by pairing $\bm \alpha \cup \alpha_N$ with $\frak T_\beta$. We could just as well applied this strategy to the attaching curve $\bm \alpha \cup \alpha_K$. However, in this case, we obtain a hypercube
\begin{align*}
    \CF(\bm \alpha \cup \alpha_K, \frak T_\beta).
\end{align*}
This is a 3-dimensional hypercube, though we view it as a 2-dimensional hypercube with vertices 1-dimensional hypercubes
\begin{align*}
        \CF(\bm \alpha \cup \alpha_K, \bm\beta\cup\beta_{km},[s]) \xra{F^{\beta_{km} \ra \beta_{(k+1)m}}} \CF(\bm \alpha \cup \alpha_K, \bm\beta\cup\beta_{(k+1)m})
    \end{align*}
whose vertices are models for $\CFK_\cR(S^3, K)$, since, rather than filling $S^3 \smallsetminus K$ with the $N$-framed solid torus specified by $\alpha_N$, we fill it with the $\infty$-framed solid torus specified by $\alpha_K$. The same construction as above yields a 2-dimensional $\F[\frU,\frV]$-quasimodule:
\begin{align*}
    \begin{tikzcd}[ampersand replacement = \&,column sep = huge, row sep = huge]
    \vdots 
    \ar[d,"\frV"]
    \&
    \vdots
    \ar[d,"\frV"]
    \& 
    \vdots
    \ar[d,"\frV"]
    \&
    \\
    \CFK(S^3, K)
    \ar[d,"\frV"]
    \& 
    \CFK(S^3, K)
    \ar[d,"\frV"]
    \ar[l,"\frU"]
    \ar[dl,"\frak W", dashed]
    \& 
    \CFK(S^3, K)
    \ar[d,"\frV"]
    \ar[l,"\frU"]
    \ar[dl,"\frak W", dashed]
    \& 
    \hdots
    \ar[l,"\frU"]
    \\
    \CFK(S^3, K)
    \ar[d,"\frV"]
    \& 
    \CFK(S^3, K)
    \ar[d,"\frV"]
    \ar[l,"\frU"]
    \ar[dl,"\frak W", dashed]
    \& 
    \CFK(S^3, K)
    \ar[d,"\frV"]
    \ar[l,"\frU"]
    \ar[dl,"\frak W", dashed]
    \& 
    \hdots
    \ar[l,"\frU"]
    \\
    \CFK(S^3, K)
    \& 
    \CFK(S^3, K)
    \ar[l,"\frU"]
    \& 
    \CFK(S^3, K)
    \ar[l,"\frU"]
    \& 
    \hdots
    \ar[l,"\frU"]
    \end{tikzcd}.
\end{align*}
Here, by abuse of notation, the endomorphisms induced by pairing with the hypercubes $\cL_\frU$ and $\cL_\frV$ are also labeled $\frU$ and $\frV$. By taking the total telescope, we produce a $\F[\frU,\frV]$-module structure on $\CFK_\cR(S^3,K)$. We denote this module $\CFK_{\F[\frU,\frV]}(S^3, K)$, to distinguish it from $\CFK_\cR(S^3,K)$ with its usual $\F[U,V]$-module structure (though we will show the two objects are equivalent in the next section).

\begin{prop}
    The hypercube of attaching curves 
    \begin{align*}
        \cL_\alpha = \bm \alpha \cup \alpha_N \ra \bm \alpha \cup \alpha_K,
    \end{align*}
    induces a homotopy equivalence 
    \begin{align*}
        \Gamma: \frak{CFK}(K) \ra \CFK_{\F[\frU,\frV]}(S^3, K),
    \end{align*}
    as $\F[\frU, \frV]$-modules.
\end{prop}
\begin{proof}
    By definition, $\frak{CFK}(K)$ and $\CFK_{\F[\frU,\frV]}(S^3, K)$ are defined as the total telescope of 2-rays constructed from the hypercube of attaching curves $\cL_\beta$. A map between these 2-rays is a map of hypercubes between their $(i,j)$-slices. 

    The first observation is that by construction, for any $\beta_i$, pairing with $\cL_\alpha$ exactly realizes the large surgery map:
    \begin{align*}
        \CFh(S^3_{N+mk}(K),[s]) = \CF(\bm \alpha \cup \alpha_N, \bm \beta \cup \beta_i,[s]) \xra{\Gamma_s} \CFK_\cR(\bm \alpha \cup \alpha_K, \bm \beta \cup \beta_i) = \CFK(S^3, K, s).
    \end{align*}
    This map induces a homotopy equivalence 
    \begin{align*}
        \Tel(K, s) \ra \CFK_\cR(S^3, K, s)
    \end{align*}
    of $\F[U,V]$-modules. These form the length one arrows in our hypercube.
    
    \begin{align}
        \begin{tikzcd}[column sep={3.2cm,between origins},row sep={1.5cm,between origins},labels=description,ampersand replacement = \&]
        \Tel(K, s)
        	\ar[dd, swap,"\frV"]
        	\ar[dr,  "\Gamma_{s}"]
        	\ar[rr, " \frU"]
        \&\&[-.8cm]
        \Tel(K, s-1)
        	\ar[dd, "\frV"]
        	\ar[dr,"\Gamma_{s-1}"]
        \&
        \\
        \&\CFK(S^3, K, s)
        \&\&
        \CFK(S^3, K, s-1)
        	\ar[dd, "\frV"]
        	\ar[from=ll,crossing over, "\frU"]
        \\[2cm]
        \Tel(K, s+1)
        	\ar[rr, "\frU"]
        	\ar[dr,"\Gamma_{s+1}"]
        \&\&\Tel(K, s)
        	\ar[dr, " \Gamma_{s}"]	
        \&
        \\
        \&
        \CFK(S^3, K, s+1)
        	\ar[rr, "\frU"]
        	\ar[from =uu, crossing over,"\frV"]
        	\&\&
        \CFK(S^3, K, s)
    \end{tikzcd}
    \end{align}
    The higher maps (length 2 and 3) are given by pairing the higher faces of $\cL_\beta$ with $\cL_\alpha$. That these maps satisfy the hypercube structure equations is a consequence of the fact that pairing hypercubes of attaching circles yields hypercubes of chain complexes. Since the length 1 arrows are homotopy equivalences, the total map on homology is a quasi-isomorphism of $\F[U,V]$-modules. Though, as in \Cref{lem: invertibility of hypercube qis}, this implies the existence of a homotopy equivalence between the two quasimodules.
\end{proof}

\subsection{Computing $\frU$ and $\frV$ for $\CFK$}

The final step is to relate the module structure on $\CFK_{\F[\frU,\frV]}(S^3, K)$ to the usual module structure on $\CFK_\cR(S^3, K)$. Intuitively, the actions of $\frU$ and $\frV$ arise from changing the framing of $K$. This twisting happens in a neighborhood of a meridian of $K$, so in some sense, should not depend in any intrinsic way on the choice of knot. Therefore, our strategy is to show that for the unknot, $\frU$ and $\frV$ act precisely by multiplication by $U$ and $V$ respectively. To prove the general case, we consider the link cobordism which splits off an unknot from $K$. This cobordism induces a homotopy equivalence on $\CFK$ and commutes with the change of framing. Therefore, it follows that under the identification  
\begin{align*}
    \CFK_\cR(S^3, K) \simeq \CFK_\cR(S^3,K) \otimes_\cR \CFK_\cR(S^3,U) \simeq \CFK(S^3,K) \otimes_\cR \cR,
\end{align*}
we see that $\frU$ and $\frV$ act by $1\otimes U$ and $1\otimes V$.

\begin{prop}\label{prop:module_actions_agree}
    The modules $\CFK_{\F[\frU,\frV]}(S^3, K)$ and $\CFK_\cR(S^3, K)$ are homotopy equivalent as $\F[U,V]$-modules.
\end{prop}

To prove \Cref{prop:module_actions_agree}, we appeal to a splitting lemma for hypercubes. Let $(\Sigma, \bm \alpha, \bm \beta, w, z)$ be a doubly pointed Heegaard diagram for a knot $K$ in $S^3$. Let $(\Sigma_0,\alpha_0, \beta_0, w_0, z_0)$ be a genus 1, doubly pointed Heegaard diagram for the unknot in $S^3$. There is a natural cobordism 
\begin{align*}
    (S^3, K, p) \ra (S^3, K, p)\amalg(S^3, U)
\end{align*}
which is the composition of a saddle cobordism with a 3-handle cobordism. This map can be computed as the composition:
\begin{align*}
    \CFK_\cR(\Sigma, \bm \alpha, \bm \beta, w, z) 
        & \xra{S^+_{z_0,w_0}} \CFK_\cR(\Sigma\#\Sigma_0, \bm \alpha\cup \alpha_0, \bm \beta \cup \beta_0, \{w, w_0\}, \{z,z_0\})\\
    & \xra{F_{\alpha, \beta, \beta'}} \CFK_\cR(\Sigma\#\Sigma_0, \bm \alpha\cup \alpha_0, \bm \beta \cup \beta_0', \{w, w_0\}, \{z,z_0\}) \\
    & \xra{F_3} \CFK_\cR(\Sigma \amalg \Sigma_0, \bm \alpha\cup \alpha_0, \bm \beta \cup \beta_0, \{w, w_0\}, \{z,z_0\}).
\end{align*}
This composition will be denoted $E$. By \cite[Proposition 5.1]{zemke_connectedsums}, $E$ is a homotopy equivalence. It will be convenient to extend this map to a hypercube. The quasi-stabilization map is defined combinatorially. Proving that this map is compatible with those defined by holomorphic curve counts requires a neck stretching argument. For a suitable choice of almost complex structure, under the natural identification
\begin{align*}
    \CFK_\cR(\Sigma\#\Sigma_0, \bm \alpha\cup \alpha_0, \bm \beta \cup \beta_0, \{w, w_0\}, \{z,z_0\}) \cong \CFK_\cR(\Sigma, \bm \alpha, \bm \beta, w, z) \otimes \CFK_\cR(\Sigma_0,\alpha_0, \beta_0, w_0, z_0)
\end{align*}
of underlying groups, the holomorphic polygon counting maps split as tensor products. The analytic input we need is worked out in \cite{manolescu2024heegaardfloerhomologyinteger}.

\begin{prop}{\cite[Proposition 6.18]{manolescu2024heegaardfloerhomologyinteger}}\label{prop:MO_Quasi_Stab}
    Let $\cD = (\Sigma, \bm \alpha, \bm \beta_1, \hdots, \bm \beta_n, w, z)$ be a multi-Heegaard diagram and let $\cD_0 = (\Sigma_0, \alpha_0, \beta_{0,1}, \hdots, \beta_{0,n})$ be a multi-quasi-stabilizing Heegaard diagram. Then, for a suitable almost complex structure,
    \begin{align*}
        f_{\alpha, \beta_1\cup\beta_{0,1}, \hdots, \beta_{n}\cup\beta_{0,n}}(\bm x_1, \hdots, \bm x_{n}, \bm y) = f_{\alpha, \beta_1, \hdots, \beta_{n}}(\bm x_1, \hdots, \bm x_{n})\otimes \bm y.
    \end{align*}
\end{prop}

It follows from \Cref{prop:MO_Quasi_Stab} that the quasi-stabilization maps can be arranged to commute with holomorphic polygon counting maps \emph{on the nose.} 

\begin{figure}[h]
    \def\svgwidth{.8\linewidth}
\begingroup%
  \makeatletter%
  \providecommand\color[2][]{%
    \errmessage{(Inkscape) Color is used for the text in Inkscape, but the package 'color.sty' is not loaded}%
    \renewcommand\color[2][]{}%
  }%
  \providecommand\transparent[1]{%
    \errmessage{(Inkscape) Transparency is used (non-zero) for the text in Inkscape, but the package 'transparent.sty' is not loaded}%
    \renewcommand\transparent[1]{}%
  }%
  \providecommand\rotatebox[2]{#2}%
  \newcommand*\fsize{\dimexpr\f@size pt\relax}%
  \newcommand*\lineheight[1]{\fontsize{\fsize}{#1\fsize}\selectfont}%
  \ifx\svgwidth\undefined%
    \setlength{\unitlength}{439.37007874bp}%
    \ifx\svgscale\undefined%
      \relax%
    \else%
      \setlength{\unitlength}{\unitlength * \real{\svgscale}}%
    \fi%
  \else%
    \setlength{\unitlength}{\svgwidth}%
  \fi%
  \global\let\svgwidth\undefined%
  \global\let\svgscale\undefined%
  \makeatother%
  \begin{picture}(1,0.83225806)%
    \lineheight{1}%
    \setlength\tabcolsep{0pt}%
    \put(0.50632482,0.79159556){\color[rgb]{0,0,0}\makebox(0,0)[t]{\lineheight{1.25}\smash{\begin{tabular}[t]{c}{\small$S^+_{z_0,w_0}$}\end{tabular}}}}%
    \put(0,0){\includegraphics[width=\unitlength,page=1]{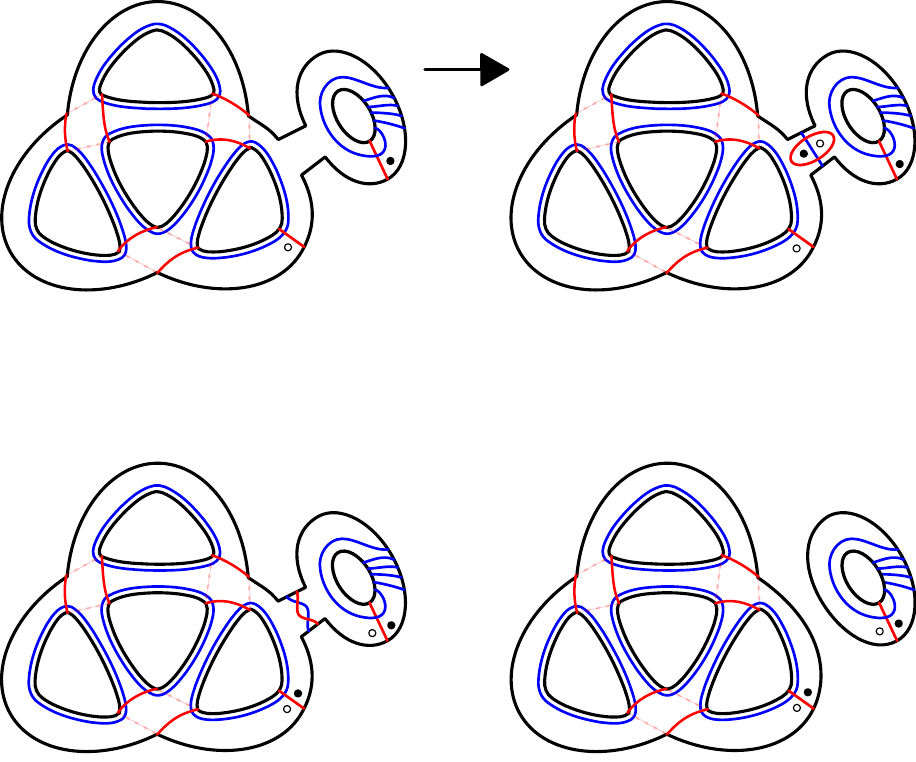}}%
    \put(0.49920261,0.30824649){\color[rgb]{0,0,0}\makebox(0,0)[t]{\lineheight{1.25}\smash{\begin{tabular}[t]{c}{\small$G_3$}\end{tabular}}}}%
    \put(0,0){\includegraphics[width=\unitlength,page=2]{heegaard_moves.pdf}}%
    \put(0.47193126,0.47732391){\color[rgb]{0,0,0}\makebox(0,0)[t]{\lineheight{1.25}\smash{\begin{tabular}[t]{c}{\small$F_{\alpha, \alpha', \beta}$}\end{tabular}}}}%
  \end{picture}%
\endgroup%

            \caption{The sequence of Heegaard diagrams representing the cobordism which splits off an unknot from $K$. The first move is a quasi-stabilization, the second is a band move, and the third is a 3-handle attachment.}
        \label{fig:heegaard_moves}
        \end{figure}

\begin{lem}\label{lem:pants_hypercube}
    Let $(\Sigma, \bm \alpha, \bm \beta, \bm \beta', w, z)$ be a doubly pointed Heegaard triple. Then, for a suitable choice of almost complex structure, the diagram 
    \begin{align}\label{diag:pants_hyper_cube}
        \begin{tikzcd}[ampersand replacement = \&, column sep = large, row sep = large]
            \CFK_\cR(\Sigma, \bm \alpha, \bm \beta, w, z) 
            \ar[d,"E"]
            \ar[r,"f_{\beta\ra\beta'}"] 
            \ar[dr,dashed]
            \& 
            \CFK_\cR(\Sigma, \bm \alpha, \bm \beta', w, z) 
            \ar[d,"E"]
            \\
            \CFK_\cR(\Sigma \amalg S^2, \bm \alpha\cup \alpha_0, \bm \beta \cup \beta_0, \{w, w_0\}, \{z,z_0\}) \ar[r,"f_{\beta\ra\beta'}"] \& \CFK_\cR(\Sigma \amalg S^2, \bm \alpha\cup \alpha_0, \bm \beta' \cup \beta_0, \{w, w_0\}, \{z,z_0\})
        \end{tikzcd}
    \end{align}
is a hyperbox of chain complexes, where $E$ denotes the maps described in \Cref{prop:MO_Quasi_Stab}.
\end{lem}
\begin{proof}
    \Cref{diag:pants_hyper_cube} is obtained by compressing the following hyperbox:
    \begin{align}
    \begin{tikzcd}[ampersand replacement = \&, column sep = large, row sep = large]
        \CFK_\cR(\Sigma, \bm \alpha, \bm \beta, w, z) \ar[d,"S^+_{z_0,w_0}"]\ar[r,"f_{\beta\ra\beta'}"] 
        \& 
        \CFK_\cR(\Sigma, \bm \alpha, \bm \beta', w, z) \ar[d,"S^+_{z_0,w_0}"]
        \\
        \CFK_\cR(\Sigma\#S^2, \bm \alpha\cup \alpha_0, \bm \beta \cup \beta_0, \{w, w_0\}, \{z,z_0\}) 
        \ar[d,"f_{\beta_0\ra \beta_0'}"]
        \ar[r,"f_{\beta\ra\beta'}"]
        \ar[dr,"h_{\beta\ra \beta', \beta_0 \ra \beta_0'}",dashed] \& 
        \CFK_\cR(\Sigma\#S^2, \bm \alpha\cup \alpha_0, \bm \beta' \cup \beta_0, \{w, w_0\}, \{z,z_0\}) 
        \ar[d,"f_{\beta_0\ra\beta_0'}"]
        \\
        \CFK_\cR(\Sigma\#S^2, \bm \alpha\cup \alpha_0, \bm \beta \cup \beta_0', \{w, w_0\}, \{z,z_0\})\ar[d,"F_3"]
        \ar[r,"f_{\beta\ra\beta'}"] \& 
        \CFK_\cR(\Sigma\#S^2, \bm \alpha\cup \alpha_0, \bm \beta' \cup \beta_0', \{w, w_0\}, \{z,z_0\})\ar[d,"F_3"]
        \\
        \CFK_\cR(\Sigma \amalg S^2, \bm \alpha\cup \alpha_0, \bm \beta \cup \beta_0, \{w, w_0\}, \{z,z_0\}) \ar[r,"f_{\beta\ra\beta'}"] \& \CFK_\cR(\Sigma \amalg S^2, \bm \alpha\cup \alpha_0, \bm \beta' \cup \beta_0', \{w, w_0\}, \{z,z_0\})
    \end{tikzcd}
\end{align}
    The top square commutes by \Cref{prop:MO_Quasi_Stab}. To see the central square is a hypercube we verify 
    \begin{align*}
        \begin{tikzcd}[column sep = huge, row sep = huge, ampersand replacement = \&]
            (\bm\beta \cup \beta_0)
            \ar[d,"\Theta_{\beta, \beta}\times \Theta_{\beta_0, \beta_0'}" description] 
            \ar[r,"\Theta_{\beta, \beta'}\times \Theta_{\beta_0, \beta_0}"]
            \ar[dr,dashed]
            \& 
            (\bm\beta' \cup \beta_0) 
            \ar[d,"\Theta_{\beta', \beta'}\times \Theta_{\beta_0, \beta_0'}" description] \\
            (\bm\beta \cup \beta_0') 
            \ar[r,"\Theta_{\beta, \beta'}\times \Theta_{\beta_0', \beta_0'}" below] \& (\bm\beta' \cup \beta_0')
        \end{tikzcd}
    \end{align*}
    is a hypercube of attaching curves. The hypercube relations follow from the associativity of holomorphic triangles. Finally, the lowest square commutes by a second application of \Cref{prop:MO_Quasi_Stab}. Note that for both the top and bottom squares we use the same almost complex structure. 
\end{proof}

Let $\frak P(\bm \beta, \bm \beta')$ be the 2-dimensional hypercube obtained by compressing \Cref{diag:pants_hyper_cube}. 

\begin{lem}\label{lem:pants_extension}
    Suppose $\cL_{\beta}$ is an $n$-dimensional hypercube of attaching curves. Consider a 1-face, $\frak F$, of $\CF(\bm\alpha, \cL_\beta)$ of the form 
    \begin{align*}
        \CFK_{\cR}(\Sigma, \bm \alpha, \bm \beta_1, w, z) \ra \CFK_{\cR}(\Sigma, \bm \alpha, \bm \beta_2, w, z).
    \end{align*}
    Then, for a suitable choice of almost complex structure, the partially defined hypercube $\CF(\bm\alpha, \cL_\beta) \cup_{\frak F} \frak P(\bm \beta, \bm \beta')\times \E^{n-1}$ admits a filling.
\end{lem}
\begin{proof}
     We will extend \Cref{diag:pants_hyper_cube} to a hyperbox and then compress. Everything follows from  \Cref{prop:MO_Quasi_Stab}: the top and bottom can be extended as all higher homotopies can be taken to be zero; to see that the central square extends, we note that from \Cref{prop:MO_Quasi_Stab} it follows that all polygon counting maps split as 
     \[
        f_{\alpha, \beta_1\cup\beta_{0,1}, \hdots, \beta_{n}\cup\beta_{0,n}}= f_{\alpha, \beta_1, \hdots, \beta_{n}}(\bm x_1, \hdots, \bm x_{n})\otimes \id_{\CF(\bm \alpha, \beta_n)}.
     \]
     Therefore, we can simply fill in the missing arrows with terms of the form $f_{\alpha, \beta_1, \hdots, \beta_{n}}\otimes \id_{\CF(\bm \alpha, \beta_n)}.$
\end{proof}

With this, we are equipped to prove \Cref{prop:module_actions_agree}.

\begin{proof}[Proof of \Cref{prop:module_actions_agree}]
    Recall that the module $\CFK_{\F[\frU,\frV]}(S^3, K)$ is the total telescope obtained by collapsing the quasimodule defined by the data $(\CFK(S^3, K), \frU, \frV, \frak W)$. Consider the hypercube $\CFK(S^3, K) \otimes \CFK_{\F[\frU,\frV]}(S^3, U)$, whose vertices are $\CFK(S^3, K) \otimes \CFK(S^3, U)$ and whose length 1 and 2 arrows are $1\otimes \frU$, $1 \otimes \frV$, and $1 \otimes \frak{W}$ respectively. The map $E$, combined with \Cref{lem:pants_hypercube} and \Cref{lem:pants_extension}, give us the following diagram:

    \begin{align}\label{diag:compute UV action on unknot}
        \begin{tikzcd}[column sep={3.9cm,between origins},row sep={1.5cm,between origins},labels=description,ampersand replacement = \&]
        \CFK(S^3, K, s)
        	\ar[dd, swap,"\frV"]
        	\ar[dr,  "E"]
        	\ar[rr, " \frU"]
        \&\&[-.8cm]
        \CFK(S^3, K, s-1)
        	\ar[dd, "\frV"]
        	\ar[dr,"E"]
        \&
        \\
        \&\CFK(S^3, K, s) \otimes_{\F[U,V]} \F[U,V]
        \&\&
        \CFK(S^3, K, s-1) \otimes_{\F[U,V]} \F[U,V]
        	\ar[dd, "1 \otimes \frV"]
        	\ar[from=ll,crossing over, "1 \otimes \frU"]
        \\[2cm]
        \CFK(S^3, K, s+1) 
        	\ar[rr, "\frU"]
        	\ar[dr,"E"]
        \&\&\CFK(S^3, K, s)
        	\ar[dr, " E"]	
        \&
        \\
        \&
        \CFK(S^3, K, s+1) \otimes_{\F[U,V]} \F[U,V]
        	\ar[rr, "1 \otimes \frU"]
        	\ar[from =uu, crossing over,"1 \otimes \frV"]
        	\&\&
        \CFK(S^3, K, s) \otimes_{\F[U,V]} \F[U,V]
    \end{tikzcd}
    \end{align}
    Therefore, it suffices to compute the action of the maps $\frU$ and $\frV$ on $\CFK(S^3, U)$. This computation can be done in the genus 1 diagram shown in \Cref{fig:model_comp}. The map $\frU$ counts triangles with a corner at $\Theta^-$ in $\SpinC$ structure $[-1]$. There is a single triangle and it covers the $\bm w$ basepoint. Hence, 
    \begin{align*}
        \frU(\bm x) = U \cdot  \bm x.
    \end{align*}
    Similarly, $\frV$ counts triangles with a single vertex at $\Theta^+$ in $\SpinC$ structure $[+1]$. Again, there is a single triangle and it covers the $\bm z$ basepoint, from which it follows that 
    \begin{align*}
        \frV(\bm x) = V \cdot  \bm x.
    \end{align*}
    as claimed. Further, all higher polygon counting maps are trivial.
    
    Finally, note that the total telescope of the front face of \eqref{diag:compute UV action on unknot} is nothing but $\Tel\Hyp(\CFK(S^3, K))$, which is homotopy equivalent to $\CFK(S^3, K)$.
\end{proof}

\begin{figure}[h]
    \def\svgwidth{.8\linewidth}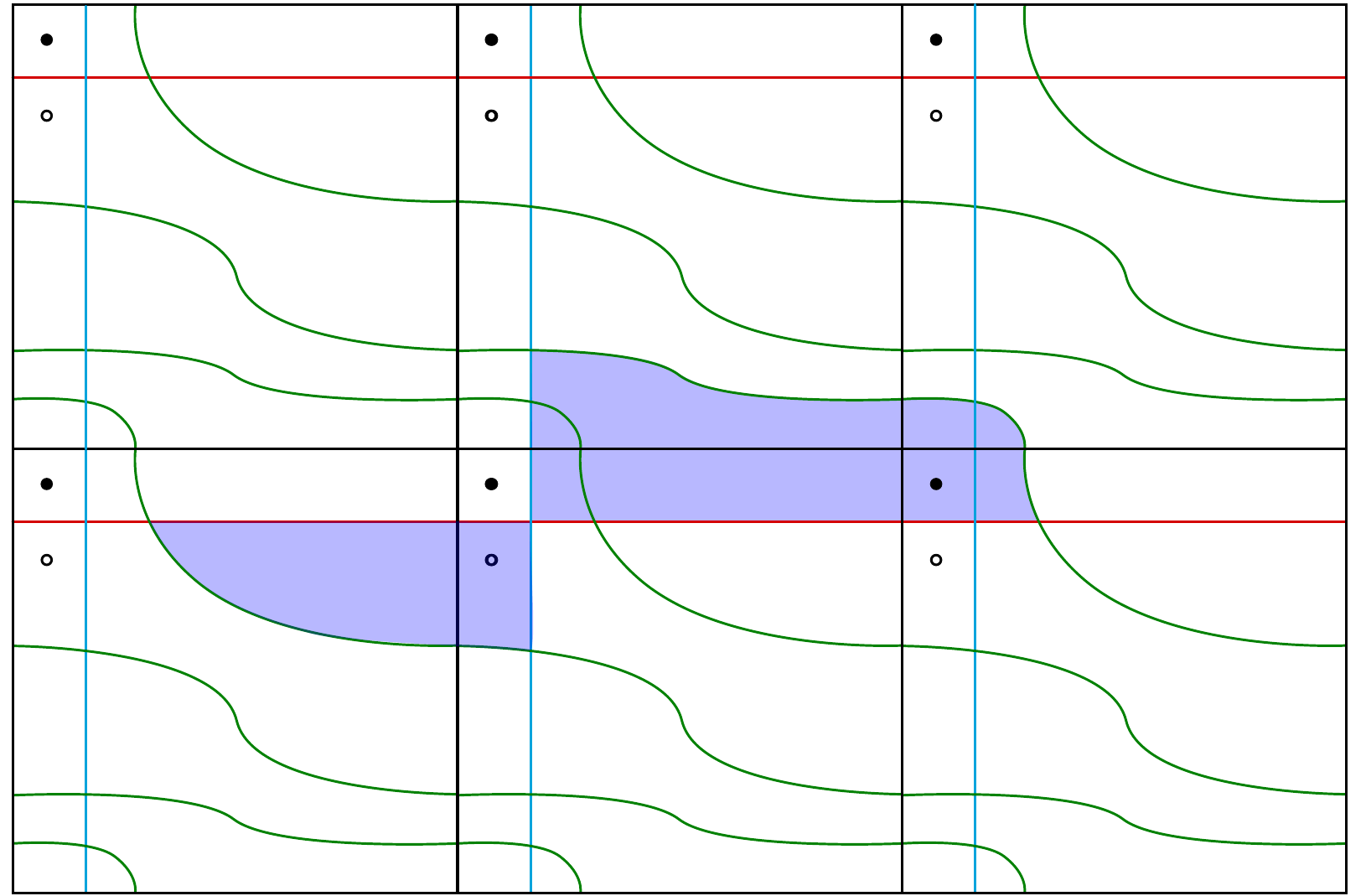
            \caption{Model computation of $\frU$ and $\frV$ in a genus 1 diagram.}
        \label{fig:model_comp}
        \end{figure}

\subsection{The Bordered Perspective}

We note that the constructions of this section fit naturally in the context of bordered Floer homology. 

Let $\CFAh(\bm O_{m})$ be the type A structure for the $m$-framed solid torus. By the morphism pairing theorem, 
\begin{align*}
    \Mor_{\cA}(\CFAh(\bm O_{km}), \CFAh(\bm O_{m(k+1)})) \simeq \CFh(L(m,1)).
\end{align*}
By working with genus 1 diagrams for $\CFAh(\bm O_{mk})$, it is easy to see that every generator of $\HFh(L(m, 1))$ can be represented by a unique intersection point between the beta curves for $\CFAh(\bm O_{mk})$ and $ \CFAh(\bm O_{m(k+1)})$. In particular, as in the previous sections, we fix intersection points representing $\SpinC$-structures $[-1]$,$[0]$, and $[1]$; counting triangles with a corner at one of these intersection points determines an $A_\infty$-morphism
\begin{align*}
    \frak f^A_{[*], i+1}\colon \CFAh(\bm O_{mk}) \otimes \cA^{\otimes i} \ra \CFAh(\bm O_{m(k+1)}).
\end{align*}
for $* \in \{-1,0,1\}$. These modules and morphisms form the vertices and edges of a 3-dimensional pre-hypercube of type A structures, which can be filled in by counting higher polygons. We will write $\frak{CFA}(\bm O)$ for this hypercube. 

Now, let $\CFAh(\bm O_\infty, \nu)$ be the type A structure for the infinity framed solid torus containing a curve $\nu$ which is isotopic to an $S^1$-fiber. By counting triangles once again, there is a map 
\begin{align*}
    \Gamma^A_{i+1}: \CFAh(\bm O_m)\otimes \cA^{\otimes i} \ra \CFAh(\bm O_\infty, \nu).
\end{align*}
Equivalently, $\Gamma^A = \{\Gamma^A_{i}\}$ determines a 1-dimensional hypercube of type A structures, which we also denote $\Gamma^A$. It follows from the pairing theorem for triangles that there is an equivalence
\begin{align*}
    \begin{tikzcd}[ampersand replacement = \&]
        \Gamma^A \boxtimes \CFDh(\KC) \simeq \left( \CFK_{\F[U]}(S^3, K) \xra{F_{\overline{W_N(K)}}} \CFh(S^3_N(K)) \right).
    \end{tikzcd}
\end{align*}
The map on the right is the total cobordism map associated to the dual of the canonical trace cobordism. Of course, provided $N$ is large, by restricting this map to Alexander grading $s$ and projecting to $\SpinC$-structure $[s]$ this pairing computes the large surgery isomorphism. 

We combine these two observations as follows. Let $P$ be the 3-manifold
\begin{align*}
    P:= S^1 \times D^2 \smallsetminus (\nu(\lambda) \cup \nu(\mu) )
\end{align*}
where $\lambda = S^1 \times \{0\}$ and $\mu$ is a small meridian of $\lambda$; here, the boundary components corresponding to $\mu$ and $\lambda$ are $0$-framed while the third boundary is infinity framed. We consider the trimodule:
\begin{align*}
    \widehat{CFDDA}(P),
\end{align*}
where we view $\partial \nu(\lambda)$ and $\partial \nu(\mu)$ as being the type D boundary components, and $\partial (S^1 \times D^2)$ as the type A boundary. Decompose
\begin{align*}
    S^3_N(K) & = \bm O_{km} \cup \bm O_{N} \cup P \cup (\KC)_0\\
    (S^3, K) & = \bm O_{km} \cup (\bm O_{\infty}, \nu) \cup P \cup (\KC)_0.
\end{align*}
This decomposition produces hypercubes
\begin{align*}
    \frak H_{\bm O} :=\frak{CFA}(\bm O) \boxtimes \CFAh(\bm O_N) \boxtimes \widehat{CFDDA}(P) \boxtimes \CFDh(\KC),
\end{align*}
and 
\begin{align*}
    \frak H_{K}:=\frak{CFA}(\bm O) \boxtimes \CFAh(\bm O_{\infty}, \nu) \boxtimes \widehat{CFDDA}(P) \boxtimes \CFDh(\KC).
\end{align*}
The map $\Gamma^A$ gives rise to a map between these hypercubes, which by the pairing theorem for polygons, induces the homotopy commutative diagram:
\begin{align*}
    \begin{tikzcd}[ampersand replacement = \&]
        \Tel(\frak H_{\bm O}) \ar[r] \ar[d] \& \Tel(\frak H_{K}) \ar[d] \\
        \Tel(\CF(\bm \alpha \cup \alpha_N, \cL_\beta)) \ar[r,"\Gamma"] \& \Tel(\CF(\bm \alpha \cup \alpha_K, \cL_\beta)).
    \end{tikzcd}
\end{align*}
Here, the vertical arrows come from the pairing theorem. This perspective will be quite useful going forward. 

\section{The $\CFK_\cR$-to-$\CFDh$ formula}\label{sec: the LOT Correspondence} 
In \cite{guthkang2024invariantsplittingprinciples}, we constructed a map 
\begin{align*}
    \Lambda: \End_h^\cA(\CFDh(\KC)) \ra \End_\cR^{h}(\CFK_\cR(S^3, K)), 
\end{align*}
which we will return to in \Cref{sec: the map Lambda}. In this section, we construct a map in the reverse direction.

In \cite[Chapter 11]{LOT_bordered_HF}, Lipshitz-Ozsv\'ath-Thurston describe a formula for determining a model for $\CFDh(\KC)$ given a model for $\CFK_\cR(S^3, K)$. At the level of objects, this formula is well-known, and has been utilized in many applications of bordered Floer homology. At the level of morphisms, things are more subtle. Given an endomorphism $f$ of $\CFDh(\KC)$, one can obtain an endomorphism of $\CFK_{V = 0}(S^3, K)$ by considering the tensor product $\bI_{\CFAh(S^1\times D^2, S^1 \times \pt)}\boxtimes f$; note that this method gives an endomorphism of the $V = 0$ truncation of $\CFK(S^3,K)$, not the truncation over $\cR$. But, they also give a construction in the reverse direction. Given a homotopy equivalence of $\CFK_\cR(S^3, K)$, they construct an endomorphism $\Omega(f)$, of (a particular model for) $\CFDh(\KC)$. Their construction applies just as well to a broader class of morphisms, which we call \emph{locally symmetric}; see \Cref{def:locally symmetric}.

\subsection{The \LOT basis free model for objects}

In \cite[Chapter 11]{LOT_bordered_HF}, Lipshitz-Ozsv\'ath-Thurston produce a model for $\CFDh(\KC)$ with framing $-N$, for $N \gg 0$; in contrast, we have been up to this point working with positive framings and surgeries. Therefore, in this section, in order to make use of the model from \cite[Theorem 11.36]{LOT_bordered_HF}, we will primarily be interested in the ``big'' knot Floer complex \cite{os_knotinvts}.

\begin{defn}\label{def: CFK big}
   Let $C_K$ be a model for the knot Floer complex of $K$. Define $C_K^b$ to be the quotient of $C_K^\infty$ which consists of monomials $U^iV^j \bm x$ with $\max(i,j) \ge 0$. Let $\widehat{C}_K^b$ be the subset consisting of monomials $U^iV^j \bm x$ with $\max(i,j) = 0$ (Ozsv\'ath and Szab\'o call this complex ``big'' $\CFK^-$). Define $\hat{A}_s^b(C_K)$ to be the subset of $\hCb$ of elements in Alexander grading $s$. See \cite[Section 4]{os_knotinvts}. The $\hat{A}_s^b(C)$ inherit a Maslov grading as well, from $C_K$.
\end{defn}

According to \cite{os_knotinvts,rasmussen_knotcompl}, the complex $\hAb_s(C_K)$ is a model for $\CFh(S^3_{-N}(K), [s])$ for large $N$. Note that $\widehat{C}_K^b$ is the complex obtained from $C_K^b$ by setting $(UV)^{-1} = 0 = (UV)$. We will frequently write $\hAb_s$ rather than $\hAb_s(C_K)$. \\

Let $\hCb/(U^{-1})$ (respectively, $\hCb/(V^{-1})$) denote the quotients of $\widehat{C}_K^b$ obtained by setting $U^{i} = 0$ (respectively, $V^i = 0$) for all $i <0$. The associated complexes $\hAb_s/(U^{-1})$ and $\hAb_s/(V^{-1})$ will become the building blocks of the Lipshitz-Ozsv\'ath-Thurston model. Since $V$ increases the Alexander grading, for sufficiently large $s$, the complexes $\hAb_s/(U^{-1})$ are identically zero. We define $M_C$ to be the maximal $s$ such that $\hAb_s/(U^{-1}) \neq 0$. Similarly, $U$ decreases the Alexander grading; hence we may define $m_C$ to be the minimal $s$ so that $\hAb_{s}/(V^{-1}) \neq 0$.\footnote{As we have not assumed that $C$ is reduced, these quantities may not agree with the Seifert genus.} Conversely, when $s$ is very small, $\hAb_s/(U^{-1})$ stabilizes, and is homotopy equivalent to $\CFh(S^3)$; likewise, when $s$ is large $\hAb_s/(V^{-1})$ stabilizes to $\CFh(S^3)$. We will also consider the complex $\hAb_s/(U^{-1}, V^{-1})$, which is homotopy equivalent to $\CFKh(S^3, K, s)$.

The basis free model for $\CFDh(\KC)$ will be built out of these quotients of $\hAb$. Fix an odd integer\footnote{The $\F$-dimension of the unstable chain depends on the parity of $N$, so simplicity, we fix $N$ odd.} $N> M_C - m_C$. Consider the vector spaces:
\begin{align*}
    \cH & = \bigoplus_{m_C \le s \le m_C + S_N} \hAb_s/(V^{-1}) \\
    \cU & = \bigoplus_{1 \le s \le 2 \lfloor N/4 \rfloor+1} H_*(\hAb_{m_C + S_N+s}/(V^{-1})) \\
    \cV & = \bigoplus_{M_C - T_N \le s \le M_C}  \hAb_s/(U^{-1})\\
    \cK & = \bigoplus_{m_C \le s \le M_C} \hAb_s/(U^{-1},V^{-1}).
\end{align*}
Here, $S_N$ and $T_N$ are constants which depend on $N$ and $C_K$ (see \cite{LOT_bordered_HF} for details) which, for us, will simply be used to record the number of summands in $\cH$ and $\cV$ respectively. Note that $\cU$ is simply a direct sum of copies of $\F$; furthermore, the map $$U^{-1}: \hat{A}^b_{m_C + S_N+s}/(V^{-1}) \ra \hat{A}^b_{m_C + S_N+s+1}/(V^{-1})$$ is an isomorphism. We will call $\cH$ the \emph{horizontal region}, $\cU$ the \emph{unstable region}, $\cV$ the \emph{vertical region}, and $\cK$ the \emph{knot region}.

We will write $\cD(C_K, N, (x,y))$ to be the Lipshitz-Ozsv\'ath-Thurston model for $\CFDh(\KC)$. This model depends on $C_K$, the framing parameter $N$, as well as a pair of elements $x \in (\hat{A}^b_{m_C + S_N}/(V^{-1})^\vee$ and $y \in \hat{A}^b_{M_C-T_N}/(U^{-1})$ which generate their homology (which we assume are 1-dimensional). As a vector space, we define $\cD(C_K, N, (x,y))$ to be
\begin{align*}
    \iota_0 \cdot \cD(C_K, N, (x,y)) & = \cK \\
    \iota_1 \cdot \cD(C_K, N, (x,y)) & = \cH \oplus \cU \oplus \cV.
\end{align*}
The coefficient maps defining the differential are defined in terms of certain canonical maps between these various quotients of $\hAb$, which we now describe.

First, consider the maps
\begin{align*}
    \hAb_{s}/(U^{-1}) \xra{V} \hAb_{s+1}/(U^{-1}), \quad \hAb_{s}/(V^{-1}) \xra{U^{-1}} \hAb_{s+1}/(V^{-1}),
\end{align*}
given by multiplication by $V$ and $U^{-1}$ respectively. As noted above, for sufficiently large values of $s$, $\hAb_{-s}/(U^{-1})$ and $\hAb_{s}/(V^{-1})$ are both models for $\CFh(S^3)$. Furthermore, we \emph{choose} homotopy equivalences
    \[
    \hat{A}^b_{m_C + S_N}/(V^{-1})\xrightarrow{\simeq} U^{-1} \cdot H_*(\hat{A}^b_{m_C + S_N}/(V^{-1})) \simeq \HFh(S^3),
    \]
and 
    \[
    \HFh(S^3) \simeq U^{-1-(2\lfloor N/4 \rfloor)} \cdot H_*(\hat{A}^b_{m_C + S_N}/(V^{-1})) \xrightarrow{\simeq} \hat{A}^b_{M_C-T_N}/(U^{-1}).
    \]
We note that choosing these homotopy equivalences is equivalent to choosing a pair of elements $x \in (\hat{A}^b_{m_C + S_N}/(V^{-1}))^\vee$ and $y \in \hat{A}^b_{M_C-T_N}/(U^{-1})$ which generate their homology. 

The above groups fit into a sequence:
\begin{align*}
    \hAb_{m_{C}}/(V^{-1}) \xra{U^{-1}} \hAb_{m_{C}+1}/(V^{-1}) \xra{U^{-1}} \hdots \xra{U^{-1}} \hAb_{m_C+S_N}/(V^{-1}) \xra[x]{\simeq} H_*(\hat{A}^b_{m_C + S_N+1}/(V^{-1})) \xra{U^{-1}} \\
    \hdots \xra{U^{-1}} H_*(\hat{A}^b_{m_C +S_N+2\lfloor N/4 \rfloor +1}/(V^{-1})) \xra[y]{\simeq} \hat{A}^b_{M_C-T_N}/(U^{-1}) \xra{V} \hdots \xra{V} \hAb_{M_C-1}/(U^{-1})\xra{V} \hAb_{M_C}/(U^{-1}).
\end{align*}
We will write $D_{23}$ for all of these maps. 

Next, we consider the two natural inclusions
\begin{align*}
    \hAb_s/(U^{-1}, V^{-1}) \hookrightarrow \hAb_s/(V^{-1}), \quad \hAb_s/(U^{-1}, V^{-1}) \hookrightarrow \hAb_s/(U^{-1}),
\end{align*}
which will be denoted $D_3$ and $D_1$ respectively. Finally, we consider a pair of maps induced by the differential on $\hAb_s/(V^{-1})$:
\begin{align*}
    D_2: \hAb_s/(V^{-1}) \xra{U^{-1}} \hAb_{s+1}/(V^{-1}) \xra{\partial} \hAb_{s+1}/(V^{-1}) \ra \hAb_{s+1}/(U^{-1}, V^{-1})
\end{align*}
as well as 
\begin{align*}
    D_{123}: \hAb_{s}/(U^{-1}, V^{-1})  \xra{U^{-1}} \hAb_{s+1}/(V^{-1}) \xra{\partial}\hAb_{s+1}/(V^{-1}) \ra \hAb_{s+1}/(U^{-1}, V^{-1}) \hookrightarrow \hAb_{s+1}/(U^{-1}).
\end{align*}

We therefore equip $\mathcal{D}(C_K, N, (x, y))$ with a differential coming from the internal differentials of the various complexes as well as the structure maps
\begin{align*}
    \delta^1 = \huge \sum_{I \in \{1, 2, 3, 23, 123\}} D_{I},
\end{align*}
defined above. 

The following diagram schematizes the structure maps between the various pieces of the model. A more detailed diagram is shown in \Cref{fig:basis_free}, which is adapted from \cite[Figure 11.7]{LOT_bordered_HF}. Of course, this is a model for $\CFDh(\KC)$, according to \cite[Theorem 11.36]{LOT_bordered_HF}.

\begin{align*}
    \begin{tikzcd}[ampersand replacement = \&, row sep = huge, column sep = huge]
    \cV\ar[loop above, "D_{23}"] \& \\
    \cU \ar[u,"D_{23}"] \& \cK \ar[dl, bend left, "D_3"] \ar[ul,"D_1"]\ar[ul, bend right,"D_{123}" right] \\
    \cH \ar[loop below, "D_{23}"]\ar[u,"D_{23}"] \ar[ur, "D_2"]\& 
\end{tikzcd}
\end{align*} 

\begin{figure}
    \centering
    \begin{align*}
    \begin{tikzcd}[ampersand replacement = \&, row sep = huge]
        \hAb_s/(U^{-1}) 
        \&
        \&
        \&
        \\
        \hAb_{s-1}/(U^{-1}) 
        \ar[u,"\rho_{23}"]
        \&
        \&
        \&
        \\
        \vdots
        \ar[u,"\rho_{23}"]
        \& 
        \&
        \&
        \hAb_s/(U^{-1}, V^{-1}) 
        \ar[llluu," \rho_1" description, bend right]
        \\
        \HFh(S^3)
        \ar[u,"\rho_{23}"]
        \&
        \&
        \&
        \hAb_{s-1}/(U^{-1}, V^{-1}) 
        \ar[llluu,"\rho_1" description]
        \ar[llluuu,"\rho_{123}" description]
        \\
        \HFh(S^3)
        \ar[u,"\rho_{23}"]
        \&
        \&
        \&
        \vdots 
        \\
        \HFh(S^3)
        \ar[u,"\rho_{23}"]
        \&
        \&
        \&
        \hAb_{-s+1}/(U^{-1}, V^{-1}) 
        \ar[llldd,"\rho_3" description]
        \\
        \vdots
        \ar[u,"\rho_{23}"]
        \&
        \&
        \&
        \hAb_{-s}/(U^{-1}, V^{-1}) 
        \ar[llldd,"\rho_3" description, bend left]
        \\
        \hAb_{-s+1}/(V^{-1}) 
        \ar[u,"\rho_{23}"]
        \&
        \&
        \&
        \\
        \hAb_{-s}/(V^{-1}) 
        \ar[u,"\rho_{23}"]
        \ar[rrruuu,"\rho_2" description,right]
        \&
        \&
        \&
        \\
    \end{tikzcd}
\end{align*}
    \caption{The \LOT basis free model for $\CFDh(\KC)$.}
    \label{fig:basis_free}
\end{figure}

\subsection{Labels}
We will often drop $N$ from the notation and will simply write $\cD_{(x,y)}(C_K)$. For the time being, we will continue to emphasize the role of the pair $(x, y)$, which we call a \emph{label} for the basis free model. In order to make sense of endomorphisms of $\CFDh(\KC)$, we will prove that the type D structures induced by different choices of labels are related by homotopy equivalences which are themselves unique up to homotopy.

Suppose $(x, y)$ and $(x',y')$ are two different labels. As $\hAb_{m_C + S_N}/(V^{-1})$ and $\hAb_{M_C-T_N}/(U^{-1})$ have one dimensional homology, the classes $x$ and $x'$ differ by a boundary, as do $y$ and $y'$. Choose elements $H_{x,x'}$ and $H_{y,y'}$ of $(\hat{A}^b_{m_C + S_N}/(V^{-1}))^\vee$ and $\hat{A}^b_{M_C-T_N}/(U^{-1})$, respectively, so that
\[
x + x' = \partial H_{x,x'} \quad \text{and} \quad y+y' = \partial H_{y,y'}.
\]
Then it is clear that $H_{x,x'}$ and $H_{y,y'}$ are unique up to boundary, as the difference between any two possible choices gives a cycle, which has to be a boundary as its homology class must be trivial for grading issues (the difference in grading between the classes $H_{x,x'}$ and $x$ is one, and $x$ generates the homology of $(\hat{A}^b_{m_C + S_N}/(V^{-1}))^\vee$.)

Now we define a type D morphism 
\[
F_{(x,y)\rightarrow (x',y')}:\cD_{x,y}(C_K)\rightarrow \cD_{x',y'}(C_K)
\]
as follows. In the region
\[
\hat{A}^b_{m_C + S_N}/(V^{-1}) \xrightarrow{\rho_{23}\cdot x} \mathbb{F}_2 = \widehat{HF}(S^3)
\]
of $\cD_{x,y}(C_K)$, we define $F_{(x,y)\rightarrow (x',y')}(w) = \rho_{23}H_{x,x'}(w) \in \mathbb{F}_2$. In the part
\[
\widehat{HF}(S^3) = \mathbb{F}_2 \xrightarrow{\rho_{23}\cdot y} \hat{A}^b_{M_C-T_N}/(U^{-1}),
\]
we define $F_{(x,y)\rightarrow (x',y')}(1)=\rho_{23}H_{y,y'} \in \hat{A}^b_{s}/(U^{-1})$, where $1$ denotes the generator of $\mathbb{F}_2$.  $F_{(x,y)\rightarrow (x',y')}$ is the identity otherwise.

\begin{rem}\label{rem:xy_to_xy_is_id}
    Note that for any label $(x,y)$ of $\CFK_\cR(S^3,K)$, we have that 
    $$F_{(x,y)\rightarrow (x,y)} = \id.$$
\end{rem}

\begin{lem} \label{lem: invariance of label change map}
    The homotopy class of $F_{(x,y)\rightarrow (x',y')}$ does not depend on the choices of $H_{x,x'}$ and $H_{y,y'}$.
\end{lem}
\begin{proof}
    Suppose that we instead chose maps $H'_{x,x'}$ and $H'_{y,y'}$ to construct another type D morphism $F'_{(x,y)\rightarrow (x',y')}$. As discussed above, we can find elements $K_x$ and $K_y$ such that 
    \[
    H_{x,x'}+H'_{x,x'} = \partial K_x,\quad H_{y,y'} + H'_{y,y'} = \partial K_y,
    \]
    and thereby can construct an $\mathcal{A}$-linear map
    \[
    G:\CFDh(\KC)_{x,y}\rightarrow \CFDh(\KC)_{x',y'}.
    \]
    It is straightforward to verify that $F_{(x,y)\rightarrow (x',y')} + F'_{(x,y)\rightarrow (x',y')} = \partial G + G\partial$. Therefore $F_{(x,y)\rightarrow (x',y')}$ is homotopic to $F'_{(x,y)\rightarrow (x',y')}$.
\end{proof}

\begin{lem} \label{lem: transitivity of label change map}
    For any labels $(x,y)$, $(x',y')$, and $(x'',y'')$ of $\CFDh(\KC)$, we have 
    \[
    F_{(x',y')\rightarrow (x'',y'')} \circ F_{(x,y)\rightarrow (x',y')} \sim F_{(x,y)\rightarrow (x'',y'')}.
    \]
\end{lem}
\begin{proof}
    This follows directly from \Cref{lem: invariance of label change map} and the definition of $F_{(x,y)\rightarrow (x'',y'')}$.
\end{proof}

\begin{lem} \label{lem: label change maps are h.e.}
    For any labels $(x,y)$ and $(x',y')$ of $\CFK_\cR(S^3,K)$, the type D morphism $F_{(x,y)\rightarrow (x',y')}$ is a homotopy equivalence, and $F_{(x',y')\rightarrow (x,y)}$ is its homotopy inverse.
\end{lem}
\begin{proof}
    Using \Cref{lem: transitivity of label change map}, we get
    \[
    \begin{split}
        F_{(x,y)\rightarrow (x',y')}\circ F_{(x',y')\rightarrow (x,y)} &\sim F_{(x',y')\rightarrow (x',y')} = \mathrm{id}_{\CFDh(\KC)_{x',y'}}, \\
        F_{(x',y')\rightarrow (x,y)} \circ F_{(x,y)\rightarrow (x',y')} &\sim F_{(x,y)\rightarrow (x,y)} = \mathrm{id}_{\CFDh(\KC)_{x,y}}.
    \end{split}
    \]
    The lemma follows.
\end{proof}

Let $\cL$ be the set of labels for $C_K$. We define \emph{the} basis free model to be the transitive system $$\left(\{\cD_{(x,y)}\}_{(x, y)}, \{F_{(x,y) \to (x',y')}\}_{(x, y), (x',y')} \right).$$

\subsection{The \LOT basis free model for morphisms} \label{subsec: LOT for morphisms}

The $\CFK_\cR$-to-$\CFDh$ formula, which takes a model $C_K$ for $\CFK_\cR(K)$ to the basis free model $\cD_{(x,y)}(C_K)$ for $\CFDh(\KC)$ described above, is functorial (up to homotopy) according to \cite[Proposition 11.38]{LOT_bordered_HF} in the following sense. 
Given a homogeneous $\cR$-linear map
\[
f:C_K \rightarrow C_K
\]
we define two $\cA$-linear maps
\[
f_\emptyset,f_2:\cD_{(x,y)}(C_K)\rightarrow \cD_{(x,y)}(C_K),
\]
as follows.
\begin{itemize}
    \item $f_\emptyset$ is canonically defined from the endomorphism $f^b_s$ of $\hat{A}^b_s$ on each $s\in \mathbb{Z}$, induced by the chain map $f$.
    \item The map $f_2$ is defined to be the composition 
    \begin{align*}
        \hAb_s/(V^{-1}) \xra{U^{-1}} \hAb_{s+1}/(V^{-1}) \xra{f} \hAb_{s+1+\deg(f)}/(V^{-1}) \ra \hAb_{s+1+\deg(f)}/(U^{-1},V^{-1}).
    \end{align*}
\end{itemize}

Let us introduce some terminology.

\begin{defn}\label{def:locally symmetric}
    Given a finitely generated free (bigraded) chain complex $C$ over $\mathcal{R}$ satisfying
    \[
    \dim_{\mathbb{F}_2} H_\ast(C\otimes_{\cR} \cR/(V,U+1) )= \dim_{\mathbb{F}_2}H_\ast(C\otimes_{\cR} \cR/(U,V+1)) = 1,
    \]
    every chain endomorphism $f:C\rightarrow C$ induces chain endomorphisms $f_U$ of $C\otimes_{\cR} \cR/(V)$ and $f_V$ of $C\otimes_{\cR} \cR/(U)$. We say that $f$ is \emph{locally symmetric} if $f_U$ and $f_V$ are either both homotopy equivalences or both nullhomotopic.
\end{defn}

\begin{defn}
    Let $\End^\cA_{h}(\CFDh(\KC))$ be the space of module endomorphisms of $\CFDh(\KC)$ and let $\End_\cR^h(\CFK_\cR(S^3, K))$ be the space of endomorphisms of $\CFK_\cR(S^3, K)$. Define $\End^{h,ls}_{\cR}(\CFK_\cR(S^3,K))$ to be the $\mathbb{F}_2$-subalgebra of $\End^h_{\cR}(\CFK_\cR(S^3,K))$ which is $\mathbb{F}_2$-linearly spanned by locally symmetric endomorphisms of $\CFK_\cR(S^3,K)$.
\end{defn}

Given an endomorphism $f$ in $\End^{h,ls}_{\cR}(\CFK_\cR(S^3,K))$, let $\omega(f)$ be the partially defined function:
\[
\omega(f) = f_\emptyset + \rho_2 f_2.
\]
Note, that we have not specified how to define $\omega(f)$ on the unstable chain. To do so, choose any label $(x,y)$ of $C_K$, and denote the endomorphism of $\mathbb{F}_2 = \widehat{HF}(S^3)$ induced by $f$ by $f_0$; this is 1 if $f$ is local and 0 if $f$ is not local. Then there exist elements $H_1\in (\hat{A}^b_{m_C+S_N}/(V^{-1}))^\ast$ and $H_2 \in \hat{A}^b_{M_C-T_N}/(U^{-1})$ such that
\[
x \circ f_U + f_0 \circ x = \partial H_1,\quad y \circ f_0 + f_V \circ y = \partial H_2.
\]
As in definition of label change maps, we may use the homotopies $(H_1,H_2)$ to define an $\cA$-linear endomorphism $K_{H_1,H_2}$ of $\cD_{x,y}(C_K)$.

\begin{defn} \label{defn: omega}
Given $f\in \End^{h,ls}_{\cR}(\CFK_\cR(S^3,K))$ and a label $(x, y)$ for $\cD(C_K)$, define
\[
\Omega(f)_{x,y} = \omega(f)+K_{H_1,H_2}.
\]
We will write $\Omega$ for the map of homotopy classes of maps as well. 
\end{defn}

Note that if we had chosen another $H'_1$ and $H'_2$ to define $\Omega(f)$ in \Cref{defn: omega}, then $H_1+H'_1$ and $H_2+H'_2$ would be boundaries, so by following the arguments in the proof of \Cref{lem: invariance of label change map}, we see that $\Omega([f])_{x,y}$ does not depend on the choice of $H_1$ and $H_2$.

\begin{lem} \label{lem: omega is well defined}
    For any locally symmetric chain endomorphism $f$ and any two labels $(x,y)$ and $(x',y')$ of $\CFK_\cR(S^3,K)$, we have 
    \[
    \Omega([f])_{x',y'} = [F_{(x,y)\rightarrow (x',y')}] \circ \Omega([f])_{x,y} \circ [F_{(x',y')\rightarrow (x,y)}].
    \]
    In other words, $\Omega$ determines a map of transitive systems:
    \begin{align*}
        \Omega([f]): \cD(C_K) \ra \cD(C_K).
    \end{align*}
\end{lem}
\begin{proof}
    This is a straightforward computation and follows directly from the definitions.
\end{proof}

\Cref{defn: omega} and \Cref{lem: omega is well defined} give us the map
\[
    \Omega: \End^{h,ls}_{\cR}(\CFK_\cR(S^3, K)) \ra \End_h^\cA(\CFDh(\KC)).
\]
It is clear that $\Omega$ is $\mathbb{F}_2$-linear.

\begin{lem}
    For any locally symmetric chain endomorphisms $f,g$ of $\CFK_\cR(S^3,K)$, we have $\Omega([f\circ g]) = \Omega([f])\circ \Omega([g])$. In other words, $\Omega$ is an $\mathbb{F}_2$-algebra homomorphism.
\end{lem}
\begin{proof}
    This is again a straightforward computation and thus also left to the reader.
\end{proof}

Finally, we note that if $K_0$ and $K_1$ are locally equivalent, then in exactly the same way, we can construct a map 
\begin{align*}
    \Omega_{K_0, K_1}:\Hom_\cR(\CFK_\cR(S^3, K_0),\CFK_\cR(S^3, K_1))\ra \Mor^\cA(\KC_0, \KC_1).
\end{align*}
We note that if $f\in \Hom_\cR(\CFK_\cR(S^3, K_0),\CFK_\cR(S^3, K_1))$ is a local equivalence and $(x, y)$ is a label for $\CFK_\cR(S^3, K_0)$, then $(f(x), f(y))$ is a label for $\CFK_\cR(S^3, K_1)$.

\begin{lem} \label{lem: locally symmetric morphisms functoriality}
    For a locally symmetric morphism $f \in \Hom_\cR(\CFK_\cR(S^3, K_0),\CFK_\cR(S^3, K_1))$ and labels $(x, y)$ and $(x',y')$ of $\CFK_\cR(S^3, K_0)$, we have that 
    \begin{align*}
        \Omega_{K_0,K_1}([f])_{x',y'}\circ [F_{(x,y)\rightarrow (x',y')}] = [F_{(f(x),f(y))\rightarrow (f(x'),f(y'))}] \circ \Omega_{K_0,K_1}([f])_{x,y},
    \end{align*}
    and therefore, we obtain a map 
\begin{align*}
    \Omega_{K_0, K_1}(f): \cD(C_{K_0}) \ra \cD(C_{K_1}).
\end{align*}
Furthermore, if $g_i \in \End_\cR(\CFK_\cR(K_i))$ for $i \in \{0,1\}$ we have 
\begin{align*}
    \Omega([f\circ g_0]) = \Omega([f])\circ \Omega([g_0]) \\
    \Omega([g_1\circ f]) = \Omega([g_1])\circ \Omega([f]).
\end{align*}
\end{lem}
\begin{proof}
    This is immediate from the definitions.
\end{proof}

\begin{rem}
    The transitive system $\mathcal{D}(C_K)$ is very complicated and thus impossible to use directly in explicit computations. However, \Cref{lem: locally symmetric morphisms functoriality} allows us to simply choose a pair of labels $(x,y)$ and use the model $\mathcal{D}_{x,y}(C_K)$ instead of the universal model $\mathcal{D}(C_K)$. Therefore, from now on, we will suppress the labels and use $\mathcal{D}_{x,y}(C_K)$ instead.
\end{rem}

\section{The box tensor product surgery hypercube}\label{sec: the LOT hypercube} 
In the previous section, we recalled how a model $C_K$ for $\CFK_\cR(S^3, K)$ determines a model $\cD(C_K)$ for $\CFDh((\KC)_{-N})$. In particular, by pairing with type A structure for the $(-m)$-framed solid torus, we obtain a model for $\CFh(S^3_{-(N+m)}(K))$. We will write $$\cLOT^-_m(C_K):= \CFAh(\bm O_{-m})\boxtimes \cD(C_K).$$ The superscript ``-'' is to emphasize that we are working with negatively framed knot complements and surgered manifolds; we will otherwise suppress the dependence on the framing of the knot complement from the notation. We will make use of this model to explicitly describe the hypercube whose total telescope is the complex $\mathfrak{CFK}(K)$ described in Section \ref{sec: Telescopes and large surgeries}.
 
The type A structure $\CFAh(\bm O_{-m})$ is given as follows. There are $m+1$ generators: a single generator $\beta$ in $\iota_1 \cdot \CFAh(\bm O_{-m})$ and $m$ generators $b_1, \hdots, b_{m}$ which belong to  $\iota_0 \cdot \CFAh(\bm O_{-m})$. For later convenience, we fix $\gr(b_1) = (0, 0,-1)$, which forces $\gr(b_{i}) = (0, 0, -i)$. The $A_\infty$-operations are determined by the graph:
\begin{align*}
    \CFAh(\bm O_{-m}):= \beta \xra{\rho_2} b_{m} \xra{\rho_{3}\otimes \rho_2} b_{m-1}\xra{\rho_{3}\otimes \rho_2} \hdots \xra{\rho_{3}\otimes \rho_2} b_1 \xra{\rho_{3}\otimes \rho_2 \otimes \rho_1}\beta.
\end{align*}
The arrows encode the $A_\infty$-operations as follows: given a path from $\eta_0$ to $\eta_\ell$, there is an associated word, $I$, in the alphabet $\{1, 2, 3\}$ given by concatenating the subscripts of the algebra elements along the path. Let $I_1 \hdots I_n$ be the minimal subdivision of $I$ into subwords in $\{1, 2, 3, 12, 23, 123\}$. This subdivision records an $A_\infty$ operation 
\begin{align*}
    m_{n+1}(\eta_0 \otimes I_1 \otimes \hdots \otimes I_n) = \eta_\ell.
\end{align*}
For example, the path 
\begin{align*}
    \beta \xra{\rho_2} b_{m} \xra{\rho_3 \otimes \rho_2} b_{m-1},
\end{align*}
corresponds to the operation $m_3(\beta \otimes \rho_{23}\otimes \rho_2) = b_{m-1}.$\\

As a vector space, $\cLOT^-_m(C_K) = \CFAh(\bm O_{-m}) \boxtimes \cD(C_K)$ is very simple. The generator $\beta$ pairs nontrivially with elements in $\cH \oplus \cU \oplus \cV$ while the elements $b_i$ pair with elements of $\cK$. In summary, we have an isomorphism of vector spaces, 
\begin{align}\label{eqn:truncated_horiz}
    \cLOT^-_{m}(C_K) \cong \left(\cK\right)^{\oplus m} \oplus \cH \oplus \cU\oplus \cV.
\end{align}
The complex $\cLOT^-_m(C_K)$ inherits a $\SpinC$-decomposition $\cLOT^-_m(C_K) = \bigoplus \cLOT^-_m(C_K,s)$. Furthermore, as $\CFAh(\bm O_{-m})\boxtimes \cD(C_K)$ is a model for $\CFh(S^3_{-(N+m)}(K))$, it also comes with a Maslov grading, though for our purposes, it will be useful to shift the grading. An element $x$ of $\cLOT^-_m(C_K,s)$ corresponding to an element of $\CFh(S^3_{-(N+m)}(K),s)$ with Maslov grading $i$ is defined to have Maslov grading $$i - \frac{(N+m)-(2\ell+(N+m))^2}{4(N+m)},$$ where $\langle c_1(\frs), [S] \rangle = 2 \ell + (N+m)$. This grading shift is exactly the degree of the large surgery isomorphism, as we will ultimately want to identify $\cLOT^-_m(C_K)$ with a subset of $\CFK^\infty(S^3, K)$, rather than $\CFh(S^3_{-(N+m)}(K),s)$.

Having described $\cLOT^-_{m}(C_K)$ as a vector space in terms of $C_K$, we turn to the box tensor product differential.

\subsubsection{The LOT Differential} We will organize the computation by considering all possible $A_\infty$-operations on $\CFAh(\bm O_{-m})$ and looking for those compatible type D operations which will give rise to terms in the differential. 

To orient the reader, we sketch the outcome of our computation. The upshot is that $\cLOT^-_m(C_K)$ is isomorphic to a truncation of $\hAb$. Better yet, provided $m$ and $N$ are large, we can explicitly identify $\cLOT^-_m(C_K,s)$ with $\hAb_s$. It may be helpful to compare the following discussion to \Cref{fig:LOT2box}. 

As we indicated above, the LOT model consists of four basic pieces: the (truncated) horizontal complex $\cH$, the (truncated) vertical complex, $\cV$, the $\HFKh$-complex $(\cK)^{\oplus m}$ (which is homotopy equivalent to $m$ copies of $\HFKh(S^3,K)$), and the unstable chain $\cU$. The horizontal, vertical, and unstable complexes pair with the unique element of $\CFAh(\bm O_{-m})$ and carry along their internal differentials. We define a map
\begin{align*}
    \Phi_{m,s}: \cLOT^-_m(C_K, s) \ra \hAb_s
\end{align*}
as follows. For $\beta\boxtimes \eta \in\cH$, we define $\Phi_{m,s}(\beta\boxtimes \eta) = U^{-m}\eta$; for $\beta\boxtimes \eta \in\cV$, we define $\Phi_{m,s}(\beta\boxtimes \eta) = \eta$; finally, for $\beta\boxtimes \eta \in\cU$, we define $\Phi_{m,s}(\beta\boxtimes \eta)$ to be the isomorphism determined by the label $(x,y)$.  We extend $\Phi_{m,s}$ over the $\HFKh$ component as follows. The $\HFKh$ portion of the basis free model pairs with the elements $b_i$, so in the box tensor product complex, we see $m$ copies of $\hAb/(U^{-1},V^{-1})$. We will formally distinguish these summands by writing 
\begin{align*}
    \cK^{\oplus m} = \bigoplus_{i=1}^m U^{-i} \cdot \cK,
\end{align*}
identifying $b_i \boxtimes \eta$ with $U^{-i} \eta$, for $\eta \in \cK$, i.e. we define $\Phi_{m,s}(b_i \boxtimes \eta) = U^{-i} \eta$. The ring $\F[U]$ acts on this complex in the obvious way. Note that the map $\Phi_{m,s}$ is grading preserving, precisely because we shifted the grading on $\cLOT^-_m(C_K,s)$. There are type D operations which start and end at $\cK$, but pass through the horizontal complex; this family of type D operations gives rise to a differential on $\cK^{\oplus m}$. It turns out, this differential is easy to describe. The complex, $\cK^{\oplus m}$ can naturally be viewed as a quotient of $\hAb/(V^{-1})$ and the box tensor product differential recovers precisely the differential induced by $\hAb/(V^{-1})$. 

There are also type D operations emanating from the horizontal region $\cH$ and terminating at the $\HFKh$ region, induced by the differential. As we shall see, it will be convenient to identify $\langle \beta \rangle \boxtimes \cH$ with $U^{-(m+1)} \cdot \cH$. As every element of $\cH$ can be written as $U^{-j}\cdot \eta$ for some $\eta \in \cK$, the explicit identification is $\beta \boxtimes (U^{-j}\eta) \mapsto U^{-(j+m)}\eta$. In particular, $\cH \oplus \cK^{\oplus m}$ is also just a truncation of the horizontal part of $C_K$, simply truncated at some larger Alexander grading. There are two kinds of terms in the box tensor product differential which start at $U^{-(m+1)} \cdot \cH$; there are internal differentials (which are indeed induced by the internal differentials of $\hAb/(V^{-1})$) and also differentials which terminate on the $\HFK$-complex. Though, with respect to our identification $\cH\oplus \cK^{\oplus m} \cong U^{-(m+1)}\cdot \cH \oplus \bigoplus_{i=1}^m U^{-i} \cdot \cK$, we shall see that these arrows are precisely those induced by the differential on $\hAb/(V^{-1})$.

The differential on $\cV$ is even simpler: the box tensor product exactly recovers the differential on $\hAb/(U^{-1})$. Finally, there are type D operations from the $\HFK$-complex to the vertical complex. These correspond exactly to those summands of $\hAb$ which have vertical and horizontal differentials. Summarizing, we have:

\begin{lem}\label{lem:LOT_model_comp}
    For $m$ sufficiently large, the map $\Phi_{m,s}$ is a grading preserving isomorphism 
    \begin{align*}
        \Phi_{m,s}: \cLOT^-_m(C_K,s) \ra \hAb_s
    \end{align*}
    of $\F$-chain complexes. 
\end{lem}

\begin{figure}
    \centering
    \includegraphics[width=1\linewidth]{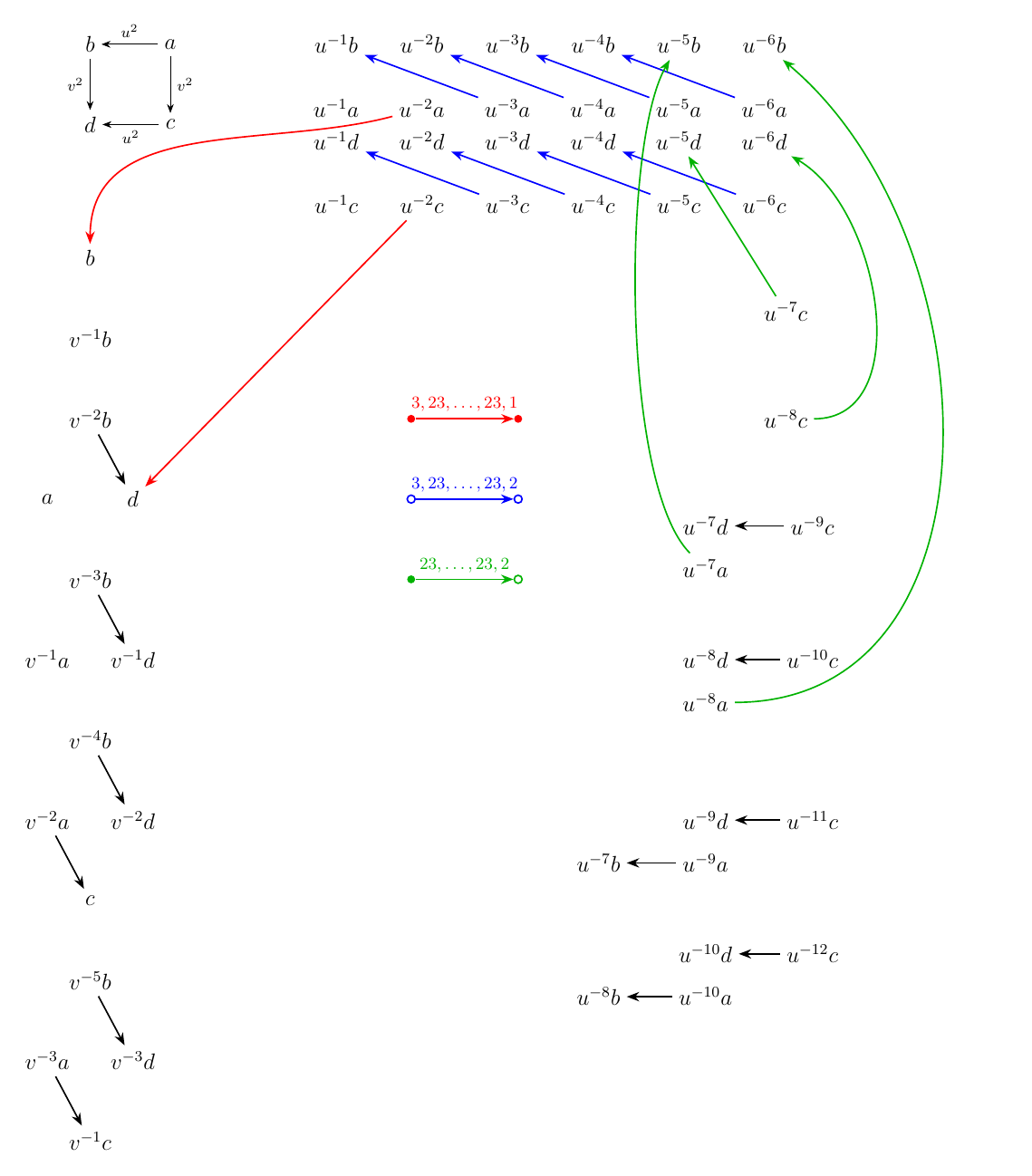}
    \caption{Complex $\CFAh(\bm O_{-m}) \boxtimes B$, where $B$ is the type D structure associated to a box complex with arrows of length 2. On the left, we have drawn the vertical complex; on the right, the horizontal complex, and in the center, the $\HFKh$ portion of the complex. We have also drawn in the differentials, color coded according to the sequence of algebra elements which produce them.}
    \label{fig:LOT2box}
\end{figure}

Having sketched the structure of the box tensor product complex, we turn to a careful computation. Don't worry -- it'll be fun! 
\begin{proof}
We proceed by considering all possible $A_\infty$-operations on $\CFAh(\bm O_{-m})$. \\
\begin{enumerate}
    \item[($\partial_\boxtimes:\cS \circlearrowleft$):] First, we consider terms in the box tensor product differential are those which preserve the decomposition $\cLOT^-_m(C_K)$. The simplest of these are induced by operations in $\CFAh(\bm O_{-m})$ of the form:
\begin{align*}
    m_2(\beta \otimes 1) = \beta.
\end{align*}
These pair with the internal differentials of $\hAb/(U^{-1})$ and $\hAb/(V^{-1})$, so that every term $\eta \ra \epsilon$ in $\hAb/(U^{-1})$ or $\hAb/(V^{-1})$ gives rise to a term $\beta \boxtimes \eta \ra \beta \boxtimes\epsilon$, i.e. the internal differentials of $\cH$ and $\cV$ agree exactly with the differentials on $\hAb$. These terms correspond to the black arrows in \Cref{fig:LOT2box}.\\

Terms from $(\cK)^{\oplus m}$ to $(\cK)^{\oplus m}$ are slightly more complicated, and are induced by the $A_\infty$ operations
\begin{align*}
    m_{\ell+3}(b_i \otimes \rho_{3} \otimes \overbrace{\rho_{23}\otimes \hdots \otimes \rho_{23}}^{\ell} \otimes \rho_2) = b_{i-\ell-1}.
\end{align*}
These operations pair with the type D structure maps given by the sequence 
\begin{align*}
    \hAb_s/(U^{-1}, V^{-1}) \xra{\rho_3} \hAb_s/(V^{-1}) \xra{\rho_{23}} \hdots \xra{\rho_{23}} \hAb_{s+\ell}/(V^{-1})  \xra{\rho_2} \hAb_{s+\ell+1}/(U^{-1}, V^{-1}).
\end{align*}
See \Cref{fig:basis_free}. Concretely, if $\eta \in \hAb_s/(U^{-1}, V^{-1})$, then the image of $\eta$ under this composition is the $U^{\ell+1}$-component of the differential of $\eta$, which we call $\vartheta$. Therefore, in the box tensor product, there are terms of the differential of the form 
\begin{align*}
    \{b_i \} \boxtimes \hAb_s/(U^{-1}, V^{-1}) \ra \{b_{i-\ell -1} \} \boxtimes \hAb_s/(U^{-1}, V^{-1}).
\end{align*}
which take $b_i \boxtimes \eta$ to $b_{i-\ell-1} \boxtimes \vartheta$. Or, under the identification \Cref{eqn:truncated_horiz}, there is an arrow 
\begin{align*}
    U^{-i} \eta \ra U^{-i+\ell+1} \vartheta.
\end{align*}
Hence, the box tensor product on this quotient agrees exactly with the differential on $U^{-1} \cdot \hAb/(U^{-(m+1)}, V^{-1})$ induced as a quotient of $\hAb/(V^{-1})$. These are the blue arrows in \Cref{fig:LOT2box}.

We note that while there are $A_\infty$-operations of the form 
\begin{align*}
    m_{k}(b_i \otimes \hdots \otimes \rho_{123} \otimes \hdots \otimes \rho_{23}\otimes \rho_2) = b_j,
\end{align*}
they cannot pair with any type D operations, so there are no additional differentials.
\\

    \item[($\partial_\boxtimes:\cH \ra \cK$):]Next, we consider differentials from the truncated horizontal complex to the $\HFK$ complex. Again, under our identification, $\cH \oplus \cK^{\oplus m}$ is nothing but a truncation of the horizontal part of $C_K$, and we expect to recover that differential. We consider operations of the form 
\begin{align*}
    m(\beta \otimes \overbrace{\rho_{23}\otimes \hdots \otimes \rho_{23}}^{\ell} \otimes \rho_2) = b_{m-\ell}.
\end{align*}
Much like above, these have matching type D operations, given by sequences 
\begin{align*}
    \hAb_s/(V^{-1}) \xra{\rho_{23}} \hdots \xra{\rho_{23}} \hAb_{s+\ell}/(V^{-1})  \xra{\rho_2} \hAb_{s+\ell+1}/(U^{-1}, V^{-1}),
\end{align*}
(again, see \Cref{fig:basis_free}). These type D operations are induced by the differential of $C_K$, taking an element $U^{-j} \eta$ to the $U^{-j + \ell + 1}$-component of $\partial(U^{-j} \eta)$. Hence, the box tensor product differential acts by 
\begin{align*}
    \beta \boxtimes U^{-j} \eta \mapsto b_{m-\ell} \boxtimes U^{-j+\ell + 1} \vartheta,
\end{align*}
which corresponds to 
\begin{align*}
     U^{-(j+m+1)} \eta \mapsto U^{-(j+m+1)+\ell + 1} \theta.
\end{align*}
I.e., this is simply one component of the horizontal differential. These are shown as the green arrows in \Cref{fig:LOT2box}.
\item[($\partial_\boxtimes: \cK \ra \cV$):] The components of the differential from the $\HFKh$ portion of the complex to the vertical are ultimately the most interesting. These come from operations of the form 
\begin{align*}
    m_{\ell+2}(b_i \otimes \rho_3 \otimes \overbrace{\rho_{23}\otimes \hdots \otimes \rho_{23}}^{\ell-1} \otimes \rho_2 \otimes \rho_1) = \beta, \quad 0 < \ell < m
\end{align*}
These may pair with type D operations of the form 
\begin{align*}
    \hAb_s/(U^{-1}, V^{-1}) \xra{\rho_3} \hAb_s/(V^{-1}) \xra{\rho_{23}} \hdots \xra{\rho_{23}} \hAb_{s+\ell}/(V^{-1}) \xra{\rho_2} \hAb_{s+\ell+1}/(U^{-1}, V^{-1}) \xra{\rho_1} \hAb_{s+\ell+1}/(U^{-1}).
\end{align*}
Again, by considering the definitions of the type D structure maps, one sees that these maps are given exactly by the horizontal differential. Since in the codomain $U^{-1}$ is set to zero, these maps are nontrivial only on terms $U^{-\ell}\eta \in U^{-\ell}\hAb/(U^{-1}, V^{-1})$ with $\langle U^{-\ell} \partial \eta, 1\rangle = 1$. These correspond to the red arrows in \Cref{fig:LOT2box}.
\\
\end{enumerate}
There is one more potential $A_\infty$-operation which could contribute to a differential from the horizontal component to the vertical component:
\begin{align*}
    m_{m+3}(\beta \otimes \overbrace{\rho_{23}\otimes \hdots \otimes \rho_{23}}^{m} \otimes \rho_2 \otimes \rho_1) = \beta.
\end{align*}
This operation could potentially have a corresponding type D sequence, were there a class $\eta$ with $\langle U^{m+1}, \eta \rangle = 1$. Though, as we have assumed $m$ is very large, we may assume that no such terms exist. 

Note that there are no nontrivial $A_\infty$-operations of the form $m_{k+1}(\alpha\otimes \rho_{23}\otimes \hdots)$ or $m_{k+1}(\alpha\otimes\hdots\otimes\rho_{23})$. Hence, the type $D$ operations
\begin{align*}
    \hdots \xra{\rho_{23}} \HFh(S^3) \xra{\rho_{23}} \HFh(S^3) \hdots \xra{\rho_{23}} \HFh(S^3) \xra{\rho_{23}} \hdots,
\end{align*}
have nothing with which to pair. Hence, there are no differentials in or out of generators coming from the unstable chain. This concludes the computation.
\end{proof}

\subsection{Quasimodule structure for the LOT model}

\begin{figure}
    \centering
    \includegraphics[width=0.75\linewidth]{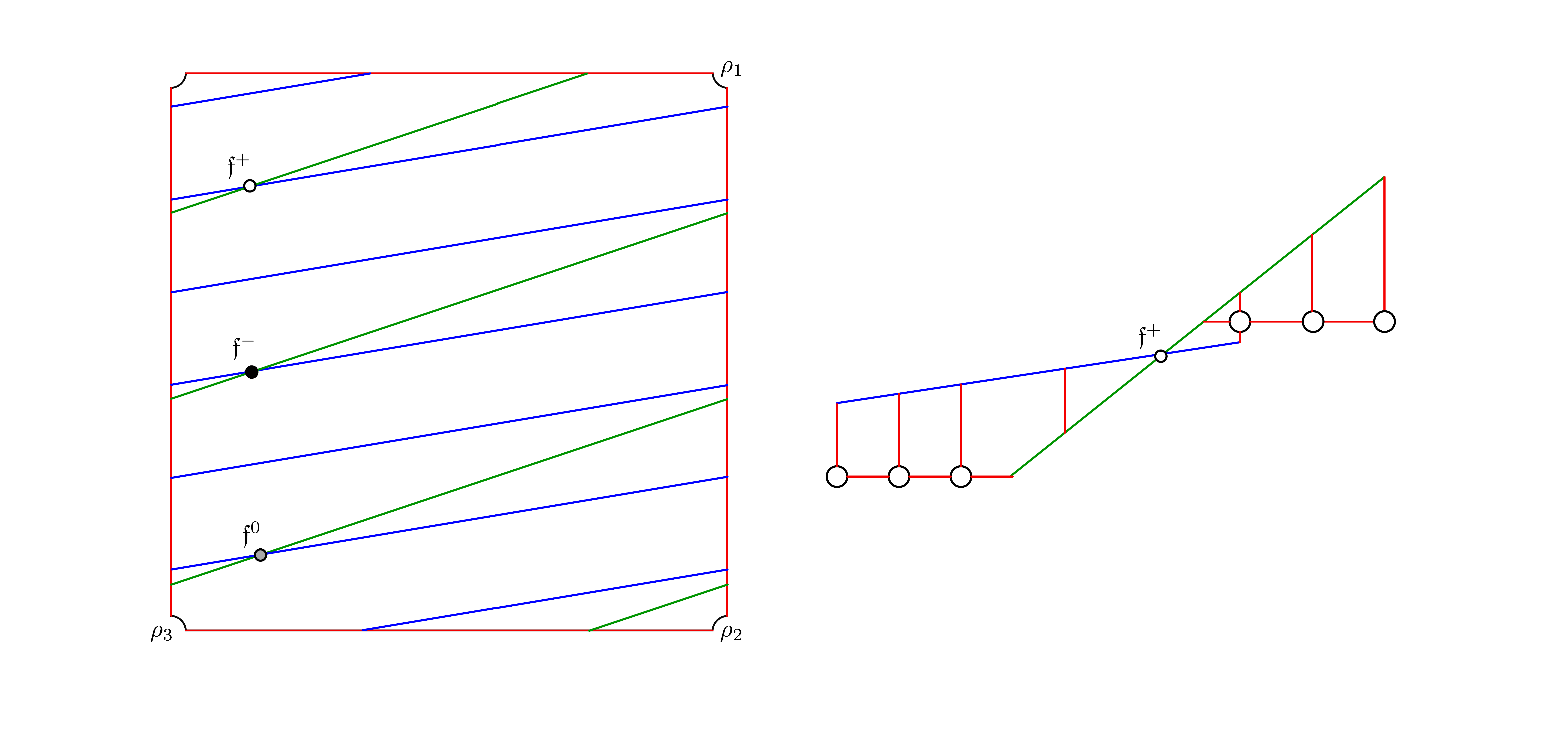}
    \caption{A bordered Heegaard triple diagram for computing the maps $\frak f^0$, $\frak f^-$, and $\frak f^+$. The two large triangles which contribute terms to $\frak f^0$ are shown on the right.}
    \label{fig:cfa_triangles}
\end{figure}
In the previous section, we studied the hypercube constructed from the maps
\begin{align*}
    \CFh(S^3_{N+km}(K),s) \xra{F_{W_k,[*]}} \CFh(S^3_{N +(k+1)m}(K),[s+*]).
\end{align*}
The dual maps, 
\begin{align*}
    \CFh(S^3_{-N -km}(K)) \xra{F_{W_k^\vee,[*]}} \CFh(S^3_{-N -(k-1)m}(K)),
\end{align*}
can be computed explicitly as 
\begin{align*}
    \CFAh(\bm O_{-km}) \boxtimes \CFDh((\KC)_{-N}) \xra{\frak f^* \boxtimes \bI} \CFAh(\bm O_{-(k-1)m}) \boxtimes \CFDh((\KC)_{-N}).
\end{align*}
Here, $\frak f^* = \{\frak f_k^* \}_{k \ge 0}$ is the map which counts holomorphic triangles in the punctured triangle with a corner at an intersection point $* \in \{+, 0, -\}$ and shifts the $\SpinC$-grading by $\{+1,0,-1\}$ just as in \Cref{sec:tel_hyp}; see \Cref{fig:cfa_triangles}. As these triangle counts occur in the torus, we can compute these maps explicitly. In particular, the bordered perspective produces a purely combinatorial model for (the dual of) $\frak{CFK}(K)$.

\begin{figure}
    \centering
\begin{align*}
    \begin{tikzcd}[ampersand replacement = \&, row sep = huge, column sep = huge]
        \beta 
        \ar[d,"\rho_2", purple] \ar[rr]
        \& \& 
        \gamma
        \ar[d,"\rho_2", blue] 
        \\
        b_{2m}
        \ar[d,"\rho_3 \otimes \rho_2" left, purple] 
        \& \& 
        c_{m}
        \ar[dd,"\rho_3 \otimes \rho_2", blue] 
        \\
        b_{2m-1}
        \ar[d,"\rho_3 \otimes \rho_2" left, purple] 
        \& \& 
        \\
        b_{2m-2}
        \ar[d,"\rho_3 \otimes \rho_2" left, purple] 
        \& \& 
        c_{m-1}
        \ar[d,"\rho_3 \otimes \rho_2", blue] 
        \\
        \vdots 
        \ar[d,"\rho_3 \otimes \rho_2" left, purple] 
        \& \&
        \vdots
        \ar[d,"\rho_3 \otimes \rho_2", blue] 
        \\
        b_{m+1}
        \ar[d,"\rho_3 \otimes \rho_2" left, purple] 
        \ar[rruuuuu,"\rho_3", purple]
        \& \& 
        c_{(m+1)/2}
        \ar[dd,"\rho_3 \otimes \rho_2", blue] 
        \\
        b_{m}
        \ar[d,"\rho_3 \otimes \rho_2" left, purple] 
        \ar[rrdddd, "\rho_1", blue]
        \& \& 
        \\
        \vdots 
        \ar[d,"\rho_3 \otimes \rho_2" left, purple] 
        \& \&
        \vdots
        \ar[d,"\rho_3 \otimes \rho_2", blue] 
        \\
        b_{2}
        \ar[d,"\rho_3 \otimes \rho_2" left, purple] 
        \& \& 
        c_{1}
        \ar[dd,"\rho_3 \otimes \rho_2 \otimes \rho_1", blue] 
        \\
        b_{1}
        \ar[d,"\rho_3 \otimes \rho_2\otimes \rho_1" left, purple] 
        \& \& 
        \\
        \beta \ar[rr]
        \& \& 
        \gamma
    \end{tikzcd}
\end{align*}
    \caption{The map $\frak f^0$. }
    \label{fig:CFA_f0}
\end{figure}
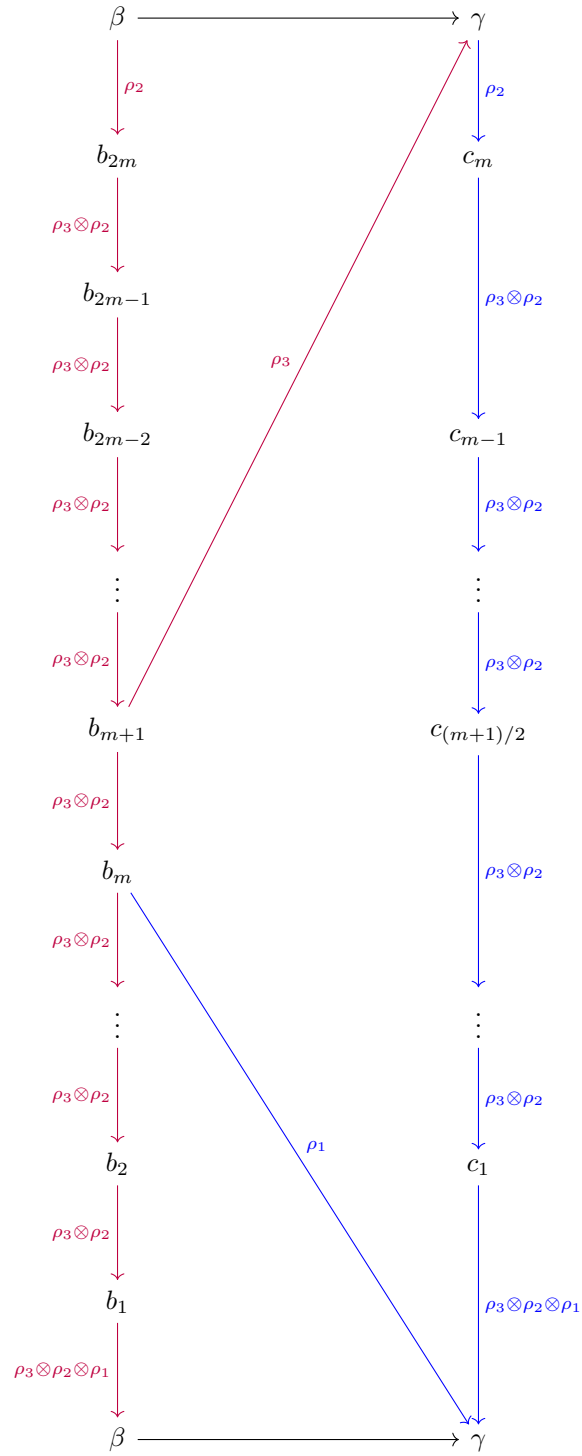

The map $\frak f^0$ is shown in Figure \ref{fig:CFA_f0} and is to be interpreted as follows. The horizontal and diagonal arrows encode the map $\frak f^0: \CFAh(O_{-2m}) \ra \CFAh(O_{-m})$. Paths along the graph along arrows of the same color correspond to terms of $\frak f^0$. For instance, the path
\begin{align*}
    b_{m+2}\xra{\rho_3 \otimes \rho_2} b_{m+1} \xra{\rho_3} \gamma,
\end{align*}
is to be interpreted as the term
\begin{align*}
    \frak{f}^0_3(b_{m+2} \otimes \rho_3 \otimes \rho_{23}) = \gamma.
\end{align*}
Roughly, there are two large triangles: $T_1$ which determines the term
\begin{align*}
    \frak{f}^0_2(b_{m+1} \otimes \rho_3) = \gamma, 
\end{align*}
and another $T_2$ which gives rise to 
\begin{align*}
    \frak{f}^0_2(b_{m} \otimes \rho_1) = \gamma.
\end{align*}
There are many other triangles containing $T_1$ or $T_2$, which look like a rectangle which has been glued to either $T_1$ or $T_2$. See \Cref{fig:cfa_triangles}.

In particular, the path $b_{m+1} \xra{\rho_3} \gamma$ can be precomposed with paths in graph defining $\CFAh(\bm O_{-2m})$ to obtain higher order terms of $\frak f^0$ and the path $b_{m} \xra{\rho_1} \gamma$ can be postcomposed with paths in graph defining $\CFAh(\bm O_{-m})$.

There are also arrows labeled by the identity element of $\cA(T^2)$ which we have not drawn for clarity. There are two such arrows pointing into each $c_i$, coming from $b_i$ and $b_{i+m}$ (this can be seen in \Cref{fig:cfa_triangles} -- both $T_1$ and $T_2$ contain subtriangles which do not touch the boundary). 

We will explicitly compute the maps induced by $\frak{f}^*$ on the complexes $\cLOT^-_m(C_K)$. Our analysis of $\frak f^0$ will be similar to that of our analysis of the differential on  $\cLOT^-_m(C_K)$; we will consider all terms of $\frak f^*$ and look for compatible type D operations.

Before explicitly computing these maps, we state a useful lemma. 

\begin{lem}\label{lem:SpinC bound}
    Suppose $s$ is an integer satisfying $|s| < m_C + m$. If $\xi$ is an element of $\cLOT^-_{m}(C_K)$ of $\SpinC$-grading $[s]$, then $\xi$ is not an element of $\CFAh(\bm O_{-m}) \boxtimes (\cH \oplus \cU)$. 
\end{lem}
\begin{proof}
    Let $\xi = \beta \boxtimes U^{-j} \cdot \eta$ for $U^{-j}\eta \in (\cH \oplus \cU)$. Under the map $\Phi_{m,s}$, the element $\xi$ corresponds to $U^{-j-m}\eta$ $(j > 0)$, which has Alexander grading $s = A(\eta)+(j+m) \ge m_C + m$. The lemma follows.
\end{proof}

Our first computation involves the maps
\begin{align*}
    \frak F^0_s:= \frak f^0 \boxtimes \bI: \cLOT^-_{2m}(K, s) \ra \cLOT^-_{m}(K, s). 
\end{align*}
We will write $\frak F^0$ for the total map. According to \Cref{lem:LOT_model_comp}, both $\cLOT^-_{2m}(C_K)$ and $\cLOT^-_{m}(C_K)$ are truncations of $\hAb$ (quotient out $\hAb$ for $|s|\gg 0$). And in fact, $\frak F^0$ respects this identification. 

\begin{lem}\label{lem: F0 is id}
Suppose that $m \gg N > M_C-m_C$ and that $|s| < m_C + m$. Then, there is a commutative square
\begin{align*}
    \begin{tikzcd}[ampersand replacement = \&]
        \cLOT^-_{2m}(C_K,s) \ar[r,"\frak F^0_s"]\ar[d,"\Phi_{2m,s}"] \& 
            \cLOT^-_{m}(C_K,s)\ar[d,"\Phi_{m,s}"] \\
        \hAb_s \ar[r,"\id"] \&
            \hAb_s
    \end{tikzcd}
\end{align*}
where the vertical arrows are the isomorphisms given by \Cref{lem:LOT_model_comp}.
\end{lem}

\begin{proof}
Let us compute. We refer the reader to \Cref{fig:lot_f0_ex} for an example to aid in following the computation. According to \Cref{lem:SpinC bound}, it suffices to consider sequences involving $\cV$ and $\cK$.

\begin{enumerate}
    \item[($\cS_{-2m} \xra{\frak f^0}\cS_{-m}$):] As in the computation of the box tensor product differential, we begin by considering terms of $\frak f^0$ which respect the vector space decompositions $\cLOT^-_{2m}(C_K) \cong \cH_{2m}\oplus \cU_{2m}\oplus \cV_{2m}\oplus (\cK)^{2m}$ and $\cLOT^-_{m}(C_K) \cong \cH_{m}\oplus \cU_{m}\oplus \cV_{m}\oplus (\cK)^{m}$. The simplest terms of $\frak f^0$ are those of the form
\begin{align*}
    \frak f^0_1(\beta) = \gamma.
\end{align*}
These clearly pair nontrivially. For each $\eta \in \cV$, terms $\gamma \boxtimes \eta$ appear as a term in $f^0 \boxtimes \bI(\beta \boxtimes \eta)$. These clearly commute with $\Phi_{m,s}$. 

Next, we consider the terms of the form 
\begin{align*}
    \frak f^0_1(b_i) = c_i =\frak f^0_1(b_{i+m}).
\end{align*}
Both maps pair to produce nontrivial morphisms on the box tensor product, giving rise to the components of $\bI \boxtimes\frak f^0$ of the form
\begin{align*}
    \left(\bigoplus_{i=1}^{m} U^{-i} \cdot \cK\right) \oplus \left(\bigoplus_{i=m+1}^{2m} U^{-i} \cdot \cK\right) \xra{[1 \, U^m]} \bigoplus_{i=1}^{m} U^{-i} \cdot \cK, \\ 
    U^{-i} \eta \mapsto  U^{-i} \eta,\\
    U^{-(i+m)} \eta \mapsto  U^{-i} \eta.
\end{align*}
Once again, $\frak F^0$ is the $\SpinC$-grading preserving part of this map, i.e. only the component taking $U^{-i} \eta \mapsto  U^{-i} \eta.$

There is also a term
\begin{align*}
    \frak{f}^0 (\beta\otimes \overbrace{\rho_{23}\otimes \hdots \otimes \rho_{23}}^{m}) = \gamma,
\end{align*}
which may pair non-trivially. But, as we have assumed $m \gg N > M_C-m_C$, we may assume that the $\overbrace{\rho_{23}\otimes \hdots \otimes \rho_{23}}^{m}$ component of $\delta^{m}$ is trivial. Higher maps from $\beta$ to $\gamma$ are ruled out similarly.
\\
\item[($\cK_{-2m}\xra{\frak f^0}\cV_{-m}$):] There are terms of the form 
\begin{align*}
    f_2(b_{m} \otimes \rho_1) & = \gamma,\\
    f_2(b_{m} \otimes \rho_{12}) & = c_m,\\
    f_{i+3}(b_{m} \otimes \rho_{123} \otimes \overbrace{\rho_{23}\otimes \hdots \otimes \rho_{23}}^{i}\otimes\rho_2) &= c_{m-i-1}, \quad i > 0.
\end{align*}
The first pairs with a matching type D operation, yielding a map 
\begin{align*}
    U^{-m}\cdot \hAb/(U^{-1}, V^{-1}) \ra \hAb/(U^{-1}), \quad b_{m}\boxtimes \eta = U^{-m} \eta \mapsto \eta = \gamma\boxtimes \eta. 
\end{align*}
This map does not preserve $\SpinC$-gradings, and therefore does not contribute to $\frak F^0$. Maps of the second and third kind have no type D operations with which to pair.
\end{enumerate}

There are also terms of the form 
\begin{align*}
    \frak f^0(b_{i}\otimes \rho_3 \otimes \overbrace{\rho_{23}\otimes \hdots \otimes \rho_{23}}^{i-1}\otimes \rho_2\otimes \rho_{123} \otimes \overbrace{\rho_{23}\otimes \hdots \otimes \rho_{23}}^{m-1}) = \gamma,
\end{align*}
which, a priori, could pair with a sequence of type D operations. But, once again, as $m$ is large relative to $N$, such terms do not contribute. Higher maps are ruled out similarly.
\end{proof}

\begin{figure}
    \centering
    \includegraphics[width=0.75\linewidth]{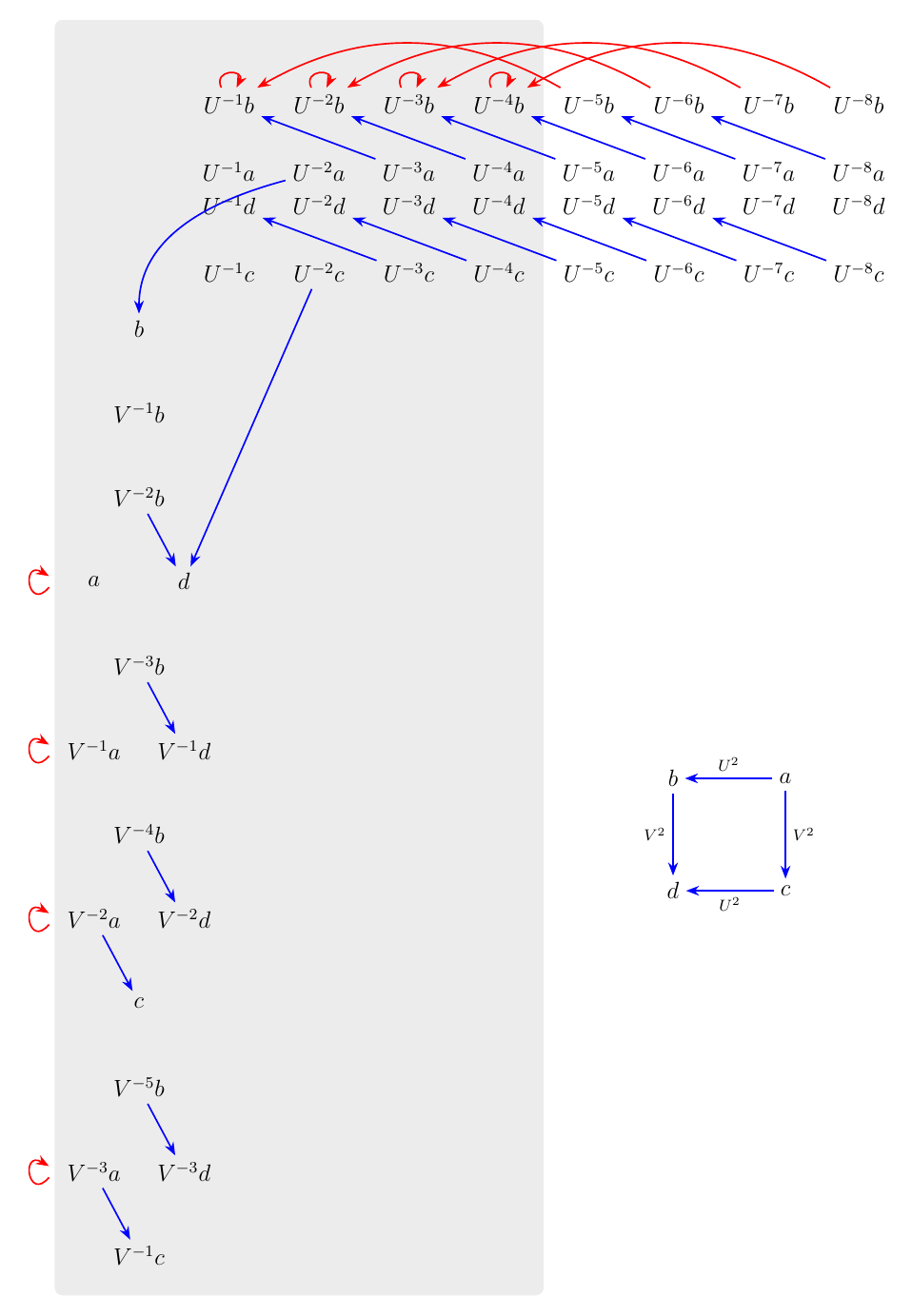}
    \caption{A schematic of the morphism $\frak f^0 \boxtimes \bI$. Here, we have drawn a portion of $\cLOT^-_{2m}(C_K,s)$; its differential is given by the blue arrows. The complex $\cLOT^-_{m}(C_K,s)$ is identified with the quotient within the shaded box. The morphism $\frak f^0 \boxtimes \bI$ is drawn in red. In particular, provided $m$ is large relative to $s$, $\frak f^0 \boxtimes \bI$ induces an isomorphism on the summand in $\SpinC$-structure $[s]$.}
    \label{fig:lot_f0_ex}
\end{figure}

\subsection{Combinatorial Hyperbox}
Having computed the map $\frak f^0 \boxtimes \bI$, we turn to the maps $\frak f^+\boxtimes \bI$ and $\frak f^-\boxtimes \bI$ which will recover the $V$ and $U$ actions. As before, the maps $\frak f^-$ and $\frak f^+$ can be computed combinatorially, and are shown in \Cref{fig:LOT_UV}. \\

\begin{figure}
\begin{align*}
    \begin{tikzcd}[ampersand replacement = \&, row sep = huge, column sep = huge]
        \beta 
        \ar[d,"\rho_2", purple] 
        \& \& 
        \gamma
        \ar[d,"\rho_2", blue] 
        \\
        b_{2m}
        \ar[d,"\rho_3 \otimes \rho_2" left, purple] 
        \& \& 
        c_m
        \ar[dd,"\rho_3 \otimes \rho_2", blue] 
        \\
        b_{2m-1}
        \ar[d,"\rho_3 \otimes \rho_2" left, purple] 
        \& \& 
        \\
        b_{2m-3}
        \ar[d,"\rho_3 \otimes \rho_2" left, purple] 
        \& \& 
        c_{m-1}
        \ar[d,"\rho_3 \otimes \rho_2", blue] 
        \\
        \vdots 
        \ar[d,"\rho_3 \otimes \rho_2" left, purple] 
        \& \&
        \vdots
        \ar[dd,"\rho_3 \otimes \rho_2", blue] 
        \\
        b_{m+2}
        \ar[d,"\rho_3 \otimes \rho_2" left, purple] 
        \ar[rruuuuu,"\rho_3", purple]
        \& \& 
        \\
        b_{m+1}
        \ar[d,"\rho_3 \otimes \rho_2" left, purple] 
        \ar[rruuuuu]
        \& \& 
        \hdots
        \ar[dd,"\rho_3 \otimes \rho_2", blue] 
        \\
        b_{m}
        \ar[d,"\rho_3 \otimes \rho_2" left, purple] 
        \ar[rruuuu]
        \& \& 
        \\
        \vdots 
        \ar[d,"\rho_3 \otimes \rho_2" left, purple] 
        \& \&
        \vdots
        \ar[d,"\rho_3 \otimes \rho_2", blue] 
        \\
        b_{2}
        \ar[d,"\rho_3 \otimes \rho_2" left, purple] 
        \ar[rr]
        \& \& 
        c_{1}
        \ar[dd,"\rho_3 \otimes \rho_2 \otimes \rho_1", blue] 
        \\
        b_{1}
        \ar[d,"\rho_3 \otimes \rho_2\otimes \rho_1" left, purple] 
        \ar[rrd,"\rho_1", blue]
        \& \& 
        \\
        \beta 
        \& \& 
        \gamma
    \end{tikzcd}
    \hspace{1cm}
    \begin{tikzcd}[ampersand replacement = \&, row sep = huge, column sep = huge]
        \beta 
        \ar[d,"\rho_2", purple] 
        \& \& 
        \gamma
        \ar[d,"\rho_2", blue] 
        \\
        b_{2m}
        \ar[d,"\rho_3 \otimes \rho_2" left, purple] 
        \ar[rru, "\rho_3", purple]
        \& \& 
        c_{m}
        \ar[dd,"\rho_3 \otimes \rho_2", blue] 
        \\
        b_{2m-1}
        \ar[d,"\rho_3 \otimes \rho_2" left, purple] 
        \ar[urr] 
        \& \& 
        \\
        b_{2m-2}
        \ar[d,"\rho_3 \otimes \rho_2" left, purple] 
        \ar[rr]
        \& \& 
        c_{m-1}
        \ar[dd,"\rho_3 \otimes \rho_2", blue] 
        \\
        \vdots 
        \ar[d,"\rho_3 \otimes \rho_2" left, purple] 
        \& \&
        \\
        b_{m}
        \ar[d,"\rho_3 \otimes \rho_2" left, purple] 
        \ar[rrdddd]
        \& \& \vdots
        \ar[ddd,"\rho_3 \otimes \rho_2" right, blue]
        \\
        b_{m-1}
        \ar[d,"\rho_3 \otimes \rho_2" left, purple]
        \ar[rrdddd,"\rho_1" left, blue]
        \& \& 
        \\
        \vdots 
        \ar[d,"\rho_3 \otimes \rho_2" left, purple] 
        \& \&
        \\
        b_{2}
        \ar[d,"\rho_3 \otimes \rho_2" left, purple] 
        \& \& 
        c_2
        \ar[d,"\rho_3 \otimes \rho_2" right, blue]
        \\
        b_{1}
        \ar[d,"\rho_3 \otimes \rho_2\otimes \rho_1" left, purple] 
        \& \& 
        c_{1}
        \ar[d,"\rho_3 \otimes \rho_2 \otimes \rho_1", blue] 
        \\
        \beta 
        \& \& 
        \gamma
    \end{tikzcd}
\end{align*}
\caption{Left: $\frak{f}^-$; right: $\frak{f}^+$}
\label{fig:LOT_UV}
\end{figure}

Let $\frak F^-_s$ be the component of $\frak f^-\boxtimes \bI$ which maps $\cLOT^-_{2m}(K, s)$ to $\cLOT^-_{m}(C_K, s-1)$.

\begin{lem}\label{lem: F- is U}
Suppose that $m \gg N > M_C-m_C$ and that $|s| < m_C + m$. Then, there is a commutative square
\begin{align*}
    \begin{tikzcd}[ampersand replacement = \&]
        \cLOT^-_{2m}(C_K,s) \ar[r,"\frak F^-_s"]\ar[d,"\Phi_{2m,s}"] \& 
            \cLOT^-_{m}(K,s-1)\ar[d,"\Phi_{m,s-1}"] \\
        \hAb_s \ar[r,"U"] \&
            \hAb_{s-1}
    \end{tikzcd}
\end{align*}
where the vertical arrows are the isomorphisms given by \Cref{lem:LOT_model_comp}.
\end{lem}

\begin{proof}
As before, we shall break the computation up in terms of the various components of $\frak f^-$. Again, we only consider sequences involving $\cV$ and $\cK$.\\

\begin{enumerate}
    \item[($\cS_{-2m} \xra{\frak f^-}\cS_{-m}$):] The terms 
\begin{align*}
    \frak f^-_m(\beta \otimes \overbrace{\rho_{23}\otimes \hdots \otimes \rho_{23}}^{m-1}) = \gamma,
\end{align*}
do not contribute as $m \gg N$.

Consider components of $\frak f^-$ of the form 
\begin{align*}
    \frak f^-_1(b_i) = c_{i-1}, \quad 2 \le i \le m+1.
\end{align*}
These certainly pair nontrivially with type D operations, and give rise to terms 
\begin{align*}
    b_i \boxtimes \eta = U^{-i} \eta \mapsto U^{-i+1} \eta = c_{i-1}\boxtimes \eta.
\end{align*}
I.e., on these components, the map is exactly multiplication by $U$. \\
\item[($\cK_{-2m}\xra{\frak f^-} \cV_{-m}$):] We ought to expect elements of the $\HFKh$-complex of the form $U^{-1} \eta$ to map to $\eta$ as an element of the vertical component. Indeed, there is a map
\begin{align*}
    \frak f^-_2(b_1 \otimes \rho_1) = \gamma,
\end{align*}
which has a matching type D operation (induced by inclusion) and gives rise to the term $U^{-1} \eta \mapsto \eta.$ \\
\end{enumerate}
The action of $\frak F^-$ is trivial on the remaining components of the complex. Though, as long as $m$ is large, little information about the $U$-action is lost. This concludes the proof.
\end{proof}

Now, for the $V$-action. Let $\frak F^+_s$ be the component of $\frak f^+\boxtimes \bI$ which maps $\cLOT^-_{2m}(K, s)$ to $\cLOT^-_{m}(C_K, s+1)$. This computation is somewhat more straightforward than the previous.

\begin{lem}\label{lem: F+ is V}
Suppose that $m \gg N > M_C-m_C$ and that $|s| < m_C + m$. Then, there is a commutative square
\begin{align*}
    \begin{tikzcd}[ampersand replacement = \&]
        \cLOT^-_{2m}(C_K,s) \ar[r,"\frak F^+_s"]\ar[d,"\Phi_{2m,s}"] \& 
            \cLOT^-_{m}(K,s+1)\ar[d,"\Phi_{m,s+1}"] \\
        \hAb_s \ar[r,"V"] \&
            \hAb_{s+1}
    \end{tikzcd}
\end{align*}
where the vertical arrows are the isomorphisms given by \Cref{lem:LOT_model_comp}.
\end{lem}
\begin{proof}
Consider now the map $\frak f^+$, which is also shown in \Cref{fig:LOT_UV}. 

\begin{enumerate}
    \item[($\cS_{-2m} \xra{\frak f^+}\cS_{-m}$):] There are components of the form 
\begin{align*}
    \frak f^+_2(\beta \otimes \rho_{23}) = \gamma.
\end{align*}
On the vertical portion of the complex, this pairs with $D_{23}$ which is precisely given by multiplication by $V$ (as predicted). 
There are also terms of the form 
\begin{align*}
    \frak f^+_1(b_{i+m-1}) = c_{i} & & 1 \le i \le m.
\end{align*}
These components do, of course, pair nontrivially, but shift the $\SpinC$-grading by $(m-1)$, and therefore do not contribute to $\frak F^+$.  Similarly, there are terms 
\begin{align*}
    \frak f^+_2(b_{m-1}\otimes \rho_{12}) = c_m & & \\
    \frak f^+_{3+i}(b_2\otimes \rho_{123} \otimes \overbrace{\rho_{23}\otimes \hdots \otimes \rho_{23}}^{i}\otimes \rho_2) = c_{m-i-1} & & 0 \le i \le m-2,
\end{align*}
but neither collection has matching type D operations, and therefore do not contribute. 

\item[($\cK_{-2m} \xra{\frak f^+}\cV_{-m}$):] Finally, there are terms 
\begin{align*}
    \frak f^+_2(b_{m-1}\otimes \rho_1) = \gamma & & 
\end{align*}
This map does not have a matching type D sequence with which to pair, and should take $U^{-(m-1)}\eta \in U^{-(m-1)} \cdot \cK$ to $\eta \in \cV$. But, as this does not have the correct grading shift, this map also does not contribute. \\ 
\end{enumerate}
\end{proof}

Hence, we can define a quasi-module
\begin{align*}
    \cT^-_K:=(\Tel(\cLOT^-_{km}(C_K)),\frak F^0, \frak F^-, \frak F^+, 0),
\end{align*}
which splits as $\cT^-_K = \bigoplus_s(\cT^-_K)_s$. In fact, this is a strict $\F[U,V]$-quasimodule. This is because the maps $\frak f^0$, $\frak f^-$, and $\frak f^+$ are computed by counting triangles in the universal cover of the torus, and it is easy to verify that there are no higher homotopies (there are no convex polygons which could furnish such maps which have the correct boundary conditions.) Hence, $\cT^-_K$ has the structure of an honest $\F[U,V]$-module. But, more is true. It actually has the structure of an $\cR$-module, namely, the product $UV$ is identically 0.

\begin{figure}
    \centering
    \begin{tikzpicture}[
    x=2.4cm, y=1.2cm, 
    >=stealth, 
    thick,
    every node/.style={font=\large}
]

    \colorlet{myred}{red!80!black}
    \colorlet{myblue}{blue!80!black}
    \colorlet{mygreen}{green!60!black}


    \node (Lx1) at (0, 13) {$\beta$};
    \foreach \i in {12,11,...,1} {
        \node (Lb\i) at (0, \i) {$b_{\i}$};
    }
    \node (Lx2) at (0, 0) {$\beta$};

    \node (Ly1) at (1, 9) {$\gamma$};
    \foreach \i in {8,7,...,1} {
        \node (Lc\i) at (1, \i) {$c_{\i}$};
    }
    \node (Ly2) at (1, 0) {$\gamma$};

    \node (Lz1) at (2, 5) {$\delta$};
    \foreach \i in {4,3,2,1} {
        \node (Ld\i) at (2, \i) {$d_{\i}$};
    }
    \node (Lz2) at (2, 0) {$\delta$};

    \draw[->, myblue] (Lx1) -- node[left, font=\scriptsize, text = myblue] {2}(Lb12);
    \foreach \i [evaluate=\i as \nexti using int(\i-1)] in {12,11,...,2} {
        \draw[->, myblue] (Lb\i) -- node[left, font=\scriptsize, text = myblue] {3,2} (Lb\nexti);
    }
    \draw[->,myblue] (Lb1) -- node[left, font=\scriptsize,text=myblue] {3,2,1}(Lx2);
    
    \draw[->,mygreen] (Ly1) -- node[right, font=\scriptsize, text = mygreen] {2}(Lc8);
    \foreach \i [evaluate=\i as \nexti using int(\i-1)] in {8,7,...,2} {
        \draw[->,mygreen] (Lc\i) -- node[right, font=\scriptsize,text=mygreen] {3,2} (Lc\nexti);
    }
    \draw[->,mygreen] (Lc1) -- node[right, font=\scriptsize,text=mygreen] {3,2,1}(Ly2);

    \draw[->,myred] (Lz1) -- node[right, font=\scriptsize,text=myred]{2}(Ld4);
    \foreach \i [evaluate=\i as \nexti using int(\i-1)] in {4,3,2} {
        \draw[->,myred] (Ld\i) -- node[right, font=\scriptsize,text=myred] {3,2} (Ld\nexti);
    }
    \draw[->,myred] (Ld1) -- node[right, font=\scriptsize,text=myred]{3,2,1}(Lz2);

    \draw[->,myblue] (Lb12) -- node[above right, inner sep=2pt,font=\scriptsize, text=myblue] {3} (Ly1);
    \foreach \i [evaluate=\i as \j using int(\i-3)] in {11,10,...,4} {
        \draw[->] (Lb\i) -- (Lc\j);
    }
    \draw[->,mygreen] (Lb3) -- node[right, font=\scriptsize,text=mygreen] {1}(Ly2);

    \draw[->, mygreen] (Lc6) -- node[above, inner sep=2pt, font=\scriptsize,text=mygreen] {3} (Lz1);
    \foreach \i [evaluate=\i as \j using int(\i-1)] in {5,4,...,2} {
        \draw[->] (Lc\i) -- (Ld\j);
    }
    \draw[->,myred] (Lc1) -- node[above, inner sep=2pt, font=\scriptsize, text = myred] {1} (Lz2);



    \node (Rx1) at (4, 13) {$\beta$};
    \foreach \i in {12,11,...,1} {
        \node (Rb\i) at (4, \i) {$b_{\i}$};
    }
    \node (Rx2) at (4, 0) {$\beta$};

    \node (Ry1) at (5, 9) {$\gamma$};
    \foreach \i in {8,7,...,1} {
        \node (Rc\i) at (5, \i) {$c_{\i}$};
    }
    \node (Ry2) at (5, 0) {$\gamma$};

    \node (Rz1) at (6, 5) {$\delta$};
    \foreach \i in {4,3,2,1} {
        \node (Rd\i) at (6, \i) {$d_{\i}$};
    }
    \node (Rz2) at (6, 0) {$\delta$};

    \draw[->, myred] (Rx1) -- node[left, font = \scriptsize] {2} (Rb12);
    \foreach \i [evaluate=\i as \nexti using int(\i-1)] in {12,11,...,2} {
        \draw[->, myred] (Rb\i) -- node[left, font=\scriptsize, text=myred] {3,2} (Rb\nexti);
    }
    \draw[->, myred] (Rb1) -- node[left, font = \scriptsize] {3,2,1} (Rx2);

    \draw[->, myblue] (Ry1) -- node[left, font = \scriptsize]{2}(Rc8);
    \foreach \i [evaluate=\i as \nexti using int(\i-1)] in {8,7,...,2} {
        \draw[->, myblue] (Rc\i) -- node[left, font=\scriptsize, text=myblue] {3,2} (Rc\nexti);
    }
    \draw[->, myblue] (Rc1) -- node[right,font=\scriptsize] {3,2,1} (Ry2);

    \draw[->, mygreen] (Rz1) -- node[right,font=\scriptsize] {3} (Rd4);
    \foreach \i [evaluate=\i as \nexti using int(\i-1)] in {4,3,2} {
        \draw[->, mygreen] (Rd\i) -- node[right, font=\scriptsize, text=mygreen] {3,2} (Rd\nexti);
    }
    \draw[->, mygreen] (Rd1) -- node[right, font=\scriptsize] {3,2,1} (Rz2);

    \draw[->, myred] (Rb10) -- node[above, inner sep=2pt,font=\scriptsize, text=myred] {3} (Ry1);
    \foreach \i [evaluate=\i as \j using int(\i-1)] in {9,8,...,2} {
        \draw[->] (Rb\i) -- (Rc\j);
    }
    \draw[->,myblue] (Rb1) -- node[above, inner sep=2pt, font=\scriptsize,text=myblue] {1}(Ry2);

    \draw[->, myblue] (Rc8) -- node[above right, inner sep=2pt, font=\scriptsize,text=myblue] {3} (Rz1);
    \foreach \i [evaluate=\i as \j using int(\i-3)] in {7,6,5,4} {
        \draw[->] (Rc\i) -- (Rd\j);
    }
    \draw[->, mygreen] (Rc3) -- node[above right, inner sep=2pt, font=\scriptsize,text=mygreen] {1} (Rz2);

\end{tikzpicture}
    \caption{The compositions $\frak f^- \frak f^+$ (left) and $\frak f^+ \frak f^-$ (right). In this example, $m = 4$ and $k = 1$.}
    \label{fig:f+f- composition}
\end{figure}
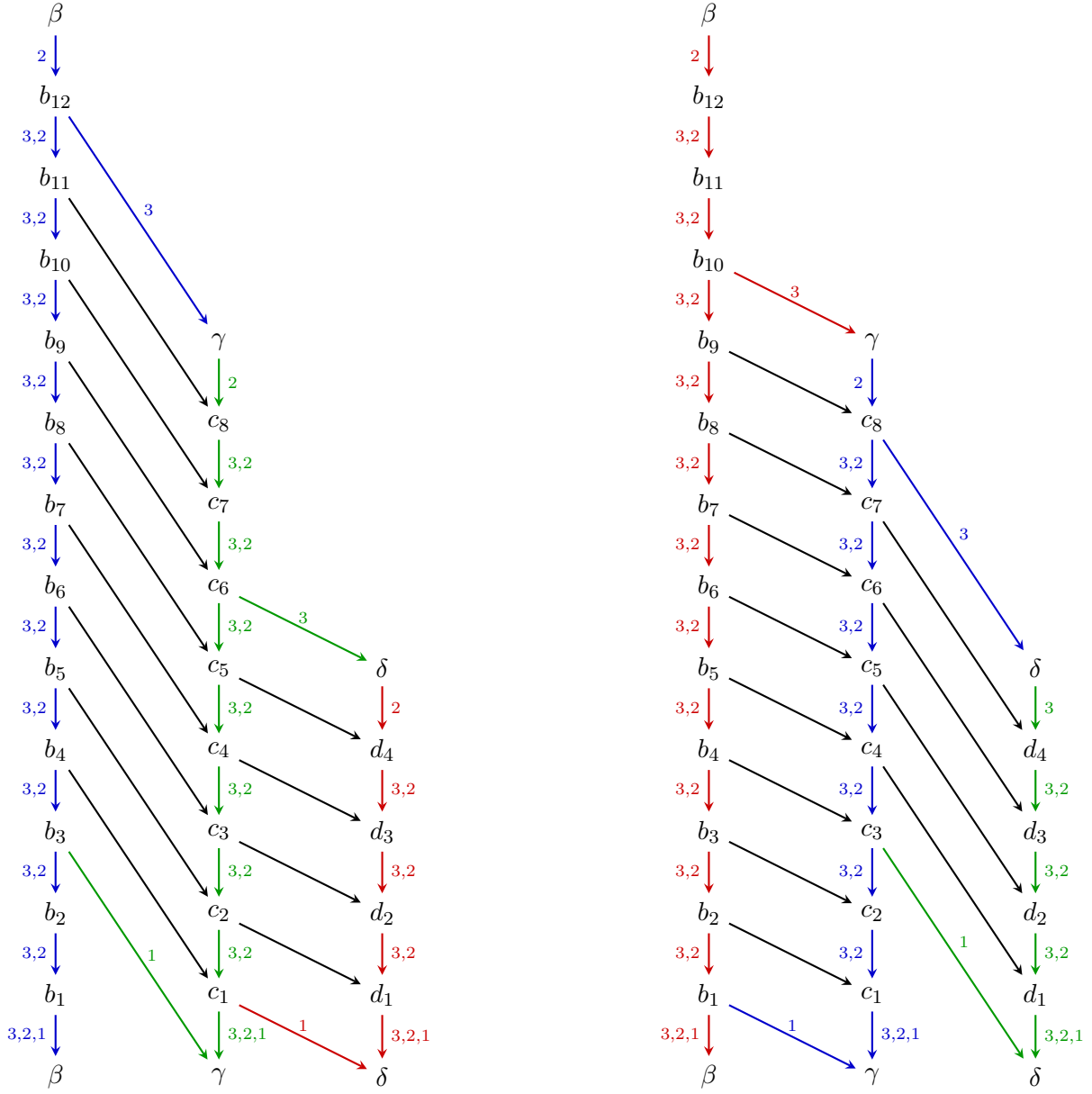

\begin{lem} \label{lem: strict R-module}
    For any integer $s$ satisfying $|s|<km+m_C$ for $m\gg 0$, the map
    \begin{align*}
        \mathfrak{f}^+\mathfrak{f}^-\boxtimes \mathrm{id}: \cLOT^-_{(k+2)m}(C_K, s) \ra \cLOT^-_{km}(C_K, s)
    \end{align*}
    is identically zero. The same statement holds for $\mathfrak{f}^-\mathfrak{f}^+\boxtimes \mathrm{id}$. 
\end{lem}
\begin{proof} 
    According to \Cref{lem:SpinC bound}, any class in $\cLOT^-_{km}(C_K, s)$ in grading $[s]$ with $s < km + m_C$ must be an element of $\cV$ or $\cK$. We now verify that for $s$ in this range, the above compositions are trivial. We will write $\{\beta, b_1 , \hdots, b_{(k+2)m}\}$,  $\{\gamma, c_1 , \hdots, c_{(k+1)m}\}$, and $\{\delta, d_1 , \hdots, d_{km}\}$ for the generators of $\CFAh(\bm O_{(k+j)m})$ for $j \in \{0,1,2\}$. The maps $\frak f^\pm: \CFAh(\bm O_{(k+1)m}) \ra \CFAh(\bm O_{km})$ can be read off from \Cref{fig:LOT_UV}. 

    The simplest terms of $\frak f^-$ and $\frak f^+$ are those of the form 
    \begin{align*}
        \frak f^-_1(b_i) = c_{i-1} & \quad & 2 \le i \le (k+1)m+1 \\
        \frak f^+_1(c_j) = d_{j-(m-1)} &  \quad & m \le j \le (k+1)m-1.
    \end{align*}
    These of course pair, giving rise to terms in $\frak f^\pm \boxtimes \bI$. Recall that by \cite[Lemma 2.3.3]{LOT_bimodules}
    \begin{align*}
        (\mathfrak{f}^+\circ \mathfrak{f}^-)\boxtimes \mathrm{id} = (\mathfrak{f}^+\boxtimes \bI) \circ (\mathfrak{f}^-\boxtimes \bI).
    \end{align*}
    It is therefore clear that these terms can pair non-trivially: for $\eta \in \cK$ with $\SpinC$-grading $[-i]$, we have non-trivial compositions of the form 
    \begin{align*}
        ((\frak f^+ \circ \frak f^-)\boxtimes \bI)(b_i \boxtimes \eta) = d_{i-m}\boxtimes \eta &  \quad & m+1 \le i \le (k+1)m +1.
    \end{align*}
    However, $b_i \boxtimes \eta$ lies in $\SpinC$-grading $[s]$, whereas $d_{i-m} \boxtimes \eta$ lies in $\SpinC$-grading $[s-m]$. Hence, when we restrict $(\frak f^+ \circ \frak f^-)\boxtimes \bI$ to $\SpinC$-grading $[s]$ in the domain and co-domain, these terms are trivial. An analogous argument holds for $(\frak f^- \circ \frak f^+)\boxtimes \bI$. \\

    We now turn to the higher order terms. By our assumption on $s$, we only need to consider terms of the form $\beta\boxtimes \eta$ for $\eta \in \cV$ and $b_i\boxtimes \vartheta$ for $\vartheta \in \cK$. Let us consider terms of the first kind. The only type D operations involving $\eta \in \cV$ are of the form $\delta^\ell(\eta) = \overbrace{\rho_{23}\otimes\hdots \otimes \rho_{23}}^\ell \otimes V^k \eta$. There are matching terms: 
    \begin{align*}
        \frak f^+_{2}(\frak f^-_m(\beta\otimes \overbrace{\rho_{23}\otimes\hdots \otimes \rho_{23}}^{m-1}) \otimes \rho_{23}) = \delta,
    \end{align*}
    giving rise to the term $\delta \boxtimes V^m \eta$ in $((\frak f^+ \circ \frak f^-)\boxtimes \bI)(\beta \boxtimes \eta)$. But, as $m$ is large relative to $S$, any term $\eta$ of the truncated vertical complex must be annihilated by $V^m$. Hence, such terms are trivial. There are no type D operations from $\cV$ to $\cK$, so all other terms in $((\frak f^+ \circ \frak f^-)\boxtimes \bI)(x \boxtimes \eta)$ are trivial. 

    There can be no sequences giving rise to terms taking elements $b_i\boxtimes \vartheta$ for $\vartheta \in \cK$ and $1 \le i \le (k+2)m$ to elements $d_j \boxtimes \vartheta$ for $\vartheta \in \cK$ for $1 \le j \le km$, as any such sequence must contain the terms of the form
    \begin{align*}
        b_i \xra{\rho_3 \otimes \rho_2} \hdots b_2 \xra{\rho_1} \hdots c_3 \xra{\rho_1} \hdots.
    \end{align*}
    On the type $D$ side, there are sequences with exactly one $\rho_1$ and exactly one $\rho_{123}$, but none with both. Hence, these type A sequences have nothing to pair with and do not contribute.
    
    Therefore, it remains to consider terms taking elements $b_i\boxtimes \vartheta$ for $\vartheta \in \cK$ and $1 \le i \le (k+2)m$ to elements $\delta \boxtimes \vartheta$ for $\vartheta \in \cV$. These must come from sequences from $b_i$ to $\delta$. There are sequences of the form 
    \begin{align*}
        b_{m(k+1)+2+i} \xra{\rho_{3}\otimes \rho_2}b_{m(k+1)+2+(i-1)} \xra{\rho_{3}\otimes \rho_2} \hdots \xra{\rho_{3}\otimes \rho_2} b_{m(k+1)+2} \xra[\frak{f}^-]{\rho_3} \gamma \xra{\rho_2} c_{m(k+1)} \xra[\frak f^+]{\rho_3} \delta,
    \end{align*}
    where $0 \le i \le m-2$, determining terms 
    \begin{align*}
     (\frak f^+ \circ \frak f^-)_{i+3}(b_{m(k+1)+2+i}\otimes \rho_{3}\otimes \overbrace{\rho_{23}\otimes\hdots \otimes \rho_{23}}^{i+1})=\gamma.
    \end{align*}
    These do have matching type $D$ operations. Given $\eta \in \cK$, there are sequences 
    \begin{align*}
        \eta \xra{\rho_3} \eta \xra{\rho_{23}} U^{-1} \eta \hdots \xra{\rho_{23}} U^{-i}\eta, & \quad & 0 \le i \le m_C + S_N - A(\eta).
    \end{align*}
    There are additional sequences which involve the unstable chain. If the complex has label $(\sigma, \tau)$, then there are additional sequences:
    \begin{align*}
        \sigma \xra{\rho_3} \sigma \xra{\rho_{23}} U^{-1} \sigma \hdots \xra{\rho_{23}} U^{-(m_C+S_N-A(\sigma))}\sigma \xra{\rho_{23}}  U^{-(m_C+S_N-A(\sigma))-1}\sigma \xra{\rho_{23}} \hdots \\
        \hdots \xra{\rho_{23}} U^{-(m_C+S_N-A(\sigma))-j}\sigma, & \quad & 1 \le j \le 2\floor{N/4}+1,
    \end{align*}
    corresponding to sequences which start at $\sigma \in \cK$ and terminate at somewhere in the unstable chain, as well as
    \begin{align*}
        \sigma \xra{\rho_3} \sigma \xra{\rho_{23}} \hdots \xra{\rho_{23}} U^{-(m_C+S_N-A(\sigma))-2N} \sigma \xra{\rho_{23}} V^{-m_C+S_N+A(\tau)}\tau \xra{\rho_{23}} \hdots\\
        \hdots\xra{\rho_{23}} V^{-(M_C + T_N)+A(\tau)+\ell} \tau, & \quad & 0 \le \ell \le M_C + T_N-A(\tau),
    \end{align*}
    corresponding to sequences starting at $\sigma \in \cK$ and ending at some $V$-power of $\tau$ in $\cV$.
    
    If $m$ is large, we expect all of these sequences to pair nontrivially. The first two sequences give rise to terms
    \begin{align*}
        U^{-(m(k+1)+2+i)} \eta \sim b_{m(k+1)+2+i} \boxtimes \eta \mapsto \delta \boxtimes U^{-(i+1)}\eta \sim U^{-(km+i+1)}\eta.
    \end{align*}
    These maps take elements from $\SpinC$-grading $[s]$ to elements in $\SpinC$-grading $[s+m+1]$, and therefore do not contribute. 

    The last sequence gives rise to terms 
    \begin{align*}
        U^{-(m(k+1)+2+i)} \sigma \sim b_{m(k+1)+2+i} \boxtimes \sigma \mapsto \delta \boxtimes V^{-(M_C + T_N)+A(\tau) + j}\tau \sim V^{-(M_C + T_N)+A(\tau) + j}\tau,
    \end{align*}
    where $j = i - (M_C + T_N) + A(\sigma) - 2N -1$ with $2N + M_C + T_N - A(\sigma) < i < m-1$. These take terms in $\SpinC$-grading $[s]$ to elements in $\SpinC$-grading $[s+m(k+1)+2(A(\sigma) - N - (M_C + T_N) +1]$. These terms will not contribute, unless 
    $$-m(k+1) + \big(2A(\tau) - 2(M_C + T_N) - 2N - 3\big) \equiv -m + \big(2A(\tau) - 2(M_C + T_N) - 2N - 3\big) \equiv 0 \pmod{mk}.$$ 
    When $k \ge 1$ and
    \begin{align*}
        m > |2(M_C + T_N) + 2N + 3 - 2A(\tau)|,
    \end{align*}
    we have $-m + \big(2A(\tau) - 2(M_C + T_N) - 2N - 3\big) \not\equiv 0 \pmod{mk}.$ Even in the case that $k = 1$, we have 
    \begin{align*}
        -m + \big(2A(\tau) - 2(M_C + T_N) - 2N - 3\big) \equiv \big(2A(\tau) - 2(M_C + T_N) - 2N - 3\big) \pmod{m}.
    \end{align*}
    The sum $\big(2A(\tau) - 2(M_C + T_N) - 2N - 3\big)$ is nonzero (it is odd) and therefore, provided $m > |2(M_C + T_N) + 2N + 3 - 2A(\tau)|$, this quantity is not congruent to $m$. Therefore, in each case, we can choose $m$ large enough that these terms do not contribute to the map.
\end{proof}

From this construction, it is then clear how to recover the quasimodule $\mathfrak{CFK}(K)$ from \Cref{sec: Telescopes and large surgeries}. For our convenience, let us introduce some notation. For $\ast\in \{0,+,-\}$, define
\[
s(\ast) = \begin{cases}
    \id &\mathrm{if}\quad \ast=0,\\
    V &\mathrm{if}\quad \ast=+,\\
    U &\mathrm{if}\quad \ast=-,
\end{cases}\qquad
n(\ast) = \begin{cases}
    0 &\mathrm{if}\quad \ast=0,\\
    1 &\mathrm{if}\quad \ast=+,\\
    -1 &\mathrm{if}\quad \ast=-.
\end{cases}
\]

As $C_K$ is a finitely generated, free chain complex representing the chain homotopy type $CFK_\mathcal{R}(S^3,K)$, its dual, $C^\vee_K$, is a model for $CFK_\mathcal{R}(S^3,\overline{K})$. The Lipshitz--Ozsv\'{a}th--Thurston basis-free model $\cD(C_K^\vee, -N)$ therefore represents the homotopy type $\widehat{CFD}((S^3 \smallsetminus \overline{K})_{-N})$, i.e. the $(-N)$-framed complement of $\overline{K}$. Since 
\[
-(S^3\smallsetminus \overline{K})_{-N} = (S^3 \smallsetminus K)_N,
\]
it follows that 
\[
\cD(C_{K}^\vee,-N)^\vee \simeq \widehat{CFA}((S^3 \setminus K)_N).
\]
Note that the dual of a type D structure is a type A structure. In particular, we may identify 
\[
\widehat{CF}(S^3 _{N+km}(K),[s]) \simeq (\widehat{CFA}(\bm O_{-km})\boxtimes \cD(C_K^\vee,-N))^\vee_{[-s]} \simeq \cLOT^-_{mk}(C^\vee_K,[-s])^\vee,
\]
where the subscript $[-s]$ indicates that we are taking the $\mathrm{Spin}^c$-graded component which corresponds to the $\mathrm{Spin}^c$ structure $[-s]$ of $S^3 _{-N-km}(\overline{K})$. In light of these observations, we define
\begin{align*}
    \cLOT^+_{m}(C_K,s):=\cLOT^-_{m}(C_K^\vee,-s)^\vee
\end{align*}
for $m > 0$. Similarly, by dualizing the maps $\frak F^*$ for $* \in \{-,0,+\}$, we obtain maps
\begin{align*}
    \cLOT^+_{km}(C_K,s) \xra{(\frak F^*_{s})^\vee} \cLOT^+_{(k+1)m}(C_K,s+n(*)).
\end{align*}
We will frequently abuse notation to write $\frak F^*$ for these maps as well. 

Moreover, the quasi-module of \Cref{sec: Telescopes and large surgeries}, $\mathfrak{CFK}(K)$, can be computed from the diagrams
\begin{align*}
    \begin{tikzcd}[ampersand replacement = \&]
        \cLOT^+_{km}(C_K,s) \ar[r,"\frak F^\ast_s"]\ar[d,"\frak F^\bullet_s"] \& 
            \cLOT^+_{(k+1)m}(K,s+n(\ast))\ar[d,"\frak F^\bullet_s"] \\
        \cLOT^+_{(k+1)m}(K,s+n(\bullet)) \ar[r,"\frak F^\ast_s"] \& 
            \cLOT^+_{(k+2)m}(K,s+n(\ast) + n(\bullet)),
    \end{tikzcd}
\end{align*}
for $*,\bullet \in \{-,0,+\}$. In particular, as noted above, this model is a strict hypercube.

\subsection{Extensions:} 

According to the previous section, there is a strict quasimodule, $\cT_K^+$ with an identification of $(\cT_K^+)_s$ with $(\hAb_s)^\vee$ compatible with the $\F[U,V]$-module action for all $s$ in a finite range. As our constructions in the previous sections required us to choose $N$ and $m$ large relative to $|s|$, we cannot fix $N$ and $m$ for all $s$ simultaneously. This is problematic, as we would like to consider the space of all endomorphisms, which consists of maps of arbitrarily large degree shift. We circumvent this issue as follows.

Consider the telescope $(\tT^+_K)_s$ of the subray 
\begin{align*}
    \cLOT^+_{N+|s|m}(C_K,s) \xra{\frak F^0_s} \cLOT^+_{N+(|s|+1)m}(C_K,s)\xra{\frak F^0_s} \hdots,
\end{align*}
and let $\tT^+_K = \bigoplus_s (\tT^+_K)_s$. As $\ell \ra \infty$, the maps $\cLOT^+_{N+(|s| + \ell)m}(C_K,s) \xra{\frak F^0_s} \cLOT^+_{N+(|s| +\ell+1)m}(C_K,s)$ become homotopy equivalences. It follows that the map induced by inclusion induces a $\F[U,V]$-linear quasi-isomorphism of their telescopes:
\begin{align*}
    \iota_s: (\tT^+_K)_s \ra (\cT_K^+)_s.
\end{align*}
It follows that $\iota$ induces an equivalence 
\begin{align*}
    \Tel\Hyp(\iota): \Tel\Hyp(\tT^+_K) \ra \Tel\Hyp(\cT_K^+).
\end{align*}
Recall that for any such quasimodule, there is a canonical projection $$p_{\tT^+_K}: \Tel\Hyp(\tT^+_K)\ra \tT^+_K,$$ which is a quasi-isomorphism by \Cref{lem:telhyp-canonical-projection-QI}.

\begin{defn}
    Let $\hat{\mathfrak{A}}_s(C_K)$ be the telescope of the ray 
    \begin{align*}
        \hAb_{-s}(C_K^\vee)^\vee\xra{\id}\hAb_{-s}(C_K^\vee)^\vee \xra{\id} \hdots,
    \end{align*}
    and define $\hat{\mathfrak{A}}^b(C):= \bigoplus_s\hat{\mathfrak{A}}_s^b(C)$. Note that we may canonically identify $\hAb_{-s}(C_K^\vee)^\vee$ with $(C_K)_s$, so really 
    \begin{align*}
        \hat{\mathfrak{A}}_s(C_K) = \Tel((C_K)_s \xra{\id}(C_K)_s \xra{\id} \hdots).
    \end{align*}
    Let $\hat{\mathfrak{A}}(C_K) := \bigoplus_s \hat{\mathfrak{A}}_s(C_K)$.
\end{defn}

Of course, $\hat{\mathfrak{A}}(C_K)$ inherits an action of $\F[U, V]$. Furthermore, the canonical map 
\begin{align*}
    \iota_{C_K}:C_K \ra \hat{\mathfrak{A}}(C_K)
\end{align*}
is an $\F[U,V]$-linear homotopy equivalence. 

In the previous section, we constructed commutative diagrams
\begin{align}\label{diag:strict hypercube for small s}
    \begin{tikzcd}[ampersand replacement = \&]
        \cLOT^-_{km}(C_K,[s]) \ar[r,"\frak F^*_s"] \ar[d,"\Phi_{km,s}"] \& \cLOT^-_{km}(C_K,[s+n(\ast)])\ar[d,"\Phi_{km,s+n(*)}"] \\
        \hAb_s \ar[r,"s(\ast)"] \& \hAb_{s+n(\ast)}
    \end{tikzcd}
\end{align}
for $|s| < B$ and $km > N > M_C-m_C$ where $B = km+m_C$. The maps
\begin{align*}
    \Tel(\Phi_{km,s}^\vee): \hat{\mathfrak{A}}_s(C_K) \ra (\tT^+_K)_s,
\end{align*}
satisfy 
\begin{align*}
     \Tel(\Phi_{km,s-1}^\vee) \circ U = U \circ \Tel(\Phi_{km,s}^\vee) \quad \Tel(\Phi_{km,s+1}^\vee) \circ V = V \circ \Tel(\Phi_{km,s}^\vee),
\end{align*}
whenever $|s| <B$. Therefore, we may define 
\begin{align*}
    \frak{P}_s:= \Tel(\Phi_{km,s}^\vee) \circ (\iota_C)_s: C_s \ra (\tT^+_K)_s.
\end{align*}
Of course, these maps also satisfy 
\begin{align*}
      \frak{P}_{s-1} \circ U = U \circ  \frak{P}_s \quad  \frak{P}_{s+1} \circ V = V \circ  \frak{P}_s,
\end{align*}
for $|s| < B-1$.

Now, provided we have chosen $m$ sufficiently large, we may assume that for $s \ge B-1$, we have that 
\begin{align*}
    V: (C_K)_{s} \ra (C_K)_{s+1} \quad  U: (C_K)_{-s} \ra (C_K)_{-s-1}
\end{align*}
are isomorphisms. We can now use linearity to define $\frak P_s$ for $|s|$ large as follows: define
\[
\mathfrak{P}_s = \begin{cases}
    V^{s-(B-1)}\mathfrak{P}_{B-1}V^{-s+(B-1)} &\text{if}\quad s\ge B, \\
    \mathfrak{P}_s &\text{if}\quad |s|<B, \\
    U^{-s+(1-B)}\mathfrak{P}_{1-B}U^{s-(1-B)} &\text{if} \quad s\le -B.
\end{cases}
\]
By construction, $\frak P_s$ now satisfies
\begin{align*}
      \frak{P}_{s-1} \circ U = U \circ  \frak{P}_s \quad  \frak{P}_{s+1} \circ V = V \circ  \frak{P}_s,
\end{align*}
for all $s \in \Z$. Hence, 
\begin{align*}
    \frak P:= \bigoplus_s \frak P_s: C_K \ra \tT^+_K
\end{align*}
is a homotopy equivalence of $\F[U,V]$-modules. In summary, we have proven the following:
\begin{prop}\label{prop:P-is-QI}
    The map $\frak P: C_K \ra \tT^+_K$ is a $\F[U,V]$-quasi-isomorphism which is bi-grading preserving. 
\end{prop}
\begin{proof}
    We will simply make a brief comment on the statement about the grading, as the fact that $\frak P$ is a quasi-isomorphism is clear by the discussion above. By construction, $\frak P$ preserves the Alexander grading. So, it suffices to check that it preserves the Maslov grading. The map $\frak P$ is defined as the composition $\Tel(\Phi_{km,s}^\vee) \circ (\iota_C)_s$. The map $(\iota_C)_s$ is clearly grading preserving. The map $\Phi_{km,s}$ is grading preserving by construction (i.e., the grading on $\cLOT^+(C_K,s)$ was defined precisely so that this map is degree zero).
\end{proof}

In the next section, we will give $\tT^+_K$ the structure of a $\End_\cR(C_K)$-module, and prove that $\frak P$ is a map of $\End_\cR(C_K)$-modules.


\section{Endomorphism actions}\label{sec: End actions}

The quasimodule $\hat{\frak{A}}(C_K)$ has an obvious $\End_\cR(C_K)$-module structure. If $f\in \End_\cR(C_K)$ is a homogeneously graded map with Alexander grading shift $A$, we have (strictly commuting) diagrams
\begin{align*}
    \begin{tikzcd}[ampersand replacement = \&]
        \hAb_s \ar[r,"f"] \ar[d,"s(*)"] \& \hAb_{s+A}\ar[d,"s(*)"] \\
        \hAb_{s+n(\ast)} \ar[r,"f"] \& \hAb_{s+n(\ast)+A}.
    \end{tikzcd}
\end{align*}
These diagrams provide the data of a 4-dimensional hypercube, by taking the higher homotopies to be identically zero. Such a hypercube furnishes a map 
\begin{align*}
    \hat{\frak{A}}(f): \hat{\mathfrak{A}}_s^b(C_K) \ra \hat{\mathfrak{A}}_{s+A}^b(C_K),
\end{align*}
as desired. It is clear that the map $\iota_{C_K}: C_K \ra \hat{\frak{A}}(C_K)$ is an $\End_\cR(C_K)$-module map. We will now use $\Omega$ to equip $\widetilde{\cT}^+_K$ with an $\End_\cR(C_K)$-module structure as well. 

By definition, $\mathcal{LOT}(C_K, [s])$ is defined to be
\begin{align*}
    (\mathfrak{CFA}(\donut)\boxtimes \cD(C_K))_{[s]},
\end{align*}
where 
\[
\mathfrak{CFA}(\donut) = (\CFAh(\donut_{(1+||\epsilon||)m}), F^\epsilon_{\epsilon_0}),
\]
is the 3-dimensional hypercube of type A structures; when $||\epsilon|| = 1$, the maps $F^\epsilon_{\epsilon_0}$ are given by $\mathfrak{f}^0,\mathfrak{f}^+,\mathfrak{f}^-$; higher length components zero; the subscript $[s]$ indicates that we are considering the summand in $\SpinC$-grading $s$. 

Given a type D endomorphism $\varphi\in \End^\cA(\cD(C_K))$, the general framework of hypercubes of type A and D modules allows us to produce an endomorphism of the hypercube $\mathcal{LOT}(C_K)$ as follows. By regarding $\varphi$ as a 1-dimensional hypercube of type $D$-structures
\begin{align*}
    \varphi: = \cD(C_K) \xra{\varphi}\cD(C_K),
\end{align*}
it can be paired with $\mathfrak{CFA}(\donut)$ to produce a 4-dimensional hypercube of chain complexes, $\mathfrak{CFA}(\donut)\boxtimes \varphi$, whose new length 1 arrows are given by $\bI \boxtimes \varphi$ (see \Cref{sec:tel_hyp}).

\begin{lem} Let $f$ be a locally symmetric chain endomorphism of $C_K$ with Alexander degree shift $A$. For integers $N$ and $s$ with $m \gg N > M_C - m_C$ and $|s| < m_C + N- 2A$, there is a strictly commutative square
\begin{align*}
    \begin{tikzcd}[ampersand replacement = \&]
        \mathcal{LOT}_{-2m}(K,s) \ar[r,"\bI \boxtimes \Omega(f)"]\ar[d,"\frak{F}^*_{s}"] \& 
            \mathcal{LOT}_{-2m}(K,s+A)\ar[d,"\frak{F}^*_{s + A}"] \\
        \mathcal{LOT}_{-m}(K,s+ n(*)) \ar[r,"\bI \boxtimes \Omega(f)"]\& 
            \mathcal{LOT}_{-m}(K,s+n(*)+A).
    \end{tikzcd}
\end{align*}
where the vertical arrows are the morphisms constructed in \Cref{sec: the LOT hypercube}.
\end{lem}
\begin{proof}
    This follows directly from the proofs of Lemmas \ref{lem: F0 is id}, \ref{lem: F- is U} ,\ref{lem: F+ is V}.
\end{proof}
By dualizing, we obtain the analogous squares for the \LOT models $\mathcal{LOT}^+_{m}(C_K,s)$ for positive surgeries.

Recall that the complexes $\mathcal{LOT}_{m}(C_K,s)$ are defined in terms of the models $\cD(C_K, N)$, which depend on the framing parameter $N$ of the complement of $K$, which we have typically suppressed from the notation. Here, let us emphasize that a locally symmetric chain endomorphism $f$ of $C_K$ with Alexander degree shift $A$ induces an endomorphism 
\begin{align*}
    \Omega(f): (\cT_K^+)_s \ra (\cT_K^+)_{s+A}
\end{align*}
only if $N$ is much larger than $A$. In the same way, provided $N$ is large, we obtain a map 
\begin{align*}
    \widetilde{\Omega}(f): (\tT_K^+)_s \ra (\tT_K^+)_{s+A},
\end{align*}
as well. Moreover, these maps fit into an obvious commutative square 
\begin{align*}
\begin{tikzcd}[ampersand replacement = \&]
    \widetilde{\cT}_K^+ \ar[r,"\widetilde{\Omega}(f)"] \ar[d,"\iota"] \& \widetilde{\cT}_K^+ \ar[d,"\iota"] \\
    \cT_K^+ \ar[r,"\Omega(f)"]\& \cT_K^+,
\end{tikzcd}
\end{align*}
where the vertical arrows are the quasi-isomorphisms induced by the obvious inclusion of rays. 

Furthermore, $\Omega(f)$ is compatible with the hypercube maps.

\begin{lem}\label{lem: omega(f) computes f} Let $f$ be a locally symmetric chain endomorphism of $C$ with Alexander degree shift $A$. For integers $N$ and $s$ with $m \gg N > M_C - m_C$ and $|s| < B - 2A$, there is a strictly commutative square
\begin{align*}
    \begin{tikzcd}[ampersand replacement = \&]
        \mathcal{LOT}_{-m}(K,s) \ar[r,"\bI \boxtimes \Omega(f)"]\ar[d,"\Phi_{m,s}"] \& 
            \mathcal{LOT}_{-m}(K,s+A)\ar[d,"\Phi_{m,s+A}"] \\
        \hAb_s \ar[r,"f"] \&
            \hAb_{s+A}
    \end{tikzcd}
\end{align*}
where the vertical arrows are the isomorphisms given by \Cref{lem:LOT_model_comp} and $B = km + m_C$.
\end{lem}
\begin{proof}
    Any morphism, $f$, can be written as a sum of terms of the form,
    \begin{align*}
        \eta \mapsto &\vartheta & \hspace{.5cm} &  \\
        U^{-i} \eta \mapsto & \vartheta & \hspace{.5cm} & i > 0\\
        U^{-i} \eta \mapsto & U^{-i + k}\vartheta & \hspace{.5cm} & i > 0, k \ge 0 \\
        V^{-j} \eta \mapsto & \vartheta & \hspace{.5cm} & j > 0 \\
        V^{-j} \eta \mapsto & V^{-j + k}\vartheta & \hspace{.5cm} & j > 0, k \ge 0.
    \end{align*}
    Hence, it suffices to prove the lemma in cases.

    Consider a morphism of the form $f: \eta \mapsto \vartheta$. Under $\Phi_{2m, s}$, $\eta$ corresponds to the element $\beta \boxtimes \eta$. Further, as $f(\eta)$ has no $U$-component, $\Omega(f)$ consists only of $f_\emptyset$. Therefore, there is a single term in $(\bI \boxtimes \Omega(f))(\beta \boxtimes \eta)= \beta \boxtimes \vartheta$. As  $\Phi_{m, s+A}(\beta \boxtimes \vartheta) = \vartheta$, the claim follows in this case. 

    The cases $V^{-j} \eta \mapsto \vartheta$ and $V^{-j} \eta \mapsto  V^{-j + k}\vartheta$ follow by an identical computation. 

    Next, consider the case $f: U^{-i} \eta \mapsto \vartheta$. The element $U^{-i} \eta$ corresponds to the element $b_i \boxtimes \eta$ under $\Phi_{2m,s}$ and $\vartheta$ corresponds to the element $\beta \boxtimes \vartheta$ under $\Phi_{m,s+A}$ (note that here we use the assumption that $s$ and $m$ are much larger than $A$ to ensure that $\eta$ is represented by an element in the $\HFKh$-component of $\mathcal{LOT}_{-m}(K, s)$). The $f_\emptyset$ component of $\Omega(f)$ is trivial in this case, and the component $f_2$ maps $U^{-k} \eta \in \hAb_s/(V^{-1})$ to $[f(U^{-k-1}\eta)] \in \hAb_s/(U^{-1}, V^{-1})$. Note that this map is trivial unless $\langle U^{k+1}, f(\eta) \rangle = 1$. For this reason there is a single term in $(\bI \boxtimes \Omega(f))(b_i \boxtimes \eta)$, given by pairing the sequence 
    \begin{align*}
        b_i \xra{\rho_3 \otimes \rho_2} b_{i-1} \xra{\rho_3 \otimes \rho_2} \hdots \xra{\rho_3 \otimes \rho_2} b_1 \xra{\rho_3 \otimes \rho_2 \otimes \rho_1} \beta,
    \end{align*}
    with the sequence 
    \begin{align*}
        \eta \xra{\rho_3} \eta \xra{\rho_{23}} U^{-1} \eta \xra{\rho_{23}} \hdots \xra{\rho_{23}}U^{-i+1} \eta \xra{\rho_2}  f(U^{-1} \cdot U^{-i+1}\eta) = \vartheta \xra{\rho_1} \vartheta. 
    \end{align*}
    Hence, $(\bI \boxtimes \Omega(f))(b_i \boxtimes \eta) = \beta \boxtimes \vartheta$, as claimed. 

    The case of morphisms of the form $f: U^{-i} \eta \mapsto U^{-i+k}\vartheta$ is similar. As $A$ is assumed to be small relative to $s$ and $m$, both $U^{-i}\eta$ and $U^{-i+k}\vartheta$ correspond under $\Phi$ to elements $b_i \boxtimes \eta$ and $b_{i-k} \boxtimes \vartheta$, which are members of the $\HFKh$-components of $\mathcal{LOT}_{-2m}(K, s)$. Again, $\Omega(f) = \rho_2 f_2$, and we pair the sequences 
    \begin{align*}
        b_i \xra{\rho_3 \otimes \rho_2} b_{i-1} \xra{\rho_3 \otimes \rho_2} \hdots \xra{\rho_3 \otimes \rho_2} b_{i-k}
    \end{align*}
    with the sequence 
    \begin{align*}
        \eta \xra{\rho_3} \eta \xra{\rho_{23}} U^{-1} \eta \xra{\rho_{23}} \hdots \xra{\rho_{23}}U^{-k+1} \eta \xra{\rho_2} f(U^{-1} \cdot U^{-k+1}\eta) = \vartheta. 
    \end{align*}
    Hence, we have that 
    \begin{align*}
        (\bI \boxtimes \Omega(f))(b_i \boxtimes \eta) = b_{i-k} \boxtimes \vartheta.
    \end{align*}
    The lemma follows. 
\end{proof}
Again, the case for positive surgeries follows by dualizing. 

Having defined $\Omega(f)$ for small $s$, we show that the finiteness assumptions on $C_K$ are sufficient to extend the definition for all $s$.

\begin{lem}\label{lem:local-sym-bound}
    The Alexander grading shift $A$ of any locally symmetric homogeneous endomorphism of $C_K$ satisfies $|A| < 2(M_C-m_C)$, where $M_C$ and $m_C$ are the maximal and minimal Alexander gradings of $C_{K}\otimes_{\cR} \cR/(U,V)$.
\end{lem}
\begin{proof}
    Write $f(w) = x + U^i y + V^j z$. If $x\neq 0$, then clearly $|A| = |A(x) - A(w)| \le |A(x)| + |A(w)| \le M_C - m_C$, so the claim follows. If $x = 0$, then by the symmetry condition, $y$ and $z$ are both nonzero. So, on the one hand, 
    \begin{align*}
        A = A(y)-A(w)-i > -(M_C - m_C) - i,
    \end{align*}
    and on the other 
    \begin{align*}
        A = A(z)-A(w)+j < M_C - m_C + j.
    \end{align*}
    Since $C$ is finitely generated, $i$ and $j$ can be no larger than $M_C - m_C$ as well. Hence, $|A| < 2(M_C - m_C)$.
\end{proof}

It follows from Lemmas ~\ref{lem: omega(f) computes f}  and ~\ref{lem:local-sym-bound}, it follows that the maps $\frak P_s$ satisfy 
\begin{align}\label{diag: frak P commutes with f}
\begin{tikzcd}[ampersand replacement = \&]
    C_s \ar[r,"f^\vee"] \ar[d,"\frak P_s"] \& C_{s+A} \ar[d,"\frak P_{s+A}"] \\
     (\widetilde{\cT}_K^+)_s \ar[r,"\widetilde{\Omega}(f^\vee)"] \& (\widetilde{\cT}_K^+)_{s+A},
\end{tikzcd}
\end{align}
whenever $|s| < B - 2(M_C-m_C) - 2$. We extend $\widetilde{\Omega}(f^\vee)_s$ for large $|s|$ by
\[
\widetilde{\Omega}(f^\vee)_s = \begin{cases}
    V^{s-(T-1)}\widetilde{\Omega}(f^\vee)_{T - 1}V^{-s+(T-1)} &\text{if}\quad s\ge T, \\
    \widetilde{\Omega}(f^\vee)_s &\text{if}\quad |s|< T, \\
    U^{-s+(1-T)}\widetilde{\Omega}(f^\vee)_{-T+1}U^{s-(1-T)} &\text{if} \quad s\le -T,
\end{cases}
\]
where $T =  B - 2(M_C-m_C) - 2.$

\begin{lem}\label{lem: extended-diag-commutes}
    For $N$ and $m$ sufficiently large, the Diagram \eqref{diag: frak P commutes with f} commutes for all integers $s$.
\end{lem}
\begin{proof}
    Write $g_s := \widetilde{\Omega}(f^\vee)_s \circ \frak P_s$ and $h_s := \frak P_{s+A}\circ f^\vee_s$, both maps $C_s \ra (\tT^+_K)_{s+A}$; we must show $g_s = h_s$ for every $s \in \Z$.

    By construction the maps $\frak P_s$ intertwine the $U$- and $V$-actions,
    \begin{align*}
        \frak P_{s-1}\circ U = U\circ \frak P_s, \qquad \frak P_{s+1}\circ V = V\circ \frak P_s \qquad (s \in \Z),
    \end{align*}
    (see the discussion preceding \Cref{prop:P-is-QI}), and since $f$ is $\cR$-linear its dual satisfies the same relations, $f^\vee_{s-1}\circ U = U\circ f^\vee_s$ and $f^\vee_{s+1}\circ V = V\circ f^\vee_s$. Finally, as $B = km + m_C$ grows with $m$, we may take $m$ large enough that $T = B - 2(M_C - m_C) - 2$ lies above the stabilization range of $C_K$; then
    \begin{align*}
        V : (C_K)_s \ra (C_K)_{s+1}\ (s \ge T), \qquad U : (C_K)_s \ra (C_K)_{s-1}\ (s \le -T)
    \end{align*}
    and the corresponding maps on $\tT^+_K$ are isomorphisms. These are the very isomorphisms used to define the extensions of $\frak P_s$ and $\widetilde{\Omega}(f^\vee)_s$ outside the range $|s| < T$.

    \emph{Case $|s| < T$.} By \Cref{lem:local-sym-bound}, $|A| < 2(M_C - m_C)$, so $|s + A| < T + 2(M_C - m_C) = B - 2$; hence $\frak P_s$ and $\frak P_{s+A}$ are the constructed maps and $\widetilde{\Omega}(f^\vee)_s$ is the map of \Cref{lem: omega(f) computes f}. The identity $g_s = h_s$ is exactly the commutativity of Diagram~\eqref{diag: frak P commutes with f}, which holds for $|s| < T$.

    \emph{Case $s \ge T$.} From the definition of the extension, $V\circ \widetilde{\Omega}(f^\vee)_s = \widetilde{\Omega}(f^\vee)_{s+1}\circ V$ for $s \ge T$. Combining this with the relations above,
    \begin{align*}
        V\circ g_s &= V\circ \widetilde{\Omega}(f^\vee)_s\circ \frak P_s = \widetilde{\Omega}(f^\vee)_{s+1}\circ V\circ \frak P_s = \widetilde{\Omega}(f^\vee)_{s+1}\circ \frak P_{s+1}\circ V = g_{s+1}\circ V, \\
        V\circ h_s &= V\circ \frak P_{s+A}\circ f^\vee_s = \frak P_{s+A+1}\circ V\circ f^\vee_s = \frak P_{s+A+1}\circ f^\vee_{s+1}\circ V = h_{s+1}\circ V .
    \end{align*}
    Hence $V\circ(g_s + h_s) = (g_{s+1} + h_{s+1})\circ V$ for all $s \ge T$. By the previous case $g_{T-1} = h_{T-1}$, and $V$ is an isomorphism at the relevant gradings (which lie above the stabilization range once $m$ is large); so $g_{s+1} + h_{s+1} = V\circ(g_s + h_s)\circ V^{-1}$, and upward induction from $g_{T-1} = h_{T-1}$ gives $g_s = h_s$ for all $s \ge T$.

    \emph{Case $s \le -T$.} Identical, with $U$ in place of $V$: the extension gives $U\circ \widetilde{\Omega}(f^\vee)_s = \widetilde{\Omega}(f^\vee)_{s-1}\circ U$ for $s \le -T$, so $U\circ(g_s + h_s) = (g_{s-1} + h_{s-1})\circ U$. Since $g_{-T+1} = h_{-T+1}$ by the first case and $U$ is an isomorphism at the relevant gradings, $g_{s-1} + h_{s-1} = U\circ(g_s + h_s)\circ U^{-1}$, and downward induction yields $g_s = h_s$ for all $s \le -T$.
\end{proof}

In summary, we have proven the following.
\begin{prop}\label{prop:P-is-End-mod-morphism}
    The map $\frak P: C_K \ra \tT_K^+$ is a quasi-isomorphism of $\End_\cR(C_K)$-modules.
\end{prop}
\begin{proof}
    That the map is a quasi-isomorphism follows from Lemma~\ref{lem:LOT_model_comp} and the definition of the extension. The fact that the map respects the endomorphism action follows from Lemmas ~\ref{lem: omega(f) computes f} and ~\ref{lem: extended-diag-commutes}.
\end{proof}

\section{The map $\Lambda$}\label{sec: the map Lambda}

In \cite{guthkang2024invariantsplittingprinciples}, the authors constructed a ring map 
\begin{align*}
    \Lambda: \End^\cA_{h,0}(\CFDh(\KC)) \ra \End^h_\cR(\CFK_\cR(S^3, K)),
\end{align*}
where $\End^\cA_{h,0}(\CFDh(\KC))$ denotes the space of homotopy classes of degree-preserving type D endomorphisms of $\CFDh(\KC)$; we will presently extend this to a map defined on the whole of $\End^\cA_h(\CFDh(\KC))$, i.e. the space of homotopy classes of type D endomorphisms (which may not be degree-preserving) of $\CFDh(\KC)$ in \Cref{subsec: lambda map}. 

\subsection{The absolute $K$-Maslov grading on $\widehat{CF}(S^3_0(K),[s])$} \label{subsec: absolute maslov grading}

Given a $\mathrm{Spin}^c$ 3-manifold $(M,\mathfrak{s})$, its hat-flavored Heegaard Floer homology $\widehat{HF}(M,\mathfrak{s})$ carries a natural relative $\mathbb{Z}/\frak{d}$-grading, where $\frak{d}$ is the maximal divisibility of $c_1(\mathfrak{s})$ \cite[Section 4]{os_holodisks}. In particular, we get a natural relative $\mathbb{Z}$-grading only when $c_1(\mathfrak{s})$ is torsion. Note that, in this case, the relative $\mathbb{Z}$-grading can be lifted to a canonical absolute $\mathbb{Q}$-grading \cite{os_holotri,OzSz_absolutely}.

However, for our purposes, we need an absolute $\mathbb{Q}$-grading on $\widehat{HF}(S^3_0(K),[s])$ for \emph{every} integer $s\in \mathbb{Z}$. Here, $[s]$ is the $\SpinC$-structure which satisfies 
\begin{align*}
    \langle c_1([s]), [H] \rangle = 2s,    
\end{align*}
for a generator $[H]$ of $H_2(S^3_0(K);\Z) \cong \Z\langle[H] \rangle$. Compare to \cite[Section 7]{OzSz_absolutely}. Since $c_1([s])$ is torsion only when $s=0$, this poses a problem. However, this can be solved naturally in the case when $K$ is slice; in this subsection, we describe how to define the desired absolute $\mathbb{Q}$-grading. This will be used in later sections.

We start by considering the following exact triangle from \cite[Proposition 9.19]{ozsvath2004holomorphicproperties}, which holds for any $s\in\mathbb{Z}$ and $N\in\mathbb{Z}_{>0}$
\[
\cdots \rightarrow \widehat{HF}(S^3) \rightarrow \bigoplus_{i\equiv s \pmod N}\widehat{HF}(S^3_0(K),[i])\rightarrow \widehat{HF}(S^3_N(K),[s]) \rightarrow \cdots.
\]
In the above exact triangle, the map
\[
\bigoplus_{i\equiv s \pmod N}\widehat{HF}(S^3_0(K),[i])\rightarrow \widehat{HF}(S^3_N(K),[s])
\]
is induced by the 4-dimensional 2-handle cobordism $W^K_{0,N}$ from $S^3 _0(K)\# L(N,1)$ to $S^3_N(K)$. More precisely, for each $i\in\mathbb{Z}$ satisfying $i\equiv s\pmod N$, the $\mathrm{Spin}^c$ structure $(0,[i])$ on $W^K_{0,N}$ (note that $H^2(W^K_{0,N};\mathbb{Z})\simeq \mathbb{Z}\oplus \mathbb{Z}/N$) restricts to the $\mathrm{Spin}^c$ structures $[0]\#[i]$ on $S^3_0(K)\# L(N,1)$ and $[s]$ on $S^3_N(K)$. 

We denote by $\mathfrak{f}^K_{0,N;s}$ the composition
\[
\widehat{HF}(S^3_0(K),[s]) \hookrightarrow \bigoplus_{i\equiv s \pmod N} \widehat{HF}(S^3_0(K),[i]) \rightarrow \widehat{HF}(S^3_N(K),[s])
\]
where the second map is the one appearing in the surgery exact triangle. Note that $\widehat{HF}(S^3_0(K))$ is finite-dimensional, which implies that if $N$ is sufficiently large, we have that
\[
\widehat{HF}(S^3_0(K),[s+kN])=0 \quad \text{for all integers} \quad k\neq 0.
\]
Hence, for any sufficiently large $N$, the inclusion $\widehat{HF}(S^3_0(K),[s])\hookrightarrow \bigoplus_{i\equiv s \pmod N}\widehat{HF}(S^3_0(K),[i])$ is an isomorphism.

\begin{lem} \label{lem: large surgery from 0 surgery is injective}
    Let $K$ be a smoothly slice knot. Then, for any fixed integer $s\neq 0$, the map $\mathfrak{f}^K_{0,N;s}$ is injective for all sufficiently large $N$. If $s=0$, then $\mathfrak{f}^K_{0,N;s}$ has 1-dimensional kernel.
\end{lem}
\begin{proof}
    The case $s=0$ was already proven in \cite[Proposition 5.2]{guthkang2024invariantsplittingprinciples}, so we may assume that $s\neq 0$. It suffices to prove that the map $\widehat{HF}(S^3)\rightarrow \widehat{HF}(S^3_0(K),[s])$ vanishes. Since $K$ is smoothly slice, there exists a smooth concordance $C$ from the unknot $U$ to $K$. Performing a 0-surgery along $C$ induces a smooth homology cobordism $X_0(C)$ from $S^3_0(U)=S^1 \times S^2$ to $S^3_0(K)$, which then induces the cobordism map
    \[
    \hat{F}_{X_0(C),[s]}:\widehat{HF}(S^1 \times S^2,[s])\rightarrow \widehat{HF}(S^3_0(K),[s]).
    \]
    Since the 2-handle attachment cobordisms which define the maps in the exact triangle commute with the cobordisms $X_n(C)$ obtained by surgeries on $C$, this map fits into the following commutative diagram.
    \[
    \xymatrix{
    \widehat{HF}(S^3) \ar[rr] \ar[d]_{X_\infty(C)} && \widehat{HF}(S^1 \times S^2,[s]) \ar[d]^{\hat{F}_{X_0(C),[s]}} \\
    \widehat{HF}(S^3) \ar[rr]  && \widehat{HF}(S^3_0(K),[s])
    }
    \]
    Since we have assumed that $s\neq 0$, we have $\widehat{HF}(S^1 \times S^2,[s])=0$. Further, $X_\infty(C)$ is simply the product $S^3 \times I$. Hence, we deduce that the map $\widehat{HF}(S^3)\rightarrow \widehat{HF}(S^3_0(K),[s])$ vanishes, as it factors through the zero module.
\end{proof}

Now, for any integers $N,N'$ satisfying $0<N<N'$ and any knot $K$, we consider the 2-handle cobordism 
\[
W^K_{N,N'}:S^3_N(K)\# L(N'-N,1) \rightarrow S^3_{N'}(K).
\]
As in the case of 0-surgery to $N$-surgery, it is straightforward to see that, for any integer $s$, the cobordism $W^K_{N,N'}$ induces the following cobordism map:
\[
\mathfrak{f}^K_{N,N';s}:\widehat{HF}(S^3_N(K),[s])\rightarrow \widehat{HF}(S^3_{N'}(K),[s]).
\]

\begin{lem} \label{lem: large relative surgery map is bijective}
    For any fixed integer $s$, the map $\mathfrak{f}^K_{N,N';s}$ is bijective whenever $N$ and $N'$ are sufficiently large.
\end{lem}
\begin{proof}
    Consider the large surgery isomorphisms
    \[
    \Gamma_{N,s}:\widehat{HF}(S^3_N(K),[s])\rightarrow \hat{A}_s(K),\quad \Gamma_{N',s}:\widehat{HF}(S^3_{N'}(K),[s])\rightarrow \hat{A}_s(K).
    \]
    These are cobordism maps induced by the canonical 4-dimensional link cobordisms 
    \[
    D^K_N:S^3_N(K)\rightarrow (S^3,K),\quad D^K_{N'}:S^3_{N'}(K)\rightarrow (S^3,K).
    \]
    Observe that $D^K_{N'}\circ W^K_{N,N'}\simeq D^K_N$. Hence, by choosing its suitable $\mathrm{Spin}^c$ structure, we see that
    \[
    \Gamma_{N,s} = \Gamma_{N',s}\circ \mathfrak{f}^K_{N,N';s}.
    \]
    When $N$ and $N'$ are sufficiently large, the maps $\Gamma_{N,s}$ and $\Gamma_{N',s}$ are isomorphisms \cite{os_knotinvts,rasmussen_knotcompl}. Therefore $\mathfrak{f}^K_{N,N';s}$ is also an isomorphism.
\end{proof}

\begin{lem} \label{lem: surgery maps are compatible}
    Let $K$ be a knot and $N,N'$ be integers satisfying $0<N<N'$. Then, for any integer $s$, we have 
    \[
    \mathfrak{f}^K_{0,N';s} = \mathfrak{f}^K_{N,N';s}\circ \mathfrak{f}^K_{0,N;s}.
    \]
\end{lem}
\begin{proof}
    Consider the cobordism 
    \begin{align*}
        W^K_{0,N,N'}: L(N'-N,1) \# L(N,1) \# S^3_0(K) \ra S^3_{N'}(K)
    \end{align*}
    shown in \Cref{fig:cobordism_commutation} consisting of two 2-handles. We can attach the 2-handles in either order, giving two natural decompositions of $W^K_{0,N,N'}$:
    \begin{align*}
        L(N'-N,1) \# L(N,1) \# S^3_0(K) \ra L(N',1) \# S^3_{0}(K) \ra S^3_{N'}(K)
    \end{align*}
    or
    \begin{align*}
        L(N'-N,1) \# L(N,1) \# S^3_0(K) \ra L(N'-N,1) \# S^3_{N}(K) \ra S^3_{N'}(K).
    \end{align*}
    Once we fix the canonical generators for $\HFh(L(N'-N,1))$, $\HFh(L(N,1))$, and $\HFh(L(N',1))$, the first composition induces the map $\mathfrak{f}^K_{0,N';s}$, while the second induces $\mathfrak{f}^K_{N,N';s}\circ \mathfrak{f}^K_{0,N;s}$.
\end{proof}

    \begin{figure}[h]
    \def\svgwidth{.8\linewidth}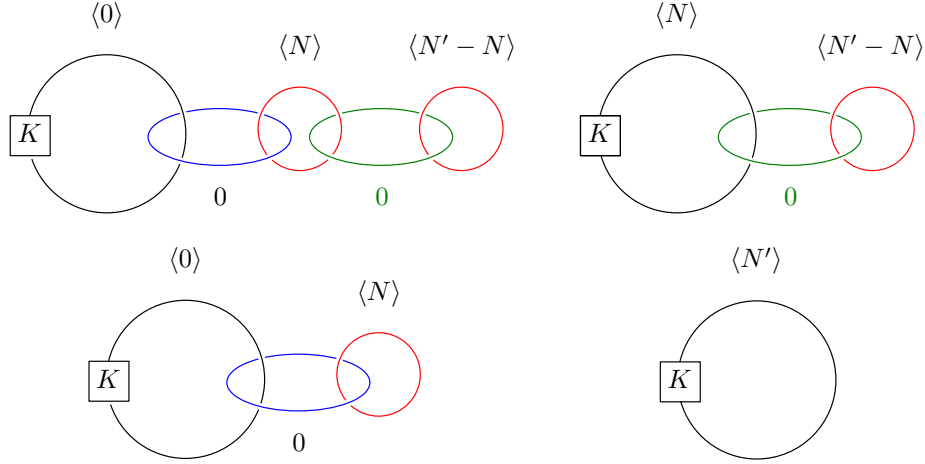
            \caption{Two decompositions of the cobordism $W^K_{0,N,N'}$.}
        \label{fig:cobordism_commutation}
        \end{figure}

Now we define the notion of absolute $K$-Maslov grading.

\begin{defn}
    Given a knot $K$ and an integer $s$, we say that an absolute $\mathbb{Q}$-grading on $\widehat{HF}(S^3_0(K),[s])$ is an \emph{absolute $K$-Maslov grading} if $\mathfrak{f}^K_{0,N;s}$ has degree shift $-\frac{3}{4}+\frac{(2s-N)^2}{4N}$, with respect to the absolute $\Q$-grading on its domain and the absolute $\mathbb{Q}$-grading on its codomain, for all sufficiently large integers $N$.
\end{defn}

\begin{lem}
    For any knot $K$, the absolute $\mathbb{Q}$-grading on $\widehat{HF}(S^3_0(K),[0])$ is an absolute $K$-Maslov grading.
\end{lem}
\begin{proof}
    Let $\frs_0$ be the unique $\SpinC$-structure on $W_{0,N}$ which restricts to the torsion $\SpinC$ structure on $S^3_0(K)$ and the canonical $\SpinC$-structure on $L(N,1)$. The grading shift formula implies 
    \begin{align*}
        \deg(\frak f_{0,N;[0]}^K) = \frac{c_1(\frs_0)^2 - 2 \chi(W_{0,N}) - 3 \sigma(W_{0,N})}{4}.
    \end{align*}
    These quantities can be computed using \cite{manolescu2024heegaardfloerhomologyinteger}, though, it is easier to note that the algebraic topology of $W_{0,N}$ does not depend on the knot $K$. Hence, it suffices to compute the grading shift in the case $K = U$. In \cite{os_HFK_integer_surgeries}, Ozsv\'ath and Szab\'o produce an exact triangle
    \begin{align*}
        \begin{tikzcd}[ampersand replacement = \&]
            \HFh(S^3_0(U)) \ar[rr,"F_{W_0,N}"] \& \& \HFh(S^3_N(U)) \ar[dl] \\
            \& \underline{\HFh}(S^3) \ar[ul] \&
        \end{tikzcd}
    \end{align*}
    where the top arrow is the total cobordism map associated to $W_{0,N}$. The map $F_{W_{0,N}} = \sum_{\frs \in \SpinC(W_{0,N})} F_{W_{0,N},\frs}$. The component of this map taking $\HFh(S^3_0(U), [0])$ to $\HFh(L(N,1), [0])$ is precisely $F_{W_{0,N},\mathfrak{s}_0}$. Since the map $\underline{\HFh}(S^3) = \F[T]/(T^N)$ takes $1 \mapsto \theta^- \in \HFh_{-1/2}(S^3_0(U), [0])$, it follows that $F_{W_{0,N},\frs_0}(\theta^+)$ is taken to the generator of $\HFh(S^3_{N}(U),[0])$. Hence, 
    \begin{align*}
        \deg(\frak f_{0,N;[0]}^K) = d(L(N,1),[0]) - \frac{1}{2} = \frac{N-1}{4} - \frac{1}{2},
    \end{align*}
    which is indeed equal to $-\frac{3}{4}+\frac{(2(0)-N)^2}{4N}$, as claimed.
\end{proof}

\begin{lem} \label{lem: uniqueness of absolute K-Maslov grading}
    For any smoothly slice knot $K$ and integer $s>0$, there exists a unique absolute $K$-Maslov grading on $\widehat{HF}(S^3_0(K),[s])$.
\end{lem}
\begin{proof}
    Hence, we now assume $s \neq 0$. Choose a sufficiently large integer $N$. According to \Cref{lem: large surgery from 0 surgery is injective} the map
    \[
    \mathfrak{f}^K_{0,N;s}:\widehat{HF}(S^3_0(K),[s])\rightarrow \widehat{HF}(S^3_N(K),[s])
    \]
    is injective. Applying the definition of absolute $K$-Maslov gradings then gives us a unique absolute $\mathbb{Q}$-grading of $\widehat{HF}(S^3_0(K),[s])$ such that the degree of $\mathfrak{f}^K_{0,N;s}$ is precisely $-\frac{3}{4}+\frac{(2s-N)^2}{4N}$. This also implies that, if an absolute $K$-Maslov grading on $\widehat{HF}(S^3_0(K),[s])$ exists, then it is unique.

    Now, for any integer $N'>N$, it follows from \Cref{lem: large relative surgery map is bijective} that the map
    \[
    \mathfrak{f}^K_{N,N';s}:\widehat{HF}(S^3_N(K),[s])\rightarrow \widehat{HF}(S^3_{N'}(K),[s])
    \]
    is bijective. Furthermore, \Cref{lem: surgery maps are compatible} tells us that 
    \[
    \mathfrak{f}^K_{0,N';s} = \mathfrak{f}^K_{N,N';s} \circ \mathfrak{f}^K_{0,N;s}.
    \]
    Observe from the proof of \Cref{lem: large relative surgery map is bijective} that 
    \[
    \deg \mathfrak{f}^K_{N,N';s} = \deg \Gamma_{N,s} - \deg(\Gamma_{N',s}) = -\frac{(2s-N)^2}{4N} + \frac{(2s-N')^2}{4N'}.
    \]
    It follows that
    \[
    \begin{split}
    \deg( \mathfrak{f}^K_{0,N';s}) &= \deg (\mathfrak{f}^K_{0,N;s}) + \deg (\mathfrak{f}^K_{N,N';s}) \\
    &= -\frac{3}{4} + \frac{(2s-N)^2}{4N} - \frac{(2s-N)^2}{4N} + \frac{(2s-N')^2}{4N'} \\
    &= -\frac{3}{4} + \frac{(2s-N')^2}{4N'}.
    \end{split}
    \]
    Hence, the given absolute $\mathbb{Q}$-grading on $\widehat{HF}(S^3_0(K),[s])$ is an absolute $K$-Maslov grading. The lemma follows.
\end{proof}

Thus, for any smoothly slice knot $K$, \Cref{lem: uniqueness of absolute K-Maslov grading} can be applied to endow $\widehat{HF}(S^3_0(K)) = \bigoplus_{s\in\mathbb{Z}} \widehat{HF}(S^3_0(K),[s])$ with an absolute $\mathbb{Q}\times \mathbb{Z}$-grading. Here, the $\mathbb{Q}$ grading is given by the absolute $K$-Maslov gradings on each $\mathrm{Spin}^c$-graded piece and the $\mathbb{Z}$-grading is simply given by the integer $s$. This gives $\widehat{HF}(S^3_0(K))$ a bigrading which is very similar to the (Alexander,Maslov) bigrading on the knot Floer homology of $K$; we call this the \emph{absolute $K$-bigrading}. Under the canonical identification
\[
\End^{\cA}_h(\CFDh(\KC)) \simeq \widehat{HF}(S^3_0(K\#\overline{K})),
\]
the absolute $K$-bigrading on $\widehat{HF}(S^3_0(K\#\overline{K}))$ induces a $\mathbb{Q}\times \mathbb{Z}$-bigrading of $\End^{\cA}_h(\CFDh(\KC))$. We will abuse notation and refer to this bigrading also as the absolute $K$-bigrading.

\begin{rem}
    As we have defined it, the absolute $K$-Maslov grading depends not just on the 3-manifold $S^3_0(K)$, but also the knot $K$. It seems there could exist a pair $(K,K')$ of knots such that $S^3_0(K) \simeq S^3_0(K')$ but the absolute $K$-Maslov grading and the absolute $K'$-Maslov grading differ. However we do not have any concrete example where such a phenomenon occurs. 
\end{rem}

\subsection{The $\Lambda$ map on $\widehat{HF}(S^3_0(K\# \overline{K}))$} \label{subsec: lambda map}

For any knot $K$, the authors constructed in \cite{guthkang2024invariantsplittingprinciples} the $\Lambda$ map from $\widehat{HF}(S^3_0(K\# \overline{K}),[0])$ to $HFK_{\mathcal{R}}(S^3,K\# \overline{K})$ as the following composition, where $N$ is any sufficiently large integer:
\[
\widehat{HF}(S^3_0(K\# \overline{K}),[0])\xrightarrow{\mathfrak{f}^{K\# \overline{K}}_{0,N;0}} \widehat{HF}(S^3 _N(K\# \overline{K}),[0]) \xrightarrow{\Gamma_{N,0}} HFK_{\mathcal{R}}(S^3,K\#\overline{K}).
\]
In this paper, however, we need its generalization to all $\mathrm{Spin}^c$ structures on $S^3_0(K\#\overline{K})$. In other words, we require a map the form
\[
\Lambda:\widehat{HF}(S^3_0(K\#\overline{K})) = \bigoplus_{s\in\mathbb{Z}}\widehat{HF}(S^3_0(K\#\overline{K}),[s])\rightarrow HFK_{\mathcal{R}}(S^3,K\#\overline{K}).
\]
whose restriction to $\widehat{HF}(S^3_0(K\#\overline{K}),[0])$ recovers the original definition of $\Lambda$ in \cite{guthkang2024invariantsplittingprinciples}.

To do this, we simply follow the same strategy. For each $s\in\mathbb{Z}$, we consider the map
\[
\Lambda_s:\widehat{HF}(S^3_0(K\#\overline{K}),[s])\xrightarrow{\mathfrak{f}^{K\#\overline{K}}_{0,N;s}} \widehat{HF}(S^3_N(K\#\overline{K}),[s])\xrightarrow{\Gamma_{N,s}} HFK_{\mathcal{R}}(S^3,K\#\overline{K})
\]
for any sufficiently large integer $N$; note that our choice of $N$ depends on $s$. 

\begin{lem}
    The map $\Lambda_s$ is independent of the choice of $N$.
\end{lem}
\begin{proof}
    Choose any sufficiently large integers $N,N'$ with $N<N'$. Recall from the proof of \Cref{lem: large relative surgery map is bijective} that 
    \[
    \Gamma_{N;s} = \Gamma_{N';s}\circ \mathfrak{f}^{\KK}_{N,N';s}.
    \]
    Then it follows from \Cref{lem: surgery maps are compatible} that
    \[
        \Gamma_{N,s} \circ \mathfrak{f}^{\KK}_{0,N;s} = \Gamma_{N',s} \circ \mathfrak{f}^{\KK}_{N,N';s} \circ \mathfrak{f}^{\KK}_{0,N;s} 
        = \Gamma_{N',s} \circ \mathfrak{f}^{\KK}_{0,N';s}. 
    \]
    The lemma follows.
\end{proof}

Hence, $\Lambda_s$ is well-defined.

\begin{defn}
    For any knot $K$, we define the map $\Lambda$ as follows:
    \[
    \widehat{HF}(S^3_0(K\#\overline{K})) = \bigoplus_{s\in\mathbb{Z}}\widehat{HF}(S^3_0(K\#\overline{K}),[s])\xrightarrow{\bigoplus_{s\in\mathbb{Z}} \Lambda_s} HFK_{\mathcal{R}}(S^3,K\#\overline{K}).
    \]
    Then, using the canonical identifications
    \[
    \widehat{HF}(S^3 _0 (K\# \overline{K}))\simeq \mathrm{End}^{\mathcal{A}}_h(\widehat{CFD}(S^3 \smallsetminus K)),\quad HFK_{\mathcal{R}}(S^3,K\#\overline{K})\simeq \mathrm{End}^h_{\mathcal{R}}(CFK_{\mathcal{R}}(S^3,K)),
    \]
    we may regard $\Lambda$ as a map of the form:
    \[
    \Lambda:\mathrm{End}^{\mathcal{A}}_h(\widehat{CFD}(S^3 \smallsetminus K))\rightarrow \mathrm{End}^h_{\mathcal{R}}(CFK_{\mathcal{R}}(S^3,K)).
    \]
\end{defn}

\subsection{Locally Symmetric Endomorphisms and $\Lambda$}

The map $\Lambda$ shifts the gradings in a well behaved manner. Since $K\#\overline{K}$ is smoothly slice, it follows from \Cref{lem: uniqueness of absolute K-Maslov grading} in the previous subsection that $\widehat{HF}(S^3_0(K\#\overline{K}),[s])$ admits an absolute $K\#\overline{K}$-grading for all $s\in\mathbb{Z}$.

\begin{lem}
    With respect to the absolute $K\# \overline{K}$-bigrading of $\widehat{HF}(S^3 _0 (K\# \overline{K}))$ and the Alexander-Maslov bigrading of $HFK_{\mathcal{R}}(S^3,K\#\overline{K})$, the map $\Lambda$ is homogeneous of bidegree $\left( -\frac{1}{2},0 \right)$.
\end{lem}
\begin{proof}
    The image of $\Lambda_s$ is contained in the image of $\Gamma_{N;s}$ for any sufficiently large integer $N$, whose elements have Alexander grading $s$. Also, with respect to the absolute $K\# \overline{K}$-grading of $\widehat{HF}(S^3 _0 (K\#\overline{K}),[s])$ and the Maslov grading of $HFK_{\mathcal{R}}(S^3,K\# \overline{K})$, the degree of $\Lambda_s$ is given by 
    \[
    \begin{split}
    \deg \Lambda_s &= \deg \Gamma_{N,s} + \deg \mathfrak{f}^K_{0,N;s} \\
    &= \frac{1}{4} - \frac{(2s-N)^2}{4N} - \frac{3}{4} + \frac{(2s-N)^2}{4N} \\
    &= -\frac{1}{2}.
    \end{split}
    \]
    Therefore $\Lambda$ is a map of bidegree $\left( -\frac{1}{2},0 \right)$, as desired.
\end{proof}

As noted at the beginning of this section, we will be interested in locally symmetric $\cR$-endomorphisms. For this reason, we briefly describe how the map $\Lambda$ is related to locally symmetric endomorphisms of $CFK_{\cR}(S^3,K)$. 

\begin{lem} \label{lem: local symmetry criterion}
    Let $K$ be a knot and $D_K$ be the canonical smooth slice disk of $K\#\overline{K}$, seen as a concordance from $K\#\overline{K}$ to the unknot. Consider the induced cobordism map $F_{D_K}$ on $\cR$-coefficient knot Floer homology, which we write as
    \[
    F_{D_K}: \End^h_{\cR}(CFK_{\cR}(S^3,K)) \ra \cR,
    \]
    under the canonical identifications 
    \[
    \End^h_{\cR}(CFK_{\cR}(S^3,K)) \simeq HFK_\cR(S^3,K\#\overline{K}),\quad HFK_\cR(S^3,U)\simeq \cR.
    \]
    Then a homogeneous chain endomorphism $f$ of $CFK_{\cR}(S^3,K)$ is locally symmetric if and only if $F_{D_K}([f])\in\{0,1\}$.
\end{lem}
\begin{proof}
    Observe that, since $D_K$ is a smooth concordance, the truncated maps
    \[
    \begin{split}
    F^{U=0}_{D_K} &: \End^h_{\F[V]}(CFK_{\cR}(S^3,K)/(U)) \ra \F[V], \\
    F^{V=0}_{D_K} &: \End^h_{\F[U]}(CFK_{\cR}(S^3,K)/(V)) \ra \F[U]
    \end{split}
    \]
    are both local, i.e. become homotopy equivalences after localizing by inverting $V$ and $U$, respectively. Suppose first that $f$ is locally symmetric, so that $f_U$ and $f_V$ are either both local or both non-local. If they are both local, then we have
    \[
    F^{U=0}_{D_K}([f_V])\neq 0 \text{ in }\F[V],\quad F^{V=0}_{D_K}([f_U])\neq 0 \text{ in }\F[U].
    \]
    It follows that $F_{D_K}([f]) \not\equiv 0\pmod U$ and $F_{D_K}([f])\not\equiv 0\pmod V$. Since $f$ is homogeneous, $F_{D_K}([f])$ is a homogeneous element of $\cR$, i.e. either 1, a power of $U$, or a power of $V$. Since powers of $U$ are zero mod $U$ and powers of $V$ are zero mod $V$, we deduce that $F_{D_K}([f])=1$. A similar argument also shows that, if $f_U$ and $f_V$ are both non-local, then $F_{D_K}([f])=0$. Hence we see that $F_{D_K}([f])\in \{0,1\}$ whenever $f$ is locally symmetric.

    Now suppose that $F_{D_K}([f])\in \{0,1\}$. Using a similar argument again, it is straightforward to observe that, if $F_{D_K}([f])=1$, then both $f_U$ and $f_V$ are local, whereas if $F_{D_K}([f])=0$, then both $f_U$ and $f_V$ are non-local. Hence we see that $f$ is locally symmetric whenever $F_{D_K}([f])\in \{0,1\}$. The lemma follows.
\end{proof}

\begin{lem} \label{lem: image of lambda is locally symmetric}
    The image of the $\Lambda$ map, i.e.
    \[
    \Lambda: \mathrm{End}^{\mathcal{A}}_h(\widehat{CFD}(S^3 \smallsetminus K))\rightarrow \mathrm{End}^h_{\mathcal{R}}(CFK_{\mathcal{R}}(S^3,K)),
    \]
    is contained in $\End^{h,ls}_{\cR}(\CFK_\cR(S^3,K))$.
\end{lem}
\begin{proof}
    View $\Lambda$ as a map
    \[
    \Lambda: \widehat{HF}(S^3_0(K\#\overline{K}))\rightarrow  HFK_\cR(S^3,K\#\overline{K}).
    \]
    According to \Cref{lem: local symmetry criterion}, in order to prove the lemma, it suffices to show that the image of the composition
    \[
    \widehat{HF}(S^3_0(K\#\overline{K})) \xrightarrow{\Lambda} HFK_\cR(S^3,K\#\overline{K}) \xrightarrow{F_{D_K}} HFK_\cR(S^3,U) \simeq \cR
    \]
    is contained in $\{0,1\}$. To show this, we start by considering the following diagram, whose commutativity follows easily from the definition of $\Lambda$ and the functoriality of Heegaard Floer homology.
    \[
    \xymatrix{
    \widehat{HF}(S^3_0(K\#\overline{K})) \ar[r]^\Lambda \ar[d]^{F_{X_0(C)}} & HFK_\cR(S^3,K\#\overline{K})  \ar[d]^{F_{D_K}} \\
    \widehat{HF}(S^3_0(U)) \ar[r]^\Lambda & HFK_\cR(S^3,U) \simeq \cR 
    }
    \]
    Here, $C$ denotes the concordance $K \# \overline{K} \ra U$ obtained by puncturing the canonical slice disk of $K\# \overline{K}$ and $X_0(C)$ is again the homology cobordism from $S^3 _0(K\#\overline{K})$ and $S^3 _0(U)$ given by performing 0-surgery along the concordance $C$. Hence the image of $F_{D_K} \circ \Lambda = \Lambda \circ F_{X_0(C)}$ is contained in the image of the bottom map
    \[
    \Lambda: \widehat{HF}(S^3_0(U)) \rightarrow HFK_\cR(S^3,U) \simeq \cR,
    \]
    which is exactly $\{0,1\}$. The lemma follows.
\end{proof}

\begin{rem} \label{rem: local symmetry of lambda}
    Using \Cref{lem: image of lambda is locally symmetric}, we can restrict the image of $\Lambda$ to  $\End^{h,ls}_{\cR}(\CFK_\cR(S^3,K))$. Therefore, from now on, we will abuse notations and write $\Lambda$ also as a map of the form
    \[
    \Lambda: \End^{\cA}_h(\CFDh(\KC)) \rightarrow \End^{h,ls}_{\cR}(\CFK_\cR(S^3,K)),
    \]
    which is still homogeneous of bidegree $\left( -\frac{1}{2},0\right)$.
\end{rem}

According to \Cref{lem: large surgery from 0 surgery is injective}, the kernel of this map is 1-dimensional, generated by a class $\theta^-_K$, which generates the kernel of the map $\frak{f}_{0, N}^K: \HFh(S^3_0(K\#\overline{K})) \ra \HFh(S^3_N(K\#\overline{K}))$. We note that according to \Cref{lem: surgery maps are compatible}, the class $\theta_K^-$ is independent of $N$. 

We will return to a more complete analysis of this class in \Cref{sec: theta class}. 
Furthermore, if $K_0$ and $K_1$ are locally equivalent, $K_0 \# \overline{K_1}$ is locally trivial, so we can mimic the constructions above to produce a map 
\begin{align*}
    \Lambda_{K_0,K_1}: \Mor^\cA_{h}(\KC_0, \KC_1) \ra \Hom_\cR^{h,ls}(K_0, K_1),
\end{align*}
in exactly the same way, and there is a class $\theta^-_{K_0, K_1}$ generating its kernel. We will return to this class in \Cref{sec: theta class}.

Finally, we discuss the bordered interpretation of $\Lambda$, as in \cite{guthkang2024invariantsplittingprinciples}. There, we worked exclusively with the $\SpinC$-structure $[0]$, though the generalization is again straightforward. 

Consider the 2-handle cobordism
\begin{align*}
    X_{N,N'}^K: (\KC)_{N} \# L(N'-N,1) \ra (\KC)_{N'},
\end{align*}
which is the bordered analogue of the cobordism $W^K_{N,N'}$. According to \cite[Lemma 5.11]{guthkang2024invariantsplittingprinciples}, there is a map $F_{X^K_{0,N}}: \CFDh((\KC)_0\#L(N,1))\ra \CFDh((\KC)_{N})$ making the following diagram commute:
\begin{center}
    \begin{tikzcd}
        \CFAh(-(\KC)_{0})\boxtimes \CFDh((\KC)_0\#L(N,1)) \ar[r] \ar[d,"\bI \boxtimes F_{X_{0,N}^K}"]
        &  \CFh(S^3_0(K\#\overline{K})\#L(N,1)) \ar[d,"F_{W_{0,N}^K}"] \\
        \CFAh(-(\KC)_0)\boxtimes \CFDh((\KC)_N)  
        \ar[r]
        &  \CFh(S^3_N(K\#\overline{K})).
    \end{tikzcd}
\end{center}
Here, the horizontal arrows are given by the pairing theorem; the vertical arrows are the \emph{total} cobordism maps, where we sum over all $\SpinC$-structures. In particular, when $N$ is large, we can compute $\frak{f}_{0,N;s}^K$ from $\bI \boxtimes F_{X_{0,N}^K}$. 

By the work of \cite{cohen2023composition}, there is a map
\begin{align*}
    \HFh(S^3_N(\KK)) \otimes_\F \HFK^-(S^3, K) \ra \HFK^-(S^3, K).
\end{align*}
In \cite{guthkang2024invariantsplittingprinciples}, we showed that this map can be upgraded to a map
\begin{align*}
    \HFh(S^3_N(\KK)) \otimes_\F \HFK_\cR(S^3, K, p) \ra \HFK_\cR(S^3, K,p),
\end{align*}
at the cost of adding a free basepoint. More precisely, we identify
\begin{align*}
    \CFh(S^3_N(\KK)) \simeq \Mor^\cA(\CFDh((\KC)_0),\CFDh((\KC)_N)),
\end{align*}
and consider the map
\begin{align*}
    \Mor^\cA(\CFDh((\KC)_0),\CFDh((\KC)_N)) \otimes (\CFAh(\mathbb{X})\boxtimes\CFDh((\KC)_0)) \ra \CFAh(\mathbb{X})\boxtimes\CFDh((\KC)_N),  \\ 
    f \otimes (a\boxtimes b) \mapsto (\bI \boxtimes f)(a\boxtimes b), \hspace{4cm}
\end{align*}
where $\bX$ is a particular triply-pointed bordered Heegaard diagram for $(S^1\times D^1, S^1 \times 0, p)$, where $p \in S^1 \times \partial D^2$. Composing with the map $\bI \boxtimes F_{X_{0,N}^K}$ above and restricting to the appropriate $\SpinC$-structures, we obtain a map
\begin{align*}
    \Lambda_p: \HFh(S^3_0(K\#\overline{K}),[s]) \ra H_*(\End_\cR^h(\CFK_\cR(S^3, K, p))).
\end{align*}
Under the natural identifications
\begin{align*}
    \CFK_\cR(S^3, K, p) \simeq \CFK_\cR(S^3, K \cup U) \simeq \CFK_\cR(S^3, K) \otimes_\cR \CFK_\cR(U \cup U) \simeq \CFK_\cR(S^3, K)\otimes_\F \cV,
\end{align*}
where $\cV = \F_{(1/2)}\oplus \F_{(-1/2)}$,
\begin{align*}
    \begin{tikzcd}[ampersand replacement = \&]
        \CFK_\cR(S^3, K, p) \ar[r,"\simeq"]\ar[d,"\Lambda_p(f^1)"] \& \CFK_\cR(S^3, K)\otimes_\F \cV \ar[d,"\Lambda(f^1)\otimes \id_{\cV}"] \\
        \CFK_\cR(S^3, K, p) \ar[r,"\simeq"] \& \CFK_\cR(S^3, K)\otimes_\F \cV,
    \end{tikzcd}
\end{align*}
See \cite[Section 5]{guthkang2024invariantsplittingprinciples} for more details. There we only considered the case $s =0$, though the general case is no different. 

\begin{lem} \label{lem: spin c shift is spin c support}
    Let $f$ be a type D endomorphism of $\widehat{CFD}((S^3 \smallsetminus K)_{-N})$ whose corresponding element
    \[
    [f] \in \mathrm{End}_h(\widehat{CFD}(S^3 \smallsetminus K))\simeq \widehat{HF}(S^3 _0 (K\# \overline{K}))
    \]
    has Alexander grading $n$, i.e. supported on the $\mathrm{Spin}^c$ structure $[n]$ for some integer $n$. Then, for any integer $m\ge 0$, the chain map
    \[
     \mathrm{id} \boxtimes f:\widehat{CFA}(\donut_m)\boxtimes \widehat{CFD}((S^3 \smallsetminus K)_{-N}) \rightarrow \widehat{CFA}(\donut_m)\boxtimes \widehat{CFD}((S^3 \smallsetminus K)_{-N})
    \]
    maps the $\mathrm{Spin}^c$ component $\widehat{CF}(S^3 _{-N-m}(K),[s])$ to $\widehat{CF}(S^3 _{-N-m}(K),[s+n])$ for any $s\in \Z$.
\end{lem}
\begin{proof}
    By \cite[Theorem 1.1]{cohen2023composition}, the map 
    \begin{align*}
        \End^\cA(\CFDh(S^3 \smallsetminus K)_{-N}) \otimes \widehat{CFA}(\donut_m)\boxtimes \widehat{CFD}((S^3 \smallsetminus K)_{-N})\ra \widehat{CFA}(\donut_m)\boxtimes \widehat{CFD}((S^3 \smallsetminus K)_{-N})
    \end{align*}
    given by $(f, a\boxtimes x) \mapsto (\bI \boxtimes f)(a\boxtimes x)$, is induced by the pair of pants cobordism $P$ for the triple $(\bm{O}_m, (S^3 \smallsetminus K)_{-N},(S^3 \smallsetminus \overline{K})_{N})$. By the computations in \Cref{sec:main_hypercube}, we have that
    \begin{align*}
        H^2(P)\cong H^2(S^3_0(K\# \overline{K}))\oplus H^2(S^3_{-(N+m)}(K))\cong \Z \oplus \Z/(-m-N)\\
        H^2(\partial_- P) \cong H^2(S^3_0(K\# \overline{K}))\oplus H^2(S^3_{-(N+m)}(K)) \cong \Z \oplus \Z/(-m-N)\\
        H^2(\partial_+ P) = H^2(S^3_{-(N+m)}(K))\cong \Z/(-m-N)
    \end{align*}
    and with respect to this identification, the restrictions are given by 
    \begin{align*}
        H^2(P) \ra H^2(\partial_- P), & \quad & (A, [\ell]) \mapsto (A, [\ell])\\
        H^2(P) \ra H^2(\partial_+ P), &\quad  &  (A, [\ell]) \mapsto [A+\ell].
    \end{align*}
    As usual, we identify these cohomology classes with $\SpinC$-structures. In particular, if we fix a $\SpinC$-structure $(n, [s])$ on $P$, the induced map takes $\widehat{CF}(S^3 _{-N-m}(K),[s])$ to $\widehat{CF}(S^3 _{-N-m}(K),[s+n])$ as claimed.
\end{proof}

\subsection{$\Lambda$ and the telescope model for $\CFK_\cR(S^3, K)$}

In this section, we give a description of the map $\Lambda$ in terms of the hypercube model for $\CFK_\cR(S^3, K)$. 

As discussed in the conclusion of \Cref{sec: Telescopes and large surgeries}, the identification 
\begin{align*}
    \frak{CFK}(K) \xra{\simeq} \CFK_\cR(S^3, K),
\end{align*}
can be computed from the bordered perspective. We decomposed 
\begin{align*}
    S^3_N(K) & = \bm O_{km} \cup \bm O_{N} \cup P \cup (\KC)_0\\
    (S^3, K) & = \bm O_{km} \cup (\bm O_{\infty}, \nu) \cup P \cup (\KC)_0.
\end{align*}
and formed the hypercubes 
\begin{align*}
    \frak H_{\bm O} :=\frak{CFA}(\bm O) \boxtimes \CFAh(\bm O_N) \boxtimes \widehat{CFDDA}(P) \boxtimes \CFDh(\KC),
\end{align*}
and 
\begin{align*}
    \frak H_{K}:=\frak{CFA}(\bm O) \boxtimes \CFAh(\bm O_{\infty}, \nu) \boxtimes \widehat{CFDDA}(P) \boxtimes \CFDh(\KC).
\end{align*}
The map $\Gamma^A: \CFAh(\bm O_N) \ra \CFAh(\bm O_\infty, \nu)$ induced a morphism between these hypercubes, and by restricting to the appropriate $\SpinC$-structures, this map induces the identification above. 

\begin{figure}
    \centering
    \includegraphics[width=0.5\linewidth]{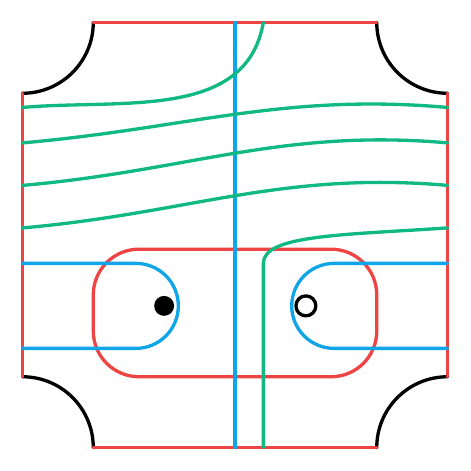}
    \caption{A bordered Heegaard triple diagram}
    \label{fig:tripleX}
\end{figure}

In order to make sense of $\Lambda$ in this context, we need to add in an extra basepoint. By considering  $\CFAh(\bm O_N, p)$ and $\CFAh(\bX)$ in place of $\CFAh(\bm O_{N})$ and $\CFAh(\bm O_{\infty}, \nu)$, the same construction gives rise to hypercubes $\frak H_{\bm O,p}$ and $\frak H_{\bX}$ as well as a map $\Gamma_{p}$ between them which computes the equivalence 
\begin{align*}
    \frak{CFK}(K, p) \xra{\simeq} \CFK_\cR(S^3, K,p),
\end{align*}
where $\frak{CFK}(K, p)$ is constructed in the same manner as $\frak{CFK}(K)$, though all diagrams are equipped with a free basepoint, $p$.

The hypercubes $\frak H_{\bm O,p}$ and $\frak H_{\bX}$ both come equipped with natural $\End^\cA(\CFDh(\KC))$-actions: $f^1 \in \End^\cA(\CFDh(\KC))$ acts on both by $\bI \boxtimes f^1$. This morphism of hypercubes gives rise to a homotopy commutative diagram of their telescopes:
\begin{align*}
    \begin{tikzcd}[ampersand replacement = \&]
        \Tel(\frak H_{\bm O,p}) \ar[r] \ar[d,"\Tel(\bI \boxtimes f^1)"] \& \Tel(\frak H_{\bX}) \ar[d,"\Tel(\bI \boxtimes f^1)"] \\
        \Tel(\frak H_{\bm O,p}) \ar[r] \& \Tel(\frak H_{\bX}).
    \end{tikzcd}
\end{align*}
Note, that $\Tel(\frak H_{\bm O,p})$ can be canonically identified with $\Tel(\frak H_{\bm O})\otimes \cV$, where $\cV = \F_{(-1/2)}\oplus \F_{(1/2)}$ (subscripts denote the Maslov grading). The map $\Gamma^A: \CFAh(\bm O_N, p) \ra \CFAh(\bX)$ is computed by counting triangles in a particular bordered Heegaard triple diagram \Cref{fig:tripleX}. Let us write $\cH_{\bm O_N, p}$ for this particular Heegaard diagram for $\CFAh(\bm O_N, p)$. Let $\cH'_{\bm O_N, p}$ be the standard genus 1 bordered Heegaard diagram for $\bm O_N$ which has been free-stabilized near the boundary basepoint. Fix a sequence of bordered Heegaard diagrams interpolating between $\cH_{\bm O_N, p}$ and $\cH'_{\bm O_N, p}$. We can arrange that this sequence only involves isotopies and handleslides. By the invariance proof of \cite{LOT_bordered_HF}, there is an induced map
\begin{align*}
    G: \CFAh(\cH_{\bm O_N, p}) \ra \CFAh(\cH'_{\bm O_N, p}),
\end{align*}
given by counting triangles which is an equivalence. 
\begin{lem}
    The map $G$ induces a map of hypercubes
    \begin{align*}
        G:\frak H_{\bm O,p} \ra \frak H_{\bm O} \otimes \cV.
    \end{align*}
    Furthermore, $G$ induces a homotopy equivalence of telescopes. 
\end{lem}
\begin{proof}
    This follows from the fact that $\CFAh(\cH'_{\bm O_N, p})$ can be canonically identified with $\CFAh(\bm O_N) \otimes \cV$ since we are in the hat-setting. The hypercube structure maps are obtained by tensoring those for $\frak H_{\bm O}$ with $\id_{\cV}$.
\end{proof}

It follows that for any $f^1 \in \End^\cA(\CFDh(\KC))$, we have a homotopy commutative square 
\begin{align*}
    \begin{tikzcd}[ampersand replacement = \&]
        \Tel(\frak H_{\bm O,p}) \ar[r] \ar[d,"\Tel(\bI \boxtimes f^1)"] \& \Tel(\frak H_{\bm O})\otimes \cV \ar[d,"\Tel(\bI \boxtimes f^1)\otimes \id_\cV"] \\
        \Tel(\frak H_{\bm O,p}) \ar[r] \& \Tel(\frak H_{\bm O})\otimes \cV.
    \end{tikzcd}
\end{align*}
In particular, it will turn out that the action of $f^1$ can be understood very concretely.

We now turn to the action of $\End^\cA(\CFDh(\KC))$ on $\frak H_{\bX}$. By the pairing theorem, $\frak H_{\bX}$ is identified with the hypercube $(\CFK_\cR(S^3, K, p), U, V, 0)$.

\begin{lem}\label{lem:Lambda absorbs framing change} Let $f \in \Mor^\cA((\KC)_0, (\KC)_0)$ be an element supported in $\SpinC$-structure $A$ (i.e., the corresponding element lies in $\HFh(S^3_0(K\#\overline{K}),[A])$). Then, there is a commutative square
    \begin{center}
        \begin{tikzcd}[column sep = small]
            (\Mor^\cA((\KC)_0, (\KC)_0))_{[A]}\otimes (\CFAx\boxtimes\CFDh((\KC)_0))_{[s]} \ar[d,"f \otimes (a \boxtimes v) \mapsto (\bI\boxtimes (F_{X_n} \circ f))(a \boxtimes v)"]\ar[r,"ev"] &
            (\CFAx\boxtimes\CFDh((\KC)_0))_{[s+A]}\ar[d," \bI \boxtimes F_{X_n}"]\\
            (\Mor^\cA((\KC)_0, (\KC)_n))_{[A]}\otimes (\CFAx\boxtimes\CFDh((\KC)_0))_{[s]}\ar[r,"ev"] &
            (\CFAx\boxtimes\CFDh((\KC)_n))_{[s+A]}.
        \end{tikzcd}
    \end{center}
Furthermore, the component of the map $\bI \boxtimes F_{X_n}$ taking $(\CFAx\boxtimes\CFDh((\KC)_0))_s$ to $(\CFAx\boxtimes\CFDh((\KC)_n))_s$ is a homotopy equivalence. 
\end{lem}
\begin{proof}
    This essentially follows from the proof of \cite[Lemma 5.17]{guthkang2024invariantsplittingprinciples}. The horizontal arrows are induced by the pair of pants cobordism maps induced by the bordered Heegaard triples $(\bm O_0, \bm O_0, \bX)$ and $(\bm O_n, \bm O_0, \bX)$ \cite{cohen2023composition}, while the vertical arrows are change of framing cobordisms; commutativity of the diagram follows from the fact that the change of framing cobordism commutes with the pair of pants cobordism. The second claim is proven directly in \cite[Lemma 5.17]{guthkang2024invariantsplittingprinciples} by splitting off a local unknot from $K$ and computing the map directly for the unknot.
\end{proof}

Note that the counter-clockwise composition in the diagram above is exactly the bordered Floer homology definition of $\Lambda_p(f)$. As a consequence, we have a commutative diagram 
\begin{align*}
    \begin{tikzcd}[ampersand replacement = \&]
        (\CFAx\boxtimes\CFDh((\KC)_0) )_{s}\ar[d," \bI \boxtimes f^1"] \ar[r,"\simeq"] \& \CFK_\cR(S^3, K, p, s) \ar[d,"\Lambda_p(f^1)"] \\
        (\CFAx\boxtimes\CFDh((\KC)_0))_{s+A} \ar[r,"\simeq"] \& \CFK_\cR(S^3, K, p, s+A).
    \end{tikzcd}
\end{align*}

\begin{lem}
    Let $P$ be the pair of pants cobordism associated to the triple $(\bm O_0, \bm O_0, \bX)$. The induced map gives rise to a homotopy commutative diagram:
    \begin{align*}
    \begin{tikzcd}[ampersand replacement = \&]
        \Tel(\frak H_{\bX}) \ar[r] \ar[d,"\Tel(\bI_{\bX} \boxtimes f^1)"] \& \Tel(\CFK_\cR(S^3, K, p), U, V, 0) \ar[r] \ar[d,"\Tel(F_P(f^1{,}{-}))"] \& \CFK_\cR(S^3, K, p) \ar[d,"\Lambda_p(f^1)"] \\
        \Tel(\frak H_{\bX}) \ar[r] \& \Tel(\CFK_\cR(S^3, K, p), U, V, 0) \ar[r] \& \CFK_\cR(S^3, K, p).
    \end{tikzcd}
\end{align*}
The left-most horizontal arrows are homotopy equivalences induced by the pairing theorem for polygons, and the right-most horizontal arrows are homotopy equivalences induced by the canonical map $\Tel(C, U, V, 0) \ra C.$
\end{lem}
\begin{proof}
    The homotopy commutativity of the first square follows from the pairing theorem for polygons \cite{LOT_spectral_sec_II} and Cohen's composition theorem \cite{cohen2023composition}. The homotopy commutativity of the second square follows from \Cref{lem:Lambda absorbs framing change}. 
\end{proof}

In summary, we have a homotopy commutative diagram 
\begin{align*}
     \begin{tikzcd}[ampersand replacement = \&]
        \Tel(\frak H_{\bm O}) \otimes \cV \ar[d,"\Tel(\bI \boxtimes f^1)\otimes \id_\cV"] \ar[r,"\simeq"] \& \CFK_\cR(S^3, K) \otimes \cV \ar[d,"\Lambda(f^1)\otimes \id_\cV"]\\
        \Tel(\frak H_{\bm O}) \otimes \cV  \ar[r,"\simeq"] \& \CFK_\cR(S^3, K) \otimes \cV,
    \end{tikzcd}   
\end{align*}
which gives us a very concrete model for the action of $\Lambda(f^1)$ in terms of hypercubes.

\subsection{Doubling}

As our construction involves adding additional basepoints, the associated knot Floer complexes end up being doubled. Before proceeding, it will be useful to prove the same happens at the level of morphisms.

Let $C$ be a chain complex over a ring $\cR$ equipped with the following additional structure: 

\begin{itemize}
    \item A Maslov bigrading $(\gr_w, \gr_z) \in \Z\oplus \Z$ with respect to which $U$ has degree $(-2, 0)$, $V$ has degree $(0, -2)$, and $\partial$ has degree $(-1,-1)$; 
    \item An Alexander grading, $A = \frac{1}{2}(\gr_w - \gr_z)$ with respect to which $U$ has degree $-1$, $V$ has degree $+1$, and $\partial$ has degree 0. 
\end{itemize}

\begin{defn} \label{defn: maslov-graded vector spaces}
    We say that a bigraded $\mathbb{F}_2$-vector space $V$ is \emph{Maslov-graded} if it is supported in Alexander grading 0, i.e. its $(\gr_z,\gr_w)$-grading lies on the diagonal $\gr_z = \gr_w$. We say that a Maslov-graded $\mathbb{F}_2$-vector space $V$ is \emph{neatly bounded below} if its Maslov grading has a minimum $M_{min}$ and the $M_{min}$-graded component of $V$ is 1-dimensional, and \emph{neatly bounded above} if its Maslov grading has a maximum $M_{max}$ and the $M_{max}$-graded component of $V$ is 1-dimensional.
\end{defn}

\begin{defn}
    Given a bigraded chain complex $C$ over $\mathcal{R}$, define $\mathrm{End}^0_{\mathcal{R}}(C)$ to be the chain complex of $\mathcal{R}$-linear endomorphisms of $C$ which preserve the Alexander grading. Note that its homology is the ring of homotopy classes of $\mathcal{R}$-linear Alexander grading preserving chain endomorphisms of $C$.
\end{defn}

\begin{defn}
    We say that an endomorphism $p \in \mathrm{End}^0_{\mathcal{R}}(C)$ is a \emph{homotopy-projection} if $p^2 \sim p$ and a \emph{projection} if $p^2 = p$. We say that two projections $p_0, p_1$ \emph{homotopy-commute} if $p_1 \circ p_0 \sim p_0 \circ p_1$. 
\end{defn}

The following lemma is standard. See \cite{guthkang2024invariantsplittingprinciples} for a proof.

\begin{lem}\label{lem:homotopy projections are homotopic to projections}
    Let $C$ be a finitely generated chain complex. Then:
    \begin{itemize}
        \item Every homotopy projection of a finitely generated chain complex is homotopic to a projection.
        \item Homotopic projections have homotopy equivalent images and kernels. 
        \item Pairwise homotopy-commuting projections are homotopic to pairwise commuting projections.
    \end{itemize}
\end{lem}

\begin{defn}
    Given a finitely generated bigraded chain complex $C$ over $\mathcal{R}$, we say that a ring endomorphism $K$ of $H_\ast(\mathrm{End}^0_{\mathcal{R}}(C))$ is \emph{summand-preserving} if for any chain endomorphism $p$ of $C$ satisfying $p^2 = p$, the image of $p$ is homotopy equivalent to the image of $K(p)$.
\end{defn}

\begin{lem}\label{lem:double_conjugation}
    Let $C$ be a finitely generated free chain complex over $\mathcal{R}$ and $f$ be an $\mathcal{R}$-linear chain endomorphism of $C$ which preserves the Alexander grading. Given a finite-dimensional Maslov-graded $\mathbb{F}_2$-vector space $\cV$ that is either neatly bounded below or neatly bounded above, suppose that there exist a summand-preserving ring endomorphism
    \[
    K:H_\ast(\mathrm{End}^0_{\mathcal{R}}(C)) \rightarrow H_\ast(\mathrm{End}^0_{\mathcal{R}}(C))
    \]
    and a homotopy autoequivalence
    \[
    F:C \otimes_{\mathbb{F}_2} \cV \xrightarrow{\simeq} C \otimes_{\mathbb{F}_2} \cV
    \]
    fitting into a homotopy commutative diagram
    \begin{align*}
        \begin{tikzcd}[ampersand replacement = \&]
            C \otimes_{\mathbb{F}_2} \cV \ar[r,"F"]\ar[d,"f \otimes \id_{\cV}"] \& C \otimes_{\mathbb{F}_2}\cV\ar[d,"K(f) \otimes \id_{\cV}"] \\
            C \otimes_{\mathbb{F}_2}\cV\ar[r,"F"] \& C \otimes_{\mathbb{F}_2}\cV\\
        \end{tikzcd}
    \end{align*}
    for all $f\in \mathrm{End}^0_{\mathcal{R}}(C)$. Then there exists a homotopy equivalence
    \[
    \widetilde{F}:C\rightarrow C
    \]
    making the diagram
    \begin{align*}
        \begin{tikzcd}[ampersand replacement = \&]
            C \ar[r,"\widetilde{F}"]\ar[d,"f"] \& C \ar[d,"K(f)"] \\
            C \ar[r,"\widetilde{F}"] \& C  \\
        \end{tikzcd}
    \end{align*}
    commute up to homotopy for all $f\in \mathrm{End}^0_{\mathcal{R}}(C)$.
\end{lem}
\begin{proof}
    By \cite[Theorem 1.1]{popovic2023link}, there exist indecomposable (homotopy) summands $C_1,\cdots,C_n$, each of which are either a snake complex or a local system, such that 
    \[
    C\simeq C_1\oplus \cdots \oplus C_n.
    \]
    For each $i=1,\cdots,n$, let $p_i$ be the projection endomorphism of $C$ which is identity on $C_i$ and zero on every other summand. Then, since $K$ is a ring homomorphism, $K(p_1),\cdots,K(p_n)$ are pairwise homotopy-commuting homotopy projections of $C$; since $C$ is finitely generated, we can homotope them to pairwise commuting projections of $C$, which we will still denote as $K(p_1),\cdots,K(p_n)$. Since $K$ is summand-preserving, we have $K(C_i)\simeq C_i$ for each $i$.

    By assumption, we have
    \[
    F\circ p_i \sim K(p_i) \circ F
    \]
    for each $i$. Hence, 
    \[
    \sum_{i=1}^n (K(p_i)\otimes \mathrm{id}_\cV) F (p_i\otimes \mathrm{id}_\cV) \sim \sum_{i=1}^n F(p_i^2\otimes \mathrm{id}_\cV) = F\circ \left(\sum_{i=1}^n p_i\otimes \mathrm{id}_\cV \right) = F\circ \mathrm{id}_{C\otimes \cV} = F,
    \]
    It follows that $F$ can be replaced with $\sum_{i=1}^n (K(p_i)\otimes \mathrm{id}_\cV) F (p_i\otimes \mathrm{id}_\cV)$; we will abuse notation and call $F$ as well. This ensures that 
    \[
    F = F_1 \oplus \cdots \oplus F_n,
    \]
    where each $F_i$ is a homotopy equivalence of the form
    \[
    F_i : C_i \otimes\cV\xrightarrow{\simeq} K(C_i) \otimes \cV.
    \]
    Let $G_i$ be a homotopy inverse of $F_i$. Without loss of generality, we will assume that $\cV$ is neatly bounded below, and write $m$ for the minimal Maslov grading of $\cV$. Choose a basis $B=\{b_1,\cdots,b_k\}$ of $\cV$, consisting of homogeneous elements, such that $b_1$ has Maslov grading $m$. For each $b\in B$, consider the canonical inclusion and projection:
    \[
    \mathrm{inc}_b:C_i\xrightarrow{x\mapsto x\otimes b} C_i\otimes \cV,\quad \mathrm{pr}_b:C_i\otimes \cV\xrightarrow{x\otimes b\mapsto x} C_i.
    \]
    Using these maps, for any $b,b^\prime \in B$, we can define
    \[
    (F_i)_{b,b^\prime} = \mathrm{pr}_{b^\prime}\circ F_i \circ \mathrm{inc}_b,\quad (G_i)_{b,b^\prime} = \mathrm{pr}_{b^\prime}\circ G_i \circ \mathrm{inc}_b.
    \]
    Then, since $G_i \circ F_i \sim \mathrm{id}_{C_i \otimes \cV}$, we have
    \[
    \sum_{j=1}^k (G_i)_{b_j,b_1} \circ (F_i)_{b_1,b_j} = (G_i \circ F_i)_{b_1,b_1} \sim \mathrm{id}_{C_i}.
    \]
    Observe that $(G_i)_{b_j,b_1} \circ (F_i)_{b_1,b_j}$ is a (bidegree-preserving) chain endomorphism of $C_i$ for each $j$. Since the homotopy endomorphism ring of $C_i$ is local by \cite[Lemma 4.30]{popovic2023link}, there exists some $j$ such that $(G_i)_{b_j,b_1} \circ (F_i)_{b_1,b_j}$ is a unit in the homotopy endomorphism ring of $C_i$ and thus admits a homotopy inverse. This would imply that $C_i [\bm d]$, where 
    \begin{align*}
        \bm d := \deg((F_i)_{b_1,b_j}) = (\gr_w(b_j), \gr_z(b_j)) - (\gr_w(b_1), \gr_z(b_1))
    \end{align*}
    is a (homotopy) direct summand of $K(C_i)$, which is homotopy equivalent to $C_i$. Since $C_i$ is finitely generated (and not acyclic), this is possible only if $\bm d=(0,0)$. Since $\cV$ is neatly bounded below and $b_1$ realizes its minimal Maslov grading, it follows that $b_j = b_1$, i.e. $j=1$. Hence it follows that $(G_i)_{b_1,b_1} \circ (F_i)_{b_1,b_1}$ is a homotopy equivalence. 

    A similar argument using $F\circ G$ instead of $G\circ F$ shows that $(F_i)_{b_1,b_1} \circ (G_i)_{b_1,b_1}$ is also a homotopy equivalence. Hence $(F_i)_{b_1,b_1}$ is a bidegree-preserving homotopy autoequivalence of $C_i$. Observe that we can also define a map
    \[
    F_{b_1,b_1}:C\rightarrow C
    \]
    in a similar manner. Then we would have
    \[
    F_{b_1,b_1} = (F_1)_{b_1,b_1} \oplus \cdots \oplus (F_n)_{b_1,b_1},
    \]
    and thus we see that $F_{b_1,b_1}$ is a homotopy autoequivalence of $C$. Then, for any bidegree-preserving chain endomorphism $f$ of $C$, we have
    \[
    F_{b_1,b_1} \circ f = (F\circ (f\otimes \mathrm{id}_\cV))_{b_1,b_1} \sim ((K(f)\otimes \mathrm{id}_\cV) \circ F)_{b_1,b_1} = K(f) \circ F_{b_1,b_1}.
    \]
    Therefore taking $\widetilde{F} = F_{b_1,b_1}$ proves the lemma.
\end{proof}

\begin{lem}
    Suppose $p \in \End_\cR^0(\CFK_\cR(S^3, K))$ is a homotopy projection. Then, $\Omega(p)$ is as well.
\end{lem}
\begin{proof}
    This is immediate from the properties of $\Omega$, since $p \sim p^2$ implies $\Omega([p]) = \Omega([p^2]) = \Omega([p])^2$.
\end{proof}

\begin{lem}\label{lem:splitting V off of LambdaOmega}
    There is a homotopy equivalence making the diagram 
    \begin{align}\label{eqn: tel to lambda}
     \begin{tikzcd}[ampersand replacement = \&]
        \Tel(\frak H_{\bm O}) \ar[d,"\Tel(\bI \boxtimes \Omega(f))"] \ar[r,"\simeq"] \& \CFK_\cR(S^3, K)  \ar[d,"\Lambda(\Omega(f))"]\\
        \Tel(\frak H_{\bm O})   \ar[r,"\simeq"] \& \CFK_\cR(S^3, K),
    \end{tikzcd}   
\end{align}
    commute up to homotopy.
\end{lem}
\begin{proof}
    Recall that we previously produced a homotopy commutative diagram 
    \begin{align*}
         \begin{tikzcd}[ampersand replacement = \&]
            \Tel(\frak H_{\bm O}) \otimes \cV \ar[d,"\Tel(\bI \boxtimes \Omega(f))\otimes \id_\cV"] \ar[r,"\simeq"] \& \CFK_\cR(S^3, K) \otimes \cV \ar[d,"\Lambda(\Omega(f))\otimes \id_\cV"]\\
            \Tel(\frak H_{\bm O}) \otimes \cV  \ar[r,"\simeq"] \& \CFK_\cR(S^3, K) \otimes \cV.
        \end{tikzcd}   
    \end{align*}
    By \Cref{lem:double_conjugation}, it suffices to prove that $\Lambda \circ \Omega$ is summand preserving. To that end, let $p \in \End^0_\cR(\CFK_\cR(S^3, K))$ satisfy $p^2 = p$. As $\Lambda$ and $\Omega$ are both ring maps, it follows that $\Lambda \circ \Omega(p)$ is a homotopy projection. Further, by applying \Cref{lem:homotopy projections are homotopic to projections}, we may assume that both are honest projections.  The lemma amounts to showing that $\im (p) \simeq \im(\Lambda\Omega(p))$. Note that $\im (p) \simeq \Cone(1 + p)\simeq \ker(1+p)$ and $\im(\Lambda\Omega(p)) \simeq \Cone(1 + \Lambda\Omega(p))\simeq \ker(1+s\Lambda\Omega(p))$. Applying the 5-lemma to the diagram
    \begin{align*}
        \begin{tikzcd}[ampersand replacement = \&]
            0 \ar[r] \ar[d] \& \Cone(1+p) \ar[r] \ar[d] \& \CFK_\cR(S^3, K) \ar[r,"1+p"]\ar[d,"\simeq"] \& \CFK_\cR(S^3, K) \ar[d,"\simeq"] \\
            0 \ar[r] \& \Cone(1+\Lambda\Omega(p)) \ar[r] \& \CFK_\cR(S^3, K) \ar[r,"1+\Lambda\Omega(p)"] \& \CFK_\cR(S^3, K).
        \end{tikzcd} 
    \end{align*}
    yields the result.
\end{proof}

\section{The $\theta^-_K$ class}\label{sec: theta class}
In the section, we study the class $\theta_K^-$, which generates the kernel of $\Lambda$. This class has a very nice interpretation in terms of immersed curves: if $\gamma$ is an immersed curve for an extendable type $D$ structure $M$, then the complex $\End^\cA(M)$ is represented by the endomorphisms of $\gamma$ in the partially wrapped Fukaya category of the punctured torus, $\End_{\cF uk}(\gamma)$. This complex is generated by intersection points between $\gamma$ and the curve $\gamma'$ obtained from $\gamma$ by a small Hamiltonian perturbation. To achieve admissibility, we must apply a fingermove, creating a pair of intersection points between $\gamma$ and $\gamma'$. One of these intersection points represents the identity morphism, and the other represents $\theta_K^-$. In light of this interpretation, we begin by considering immersed curves.

\subsection{Basis free characterization of $\theta^-_K$} \label{subsec: basis free characterization}

We start with some lemmas involving immersed curves.

\begin{lem}\label{lem:theta-minus-central}
    Let $\alpha,\alpha'$, $\beta, \beta'$, $\gamma$, and $\gamma'$ be pulled tight, and in the universal cover of the (unpunctured) torus, we have:
    \begin{itemize}
        \item $\beta$ and $\beta'$ are small Hamiltonian translates of $\alpha$ that intersect $\alpha$ at two points;
        \item $\gamma$ and $\gamma'$ are small Hamiltonian translates of $\alpha'$ that intersect $\alpha'$ at two points;
        \item $|\beta \cap \beta'| = |\gamma \cap \gamma'| = 2$.
    \end{itemize}
    Denote by $\theta^\pm _{\beta,\beta'}$ and $\theta^\pm_{\gamma,\gamma'}$ the two intersection points in $\beta \cap \beta'$ and $\gamma\cap \gamma'$, respectively, labeled so that $\gr(\theta^+_{\beta,\beta'},\theta^-_{\beta,\beta'})=\gr(\theta^+_{\gamma, \gamma'},\theta^-_{\gamma,\gamma'})=1$. Let $f \in CF(\beta,\gamma)$ be a cycle, and $f'$ be the nearest-point translate of $f$ in $CF(\beta',\gamma')$. Then, the following diagram of attaching curves is homotopy-commutative in the partially wrapped Fukaya category of the punctured torus; see \Cref{fig:theta_minus-well-defined}.
    \begin{align*}
        \begin{tikzcd}[ampersand replacement = \&]
            \beta \ar[r,"f"]\ar[d,"\theta^-_{\beta,\beta'}"] \& \gamma \ar[d,"\theta^-_{\gamma,\gamma'}"]\\
            \beta' \ar[r,"f'"] \& \gamma'.
        \end{tikzcd}
    \end{align*}
\end{lem}
\begin{proof}
    Consider the chain-level composition map
    \[
    F_{\beta,\delta,\gamma'}:CF(\beta,\delta)\otimes CF(\delta,\gamma')\rightarrow CF(\beta,\gamma')
    \]
    for $\delta\in \{\beta',\gamma\}$. It suffices to show that, for any intersection point $\mathbf{x}\in \beta \cap \gamma$ and its nearest-point translate $\mathbf{x}''\in \beta'\cap \gamma'$, we have a (strict) equality
    \[
    F_{\beta,\gamma,\gamma'}(\mathbf{x}\otimes \theta^-_{\gamma,\gamma'}) = F_{\beta,\beta',\gamma'}(\theta^-_{\beta,\beta'}\otimes \mathbf{x}'') \in CF(\beta,\gamma').
    \]
    In order to show this, we consider the nearest-point translate $\mathbf{x}'\in CF(\beta,\gamma')$ of $\mathbf{x}$. Without any loss of generality, we may assume that all immersed curves here are connected and pulled tight.

    We first consider the case when $\alpha$ and $\alpha'$ are not isotopic. Suppose that there exists a rigid holomorphic disk $\Delta$ between $\beta$, $\gamma$, and $\gamma'$ that does not cross the basepoint and has $\mathbf{x}\otimes \theta^-_{\gamma,\gamma'}$ as its incoming boundary. By isotoping $\gamma'$ to $\gamma$, we see that $\Delta$ converges to a rigid holomorphic disk between $\beta$ and $\gamma$, a contradiction. Hence, we have the strict equality
    \[
    F_{\beta,\gamma,\gamma'}(\mathbf{x}\otimes \theta^-_{\gamma,\gamma'})=0 \in CF(\beta,\gamma').
    \]
    A similar argument again gives
    \[
    F_{\beta,\beta',\gamma'}(\theta^-_{\beta,\beta'}\otimes \mathbf{x}')=0,
    \]
    and hence $F_{\beta,\gamma,\gamma'}(\mathbf{x}\otimes \theta^-_{\gamma,\gamma'}) = F_{\beta,\beta',\gamma'}(\theta^-_{\beta,\beta'}\otimes \mathbf{x})$ as desired.

    Now we consider the case when $\alpha$ and $\alpha'$ are isotopic. Then they are related by a sequence of finger moves which either create or annihilate a canceling pair of intersection points. Hence, we may further assume that $\alpha$ and $\alpha'$ intersect transversely at two points in the universal cover of the (unpunctured) torus. This would then imply that any two of the immersed curves $\beta,\beta',\gamma,\gamma'$ satisfy the same conditions. In this case, we can define the intersection points $\theta^\pm_{\delta,\epsilon}\in CF(\delta,\epsilon)$ for any $\delta,\epsilon\in \{\beta,\beta',\gamma,\gamma'\}$ with $\delta\neq \epsilon$; since they are pulled tight, we know that there are exactly two rigid holomorphic disks between $\delta$ and $\epsilon$ that {do not cross} the basepoint (which are the obvious ones from $\theta^+_{\delta,\epsilon}$ to ${\theta^-_{\delta,\epsilon}}${)}. Now, by using similar argument (i.e. degenerating a triangle to a disk) as in the previous case, we see that for any $\mathbf{x}\in CF(\beta,\gamma)$, we have
    \[
    F_{\beta,\gamma,\gamma'} (\mathbf{x}\otimes \theta^-_{\gamma,\gamma'}) = \begin{cases}
        \theta^-_{\beta,\gamma'} \quad\text{if}\quad \mathbf{x}=\theta^+_{\beta,\gamma};\\
        0 &\text{else}.
    \end{cases}
    \]
    Similarly, we have
    \[
    F_{\beta,\beta',\gamma'} (\theta^-_{\beta,\beta'}\otimes \mathbf{x}) = \begin{cases}
        \theta^-_{\beta,\gamma'} \quad\text{if}\quad \mathbf{x}=\theta^+_{\beta',\gamma'};\\
        0 &\text{else}.
    \end{cases}
    \]
    Hence, we deduce that $F_{\beta,\gamma,\gamma'}(\mathbf{x}\otimes \theta^-_{\gamma,\gamma'}) = F_{\beta,\beta',\gamma'}(\theta^-_{\beta,\beta'}\otimes \mathbf{x})$ in this case as well. The lemma follows.
\end{proof}

\begin{figure}
    \centering
    \includegraphics[width=0.75\linewidth]{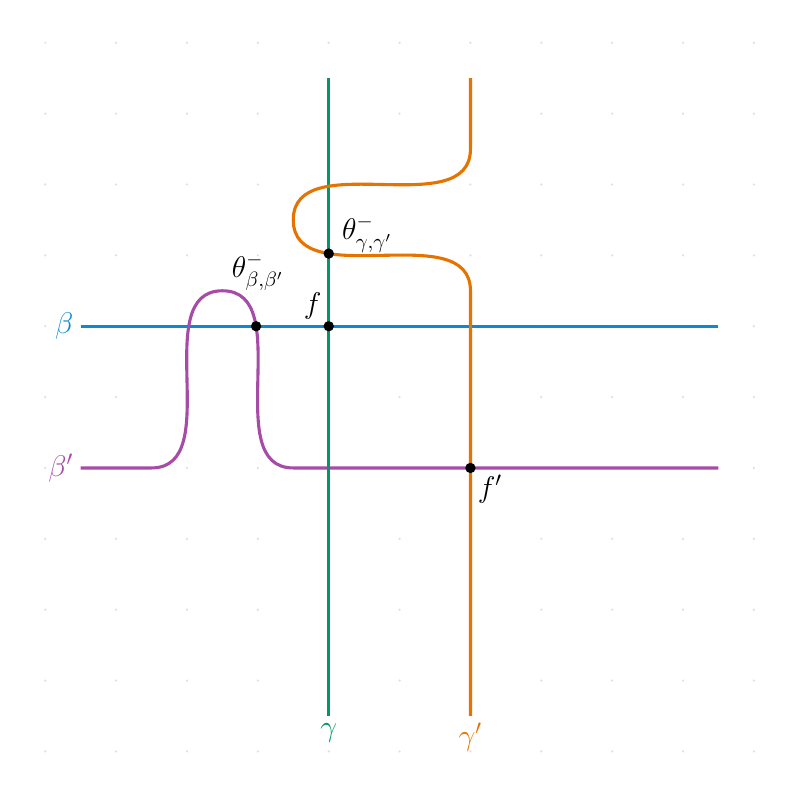}
    \caption{The configuration of curves considered in \Cref{lem:theta-minus-central}. }
    \label{fig:theta_minus-well-defined}
\end{figure}

Using \Cref{lem:theta-minus-central}, we can associate a distinguished homology class $\theta^-_{\gamma,\gamma'}\in HF(\gamma,\gamma)$ for any immersed curve (possibly with an irreducible local system) $\gamma$ on a punctured torus, or equivalently, a distinguished (homotopy class of an) endomorphism $\theta^-_M\in H_\ast \mathrm{End}(M)$ for any extendable type D structure $M$ over the torus algebra, as follows.

\begin{defn}
    Given any pulled-tight immersed curve (possibly with an irreducible local system) $\gamma$ on a punctured torus $T_\ast$, choose its Hamiltonian translate $\gamma'$ that intersects $\gamma$ transversely so that their lifts to the universal cover $\tilde{T}_\ast$ of $T_\ast$ intersect at exactly two points, say $\tilde\theta^\pm _{\gamma,\gamma'}$. The signs are chosen so that the corresponding cycles $\theta^\pm_\gamma\in CF(\gamma,\gamma')$ given {by} projecting $\tilde\theta^\pm_{\gamma,\gamma'}$ to $T_\ast$ satisfy $\mathrm{gr}(\theta^+_{\gamma,\gamma'},\theta^-_{\gamma,\gamma'})=1$. By \Cref{lem:theta-minus-central}, we see that the homology classes of $\theta^\pm_{\gamma,\gamma'}$ {do} not depend on the chosen representative of $\gamma$ in its homotopy equivalence class and the choice of a Hamiltonian translate $\gamma'$, and thus give well-defined classes in $HF(\gamma,\gamma)$; we denote them again as $\theta^\pm_{\gamma}:=\theta^\pm_{\gamma,\gamma'}$ and call them {the} canonical positive (negative) endomorphism {classes} of $\gamma$.
\end{defn}

The following lemma is then an immediate corollary of \Cref{lem:theta-minus-central}, which we record separately as it will be used later.

\begin{cor} \label{cor: theta minus vanishing lemma}
    Let $\alpha,\beta$ be immersed curves that are not isotopic to each other. Then the map
    \[
    -\circ \theta^-_{\alpha',\alpha}:HF(\alpha,\beta)\rightarrow HF(\alpha,\beta)
    \]
    is zero. Similarly, the map
    \[
    \theta^-_{\beta,\beta'}\circ -:HF(\alpha,\beta)\rightarrow HF(\alpha,\beta)
    \]
    is also zero.
\end{cor}
\begin{proof}
    This follows directly from the proof of \Cref{lem:theta-minus-central}.
\end{proof}

Using the canonical negative endomorphism classes of immersed curves, we can define a distinguished endomorphism class of a knot as follows.

\begin{defn}
    Given any knot $K$, its immersed curve invariant $\gamma_K$ has a unique component $\gamma^0_K$ (with trivial local system) whose homology class in the (unpunctured) torus is nonzero; see the discussion preceding \cite[Theorem 1]{hanselman_watson_cabling}.\footnote{This corresponds to the unique \emph{snake complex} summand of $CFK_\mathcal{R}(S^3,K)$ in the sense of \cite{popovic2023link}.} We define the canonical negative endomorphism class of $K$ as the homotopy class
    \[
    \Delta^-_K := \theta_{\gamma^0_K}^- \in HF(\gamma^0_K,\gamma^0_K) \subset HF(\gamma_K,\gamma_K) = H_\ast \mathrm{End}(\widehat{CFD}(S^3 \smallsetminus K)).
    \]
\end{defn}

\begin{rem} \label{rem: theta minus is central}
    It follows from \Cref{lem:theta-minus-central} that, for any immersed curve $\gamma$, possibly with an irreducible local system, its canonical negative endomorphism class $\theta^-_{\gamma,\gamma'}\in HF(\gamma,\gamma)$ is \emph{central}, i.e. for any $f\in HF(\gamma,\gamma)$, we have $\theta^-_{\gamma,\gamma'} \circ f = f\circ \theta^-_{\gamma,\gamma'}$ under the composition map
    \[
    \circ:HF(\gamma,\gamma)\otimes HF(\gamma,\gamma)\rightarrow HF(\gamma,\gamma).
    \]
\end{rem}

\begin{lem} \label{lem: Delta class is nonzero}
    The homology class $\Delta^-_K \in \HFh(S^3_0(K\# \overline{K}),{[0])}$ is nonzero.
\end{lem}
\begin{proof}
    Since $\Delta^-_K$ depends only on the homotopy equivalence class of $\gamma^0_K$, we see that $\gamma^0_K$ has been pulled tight, so that the only bigons between $\gamma^0_K$ and $\gamma^0_{K'}$ are the two bigons from the intersection point $\bm p_{\id}$ representing the identity to the intersection point $\Delta^-_K$. Their contributions cancel each other in the coefficient of $\Delta^-_K$ in $\partial \bm p_{\id}$ inside the chain complex $\widehat{CF}(\gamma^0_K,\gamma^0_{K'})$. It follows that $\Delta^-_K$ is not a boundary, i.e. it defines a nonzero homology class.
\end{proof}

The goal of this subsection is to prove the following theorem.

\begin{thm} \label{thm: general computation of theta minus class}
    Given any knot $K$, the (homotopy class of) endomorphism $\theta_K^-$ of $\widehat{CFD}(S^3 \smallsetminus K)$ is given by $\Delta^-_K$ and thus is combinatorially computable from any finitely generated free representative of $CFK_\mathcal{R}(S^3,K)$.
\end{thm}

In order to prove \Cref{thm: general computation of theta minus class}, we need some algebraic lemmas. We make use of the following notation:
\begin{itemize}
    \item $\mathrm{Mor}^\mathcal{A}_0(S^3 \smallsetminus K_0,S^3 \smallsetminus K_0)$ denotes the space of homotopy classes of grading-preserving type D endomorphisms of $\widehat{CFD}(S^3 \smallsetminus K_0)$;
    \item For any integer $n>0$, $\mathrm{Mor}^\mathcal{A}_0(S^3 \smallsetminus K_0,S^3 \smallsetminus K_n)$ denotes the space of homotopy classes of homogeneous type D morphisms from $\widehat{CFD}(S^3 \smallsetminus K_0)$ to $\widehat{CFD}(S^3\smallsetminus K_n)$ that correspond to elements of $\widehat{HF}(S^3 _n (K\# {\overline{K}}),[0])$ under the morphism pairing map
    \[
    \mathrm{Mor}^\mathcal{A}((S^3 \smallsetminus K)_0,(\KC)_n)\simeq \widehat{HF}(-(S^3 \smallsetminus K)_0\cup (S^3 \smallsetminus K)_n) = \widehat{HF}(S^3 _n (K\# {\overline{K}})).
    \]
\end{itemize}

Given a knot $K$, let $W_N$ be the change of framing cobordism $S^3_0(K) \# L(n,1) \ra S^3_n(K)$. As in \Cref{sec: the map Lambda}, we can compute this map in terms of the change of framing cobordism $X_n: (\KC) \#L(n,1) \ra (\KC)$, by tensoring with the identity map for the type A structure for the solid torus. By fixing the generator of $\HFh(L(n,1))$ in the canonical $\SpinC$-structure, we obtain the map in $\Mor^\cA(\KC_0, \KC_n)$, which was denoted $\frak f_{0,N;[0]}^K$ {in} \Cref{sec: the map Lambda}. Consider the map 
\begin{align*}
    \Mor^{\cA}(\KC_0, \KC_0) \ra \Mor^{\cA}(\KC_0, \KC_n),\quad  f \mapsto \frak f_{0,N;[0]}^K\circ f.
\end{align*}
This is {the} map associated to the change of framing cobordism from $S_0^3(K\#\overline K)$ to $S_n^3(K\#\overline K)$, and its kernel is precisely the kernel of $\Lambda_K$.

\begin{lem} \label{lem: type D representative of theta minus}
    Consider the immersed curve component $\gamma^{2\tau(K)}_K$ for the $2\tau(K)$-framed knot complement $\widehat{CFD}(S^3 \smallsetminus K_{2\tau(K)})$, which consists of stable chains and a single horizontal unstable chain, which we denote by $\xi_0 \xrightarrow{\rho_{12}} \eta_0$; note that we may have $\xi_0 = \eta_0$.\footnote{It is straightforward to see that this happens exactly when $\tau(K)=0$ and $CFK_\mathcal{R}(S^3,K)$ is locally trivial, i.e. has $CFK_\mathcal{R}(S^3,U)$ as a summand.} Then the canonical negative endomorphism class 
    \[
    \Delta^-_K \in H_\ast \mathrm{End}(\widehat{CFD}(S^3 \smallsetminus K)) \simeq H_\ast \mathrm{End}(\widehat{CFD}(S^3 \smallsetminus K_{2\tau(K)}))
    \]
    is represented by the type D endomorphism
    \[
    \Delta^-_K(\xi_0) = \rho_{12}\eta_0; \quad \Delta^-_K(\mathbf{x})=0\quad \text{for any other generator}\quad \mathbf{x}.
    \]
\end{lem}
\begin{proof}
    Let $Y_{2\tau(K)}$ be the  mapping cylinder associated to the $2\tau(K)$-framed Dehn twist (which is topologically a copy of $T^2 \times I$). Consider the isomorphism 
    \[
    f\mapsto \mathrm{id}_{\widehat{CFDA}(Y_{2\tau(K)})} \boxtimes f: H_\ast\mathrm{End}(\widehat{CFD}(S^3 \smallsetminus K_0)) \xrightarrow{\simeq} H_\ast \mathrm{End}(\widehat{CFD}(S^3 \smallsetminus K_{2\tau(K)})),
    \]
    which, interpreted in terms of immersed curves, gives an isomorphism
    \[
    \HF(\gamma^0_K,\gamma^0_K) \xrightarrow{\simeq} \HF(\gamma^{2\tau(K)}_K,\gamma^{2\tau(K)}_K),
    \]
    which is given by pushing forward elements of $\HF(\gamma^0_K,\gamma^0_K)$ under the Dehn twist. By \Cref{lem:theta-minus-central}, the image of $\Delta^-_K = \theta^-_{\gamma^0_K}$ under this identification is the canonical negative endomorphism class $\theta^-_{\gamma^{2\tau(K)}_K}$ of $\gamma^{2\tau(K)}_K$. Hence it suffices to compute $\theta^-_{\gamma^{2\tau(K)}_K}$ combinatorially; this can be done by counting rigid triangles between $\gamma^{2\tau(K)}_K$, its small Hamiltonian translate, and the meridian and longitude containing the basepoint of the torus (on which the immersed curves lie), with $\theta^-_{\gamma^{2\tau(K)}_K}$ as one of its corners, as explained in the discussions following the proof of \cite[Theorem 2]{cohen2025immersedcurves4manifoldinvariants}.\footnote{See \cite[Example 3.5]{cohen2025immersedcurves4manifoldinvariants} for an instructive example of how this computation is done.} Clearly, the only relevant triangle is the small triangle with corners $\xi_0$, $\theta^-_{\gamma^{2\tau(K)}_K}$, and $\eta_0$, which has a unique holomorphic representative by {the} Riemann mapping theorem. The lemma follows.
\end{proof}

\begin{lem} \label{lem: theta lemma 1}
    For each $g\in \mathrm{Mor}_0^\mathcal{A}(S^3\smallsetminus K_0,S^3\smallsetminus K_n)$, consider the composition map
    \[
    g_*:\mathrm{Mor}_0^\mathcal{A}(S^3\smallsetminus K_0,S^3\smallsetminus K_0)\xrightarrow{f\mapsto g\circ f} \mathrm{Mor}_0^\mathcal{A}(S^3\smallsetminus K_0,S^3\smallsetminus K_n).
    \]
    Then, whenever the integer $n>0$ is sufficiently large, we have
    \[
    \Delta^-_K \in \ker (g_*).
    \]
\end{lem}
\begin{proof}
    Fix $g \in \Mor^\cA_0(\KC_0, \KC_n)$. The post-composition with $F_{X_n}$ induces a surjection on the homotopy classes of $\Mor^\cA_0(\KC_0, \KC_0) \ra \Mor^\cA_0(\KC_0, \KC_n)$. Hence, $g$ is homotopic to $F_{X_n}\circ h$ for some ${h} \in \Mor^\cA_0(\KC_0, \KC_0)$. But, as $\Delta^-_K$ is central, we have 
    \begin{align*}
        g \circ \Delta^-_K \sim F_{X_n}\circ h \circ \Delta^-_K \sim F_{X_n}\circ \Delta^-_K\circ h  \sim 0 
    \end{align*}
    by \Cref{lem:theta-minus-central,cor: theta minus vanishing lemma}. Hence, for any $g$, we have $\Delta^-_K \in \ker(g_*)$, as desired.
\end{proof}

Note that we can see from the proof of \Cref{lem: theta lemma 1}  why we had to use only the $\gamma^0_K$ component in the definition of $\Delta^-_K$. Since other components consist only of stable chains, changing the framing of the knot complement does not change their isotopy class on the punctured torus. Hence, \Cref{cor: theta minus vanishing lemma} cannot be applied if we choose the canonical negative endomorphism classes of such components.

\begin{lem} \label{lem: theta lemma 2}
    Let $n>0$ be a sufficiently large integer. If $f\in \mathrm{Mor}_0^\mathcal{A}(S^3\smallsetminus K_0,S^3\smallsetminus K_0)$ is homotopically nonzero and the composition map
    \[
    f_*:\mathrm{Mor}_0^\mathcal{A}(S^3\smallsetminus K_0,S^3\smallsetminus K_n)\xrightarrow{g\mapsto g\circ f} \mathrm{Mor}_0^\mathcal{A}(S^3\smallsetminus K_0,S^3\smallsetminus K_n)
    \]
    is the zero map for all $g \in \Mor_0^\cA(\KC_0, \KC_n)$, then $f = \theta_K^-$.
\end{lem}
\begin{proof}
    Suppose that $f$ is nontrivial and $g \circ f \equiv 0$ for all $g \in \Mor^\cA_0(\KC_0, \KC_n)$. In particular, $F_{X_n}\circ f \equiv 0$, i.e. $f \in \ker(F_{X_n}\circ -)$. Since the map $\Lambda$ factors through the map $F_{X_n}\circ -$ and $\ker \Lambda$ is 1-dimensional and spanned by $\theta^-_K$ as mentioned in \Cref{lem: large surgery from 0 surgery is injective}, we see that $\ker(F_{X_n}\circ -)$ is either zero or spanned by $\theta^-_K$. Since we already know that $f\in \ker(F_{X_n}\circ -)$, we know that $\ker(F_{X_n}\circ -)$ is nonzero and thus is spanned by $\theta^-_K$, i.e. $\theta^-_K$ is its only nonzero element. Therefore $f=\theta^-_K$, as desired.
\end{proof}

We are finally {ready} to prove \Cref{thm: general computation of theta minus class}.

\begin{proof}[Proof of \Cref{thm: general computation of theta minus class}]
    The canonical negative endomorphism class $\Delta^-_K$ is {combinatorially} computable by \Cref{lem: type D representative of theta minus}. Hence, by \Cref{lem: theta lemma 2}, it suffices to prove that $\Delta^-_K\neq 0$ and $g\circ \Delta^-_K=0$ for all elements  $g\in \Mor^\cA_0(\KC_0, \KC_n)$. The former follows from \Cref{lem: Delta class is nonzero} and the latter follows from \Cref{lem: theta lemma 1}. The theorem follows.
\end{proof}

As an immediate consequence, we obtain the centrality of $\theta^-_K$, which was asserted in \Cref{sec:intro}. Note that centrality was established for $\Delta^-_K$ in \Cref{rem: theta minus is central}, independently of \Cref{thm: general computation of theta minus class}, so no circularity is involved.

\begin{cor} \label{cor: theta minus K is central}
    For any knot $K$, the class $\theta^-_K \in H_\ast\End^\cA(\CFDh(\KC))$ is central.
\end{cor}
\begin{proof}
    By \Cref{thm: general computation of theta minus class} we have $\theta^-_K = \Delta^-_K$, which is central by \Cref{rem: theta minus is central}.
\end{proof}

\section{Proof of the main theorems}\label{sec:Doubling}
We restate \Cref{thm: main1} below and provide its proof.

\begin{thm}\label{thm:main thm}
    Suppose $f \in \End_\cR(\CFK_\cR(S^3, K))$ is a locally symmetric endomorphism of Alexander degree shift $A$. Then, there is the following homotopy commutative diagram in the category of $\cR$-complexes:
    \begin{align*}
        \begin{tikzcd}[ampersand replacement = \&]
            \CFK_\cR(S^3, K, s) \ar[d,"f"] \ar[r,"\simeq"] \& \CFK_\cR(S^3, K, s) \ar[d,"\Lambda(\Omega(f))"] \\
            \CFK_\cR(S^3, K, s+A)\ar[r,"\simeq"]  \& \CFK_\cR(S^3, K, s+A),
        \end{tikzcd}
    \end{align*}
    for each $s \in \Z$. In particular, $\Lambda \circ \Omega(f)$ is conjugate to $f$, and thus the endomorphism $\Lambda \circ \Omega$ of $\mathrm{End}_\mathcal{R}^{ls}(\CFK_\cR(S^3,K))$ is an isomorphism. It follows that the map
    \[
    \Lambda: \mathrm{End}_h(\widehat{CFD}(S^3 \smallsetminus K))\rightarrow  \mathrm{End}_\mathcal{R}^{ls}(\CFK_\cR(S^3,K))
    \]
    is surjective. It has 1-dimensional kernel, which is generated by the morphism $\theta^-_K$ in \Cref{lem: large surgery from 0 surgery is injective}.
\end{thm}
\begin{proof}
    Fix a model $C_K$ for $\CFK_\cR(S^3, K)$. Throughout the proof, we resrtict to the summands in Alexander grandigs $s$ and $s+A$, though to simplify the exposition, we will drop the gradings from the notation. According to \Cref{sec: End actions}, we have a diagram
    \begin{align*}
        \begin{tikzcd}[ampersand replacement = \&]
            C_K \ar[d,"f"]\ar[r," \frak P\circ\iota_{C_K}"]
            \& \tT^+_K \ar[d,"\bI \boxtimes \Omega(f)"]
            \& \Tel\Hyp(\tT^+_K) \ar[l,"p_{\tT^+_K}" above]\ar[r,"\iota"] \ar[d,"\bI \boxtimes \Omega(f)"]
            \&\Tel\Hyp(\cT^+_K) \ar[d,"\bI \boxtimes \Omega(f)"]
            \\
            C_K \ar[r,"\frak P\circ \iota_{C_K}" below]
            \& \tT_K^+
            \& \Tel\Hyp(\tT^+_K) \ar[l,"p_{\tT^+_K}"]\ar[r,"\iota" below]
            \&\Tel\Hyp(\cT^+_K).
        \end{tikzcd}
    \end{align*}
    The first and third arrows are $\F[U,V]$ homotopy equivalences, while the second is a quasi-isomorphism. The first and second squares are strictly commutative, while the third is homotopy commutative. Additionally, in \Cref{sec: the map Lambda}, we produced a homotopy commutative diagram:
    \begin{align*}
        \begin{tikzcd}[ampersand replacement = \&]
            \Tel\Hyp(\cT_K^+) \ar[d,"\bI \boxtimes \Omega(f)"]\ar[r,"\simeq"]
            \& C_K \ar[d,"\Lambda(\Omega(f))"] \\
            \Tel\Hyp(\cT_K^+)\ar[r,"\simeq"]
            \& C_K,
        \end{tikzcd}
    \end{align*}
    coming from the pairing theorem. Together, these two diagrams give us a zigzag of quasi-isomorphisms from $C_K$ to itself compatible with the actions of $\End_\cR(C_K)$.

    According to \Cref{sec:tel_hyp}, there is a functor from the derived category of bigraded $\F[U,V]$-complexes, $\cD^\infty \mathrm{Mod}_{\F[U,V]}$ to the dg-category of bigraded dg $\F[U,V]$-modules, $\cD^\infty_{\F[U,V]}$. We will abuse notation and continue to write $C_K$, $\tT^+_K$, etc. for their induced dg modules. Furthermore, by \Cref{lem:invertibility for fg free}, since $C_K$ is finitely generated and free, the inclusion
    \begin{align*}
        [C_K, \Tel\Hyp(\cT^+_K)]_{\mathrm{dgMod}^h_{\F[U,V]}}\ra [C_K, \Tel\Hyp(\cT^+_K)]_{\cD^\infty_{\F[U,V]}}
    \end{align*}
    is a bijection. Therefore, we can invert the second arrow in the diagram above to obtain an honest endomorphism of $C_K$.

    We now apply the derived functor $(-)\otimes^L_{\F[U,V]}\cR$. The complex $\cD^\infty_{\F[U,V]}(C_K) \otimes_{\F[U,V]}^L \cR$ has a simple model: choose some free $\F[U, V]$-complex $\Tilde{C}_K$ with the property that $\Tilde{C}_K\otimes_{\F[U, V]}\cR \cong C_K$; since
    \begin{align*}
        \widetilde{C}_K \xra{UV} \widetilde{C}_K \ra C_K,
    \end{align*}
    is a free resolution of $C_K$, we have that $\cD^\infty_{\F[U,V]}(C_K) \otimes_{\F[U,V]}^L \cR$ is isomorphic to the complex 
    \begin{align*}
        (C_K \xra{0} C_K) \cong C_K \otimes_\F \cV.
    \end{align*}
    By applying \Cref{lem: main appendix lem} to the finitely generated module $C_K$, the same also holds for the dg modules $\tT^+_K$ and $\Tel\Hyp(\tT^+_K)$, and therefore in the derived category, we have a diagram 
    \begin{align*}
        \begin{tikzcd}[ampersand replacement = \&]
            \CFK_\cR(S^3,K) \otimes_{\F} \cV \ar[r,"\simeq"] \ar[d,"f \otimes_{\F} \id_\cV"]\& \CFK_\cR(S^3,K) \otimes_{\F} \cV\ar[d,"\Lambda (\Omega(f))\otimes_{\F} \id_\cV"]\\
            \CFK_\cR(S^3,K) \otimes_{\F} \cV \ar[r,"\simeq"] \&\CFK_\cR(S^3,K) \otimes_{\F} \cV.
        \end{tikzcd}
    \end{align*}
    Finally, according to the proof of \Cref{lem:splitting V off of LambdaOmega}, $\Lambda \Omega$ is summand preserving, so we can apply \Cref{lem:double_conjugation} once more to obtain the result. The rest of the theorem then follows from \Cref{lem: large surgery from 0 surgery is injective}.
\end{proof}

\subsection{Concordances and complements} 

Let $C: K_0 \ra K_1$ be a concordance, and let $X(C)$ be its complement, which we view as a cobordism with corners from $\KC_0$ to $\KC_1$. Recall from \cite{guthkang2024invariantsplittingprinciples} that we have a ring homomorphism
\[
\Lambda_{K_0,K_1}:H_\ast \mathrm{Mor}(\CFDh(\KC_0),\CFDh(\KC_1))\rightarrow H_\ast \mathrm{Mor}(\CFK_\cR(S^3,K_0),\CFK_\cR(S^3,K_1)).
\]
As in the case $K_0=K_1$, it is straightforward to observe using arguments in \Cref{subsec: absolute maslov grading} that the domain and the codomain of $\Lambda_{K_0,K_1}$ are bigraded in a way that makes $\Lambda_{K_0,K_1}$ bidegree-preserving. Furthermore, the arguments in \Cref{sec: End actions} can also be adapted to this setting to obtain a bidegree-preserving homomorphism
\[
\Omega_{K_0,K_1}:H_\ast \mathrm{Mor}(\CFK_\cR(S^3,K_0),\CFK_\cR(S^3,K_1)) \rightarrow H_\ast \mathrm{Mor}(\CFDh(\KC_0),\CFDh(\KC_1));
\]
then the proofs of \Cref{thm:main thm,thm: general computation of theta minus class} can be immediately generalized to show that $\Lambda_{K_0,K_1}\Omega_{K_0,K_1}=\mathrm{id}$ and $\ker \Lambda_{K_0,K_1}$ is generated by the type D morphism $\theta^-_{K_0,K_1}$, which we describe below. Observe that, since $K_0$ and $K_1$ are concordant, their immersed curve invariants share the same ``unstable'' summand, which we denote by $\gamma$; this corresponds to a type D structure $M_\gamma$ under the Lipshitz--Ozsv\'{a}th--Thurston correspondence that appears as a direct summand of both $\CFDh(\KC_0)$ and $\CFDh(\KC_1)$. Then $\theta^-_{K_0,K_1}$ is defined as the following composition:
\[
\theta^-_{K_0,K_1}:\CFDh(\KC_0)\xrightarrow{\text{projection}}M_\gamma \xrightarrow{\theta^-_\gamma}M_\gamma \xrightarrow{\text{inclusion}} \CFDh(\KC_1).
\]
This composition is well-defined up to homotopy by \Cref{lem:theta-minus-central}.

\begin{rem}
    When $K_0=K_1$, the maps $\Lambda_{K_0,K_1}$ and $\Omega_{K_0,K_1}$ coincide with our maps $\Lambda$ and $\Omega$ (for the knot $K_0$). See \cite{guthkang2024invariantsplittingprinciples} for the definition and various properties of $\Lambda_{K_0,K_1}$.
\end{rem}

In \cite{guth_one_not_enough_exotic_surfaces}, the first author constructed a map 
\begin{align*}
    F_{X(C), \frak h}: \CFDh(\KC_0) \ra \CFDh(\KC_1),
\end{align*}
associated to the complement of concordance $C: K_0 \ra K_1$ with a handle decomposition $\frak h$. Such a map can be paired with the identity map for the type A structure for any Heegaard diagram $\cH$ representing a satellite pattern $P$ in the solid torus to compute the map induced by the concordance $P(C): P(K_0) \ra P(K_1)$. In particular, we may take $\bX$ to be a triply-pointed bordered Heegaard diagram representing $(S^1 \times D^2, S^1 \times \{0\}, p)$, where $p$ is a free basepoint. More precisely, we have the following.

\begin{prop}\cite[Theorem 2]{guth_one_not_enough_exotic_surfaces}
    Let $C: K_0 \ra K_1$ be a concordance. Let $p_i$ be a point in $(S^3, K_i)$ disjoint from $K_i$. Let $\gamma_p$ be a path in $S^3 \times [0,1]$ disjoint from $C$ connecting $p_0$ and $p_1$. Then, for any choice of handle decomposition $\frak h$ for $X(C)$, there is a map $F_{X(C), \frak h}: \CFDh(\KC_0) \ra \CFDh(\KC_1)$ with the property that the diagram 
    \begin{align*}
        \begin{tikzcd}[ampersand replacement = \&]
            \CFAh(\bX)\boxtimes \CFDh(\KC_0) \ar[d,"\bI_{\CFAh(\bX)}\boxtimes F_{X(C),\mathfrak{h}}"] \ar[r,"\simeq"] \&
            \CFK_\cR(S^3, K_0, p) \ar[d,"F_{C,\gamma_p}"]\\
            \CFAh(\bX)\boxtimes \CFDh(\KC_1)  \ar[r,"\simeq"] \&
            \CFK_\cR(S^3, K_1, p),
        \end{tikzcd}
    \end{align*}
    is homotopy commutative; here the horizontal arrows are given by the pairing theorem and $F_{C, \gamma_p}$ is the map induced by the concordance $(C, \gamma_p)$.
\end{prop}

Next, we show that for any choice of $F_{X(C), \frak h}$, the map $F_C$ is homotopic to $\Lambda_{K_0, K_1}(F_{X(C), \frak h})$.

\begin{prop} \label{prop: Lambda and concordance maps}
    Let $C: K_0 \ra K_1$ be a concordance. Fix a handle decomposition $\frak h$ for $X(C)$ and let $F_{X(C), \frak h}$ be the associated map. Then, $\Lambda_{K_0, K_1}(F_{X(C), \frak h}) \sim F_{C}$.
\end{prop}
\begin{proof}
    By \cite{guthkang2024invariantsplittingprinciples}, there is a homotopy commutative diagram
    \begin{align*}
        \begin{tikzcd}[ampersand replacement = \&]
            \CFAh(\bX)\boxtimes \CFDh(\KC_0) \ar[d,"\bI_{\CFAh(\bX)}\boxtimes F_{X(C), \frak h}"] \ar[r,"\simeq"] \&
            \CFK_\cR(S^3, K_0, p) \ar[d,"F_{C,\gamma_p}"] \ar[r,"\simeq"] 
            \& \CFK_\cR(S^3, K_0) \otimes_\F \cV \ar[d,"F_C \otimes\id_\cV"]
            \\
            \CFAh(\bX)\boxtimes \CFDh(\KC_1)  \ar[r,"\simeq"] \&
            \CFK_\cR(S^3, K_1, p) \ar[r,"\simeq"] 
            \& \CFK_\cR(S^3, K_1) \otimes_\F \cV.
        \end{tikzcd}
    \end{align*}
    The map $\Lambda_{K_0, K_1}(F_{X(C), \frak h})$ is characterized by the property that $\bI_{\CFAh(\bX)} \boxtimes F_{X(C), \frak h}$ is homotopic to $\Lambda_{K_0, K_1}(F_{X(C), \frak h})\otimes \id_\cV$. Hence, by the diagram above, $\Lambda_{K_0, K_1}(F_{X(C), \frak h})$ is homotopic to $F_C$.
\end{proof}

Note that the invariance of (the homotopy class of) $F_{X(C)}$ under choices of handle decompositions of $X(C)$ was not shown in \cite{guth_one_not_enough_exotic_surfaces}; \Cref{prop: Lambda and concordance maps} can thus also be considered as the proof of its invariance.

\begin{cor}
    Let $\frak h_0$ and $\frak h_1$ be two handle decompositions of $X(C)$ and let $F_{X(C), \frak h_i}$ be the associated maps. Then, $F_{X(C), \frak h_0}\sim F_{X(C), \frak h_1}$.
\end{cor}
\begin{proof}
    According to the naturality of bordered Floer homology \cite{guthkang2024invariantsplittingprinciples}, the morphisms $F_{X(C), \frak h_0}$ and $ F_{X(C), \frak h_1}$ are well defined elements of $\Mor^\cA(\KC_0, \KC_1)$. Consider $\Lambda_{K_0, K_1}(F_{X(C), \frak h_0})$ and $\Lambda_{K_0,K_1}(F_{X(C), \frak h_1})$. By the previous proposition, $\Lambda_{K_0, K_1}(F_{X(C), \frak h_0})$ and $\Lambda_{K_0, K_1}(F_{X(C), \frak h_1})$ are both homotopic to $F_C$. But, $\Mor^\cA(\KC_0, \KC_1)$ splits as $\langle\theta^-_{K_0,K_1}\rangle\oplus M$, and both $F_{X(C), \frak h_0}$ and $F_{X(C), \frak h_1}$ are contained in $M$. Furthermore, the restriction $\Lambda|_{M}$ is injective. Therefore, we have that $F_{X(C), \frak h_0}$ and $F_{X(C), \frak h_1}$ are homotopic. 
\end{proof}

\begin{prop}\label{prop: omega and concordance maps}
    Let $C: K_0 \ra K_1$ be a decorated concordance and let $F_{C}$ and $F_{X(C)}$ be the induced maps on $\CFK_\cR$ and $\CFDh$. Then, $\Omega_{K_0,K_1}(F_{C}) \sim F_{X(C)}$.
\end{prop}
\begin{proof}
    By \Cref{thm:main thm}, we have that $\Lambda_{K_0,K_1}\Omega_{K_0,K_1}(F_{C}) \sim F_C$. According to \Cref{prop: Lambda and concordance maps}, we have $F_{C} \sim \Lambda_{K_0,K_1}(F_{X(C)})$. It follows then that $F_{X(C)} +\Omega_{K_0,K_1}(F_{C})$ is homotopic to $\lambda \cdot \theta_{K_0,K_1}^-$ for $\lambda \in \F$. Since $F_{X(C)}$ and $\Omega_{K_0,K_1}(F_{C})$ are degree preserving, their difference must also be degree preserving. But, $\theta_{K_0,K_1}^-$ has the same degree as the differential. Hence, $F_{X(C)}$ and $\Omega_{K_0,K_1}(F_{C})$ are homotopic.
\end{proof}

We now prove \Cref{thm: main2}, whose statement we recall for the reader's convenience.

\begin{thm}\label{thm:satellite-concordances}
    Let $C:K_0 \ra K_1$ be a concordance and let $P \sub S^1 \times D^2$ be a satellite pattern.  Then, the induced map $F_{P(C)}: \CFK_\cR(S^3, P(K_0)) \ra \CFK_\cR(S^3,P(K_1))$ is determined by the map $F_C:\CFK_\cR(S^3, K_0) \ra \CFK_\cR(S^3, K_1)$ and $\CFDAh((S^1\times D^2) \smallsetminus P)$ up to precompositions and postcompositions of homotopy autoequivalences of its domain and codomain.
\end{thm}
\begin{proof}
    According to \Cref{prop: omega and concordance maps}, $\Omega_{K_0,K_1}(F_C) \sim F_{X(C)}$. This gives rise to a map
    \begin{align*}
        \bI_{\CFDAh((S^1\times D^2) \smallsetminus P)}\boxtimes F_{X(C)} \in \Mor^\cA(S^3\smallsetminus P(K_0), S^3\smallsetminus P(K_1)).
    \end{align*}
    Under the identification given by the pairing theorem, we have that $\bI_{\CFDAh((S^1\times D^2) \smallsetminus P)}\boxtimes F_{X(C)} \sim F_{X(P(C))}$. Applying \Cref{prop: Lambda and concordance maps}, we have that $\Lambda_{P(K_0),P(K_1)}(F_{X(P(C))}) \sim F_{P(C)}$. This proves the theorem.
\end{proof}

\bibliographystyle{alpha}
\bibliography{ref}

\end{document}